\documentclass[11pt,
a4paper,
BCOR=5mm,
DIV=25,
headinclude=true,
footinclude=false,
abstract=on]
{scrartcl}

\usepackage[utf8]{inputenc}
\usepackage[english]{babel}
\usepackage{amsmath}
\usepackage{lipsum}
\usepackage{amssymb}
\usepackage{amsthm}
\usepackage{bbm}
\usepackage{mathtools}
\usepackage{enumitem}
\usepackage{graphicx}
\usepackage{natbib}
\usepackage{subcaption}
\usepackage{xcolor}
\usepackage{hyperref}

\setlist{
partopsep=0pt,
topsep=0pt,
itemsep=0pt,
}

\setenumerate[1]{label={(\roman*)}}
\setenumerate[2]{label={(\alph*)}}

\theoremstyle{definition}
\newtheorem{definition}{Definition}[section]
\newtheorem{remark}[definition]{Remark}
\newtheorem{example}[definition]{Example}
\newtheorem{assumption}{Assumption}
\newtheorem*{assumption*}{Assumption}

\theoremstyle{plain}
\newtheorem{theorem}[definition]{Theorem}
\newtheorem{proposition}[definition]{Proposition}
\newtheorem{lemma}[definition]{Lemma}
\newtheorem*{lemma*}{Lemma}
\newtheorem{corollary}[definition]{Corollary}

\newcommand\norm[1]{\left\lVert#1\right\rVert}

\newcommand{\NN}{\mathbb{N}}
\newcommand{\ZZ}{\mathbb{Z}}

\newcommand{\RR}{\mathbb{R}}

\newcommand{\ind}{\mathbbm{1}}

\newcommand{\PP}{\mathbf{P}}
\newcommand{\EE}{\mathbf{E}}
\newcommand{\Var}{\mathrm{Var}}
\newcommand{\Cov}{\mathrm{Cov}}

\title{Central limit theory for serial tail dependence estimators in heavy-tailed long memory linear time series}

\author{Ioan Scheffel$^{a}$, Marco Oesting$^{a,b}$, Gilles Stupfler$^{c}$}
\date{$^{a}$ {\small Institute for Stochastics and Applications, University of Stuttgart, D-70563 Stuttgart, Germany} \\[1ex]
$^{b}$ {\small Stuttgart Center for Simulation Science (SC SimTech), University of Stuttgart, D-70569 Stuttgart, Germany} \\[1ex]
$^{c}$ {\small Univ Angers, CNRS, LAREMA, SFR MATHSTIC, F-49000 Angers, France}}

\begin{document}

\maketitle

\begin{abstract}
  We prove multiple central limit theorems for serial tail dependence estimators in heavy-tailed long memory linear time series. The main theoretical tools are two novel multivariate reduction principles for partial sums of heavy-tailed long memory linear time series, subordinated over sliding windows and above a threshold growing with sample size. This requires addressing several substantial difficulties, including handling a nonlinear, sample-size dependent, and multivariate subordination mechanism, the dependence between several overlapping linear processes, and the lack of higher-order moments of the marginal distribution. Despite these obstacles, our assumptions are mild and, in particular, the innovation process is allowed to have infinite variance. A key feature of our theory is that our second reduction principle holds uniformly in the threshold, allowing central limit theory for empirical extremograms with sample quantiles as thresholds.
  This question has received little attention in the literature on serial extremal dependence estimation even though the version of empirical extremograms with random thresholds is ubiquitous in practice. We compare our results in several respects with those that may be obtained under short-range dependence, thereby discovering markedly different convergence rates and limit laws in our long memory setting.
\end{abstract}

\section{Introduction}
\label{sec:intro} 

Traditionally, dependence in time series is characterized by the autocorrelation function (ACF).
Since the ACF emphasizes dependence at central levels, it does not accurately capture dependence at extreme levels.
Measuring and estimating tail dependence is therefore a separate task.
The analog of the ACF in extreme value analysis is the so-called extremogram, proposed by~\cite{davisExtremogramCorrelogramExtreme2009}.
Given a stationary time series $(X_t)_{t\in\mathbb{Z}}$ whose marginal distribution has upper endpoint $+\infty$, the extremogram is defined as the limiting quantity
\begin{align*}
	\gamma_{A,B}(h)
	\ := \ \lim_{u\to\infty}
	\frac{
		\PP\left[
			[X_{\ell}]_\ell \in u \cdot A
			\,,
			[X_{\ell+h}]_\ell \in u \cdot B
			\right]}
	{\PP[X_0>u]}
	\,,\qquad h \in \mathbb{Z}\,,  A,B\subset \RR^T\,, T\in\NN
	\,,
\end{align*}
provided that the limit exists, where throughout
$[X_{\ell}]_{\ell}:=[X_0,\ldots, X_{T-1}]$.
The simplest example is for $A=B=(1,\infty)$, when the extremogram coincides with the tail-dependence coefficient
\begin{align*}
	\chi(h)
	\ := \ \lim_{u\to\infty}
	\PP\left[
		X_h > u
		\mid
		X_0 > u
		\right]
	\,,\qquad h \in \mathbb{Z}\,.
\end{align*}
With this quantity, one can answer the question
whether an extreme value at time $0$ makes an extreme value at time $h$ more likely.
Even though the probability of rare events vanishes, that is, $\PP[X_0>u]\to 0$ as $u\to \infty$, the tail dependence coefficient $\chi(h)$ can be positive.
We then call the stationary time series asymptotically dependent;
see Chapter~9.5 in~\cite{beirlantStatisticsExtremesTheory2006}
for a detailed discussion of modeling asymptotic dependence using tail dependence coefficients.
In general, the extremogram is estimated 
from observations of $X_1,\ldots,X_{n}$ by the natural estimator
\begin{align*}
  &
	\widehat{\gamma}_{A,B}(h)
  \\&
	\ := \
	\dfrac{
    \frac{1}{n}
    \sum_{t=1-(0\land h)}^{n-T+1-(0\lor h)}
		\ind\left\{
		[X_{t+\ell}]_\ell \in u_n \cdot A
		\,,
		[X_{t+\ell+h}]_\ell \in u_n \cdot B
		\right\}
	}
	{
		\frac{1}{n}
		\sum_{t=1}^n
		\ind\{X_t>u_n\}
	}
	\,,\qquad h \in \mathbb{Z}\,,  A,B\subset \RR^T\,, T\in\NN
	\,,
\end{align*}
where the threshold sequence $(u_n)$ is chosen such that $u_n\to\infty$ and $n\PP[X_0>u_n]\to\infty$. We call this the empirical extremogram.
The asymptotic analysis of the empirical extremogram requires dedicated techniques that differ from the usual ACF analysis; see~\cite{brodav1991} for a discussion of asymptotic properties of the empirical ACF in standard time series settings.
The asymptotic behavior of the empirical extremogram has been studied in~\cite{davisExtremogramCorrelogramExtreme2009} for sufficiently strong $\alpha$-mixing regularly varying time series satisfying an anti-clustering condition;
see~\cite{basrakRegularlyVaryingMultivariate2009} for a reference on regularly varying time series.
The question of handling tail dependence estimation
when mixing conditions of this type do not hold
has remained much less explored.
This is particularly true for long memory settings.

There is evidence, for example, that in hydrological time series, the assumptions mentioned above (implying short memory) are violated.
In contrast, fractionally differenced ARIMA models have been used for hydrological time series \citep{montanariFractionallyDifferencedARIMA1997}, long memory in river flows has been linked to prolonged droughts and temporal clustering of extreme floods \citep{mudelseeLongMemoryRivers2007}, and flood peak distributions often exhibit heavy-tail behavior~\citep{merzUnderstandingHeavyTails2022}.
There is still debate over whether long memory is present in financial data; see for example~\cite{rossiLongMemoryTail2013} or~\cite{contLongrange2005}.

In this work, we study tail dependence estimators for long memory linear time series with possibly infinite variance.
More precisely, the main contribution of our article is to derive central limit theorems for tail dependence estimators in long memory heavy-tailed linear time series.
For this, the available theory does not provide the required tools, since it is essentially based on one-dimensional reduction principles~\citep{scheffelCentralLimitTheory2025} and does not cover the multivariate nature of dependence estimators.
The empirical process theory of~\cite{kulikTailEmpiricalProcess2011a} is developed exclusively for stochastic volatility models and is therefore not applicable for the current class of processes.
We close this gap by developing multivariate reduction principles for threshold-dependent indicator functions.
This is not a routine extension of~\cite{scheffelCentralLimitTheory2025}, because the multivariate setting creates dependence between several overlapping linear processes, and the bounds needed to handle this dependence have to be developed separately. Classical tools used to control empirical ACFs, as in~\cite{hosking_asymptotic_1996}, cannot be used either, because the extreme value context involves estimators whose summands depend on the sample size~$n$. Accounting for long-range dependence and a subordination mechanism depending on sample size in a nonlinear way makes the arguments substantially different from those of~\cite{betmicsch2026}, in which short-range dependence is considered and the subordination mechanism varies with $n$ through scaling constants.

We prove a general reduction principle that allows us to derive central limit theory
for a large class of estimators that contains in particular all extremogram estimators considered in~\cite{davisExtremogramCorrelogramExtreme2009}.
We then restrict the class of admissible sets to obtain a uniform version of this reduction principle
that
allows us to consider extremogram estimators with random thresholds, that is,
\begin{align*}
	\widetilde{\gamma}_{A,B}(h)
	\ := \
	\frac{1}{k}
  \sum_{t=1-(0\land h)}^{n-T+1-(0\lor h)}
	\ind\left\{
	[X_{t+\ell}]_\ell \in X_{n-k:n} \cdot A
	\,,
	[X_{t+\ell+h}]_\ell \in X_{n-k:n} \cdot B
	\right\}
	\,,\quad h \in \mathbb{Z}\,,  A,B\subset \RR^T\,, T\in\NN
	\,,
\end{align*}
where $X_{1:n}\leq X_{2:n} \leq \cdots \leq X_{n:n}$ are the order statistics of the sample $X_1,\ldots,X_n$. 
This feature of our approach is uncommon in theoretical investigations: for example, the estimation theory in~\cite{davisExtremogramCorrelogramExtreme2009} covers only deterministic thresholds. However, the version with random thresholds $\widetilde{\gamma}_{A,B}(h)$ of the empirical extremogram is used in practice~\cite[Section~3.4]{davisExtremogramCorrelogramExtreme2009}. As far as we are aware, only~\cite{dre2015} establishes the asymptotic behavior of the empirical extremogram with random thresholds in a weakly dependent setting and under several technical assumptions. See Corollary 3.6 therein, in which it appears that the weak limits of $\widehat{\gamma}_{A,B}(h)$ and $\widetilde{\gamma}_{A,B}(h)$ are identical when short-range dependence is present. In general, replacing deterministic with random thresholds in asymptotic theory is known to be difficult, especially in the presence of serial dependence, see~\cite{drekne2020}.

We also show in this work that, by contrast with the short-range dependence setting, the weak limits of $\widehat{\gamma}_{A,B}(h)$ and $\widetilde{\gamma}_{A,B}(h)$ are different for $A$ and $B$ in a wide class of sets relevant in applications. In particular, we show that the estimators
\[
	\frac{1}{k}
	\sum_{t=1}^{n-T+1}
	\ind\left\{
	X_t > X_{n-k:n}, X_{t+1} > X_{n-k:n}, \ldots, X_{t+T-1} > X_{n-k:n}
	\right\}
\]
and
\[
	\frac{1}{k}
	\sum_{t=1}^{n-T+1}
	\ind\left\{
	X_t > X_{n-k:n}, X_{t+T-1} > X_{n-k:n}
	\right\},
\]
of the tail-dependence coefficients
\[
	\lim_{u\to\infty}
	\PP\left[
		X_0 > u,\ldots,X_{T-1}>u \, | \, X_0>u \right]
\]
and
\[
	\lim_{u\to\infty}
	\PP\left[
		X_{T-1}>u \, | \, X_0>u \right] \ =\  \chi(T-1),
\]
respectively, have a degenerate asymptotic distribution which is a hitherto unknown phenomenon within the long memory setting.

This paper is organized as follows. In Section~\ref{sec:model} we introduce our model and main assumptions. In Section~\ref{sec:reduction_principle} we derive reduction principles for tail dependence estimators that we use to derive the central limit theory of Section~\ref{sec:clt}.
In Section~\ref{sec:rect} 
we show that our theory for random thresholds is applicable in the class of rectangular sets by providing simplified assumptions and showing their validity.
Then we 
investigate the cancellation in the asymptotic scale for rectangular sets. 
Since our theory is also valid in the univariate setting, we finally compare results for random thresholds with results that are readily derived from existing theory in the i.i.d. and weakly dependent setting.
\section{Model and main assumptions}
\label{sec:model}
\subsection{Notation}
\label{sec:model:notation}

Throughout this article, $(\Omega,\mathcal{A}, \PP)$ denotes a probability space rich enough to support all random variables under consideration. The symbols $\PP^{*}$ and $\EE^{*}$ denote outer probability and expectation, respectively~\cite[Section~1.2]{vandervaartWeakConvergenceEmpirical2023}. For $p\ge 1$, we write $L^p(\PP)$ for the space of random variables with finite $p$-th moment.
For a generic, continuous random variable $Z$, we let $F_Z$ (resp.~$f_Z$, $q_Z$) denote its cumulative distribution function (resp.~probability density function, quantile function).
We write $a\land b = \min\{a,b\}$ and $a\lor b = \max\{a,b\}$ for $a,b\in\mathbb{R}$, and use the symbol $\sim$ to denote asymptotic equivalence of sequences and functions.
We write $\lesssim$ to mean inequality up to a generic multiplicative constant $C>0$ that can change from place to place. If the dependence of this constant on the surrounding variables is relevant to the argument, we will make this explicit, by saying, for example, that $C=C_n$ depends on $n$, or that $C$ is independent of $n$.
Throughout we will write $\mathbf{x}_{}=[x_\ell]_\ell=[x_\ell]_{\ell\in\{0,\ldots,T-1\}}\in\RR^T$
where $T\ge 1$ is fixed. If the index set is different from $\{0,\ldots,T-1\}$, we will make this explicit.
For $p\ge 0$, the symbol $C^p(\RR^T)$ denotes the vector space of real-valued functions on $\RR^T$ with continuous partial derivatives up to order $p$. 
For a set $A\subset \RR^T$ and $s\in \RR\cup\{+\infty\}$ we write $s\cdot A = \{s\cdot x\mid x\in A\}$, and set by convention $\infty\cdot A = \emptyset$.
Finally, we write $\NN:=\{1,2,\ldots\}$ and $\NN_0:=\NN\cup \{0\}$. 
\subsection{Assumptions on the model}
\label{sec:model:assumptions}

Let $(\varepsilon_t)_{t\in \ZZ}$ be a sequence of independent and identically distributed (i.i.d.)~copies of a random variable $\varepsilon$. In this work, we focus on causal linear (or MA($\infty$)) time series $X_t=\sum_{j=0}^{\infty} a_j \varepsilon_{t-j}$, where $(a_j)_{j\in \mathbb{Z}}$ 
is a sequence of real-valued coefficients such that $a_0=1$ and $a_{-j}=0$ for all $j\ge 1$. Our regularity assumptions focus on the distribution of the innovation sequence, on the rate of decay of the coefficients $a_j$ which should be slow enough to yield a divergent series, yet fast enough to accommodate tail heaviness of the innovations, and on the tail of its marginal distribution. 
We consider the case of heavy-tailed time series whose innovations have a finite first moment, but possibly infinite second moment. 
\begin{assumption}[Regularity of the innovations]
	\label{asu:f}
  Suppose that $F_\varepsilon\in C^2(\RR)$, and that $\varepsilon$ has a symmetric distribution, and that there is a finite positive constant $L$ and $\nu \in (1,\infty)\setminus\{2\}$ such that
	$\lim_{x\to\infty} x^{\nu} \PP[\varepsilon > x] = \lim_{x\to\infty} x^{\nu} \PP[\varepsilon < -x] = L/2$, and suppose that the probability density function $f_\varepsilon \in C^1(\mathbb{R})$ satisfies:
	\begin{align*}
		f_{\varepsilon}(x)
		\lor
		|f'_{\varepsilon}(x)|
		 &
		\ \lesssim \
	\frac{1}{(1+|x|)^{\alpha}}
		\qquad\text{for}\ x\in\RR\,,
		\\
		|
		f'_\varepsilon(x)
		-
		f'_\varepsilon(y)
		|
		 &
		\ \lesssim \
		|x-y|
		\cdot
	\frac{1}{(1+|x|)^{\alpha}}
		\qquad\text{for}\ x,y\in\RR\,,\,|x-y|<1\,,
	\end{align*}
	where $\alpha := 2\land \nu$, and $\lesssim$ denotes inequality up to a multiplicative constant depending only on the distribution of $\varepsilon$.
\end{assumption}

\begin{assumption}[Long memory]
	\label{asu:coef}
	$(n^{1-d} a_n)$ converges to a finite, nonzero limit $c_a$
	for some $d\in (0,1-1/\alpha)$.
\end{assumption}
\begin{assumption}[Regular variation of the marginal density]
	\label{asu:rv_f}
	The density $f_{X_0}$ is regularly varying at $+\infty$ with index $-1-\nu$.
\end{assumption}
\begin{remark}[Assumptions~\ref{asu:f},~\ref{asu:coef} and~\ref{asu:rv_f}]
	These assumptions are already used in~\cite{scheffelCentralLimitTheory2025}.
 They are reasonably weak, given the difficulty of the heavy-tailed setting.
 The edge case $\nu=2$ is excluded to streamline the analysis.  
\end{remark}
\begin{remark}[Link between Assumption~\ref{asu:rv_f} and unimodality] \label{rem:unimodality}
	If the distribution of $\varepsilon$ is unimodal, meaning that its distribution function is convex on $(-\infty,0)$ and concave on $(0,\infty)$~\cite[Definition~1.1]{DharmadhikariJoagDev1988}, then
  by Theorem~1.1 in~\cite{DharmadhikariJoagDev1988} it holds that any 
  convergent moving average also has a unimodal distribution function. On the level of densities, this means monotonicity on $(-\infty,0)$ and $(0,\infty)$. Then, from the regular variation of the moving average~\cite[(15.3.5)]{kulikHeavyTailedTimeSeries2020} and the monotone density theorem~\cite[Proposition B.1.9.11]{dehaanExtremeValueTheory2006}, we get that the density of the moving average is also regularly varying with index $-1-\nu$. In particular, Assumption~\ref{asu:rv_f} is satisfied under Assumption~\ref{asu:f} and unimodality of the innovations.
\end{remark}

\subsection{Assumptions on the multivariate subordination mechanism}
\label{sec:model:subordination}

We want to derive central limit theory for partial sums of long memory linear time series subordinated over sliding windows, that is, for
\begin{align*}
  \sum_{t=1}^{n-T+1} (G_n([X_{t+\ell}]_{\ell}) - \EE[G_n([X_{\ell}]_\ell)])
\end{align*}
where $(G_n:\RR^T\to \RR)$ is a sequence of subordination mechanisms.
The following definition specifies the class of such mechanisms considered in this work.
\begin{definition}[Subordination mechanism]
	\label{asu:G_n}
	Let $(u_n)\subset[1,\infty)$ be a threshold sequence with $u_n \to\infty$ as $n\to\infty$, let
			$A\subset \RR^T$ be a Borel set, and let $s_1,s_2 \in (0,\infty]$ satisfy $s_1\le s_2$.
	We define
	\begin{align*}
		G_{n}(\mathbf{z},s_1,s_2, A)
		\ := \
		\ind\left\{\mathbf{z} \ \in\  u_n\cdot((s_1\cdot A )\setminus (s_2\cdot A))\right\}
		\,,
		\qquad
		\mathbf{z}\in \RR^T
		\,.
	\end{align*}
\end{definition}
	Putting $s_1=1$ and $s_2=\infty$ reflects extremogram-type tail-dependence estimators as in~\cite{davisExtremogramCorrelogramExtreme2009}, where 
   we require the set $A\subset \RR^T$ to be bounded away from 0.
   Apart from that,
	we are interested in the asymptotic behavior of the (rescaled) empirical extremogram 
	\[
		\sum_{t=1}^{n-T+1}
		\ind\{[X_{t+\ell}]_{\ell}\in X_{n-k:n} \cdot A\}.
	\]
	Write
	\[
		\sum_{t=1}^{n-T+1}
		\ind\{[X_{t+\ell}]_{\ell}\in X_{n-k:n} \cdot A\} = \sum_{t=1}^{n-T+1}
		\ind\{[X_{t+\ell}]_{\ell}\in u_n \cdot (X_{n-k:n}/u_n) \cdot A\},
	\]
  where the deterministic threshold $u_n$ is replaced by the sample quantile at level $1-k/n$. This suggests, 
under the assumption that $X_{n-k:n}/u_n\to_{\PP}1$
as $n\to\infty$, 
  to investigate the asymptotic behavior of
	\[
		\sum_{t=1}^{n-T+1}
		\ind\{[X_{t+\ell}]_{\ell}\in u_n \cdot s \cdot A\}
	\]
	uniformly in $s$ in a neighborhood of 1. 
    Our argument relies on a chaining procedure.
  We first analyze the pointwise asymptotic behavior of the relevant quantity and then show that
  for neighboring values $s$ and $s'$ around 1, the difference 
	\[
		\sum_{t=1}^{n-T+1}
		\ind\{[X_{t+\ell}]_{\ell}\in u_n \cdot s \cdot A\} 
        \ -\  \sum_{t=1}^{n-T+1}
		\ind\{[X_{t+\ell}]_{\ell}\in u_n \cdot s' \cdot A\}
	\]
  converges to $0$; see Section~\ref{sec:reduction_principle:chaining} for further details. 
    For sets $A$ satisfying $s_2 \cdot A \subset s_1\cdot A$ if $s_1<s_2$, this takes the form
    \begin{align*}
		\sum_{t=1}^{n-T+1}
        G_n([X_{t+\ell}]_\ell, s_1,s_2,A)
    \end{align*}
    which justifies Definition~\ref{asu:G_n}. The next definition fixes the properties of $A$ that are needed for this extended analysis.
    Apart from the nesting property, we also need
that
differences of the form $(s_1\cdot A)\setminus (s_2\cdot A)$ are controlled by
marginal threshold exceedances.
\begin{definition}[Class of sets for the uniform reduction principle]
	Define $\mathcal{C}_T$ to be the family of all Borel sets $\emptyset \neq A\subset \RR^T$ bounded away from 0 such that there exist 
  $\{\mathfrak{a}_j\}_{j\in \{1,\ldots,T\}}\subset (0,\infty)$
  such that for all $s_1,s_2\in (0,\infty)$ satisfying $s_1<s_2$, it holds $s_2\cdot A \subset s_1 \cdot A$ and 
	\begin{align*}
		(s_1\cdot A) \setminus (s_2\cdot A) \ \subset\
		\left( \bigcup_{j=1}^T \{|x_{j-1}|\in \mathfrak{a}_j \cdot [s_1,s_2] \} \right)
		\,.
	\end{align*}
  If $A\in \mathcal{C}_T$ with a particular $\mathfrak{a}=\{\mathfrak{a}_j\}$, then we write $A\in \mathcal{C}_T(\mathfrak{a})$.
\end{definition}
\begin{example}[Examples, counterexamples and properties of sets in $\mathcal{C}_T$]
  \label{ex:ct_sets}
   \
	\begin{enumerate}
          \item{[Cartesian product]}
          Let $n,m\in \NN$. 
          If $A\in \mathcal{C}_n(\mathfrak{a})$ and $B\in \mathcal{C}_m(\mathfrak{b})$, then $A\times B \in \mathcal{C}_{n+m}(\mathfrak{c})$ with $\mathfrak{c}_j = \mathfrak{a}_j$ for $j\in\{1,\ldots,n\}$ and $\mathfrak{c}_{j}=\mathfrak{b}_{j-n}$ for $j\in \{n+1,\ldots,n+m\}$.
    \item{[Unboundedness]}. 
    Since $A$ is bounded away from 0, there exists $\mathbf{x}\in A$ such that
    \begin{align*}
    \max_{j\in\{1,\ldots,T\}}|x_{j-1}|\ >\ \delta
      \qquad\text{for some}\ \delta>0\,.
    \end{align*}
    Then, for $R>\delta$,
    \begin{align*}
      \frac{R}{\delta}
      \cdot \mathbf{x}
      \ 
      \in 
      \ 
      \frac{R}{\delta}
      \cdot A
      \ 
      \subset
      \ 
      A
      \qquad\text{and}\qquad
    \max_{j\in\{1,\ldots,T\}}
      \left|
      \frac{R}{\delta}
      x_{j-1}
      \right|
      \ > \ R
      \,.
    \end{align*}
    Therefore, $A$ is unbounded. 
		\item{[Rectangular sets with positive active coordinates]}
		      Let $A=\bigtimes_{j=1}^{T}(\mathfrak{a}_j,\infty)$ with $\mathfrak{a}_j\in (0,\infty)\cup \{-\infty\}$, and $\emptyset\neq J_A =\{j\mid \mathfrak{a}_j > 0\}$.
		      It holds for $s_1,s_2>0$ satisfying $s_1<s_2$
		      \begin{align*}
			      s_2\cdot A
			      \   =\
			      \bigcap_{j\in J_A}\{x_{j-1} \in (s_2\cdot \mathfrak{a}_j,\infty)\}
			      \ \subset\
			      \bigcap_{j\in J_A}\{x_{j-1} \in (s_1\cdot \mathfrak{a}_j,\infty)\}
			      \ = \
			      s_1\cdot A
			      \,,
		      \end{align*}
		      as well as
		      \begin{align*}
			      (s_{1}\cdot A) \setminus (s_2 \cdot A)
			       &
			      \ = \
			      \left( \bigcap_{j\in J_A}\{x_{j-1} \in (s_1\cdot \mathfrak{a}_j,\infty)\} \right)
			      \
			      \cap
			      \
			      \left( \bigcup_{j\in J_A}\{x_{j-1} \in (-\infty,s_2\cdot \mathfrak{a}_j]\} \right)
			      \\ &
			        \ \subset\
			        \bigcup_{j\in J_A}\{|x_{j-1}| \in \mathfrak{a}_j\cdot[s_1,s_2]\}
		      \end{align*}
            so that $A\in \mathcal{C}_T$ 
            with arbitrary $\mathfrak{a}_j\in (0,\infty)$ for $j\notin J_A$.
        \item{[Complements of squares centered at the origin]}
        Let $\| \cdot \|$ be 
        the supremum norm on $\RR^T$, $R>0$, and $A=\{ \| x \| > R \}$. Observe that $s\cdot A=\{ \| x \| > sR \}$ for any $s>0$. Pick $s_1,s_2>0$ with $s_1 < s_2$. Then $s_2\cdot A \subseteq s_1\cdot A$ and $(s_1\cdot A) \setminus (s_2 \cdot A)=\{ \| x \| \in (s_1 R, s_2 R] \}$. 
       We conclude that 
               \begin{align*}
          (s_1\cdot A) \setminus (s_2 \cdot A) &\ = \ \left( \bigcap_{j=1}^T \{ |x_{j-1}| \in [0, s_2 R] \} \right)
			      \
			      \cap
			      \
			      \left( \bigcup_{j=1}^T \{ |x_{j-1}| \in (s_1 R,\infty) \} \right) \\
			    &\ \subset \
			        \bigcup_{j=1}^T \{ |x_{j-1}| \in R\cdot[s_1,s_2]\}
        \end{align*}
        so that $A\in \mathcal{C}_T$ 
       with $\mathfrak{a}_j=R\in (0,\infty)$ for $j\in \{ 1,\ldots,T\}$.
		\item{[Stability under set operations with scale invariant sets]}
          Let $B$ be a scale invariant set, that is, for $s>0$, $s\cdot B=B$, and let $A\in \mathcal{C}_T(\mathfrak{a})$.
          Note that $s\cdot B^c= (s\cdot B)^c = B^c$ for $s>0$, that is, $B^c$ is also scale invariant. Therefore,  
		      \begin{align*}
			      s_2\cdot (A\cap B)
            &
			      \ = \
			      (s_2\cdot A)\cap B
			      \ \subset\
			      (s_1\cdot A)\cap B
			      \ = \
			      s_1\cdot (A\cap B)
            \,,
            \\
 s_2\cdot (A\cup B)
            &
			      \ = \
			      (s_2\cdot A)\cup B
			      \ \subset\
			      (s_1\cdot A)\cup B
			      \ = \
			      s_1\cdot (A\cup B)
            \,,
		      \end{align*}
          and
\begin{align*}
			      \left(
			      s_1\cdot (A\cap B)
			      \right)
			      \setminus
			      \left(
			      s_2\cdot (A\cap B)
			      \right)
			       &
			      			      \ = \
			      \left(
			      (s_1\cdot A)
			      \setminus (s_2\cdot A)
			      \right)
			      \cap B
			        \ \subset\
			        (s_1\cdot A)
			        \setminus (s_2\cdot A)
              \,,
              \\
 \left(
			      s_1\cdot (A\cup B)
			      \right)
			      \setminus
			      \left(
			      s_2\cdot (A\cup B)
			      \right)
			       &
			      			      \ = \
			      \left(
			      (s_1\cdot A)
			      \setminus (s_2\cdot A)
			      \right)
			      \cap B^c
			        \ \subset\
			        (s_1\cdot A)
			        \setminus (s_2\cdot A)
\,.
		      \end{align*}
          If, in addition, $A\cap B \neq \emptyset$, then $A\cap B \in \mathcal{C}_T(\mathfrak{a})$, and if $A\cup B$ is bounded away from $0$, then $A\cup B \in \mathcal{C}_T(\mathfrak{a})$. Generally, $\mathcal{C}_T$ is not closed under these set operations. 
	\end{enumerate}
\end{example}
\section{Multivariate Reduction Principle}
\label{sec:reduction_principle}
\subsection{A general pointwise reduction principle}
\label{sec:reduction_principle:ptwise}

We give a heuristic overview of how the multivariate reduction principle works.
For $k\in \mathbb{Z}$, let $\mathcal{F}_k=\sigma(\varepsilon_{k},\varepsilon_{k-1},\ldots)$ denote the past $\sigma$-algebra generated by the sequence $(\varepsilon_{t})$ up to index $k$. For $k\ge 0$, let $X_{t,k}:=\sum_{j=0}^k a_j\varepsilon_{t-j}$ be the $k$-truncated time series, 
let $A\subset \RR^T$ be bounded away from 0, $s_1,s_2\in (0,\infty]$ 
satisfying $s_1<s_2$, and write 
\begin{align*}
  G_{k,n}(\mathbf{y})	
  &
  \ = \ 
  G_{k,n}(\mathbf{y},s_1,s_2,A)
	\ = \
  \EE[G_n([X_{\ell,k+\ell}]_\ell+\mathbf{y},s_1,s_2,A)]
  \\&
  \ = \ 
  \PP[[X_{\ell,k+\ell}]_\ell+\mathbf{y} \in u_n \cdot (s_1\cdot A)\setminus (s_2\cdot A)]
  \,,
  \end{align*}
  and,
  by Lemma~\ref{lem:multi_swap} in the supplementary material,
  \begin{align}
    \label{eq:need}
    \nabla G_{k,n}(\mathbf{z})
    \ = \ 
    \nabla_{\mathbf{y}} G_{k,n}(\mathbf{y},s_1, s_2,A)\Big|_{\mathbf{y}=\mathbf{z}}
	\ = \
	-
	\int_{\RR^T}
    \ind\{\mathbf{x}\in u_n \cdot (s_1\cdot A)\setminus(s_2\cdot A)\}
	\nabla
	f_{[X_{\ell,k+\ell}]_\ell}
	(\mathbf{x}-\mathbf{z})
	\,\mathrm{d}\mathbf{x}
	\,,
\end{align}
where we set $G_{\infty,n}(\mathbf{z}):=\EE[G_n([X_{\ell,\infty}]_\ell + \mathbf{z})]=\EE[G_n([X_{\ell}]_\ell + \mathbf{z})]$ for $k=\infty$ in the natural way.
 The reduction principle consists of showing that
\begin{align}
	\label{eq:heuristic}
	\begin{split}
		 &
		\sum_{t=1}^{n-T+1}
		\left(
		G_n([X_{t+\ell}]_\ell)
		\ - \
		\EE[G_n([X_\ell]_\ell)]
		\right)
		\\ &
		  \ = \
		  \sum_{t=1}^{n-T+1}
		  \sum_{k=0}^{\infty}
		\left(
		  \EE
		  \left[
				  G_n([X_{t+\ell}]_\ell)
				  \mid \mathcal{F}_{t-k}
				\right]
		\ - \
		  \EE
		  \left[
				  G_n([X_{t+\ell}]_\ell)
				  \mid \mathcal{F}_{t-(k+1)}
				\right]
		\right)
		\\ &
		  \ = \
		  \sum_{t=1}^{n-T+1}
		  \sum_{k=0}^{\infty}
		\left(
      G_{k-1,n}([X_{t+\ell}-X_{t+\ell,k+\ell-1}]_{\ell})
		\ - \
      G_{k,n}([X_{t+\ell}-X_{t+\ell,k+\ell}]_{\ell})
		\right)
		\\ &
		  \ \approx\
		  \sum_{t=1}^{n-T+1}
		  \sum_{k=0}^{\infty}
		\left(
      G_{k,n}([X_{t+\ell}-X_{t+\ell,k+\ell-1}]_{\ell})
		\ - \
      G_{k,n}([X_{t+\ell}-X_{t+\ell,k+\ell}]_{\ell})
		\right)
		\\ &
		  \ \approx\
		  \sum_{t=1}^{n-T+1}
		  \sum_{k=0}^{\infty}
		\varepsilon_{t-k}
    \cdot
		  [a_{k+\ell}]_\ell^{\top}
      \cdot
		\nabla
    G_{k,n}([X_{t+\ell} -X_{t+\ell,k+\ell} ]_\ell)
		\\ &
		  \ = \
		\sum_{t=1}^{n-T+1}
      [X_{t+\ell}]_\ell^{\top}
      \cdot
      \nabla
      G_{\infty,n}(\mathbf{0}_T)
		\\ &
		  \qquad + \
		  \sum_{t=1}^{n-T+1}
		  \sum_{k=0}^{\infty}
		\varepsilon_{t-k}
    \cdot
		  [a_{k+\ell}]_\ell^{\top}
		\left(
		  \nabla
    G_{k,n}([X_{t+\ell} -X_{t+\ell,k+\ell} ]_\ell)
		\ - \
		  \nabla
    G_{\infty,n}(\mathbf{0}_T)
		\right)
		  \,.\end{split}
\end{align}
The first term is the leading term that will determine the asymptotic profile, described in the next theorem. 
\begin{theorem}[Central limit theorem for multivariate partial sums]
	\label{thm:linear_clt_multi}
	Let Assumptions~\ref{asu:f} and~\ref{asu:coef} hold.
	Then
	\begin{align*}
		n^{-d-1/\alpha}
		\sum_{t=1}^{n-T+1}
		[X_{t+\ell}]_\ell
		\ \to_d \
		Z_\alpha \cdot \mathbf{1}_{T} 
		\qquad\text{as}\ n\to\infty\,,\qquad
	\end{align*}
	where
	\begin{align*}
		Z_\alpha
		 &
		\ \
		\text{has distribution}
		\ \
		\begin{cases}
			\mathcal{N}(0,\sigma^2)
			 & \ \ \text{if}\ \ \alpha = 2 \ \text{with $\EE[\varepsilon^2]<\infty$,}
			\\
			S\alpha S(\eta)
			 & \ \ \text{if}\ \
			\alpha\in(1,2)\,,
		\end{cases}
	\end{align*}
	with
	variance
	\begin{align*}
		\sigma^2
		 & :=
		c_a^2\EE[\varepsilon^2]
		\frac
		{
			\Gamma(1-2d)\Gamma(d)}
		{d(2d+1)\Gamma(1-d)} \ \mbox{ when } \EE[\varepsilon^2]<\infty,
		\intertext{or scale
		}
		\eta
		 & :=
		\frac{c_a}{d}
		\cdot
		\left(
		L
		\cdot
		\frac
		{\Gamma(2-\alpha)\cos(\pi\alpha/2)}
		{1-\alpha}
		\right)^{1/\alpha}
		\left(
		\int_{-\infty}^{1}
		\left(
		(1-v)^{d}_+
		-
		(-v)^d_+
		\right)^{\alpha}
		\,\mathrm{d}v
		\right)^{1/\alpha}
		\,.
	\end{align*}
	Here, $S\alpha S(\eta)$ denotes the symmetric $\alpha$ stable distribution with characteristic function $x\mapsto \exp(-\eta^\alpha|x|^{\alpha})$.
\end{theorem}
Having determined the asymptotic profile of the linear approximation, 
we make the reduction principle in~\eqref{eq:heuristic} rigorous.
\begin{theorem}
	\label{lem:reduction_principle_optimal}
	Let Assumptions~\ref{asu:f} and~\ref{asu:coef} hold,
	and let either of the following conditions hold:
	\begin{enumerate}
		\item
		      $A$ is bounded away from 0, $s_1\in (\underline{s},\infty)$ for some $\underline{s}>0$, and $s_2 = \infty$
		\item
		      $A\in \mathcal{C}_T$ and, for some $0<\underline{s}<1<\overline{s}<\infty$, $s_1,s_2\in (\underline{s},\overline{s}]$ satisfy $s_1<s_2$
	\end{enumerate}
	Then for all $\delta>0$ and for $n\in \NN$,
	\begin{align*}
		 &
		\EE
		\left[
			\left|
      \sum_{t=1}^{n-T+1}
			G_n([X_{t+\ell}]_\ell, s_1,s_2,A)
			\ - \
			\EE[G_n([X_{t+\ell}]_\ell,s_1,s_2,A)]
			\ - \
			\left(
			\nabla_{\mathbf{y}} G_{\infty,n}(\mathbf{y},s_1,s_2,A)\Big|_{\mathbf{y}=\mathbf{0}_T}
			\right)^{\top}
			\cdot
			[X_{t+\ell}]_{\ell}
			\right|^{r_0}
			\right]
		\\ &
		  \ \lesssim \
		  ((s_2-s_1)\land 1)
		  \cdot
		  n^{
				  r_0(\kappa_0 +\delta/2)
				  }
		\,,
	\end{align*}
	where
	\begin{equation*}
		\label{eq:gamma0r0}
		r_0 \ := \ \left\{ \begin{array}{l}
			\alpha(1-d)
			\ \text{ if } \
			\dfrac{1}{(1-d)(1-2d)}
			< \alpha\,, \\[10pt]
			\dfrac{1}{2}
			\left(
			\alpha + \dfrac{1}{1-d}
			\right)
			\ \text{ otherwise,}
		\end{array} \right.
	\end{equation*}
	and
	\begin{align}
		\label{eq:kappa_0}
		\kappa_0
		\ =\
		\left\{ \begin{array}{l}
			        \dfrac{1}{\alpha(1-d)}
			        \ \text{ if } \
			        \dfrac{1}{(1-d)(1-2d)}
			        < \alpha\,, \\[10pt]
			        \dfrac{2(1-d) + (1-\alpha(1-d)(1-2d))}{\alpha(1-d)+1} = d
			        +
			        (1-d)
			        \dfrac{3-\alpha(1-d)}{\alpha(1-d)+1}
			        \ \text{ otherwise.}
		        \end{array} \right.
	\end{align}
\end{theorem}
As explained in Remark 3.4 in \cite{scheffelCentralLimitTheory2025},
$\kappa_0-d-1/\alpha<0$, so that there is a choice of $u_n\to\infty$ such that $
	u_n=o(n^{(d+1/(\nu \land 2)-\kappa_0-\delta)/(\nu+1)})
$ for $\delta>0$ small enough. 
Therefore,
\begin{align}
	\label{eq:cot2}
	\left(
	n^{d+1/\alpha}
	f_{X_0}(u_n)
	\right)^{-1}
	\cdot
	n^{\kappa_0 + \delta/2}
	\ = \
	\frac{n^{\kappa_0+\delta-d-1/\alpha}}{f_{X_0}(u_n)}
  \cdot
  n^{-\delta/2}
	\ \to \
	0\,.
\end{align}
By Theorem~\ref{lem:reduction_principle_optimal} applied with $s_1=s$ and $s_2=\infty$ and a set $A$ that is bounded away from $0$, this proves the following result.
\begin{corollary}
	\label{thm:main_simple}
	Let Assumptions~\ref{asu:f},~\ref{asu:coef}, and~\ref{asu:rv_f} hold, and assume that the threshold sequence $u_n\to\infty$ is such that
	$
		u_n=o(n^{(d+1/(\nu \land 2)-\kappa_0-\delta)/(\nu+1)})
	$
	for some $\delta>0$.
	Then, for all sets $A$ bounded away from $0$, and all $s\in (0,\infty)$ 
	\begin{align*}
    &
    \Bigg|
		\frac{
			n^{-d-1/\alpha}
		}{f_{X_0}(u_n)}
		\sum_{t=1}^{n-T+1}
    \Bigg(
		\ind\{[X_{t+\ell}]_{\ell}\in u_n \cdot s \cdot A\}
		\ - \
		\PP\left[[X_{t+\ell}]_{\ell}\in u_n \cdot s \cdot A\right]
    \\&\qquad
    \qquad
    \qquad
    \qquad
    \qquad
    \qquad
		\ - \
		\left(
		\nabla_{\mathbf{y}}G_{\infty,n}(\mathbf{y},s,\infty, A)\left.\right|_{\mathbf{y}=\mathbf{0}_T}
		\right)^{\top}
			[X_{t+\ell}]_{\ell}
    \Bigg)
    \Bigg|
	\end{align*}
	vanishes in probability
	as $n\to\infty$.
\end{corollary}
\subsection{Uniform reduction principle}
\label{sec:reduction_principle:chaining}

As explained in Section~\ref{sec:model:subordination}, we aim for a reduction principle that is uniform in $s$ and therefore allows $u_n$ to be replaced by a random threshold. Before we state the result, we give an overview of how it is derived by combining
Theorem~\ref{lem:reduction_principle_optimal} under Condition~(ii) therein, and a
chaining technique as in \cite[Proof of Theorem~2.1]{koulAsymptoticsEmpiricalProcesses2001}.
We use, for $k\in\NN_0$ and $j\in\{0,\ldots,2^k\}$, 
a suitable partition $(\pi_{j,k})$
satisfying
\begin{align*}
	\underline{s}
	\
	=
	\
	\pi_{2^k,k}
	\ < \
	\pi_{2^{k}-1,k}
	\ < \
	\cdots
	\ < \
	\pi_{1,k}
	\ < \
	\pi_{0,k}
	\ = \
	\overline{s}
	\,,
\end{align*}
and define, for $k\in\NN_0$ and $s\in (\underline{s},\overline{s}]$, the unique index $j^s_k \in\{0,\ldots,2^k-1\}$ 
satisfying
\begin{align*}
	\pi_{j^s_k+1,k}
	\ < \
	s
	\ \le \
	\pi_{j^s_k,k}
	\,.
\end{align*}
Choosing $K=K_n$, this gives a chain 
\begin{align*}
	\pi_{j^s_{K}+1,K}
	\ < \
	s
	\ \le \
	\pi_{j^s_{K},K}
	\ \le \
	\cdots
	\ \le \
	\pi_{j^s_{1},1}
	\ \le \
	\pi_{j^s_{0},0}
	\ = \
	\overline{s}
	\,,
\end{align*}
linking $\overline{s}$ to $s$, where we tighten the chain by choosing $K_n\to\infty$ at a suitable rate.
Write
\begin{align*}
  &
	\mathcal{S}_n(s_1,s_2,A)
	\ = \
	\sum_{t=1}^{n-T+1}
  \Bigg(
	G_n([X_{t+\ell}]_{\ell},s_1,s_2,A)
	\ - \
	\EE[
		G_n([X_{t+\ell}]_{\ell},s_1,s_2,A)
	]
  \\
  &
  \qquad
  \qquad
  \qquad
  \qquad
  \qquad
  \qquad
  \qquad
  \qquad
  - \ 
	  \left(
	  \nabla_{\mathbf{y}} G_{\infty,n}(\mathbf{y},s_{1},s_2,A)
	  \Big|_{\mathbf{y}=\mathbf{0}_T}
	\right)
	  ^{\top}
	\cdot [X_{t+\ell}]_{\ell}
  \Bigg)
  \,,
\end{align*}
and decompose
\begin{align*}
	\mathcal{S}_n(s,\infty,A)
	 &
	\ = \
	\mathcal{S}_n(s,\overline{s},A)
	\ + \
	\mathcal{S}_n(\overline{s},\infty,A)
  \\&
	\ = \
	\sum_{k=1}^{K}
	\mathcal{S}_n(\pi_{j^s_{k},k},\pi_{j^s_{k-1},k-1},A)
	\ + \
	\mathcal{S}_n(s,\pi_{j^s_{K},K},A)
	\ + \
	\mathcal{S}_n(\overline{s},\infty,A)
	\,,
\end{align*}
so that
\begin{align*}
	 &
	\sup_{s\in (\underline{s},\overline{s})}
	\left|
	\mathcal{S}_n(s,\infty,A)
	\right|
	\\ &
	  \ \le \
	  \sum_{k=1}^{K}
	\sup_{s\in (\underline{s},\overline{s})}
	\left|
	  \mathcal{S}_n(\pi_{j^s_{k},k},\pi_{j^s_{k-1},k-1})
	  \right|
	  \ + \
	  \sup_{s\in (\underline{s},\overline{s})}
	\left|
	  \mathcal{S}_n(s,\pi_{j^s_{K},K})
	  \right|
	  \ + \
	  |\mathcal{S}_n(\overline{s},\infty,A)|
	\\ &
	  \ \approx\
	  \sum_{k=1}^{K}
	\sum_{0\le i \le 2^{k}-1}
	\left|
	  \mathcal{S}_n(
	  \pi_{i+1,k},
	  \pi_{i,k},A
	  )
	  \right|
	  \ + \
	  |\mathcal{S}_n(\overline{s},\infty,A)|
	  \,.
\end{align*}
By Theorem~\ref{lem:reduction_principle_optimal},
\begin{align*}
  &
	\sum_{k=1}^{K}
	\sum_{0\le i \le 2^{k}-1}
	\EE
	\left[
		\left|
		\mathcal{S}_n(
		\pi_{i+1,k},
		\pi_{i,k},A
		)
		\right|^{r_0}
		\right]
	\ \lesssim \
	n^{r_0(\kappa_0 + \delta/2)}
	\sum_{k=1}^{K}
	\sum_{0\le i \le 2^{k}-1}
	\left(
	\pi_{i,k}
	\ - \
	\pi_{i+1,k}
	\right)
  \\&
	\ = \
	n^{r_0(\kappa_0 + \delta/2)}
	\cdot K_n
	\cdot (\overline{s}-\underline{s})
	\,,
\end{align*}
and
$
    \EE[|\mathcal{S}_n(\overline{s},\infty,A)|^{r_0}]\lesssim n^{r_0(\kappa_0+\delta/2)}
$.
Choosing $K_n\to\infty$ appropriately, it follows from \eqref{eq:cot2} that
\begin{align*}
  \left(
  \frac{
	n^{-d-1/\alpha}
  }{f_{X_0}(u_n)}
	\cdot
	n^{\kappa_0 + \delta/2}
  \right)^{r_0}
  \cdot
  K_n
	\cdot (\overline{s}-\underline{s})
	\ \to \ 0\,,
\end{align*}
so that
\begin{align*}
  \frac{
	n^{-d-1/\alpha}
  }{f_{X_0}(u_n)}
	\sup_{s\in (\underline{s},\overline{s})}
	\left|
	\mathcal{S}_n(s,\infty,A)
	\right|
	\ \to \ 0
	\,,
\end{align*}
in outer probability. This leads to the following uniform reduction principle.
\begin{theorem}
	\label{thm:main}
	Let Assumptions~\ref{asu:f},~\ref{asu:coef}, and~\ref{asu:rv_f} hold and assume that the threshold sequence $u_n\to\infty$ is such that
	$
		u_n=o(n^{(d+1/(\nu \land 2)-\kappa_0-\delta)/(\nu+1)})
	$
	for some $\delta>0$.
	Then, for all $A\in \mathcal{C}_T$, and all $\underline{s}\in (0,1)$ and $\overline{s}\in (1,\infty)$,
	\begin{align*}
    &
		\sup_{s\in (\underline{s},\overline{s})}
    \Bigg|
		\frac{
			n^{-d-1/\alpha}
		}{f_{X_0}(u_n)}
		\sum_{t=1}^{n-T+1}
		\Bigg(
		\ind\{[X_{t+\ell}]_{\ell}\in u_n \cdot s \cdot A\}
		\ - \
		\PP\left[[X_{t+\ell}]_{\ell}\in u_n \cdot s \cdot A\right]
    \\&
    \qquad
    \qquad - \
		\left(
		\nabla_{\mathbf{y}}G_{\infty,n}(\mathbf{y},s,\infty, A)\Big|_{\mathbf{y}=\mathbf{0}_T}
		\right)^{\top}
			[X_{t+\ell}]_{\ell}
		\Bigg)
    \Bigg|
	\end{align*}
	vanishes in outer probability
	as $n\to\infty$.
\end{theorem}
\begin{remark}[Additional assumptions about sets in Theorem~\ref{thm:main}]
We already see in Corollary~3.6 in~\cite{dre2015} that additional assumptions on the class of sets are needed to accomplish the transition from deterministic to random thresholds. We introduce these assumptions in a transparent way, showing in Example~\ref{ex:ct_sets} that the restricted class $\mathcal{C}_T$ is sufficiently large from a statistical standpoint. 
\end{remark}
In the next section, we derive central limit theory from the results developed so far.
\section{Central limit theory for extremogram-type tail-dependence estimators}
\label{sec:clt}
We now apply our reduction principles to establish central limit theorems.
We provide proofs of the results in Section~\ref{app:proofsmain} of the supplementary material.
\subsection{Central limit theory for empirical tail measures}
\label{sec:clt:tailmes} 

Let $B\subset\RR$ be a Borel set bounded away from $0$, whose boundary has Lebesgue measure 0. Then, by regular variation and symmetry of $\varepsilon$,
\begin{align*}
  \frac{
  \PP[\varepsilon\in u_n \cdot B]
  }{\PP[|\varepsilon|>u_n]}
  \ \to \ 
  \mu_{\varepsilon}[B]
  \ := \ 
  \int_{B}
  \frac{\nu}{2}
  |z|^{-1-\nu}
  \,\mathrm{d}z
  \,.
\end{align*}
Let $A\subset \RR^T$ be bounded away from 0 such that 
for all $j\ge -(T-1)$, 
\begin{align*}
 \mu_\varepsilon\left[
\left\{
x\in\RR
\mid
  x\cdot [a_j,\ldots,a_{j+T-1}]
  \in \partial A
\right\}
  \right]
  \ = \ 0\,,
\end{align*}
where $\partial A$ is the boundary of the set $A$, 
and define the tail measure at lag $T-1$ by 
\begin{align*}
	\mu_T(A)
	\ = \
	\frac{1}{\sum_{i=0}^{\infty}|a_i|^{\nu}}
  \sum_{j=-(T-1)}^{\infty}
  \mu_\varepsilon[
\left\{
x\in\RR
\mid
  x\cdot [a_j,\ldots,a_{j+T-1}]
  \in A
\right\}
  ]
	\,.
\end{align*}
Since $A$ is a continuity set of $\mu_T$ that is bounded away from 0, it holds for all 
$s>0$
by Proposition~5.2.8 in~\cite{mikoschExtremeValueTheory2024},
\begin{align*}
	\frac{
		\PP[[X_{\ell}]_{\ell}\in u_n\cdot s \cdot A]
	}{
		\PP[|X_0|>u_n]
	}
	\ \to\
	\mu_T(s\cdot A)
  \ = \ 
  s^{-\nu}\cdot \mu_T(A)
	\qquad\text{as}\ n\to\infty\,.
\end{align*}
This convergence holds locally uniformly in $s$, due to local uniformity of univariate regular variation.
For $s>0$ 
denote the empirical tail measure and its expectation as
\begin{align*}
	\widehat{\PP}_n(s,A)
	\ := \
  \frac{1}{n-T+1}
	\sum_{t=1}^{n-T+1}
	\frac{\ind\{[X_{t+\ell}]_\ell \in u_n \cdot s \cdot A\}}{\PP[|X_0|>u_n]}
	\qquad\text{and}\qquad
	\PP_n(s,A)
	\ := \
	\frac{\PP\left[[
					X_{\ell}]_\ell \in u_n \cdot s \cdot A\right]}{\PP[|X_0|>u_n]}
	\,.
\end{align*}
\subsubsection{Result for deterministic thresholds}
 Setting $s=1$, the reduction principle tells us that the asymptotic profile is determined by
  \begin{align*}
    \left(
		\nabla_{\mathbf{y}}G_{\infty,n}(\mathbf{y},1,\infty, A)\left.\right|_{\mathbf{y}=\mathbf{0}_T}
    \right)^{\top}
    \sum_{t=1}^{n-T+1}
    [X_{t+\ell}]_\ell
    \,,
  \end{align*}
  where we see from Theorem~\ref{thm:linear_clt_multi} that the correct rate of convergence for the multivariate partial sum is $n^{-d-1/\alpha}$.
  Since
$
			\nabla_{\mathbf{y}} G_{\infty,n}(\mathbf{y},1,\infty, A)\Big|_{\mathbf{y}=\mathbf{0}_T} 
$, defined in~\eqref{eq:need},
  is an integral of an integrable function over an asymptotically vanishing domain, it generally converges to $\mathbf{0}_T$ as $n\to\infty$.
 Therefore, we need an assumption on its rate of decay to identify the limit.
  We verify this assumption in a non-trivial setting in Theorem~\ref{thm:example}.
\begin{assumption}[Convergence of scaled gradient]
	\label{asu:nabla_G}
 	For the set $A$ there exists $H_A(1)\in \RR^T$ such that
  \begin{align*}
		\left|
		\frac{
			\nabla_{\mathbf{y}} G_{\infty,n}(\mathbf{y},1,\infty, A)\Big|_{\mathbf{y}=\mathbf{0}_T}
		}
		{f_{X_0}(u_n)}
		\ - \ H_A(1)
		\right|
		\ \to \ 0
		\qquad\text{as}\ n\to\infty\,.\qquad
	\end{align*}
\end{assumption}
\begin{remark}[Assumption~\ref{asu:nabla_G}]
  \label{rem:asu_4}
  Later we allow $s$ to vary around $1$ in  
  $
			\nabla_{\mathbf{y}} G_{\infty,n}(\mathbf{y},s,\infty, A)\Big|_{\mathbf{y}=\mathbf{0}_T}
  $ which leads to a similar assumption but with a vector-valued function $H_A:s\mapsto H_A(s)$ instead of a constant vector $H_A(1)$. 
\end{remark}
\begin{theorem}[Empirical tail measure with deterministic thresholds]
  \label{thm:conv_multi_simple}
	Let Assumptions~\ref{asu:f},~\ref{asu:coef}, and~\ref{asu:rv_f} hold, let $A_0,\ldots,A_h$, $h\in\NN_0$, be bounded away from 0, and let $u_n\to\infty$ with
	$
		u_n=o(n^{(d+1/(\nu \land 2)-\kappa_0-\delta)/(\nu+1)})
	$ for some $\delta>0$.
		If 
        $A_0,\ldots,A_h$ satisfy Assumption~\ref{asu:nabla_G} with $H_{A_0}(1),\ldots, H_{A_h}(1)$, then
		      \begin{align*}
			      n^{1-d-1/\alpha}
			      u_n
			      \begin{bmatrix}
				      \widehat{\PP}_n(1,A_i)
				      \ - \
				      \PP_n(1,A_i)
			      \end{bmatrix}
			      _{i=0,\ldots,h}
			      \ \to_d \
            \frac{
			      \nu
            }{2} Z_\alpha
			      [H_{A_i}(1)^{\top}\mathbf{1}_{T}]_{i=0,\ldots,h}
			      \qquad\text{as}\ n\to\infty\,. \qquad
		      \end{align*}
\end{theorem}
\begin{remark}[Increased speed in the extreme value setting]
  \label{rem:increased_speed}
  Compared to standard central limit theory for long memory linear time series,
  the increased speed
  $n^{1-d-1/\alpha}u_n$ has already been observed in the univariate setting of~\cite{scheffelCentralLimitTheory2025}, see Remark~3.7 therein. 
\end{remark}
\subsubsection{Heuristic for the transition to sample quantiles as thresholds}
\label{sec:sub:heur}
A consequence of the uniform reduction principle stated in Theorem~\ref{thm:main} is that if $A_0,\ldots,A_h\in \mathcal{C}_T$ satisfy a uniform version of Assumption~\ref{asu:nabla_G} (see Assumption~\ref{asu:nabla_G_unif}) and 
		      if $(Y^{(i)}_n)$, for $i\in \{0,\ldots,h\}$, are random sequences that satisfy $Y^{(i)}_n/u_n\to_{\PP}1$ as $n\to\infty$, we have
		      \begin{align}
            \label{eq:strange_conv}
            \begin{split}
            &
			      n^{1-d-1/\alpha}
			      u_n
			      \begin{bmatrix}
				      \widehat{\PP}_n(Y^{(i)}_n/u_n,A_i)
				      \ - \
				      \PP_n(s,A_i)
				      \Big|_{s=Y_n^{(i)}/u_n}
			      \end{bmatrix}
			      _{i=0,\ldots,h}
                  \\&
			      \ \to_d \
            \frac{\nu}{2} Z_\alpha
			      [H_{A_i}(1)^{\top}\mathbf{1}_{T}]_{i=0,\ldots,h}
			      \qquad\text{as}\ n\to\infty\,; \qquad
            \end{split}
		      \end{align}
          see Theorem~\ref{thm:conv_multi_inter} in the supplementary material for a rigorous statement.
          From Lemma~C.1 and Lemma~D.1 in~\cite{scheffelCentralLimitTheory2025}
we know that $Y_n^{(i)} = Y_n = X_{n-k:n}$ is a valid choice when $u_n$ is the marginal quantile $q_{X_0}(1-k/n)$, under suitable assumptions on the growth of $k=k_n \to\infty$.
To obtain a version of this convergence with centering
$
\PP_n(1,A_i)
$ instead of $\PP_n(s,A_i)\Big|_{s=X_{n-k:n}/u_n}$, that is, with deterministic rather than random centering, we 
use the following Taylor approximation:
\begin{align}
	\label{eq:prev_decomp_2}
	\begin{split}
		 &
		\widehat{\PP}_n(X_{n-k:n}/u_n,A_i)
		\ - \
		\PP_n(1,A_i)
		\\ &
		  \ = \
		  \widehat{\PP}_n(X_{n-k:n}/u_n,A_i)
		  \ - \
		  \PP_n(s,A_i)\Big|_{s=X_{n-k:n}/u_n}
		\\ &
		  \qquad + \
		  \PP_n(s,A_i)\Big|_{s=X_{n-k:n}/u_n}
		\ - \
		  \PP_n(1,A_i)
		\\ &
		  \ \approx \
		  \widehat{\PP}_n(X_{n-k:n}/u_n,A_i)
		  \ - \
		  \PP_n(s,A_i)\Big|_{s=X_{n-k:n}/u_n}
		\\ &
		  \qquad + \
		  \frac{\partial}{\partial s}
		\PP_n(s,A_i)
		\Big|_{s=
				  1
				  }
		\cdot
		  \left(
		  \frac{X_{n-k:n}}{u_n}
		  \ - \ 1
		  \right)
		\,.
	\end{split}
\end{align}
Following Lemma~C.2 in~\cite{scheffelCentralLimitTheory2025}, we can analyze the joint convergence of these terms for $i\in\{0,\ldots,h\}$ as follows:
	let $s_0,\ldots,s_h, s^*\in \RR$ and write 
	\begin{align*}
		E_n(s_0,\ldots,s_h)
		\ := \
		\bigcap_{i=0}^h
		\left\{
    n^{1-d-1/\alpha}
    u_n
		\left(
		\widehat{\PP}_n(X_{n-k:n}/u_n,A_i)
		\ - \
		\PP_n(s,A_i)\Big|_{s=X_{n-k:n}/u_n}
		\right)
		\ \le \ s_i
		\right\}
    \,.
	\end{align*}
  Then one can show that 
	\begin{align}
    \label{eq:eq_eq}
    \begin{split}
		 &
		\PP
		\left[
			E_n(s_0,\ldots,s_h)
			\,,
			\quad
    n^{1-d-1/\alpha}
    u_n
			\left(
			\frac{X_{n-k:n}}{u_n}
			\ - \ 1
			\right)
			\ \le \ s^*
			\right]
		\\ &
		  \ = \
		  \PP
		  \left[
			  E_n(s_0,\ldots,s_h)
			  \,,
			  \quad
			  \frac{
    n^{1-d-1/\alpha}
    u_n
        }{\nu/2}
			\left(
			  \widehat{\PP}_n(1, B)
			  \ - \
			  \PP_n(1,B)
			\right)
			  \ \le \ s^*\cdot (1+o(1)) + o_{\PP}(1)
			  \right]
    \end{split}
	\end{align}
  for a suitable set $B$. This
allows us to derive the joint limit using a suitable version of the convergence in \eqref{eq:strange_conv} with $Y_n^{(i)}=X_{n-k:n}$ for $i\in \{0,\ldots,h\}$ and $Y_n^{(h+1)}=u_n$.
\subsubsection{Assumptions}
\label{sec:clt:tailmes:assumptions} 
Having established the joint convergence, the approximation in 
	\eqref{eq:prev_decomp_2} motivates us to study also 
  $
		  \frac{\partial}{\partial s}
		\PP_n(s,A_i)
		\Big|_{s=
				  1
				  }
  $ which 
  comes from
the derivative of $G_{\infty,n}(\mathbf{y},s,\infty,A)$ with respect to the scaling variable $s$.
\begin{assumption}
	\label{asu:f_more}
	$(x\mapsto x\cdot f'_{\varepsilon}(x))\in L^{1}(\RR)$.
\end{assumption}
\begin{remark}[Assumption~\ref{asu:f_more}]
  \label{rem:bremen}
  This holds if $f'_{\varepsilon}$ is regularly varying, since then the index of regular variation is $-2-\nu$ by Assumption~\ref{asu:f}. This is a very weak assumption, which in turn guarantees, by Lemma~\ref{lem:multi_swap}(ii) in the supplementary material,
  the existence and continuity of $\frac{\partial}{\partial s}\PP_n(s,A)$, together with the identity 
\begin{align}
  \label{eq:ps_main_text}
	\frac{\partial}{\partial s}
	\PP_n(s,A)
	\ =\
	\frac{1}{s}
	\left(
	T
	\frac{
			\PP[[X_{\ell}]_\ell\in u_n\cdot s\cdot A]
		}{\PP[|X_0|>u_n]}
	\ + \
	\int_{\RR^T}
	\ind\{\mathbf{x}\in u_n \cdot s\cdot A\}
	\cdot
	\frac{
			\mathbf{x}^{\top}
			\cdot
			\nabla f_{[X_{\ell}]_\ell}(\mathbf{x})
		}{\PP[|X_0|>u_n]}
	\,\mathrm{d}\mathbf{x}
	\right)
  \,.
\end{align}
  For rectangular sets as in Example~\ref{ex:ct_sets}(iii),
  existence and continuity of the derivative are readily obtained by partial integration without the need of Assumption~\ref{asu:f_more}. This, however, is not true for sets having a more complicated geometry like, for example, intersections with scale invariant sets that are covered in Example~\ref{ex:ct_sets}(v).
\end{remark}
While Assumption~\ref{asu:f_more} guarantees the existence of 
$\frac{\partial}{\partial s} \PP_n(s,A)$ with identity~\eqref{eq:ps_main_text}, 
we also need assumptions that control its convergence as $n\to\infty$.
  The following assumption is on the convergence of the second term in \eqref{eq:ps_main_text}, while
  the convergence of the first term in~\eqref{eq:ps_main_text} follows immediately from multivariate regular variation.
\begin{assumption}[Convergence of the second term in~\eqref{eq:ps_main_text}]
  \label{asu:cond_limit}
  The set $A$ is a continuity set of $\mu_T$ and there exists a compact neighborhood $1\in N_A\subset (0,\infty)$ and a function $p_A\colon N_A \to \RR$ such that 
	\begin{align*}
		\int_{\RR^T}
		\ind\{\mathbf{x}\in u_n \cdot s \cdot A\}
		\cdot
		\frac{
			\mathbf{x}^{\top}
			\nabla f_{[X_\ell]_\ell}(\mathbf{x})
		}{\PP[|X_0|>u_n]}
		\,\mathrm{d}\mathbf{x}
		\ \to \
		s^{-\nu}
		\left(
    \frac{\nu}{2}
		\cdot
		p_{A}(s)
		\ - \
		T
		\mu_T(A)
		\right)
	\end{align*}
  uniformly in $s\in N_A$ as $n\to\infty$. 
\end{assumption}
\begin{remark}[Convergence in~\eqref{eq:ps_main_text}]
\label{rem:stuttgart}
This assumption guarantees the convergence
	\begin{align*}
		 &
		\frac{\partial}{\partial s}
		\PP_n(s,A)
		  \ \to \
		s^{-1-\nu}
		\left(
		  T\cdot \mu_T(A)
		  \ + \
      \frac{\nu}{2}
		  \cdot
		  p_A(s)
		  \ - \
		  T\cdot \mu_T(A)
		  \right)
		  \ = \
		  \frac{\nu}{2}
      \cdot
		  s^{-1-\nu}\cdot p_A(s)
	\end{align*}
	uniformly in a neighborhood of $1$ as $n\to\infty$.
  The factor $\nu/2$ appears in the joint limit derived from \eqref{eq:eq_eq}, so Assumption~\ref{asu:cond_limit} is tailored to give a simple formula.
  Note that, by Lemma~\ref{lem:multi_swap}(ii) in the supplementary material and local uniformity of regular variation, $p_A$ is continuous on $N_A$. 
\end{remark}
As highlighted in Remark~\ref{rem:asu_4} and before~\eqref{eq:strange_conv}, we now state a uniform version of Assumption~\ref{asu:nabla_G}.
\begin{assumption}[Uniform version of Assumption~\ref{asu:nabla_G}]
	\label{asu:nabla_G_unif}
	For the set $A$
	there exists a compact neighborhood
	$1\in N_A\subset (0,\infty)$ 
	and a function $H_A\colon N_A\to \RR^T$ such that
	\begin{align*}
		\sup_{s\in N_A}
		\left|
		\frac{
			\nabla_{\mathbf{y}} G_{\infty,n}(\mathbf{y},s,\infty, A)\Big|_{\mathbf{y}=\mathbf{0}_T}
		}
		{f_{X_0}(s\cdot u_n)}
		\ - \ H_A(s)
		\right|
		\ \to \ 0
		\qquad\text{as}\ n\to\infty\,.\qquad
	\end{align*}
\end{assumption}
\begin{remark}[Continuity of the limit in Assumption~\ref{asu:nabla_G_unif}]
  The continuity of $(\mathbf{y},s)\mapsto \nabla_{\mathbf{y}}G_{\infty,n}(\mathbf{y},s,\infty,A)$  obtained from Lemma~\ref{lem:multi_swap}(ii) in the supplementary material, together with the uniform convergence in $N_A$ guarantees that $H_A$ is continuous on $N_A$.
\end{remark}
\begin{remark}[Simplification and validity of Assumptions~\ref{asu:cond_limit} and~\ref{asu:nabla_G_unif} for rectangular sets]
  \label{rem:sim}
In Proposition~\ref{prop:imply} we show
for rectangular sets that Assumptions~\ref{asu:cond_limit} and~\ref{asu:nabla_G_unif} are implied by a simpler assumption, which is then verified in Theorem~\ref{thm:example} in a non-trivial setting. 
\end{remark}
\subsubsection{Results for sample quantiles as thresholds}
The next theorem is the analog of Theorem~\ref{thm:conv_multi_simple} with sample quantiles as thresholds. 
There we observe a speed increase by the factor $u_n$
compared with the classical long memory rate $n^{-d-1/\alpha}$; see Remark~\ref{rem:increased_speed}. 
In the next result, we have the same increased speed, where we replace the factor $u_n$ by $q_{X_0}(1-k/n)$, with $k=\lfloor n\cdot \PP[X_0>u_n]\rfloor$; note that by regular variation of $X_0$, $u_n/q_{X_0}
		(
		1 - k/n
		) \to 1$ as $n\to\infty$. 
\begin{theorem}[Empirical tail measure with sample quantiles as thresholds]
	\label{thm:det_centering}
  Let Assumptions~\ref{asu:f},~\ref{asu:coef}, and~\ref{asu:rv_f} hold, let $A_0,\ldots,A_h$, $h\in\NN_0$, 
  be continuity sets of $\mu_T$ bounded away from 0, and let $u_n\to\infty$ with
	$
		u_n=o(n^{(d+1/(\nu \land 2)-\kappa_0-\delta)/(\nu+1)})
	$ 
  for some $\delta>0$.
  If Assumption~\ref{asu:f_more} holds,
  and
    if 
    $A_0,\ldots,A_h$ are in $\mathcal{C}_T$ and satisfy Assumption~\ref{asu:cond_limit} and~\ref{asu:nabla_G_unif} with functions 
    $p_{A_0},\ldots,p_{A_h}$, 
    and 
    $H_{A_0},\ldots,H_{A_h}$, respectively,
    then, for
          $k=\lfloor n\cdot \PP[X_0>u_n]\rfloor$,
	\begin{align*}
		 &
		n^{1-d-1/\alpha}
		q_{X_0}
		\left(
		1 - \frac{k}{n}
		\right)
		\begin{bmatrix}
			\widehat{\PP}_n(X_{n-k:n}/u_n,A_i)
			\ - \
			\PP_n(1,A_i)
		\end{bmatrix}
		_{i=0,\ldots,h}
		\\ &
		  \ \to_d \
      \frac{
		  \nu
      }{2}
		  Z_{\alpha}
		  \left[
			  \left(
			  H_{A_i}(1)
			  \right)^{\top}
			\mathbf{1}_T
			  \ +\
			  p_{A_i}(1)
			  \right]_{i=0,\ldots, h}
		\qquad\text{as}\ n\to\infty\,. \qquad
	\end{align*}
\end{theorem}
\begin{remark}[Different limiting behavior for random and deterministic thresholds]
  \label{rem:markedly}
  It may happen that
  $H_A(1)^{\top}\mathbf{1}_T = -p_A(1)$ while $H_{A}(1)^{\top}\mathbf{1}_T\neq 0$. Then, in the setting of deterministic thresholds (Theorem~\ref{thm:conv_multi_simple}), we get a non-degenerate limit, while sample quantiles as thresholds (Theorem~\ref{thm:det_centering}) give a degenerate limit at the same rate.
  In Section~\ref{sec:rect:cancellation}, we show that this happens for some rectangular sets that are relevant in practice.
\end{remark}
\subsection{Central limit theory for the empirical extremogram}
\label{sec:clt:extremogram} 

The final result is for extremogram estimators similar to those of Section~3.3 in~\cite{davisExtremogramCorrelogramExtreme2009}.
\begin{theorem}[Central limit theorems for empirical extremogram with deterministic and random thresholds]
	\label{thm:final}
  Let Assumptions~\ref{asu:f},~\ref{asu:coef}, and~\ref{asu:rv_f} hold, let $A_0,\ldots,A_h$, $h\in\NN_0$, be continuity sets of $\mu_T$ bounded away from 0, with $\mu_T(A_0)>0$, and let $u_n\to\infty$ with
	$
		u_n=o(n^{(d+1/(\nu \land 2)-\kappa_0-\delta)/(\nu+1)})
	$ for some $\delta>0$.
	\begin{enumerate}
		\item
    If $A_0,\ldots,A_h$ satisfy Assumption~\ref{asu:nabla_G} with $H_{A_0}(1),\ldots,H_{A_h}(1)$,
		      \begin{align*}
			       &
			      n^{1-d-1/\alpha}
			      u_n
			      \cdot
			      \left[
				      \frac{\widehat{\PP}_n(1,A_i)}{\widehat{\PP}_n(1,A_0)}
				      \ - \
				      \frac{\PP_n(1,A_i)}{\PP_n(1,A_0)}
				      \right]_{i=1,\ldots,h}
			      \\ &
			        \ \to_d \
              \frac{
			        \nu
              }{2}
			        Z_{\alpha}
			        \left[
				        \frac{
					        (\mu_T(A_0)\cdot H_{A_i}(1)- \mu_T(A_i)\cdot H_{A_0}(1))^{\top} \mathbf{1}_{T}
				      }{
					        \left(
					        \mu_T(A_0)
					      \right)^2
					        }
				      \right]
			      _{i=1,\ldots,h}
            \qquad\text{as}\ n\to\infty\,.
		      \end{align*}
		\item
    If Assumption~\ref{asu:f_more} holds, and
    if $A_0,\ldots,A_h$ are in $\mathcal{C}_T$ and satisfy Assumption~\ref{asu:cond_limit} and~\ref{asu:nabla_G_unif} with functions 
    $p_{A_0},\ldots,p_{A_h}$
    and 
    $H_{A_0},\ldots,H_{A_h}$, respectively, then, for
          $k=\lfloor n\cdot \PP[X_0>u_n]\rfloor$,
		      \begin{align*}
			       &
			      n^{1-d-1/\alpha}
			      q_{X_0}
			      \left(
			      1 - \frac{k}{n}
			      \right)
			      \left[
				      \frac{\widehat{\PP}_n(X_{n-k:n}/u_n,A_i)}{\widehat{\PP}_n(X_{n-k:n}/u_n,A_0)}
				      \ - \
				      \frac{\PP_n(1,A_i)}{\PP_n(1,A_0)}
				      \right]_{i=1,\ldots,h}
			      \\ &
			        \ \to_d\
			        \frac{\nu/2
				        }{
				        \left(
				        \mu_T(A_0)
				        \right)^2
				        }
			      Z_{\alpha}
			        \left[
				        \mu_T(A_0)
				      \left(
				        H_{A_i}(1)^{\top} \mathbf{1}_{T}
				      \ + \ p_{A_i}(1)
				        \right)
				        \right.
			      \\
			       &
				      \left.\qquad
				      \qquad
				      \qquad
				      \qquad
				      \qquad
				      - \
				      \mu_T(A_i)
				      \left(
				      H_{A_0}(1)^{\top} \mathbf{1}_{T} \ +\
				      p_{A_0}(1)
				      \right)
				      \right]
			      _{i=1,\ldots,h}
            \qquad\text{as}\ n\to\infty\,.
		      \end{align*}
	\end{enumerate}
\end{theorem}
\begin{remark}[Comparison with existing literature]
  Comparing the proofs of Corollary~3.4 in~\cite{davisExtremogramCorrelogramExtreme2009} and Theorem~\ref{thm:final} in our work, we see that both
  feature the term $\mathbf{F}S_n$, where
\begin{align*}
  \mathbf{F}
  \ = \ 
		\begin{bmatrix}
			\mu_T(A_0)\cdot \mathbf{I}_{h} & [-\mu_T(A_i)]_{i=1,\ldots,h}
		\end{bmatrix}
\end{align*}
(compare with $\mathbf{F}$ in Equation~(3.18) in~\cite{davisExtremogramCorrelogramExtreme2009}),
and
\begin{align*}
  S_n
  \ = \
  n^{1-d-1/\alpha}
  u_n
\begin{bmatrix}
			\left[
				\widehat{\PP}_n(1,A_i)
				\ - \
				\PP_n(1,A_i)
				\right]_{i=1,\ldots,h}
			\\
			\widehat{\PP}_n(1,A_0)
			\ - \
			\PP_n(1,A_0)
		\end{bmatrix} 
\end{align*}
(compare with $S_n$ in Equation~(3.14) in~\cite{davisExtremogramCorrelogramExtreme2009}).

  While Theorem~3.2 in~\cite{davisExtremogramCorrelogramExtreme2009} gives a non-degenerate multivariate Gaussian limit for their version of $S_n$, Theorem~\ref{thm:conv_multi_simple} in our work gives, for our version of $S_n$, a stable random vector with linearly dependent components. Therefore, in our long memory setting, the randomness in the limit comes from one random variable $Z_{\alpha}$, rather than from a genuinely multivariate random vector. Since the theory of~\cite{davisExtremogramCorrelogramExtreme2009} is restricted to deterministic thresholds, we now compare Theorem~\ref{thm:final}(ii) (long memory---random thresholds) with Corollary~3.6 in~\cite{dre2015}(short memory---random thresholds). In the short memory case, limits for random and deterministic thresholds are generally the same, while the limits in (i) and (ii) of Theorem~\ref{thm:final} in our work can markedly differ; see Remark~\ref{rem:markedly}.
\end{remark}
\section{Theory for rectangular sets under random thresholds}
\label{sec:rect}

Throughout this section we consider
  $A=\bigtimes_{j=1}^T (\mathfrak{a}_j, \infty)$ with $\{\mathfrak{a}_j\}\subset (0,\infty)\cup\{-\infty\}$ and $\emptyset \neq J_A=\{j\mid \mathfrak{a}_j > 0\}\subset \{1,\ldots,T\}$. By Example~\ref{ex:ct_sets}(iii), $A\in \mathcal{C}_T$.
Proofs of the results of this section are provided in Section~\ref{app:proofsrect} of the supplementary material.
\subsection{Simplified assumptions}
\label{sec:rect:assumptions}
We show that our theory for random thresholds is applicable
for rectangular sets $A$ by providing simplified assumptions and showing their validity; see Remark~\ref{rem:sim}.
\begin{assumption}
  \label{asu:simplified}
  For the set $A$ there exist $(p_{j,A})_{j\in J_A}\subset [0,1]$, such that
	\begin{align}
    \label{eq:final_cond}
		\PP
		\left[
		\left.
		[X_{\ell}]_{\ell}
		\in u\cdot A
		\
		\right|
		\
		X_{j-1} = u\cdot \mathfrak{a}_j
		\right]
    \ \to \ 
    p_{j,A}
		\qquad\text{as}\ u\to\infty\,.\qquad
	\end{align}
\end{assumption}
\begin{remark}[Comparison with multivariate regular variation I]
\label{rem:classical_I}
Replacing the condition
\begin{align*}
		X_{j-1} \ =\  u\cdot \mathfrak{a}_j
        \qquad\text{by}\qquad
		X_{j-1} \ >\  u\cdot \mathfrak{a}_j
    \,,
\end{align*}
  classical multivariate regular variation yields
  \begin{align*}
    &
	\PP
		\left[
		\left.
		[X_{\ell}]_{\ell}
		\in u\cdot A
		\
		\right|
		\
		X_{j-1} > u\cdot \mathfrak{a}_j
		\right]
    \ = \ 
	\PP
		\left[
    \forall i\in J_A,
    \
    X_{i-1}> 
    \left(
    u\cdot
    \mathfrak{a}_j
    \right)
    \frac{
    \mathfrak{a}_i
    }{
    \mathfrak{a}_j
    }
    \ 
    \Bigg|\ 
		X_{j-1} > u\cdot \mathfrak{a}_j
				\right]
	    \\&
    \ \to\ 
    2
    \cdot
    \mu_T
    \left[
    \bigtimes_{i=1}^{T}
    \left(
    \frac{\mathfrak{a}_i}{\mathfrak{a}_j}
    , \infty
    \right)
    \right]
    \qquad\text{as}\ u\to\infty
    \,.
      \end{align*}
\end{remark}
The next result shows that Assumption~\ref{asu:simplified} implies Assumptions~\ref{asu:cond_limit} and~\ref{asu:nabla_G_unif}, and is therefore sufficient for the application of Theorem~\ref{thm:final}.
\begin{proposition}[Simplification of Assumptions~\ref{asu:cond_limit} and~\ref{asu:nabla_G_unif}]
  \label{prop:imply}
  Let Assumptions~\ref{asu:f},~\ref{asu:coef},~\ref{asu:rv_f}, and~\ref{asu:f_more} hold and let $A$ be a continuity set of $\mu_T$. Then 
  Assumption~\ref{asu:simplified} implies
  Assumptions~\ref{asu:cond_limit} and~\ref{asu:nabla_G_unif},
  with
	\begin{align*}
		p_A(s)
		\ = \
		-
		\sum_{j\in J_A}
		\mathfrak{a}_j^{-\nu}
		p_{j,A}
    \qquad\text{and}\qquad
		H_A(s)
		\ = \
		\sum_{j\in J_A}
		\mathfrak{a}_j^{-(\nu+1)}
		\cdot
		p_{j,A}\cdot \mathbf{e}_j
		\,.
	\end{align*}

\end{proposition}
The next result verifies Assumption~\ref{asu:simplified} in a non-trivial setting.
\begin{theorem}[Validity of Assumption~\ref{asu:simplified}]
	\label{thm:example}
		Let $T=2$, $A=(1,\infty)\times(\mathfrak{a}_1,\infty)$, with $\mathfrak{a}_1\in [1,\infty)$.
  Assume that $a_0=1$, 
  and that $(a_i)_{i\in\NN}$ is positive and strictly decreasing. 
    If Assumptions~\ref{asu:f} and~\ref{asu:coef} hold, and the distribution of $\varepsilon$ is unimodal, then Assumption~\ref{asu:simplified}
  holds with
	\begin{align*}
      p_{1,A}
      \ = \ 
      0
      \qquad\text{and}\qquad
      p_{2,A}
    \ = \ 
			\dfrac{\sum_{i=1}^{\infty}a_i^{\nu}}{\sum_{i=0}^{\infty}a_i^{\nu}}
		\ =\
			1
			\ - \
			\dfrac{1}{\sum_{i=0}^{\infty}a_i^{\nu}}
		\,,
	\end{align*}
  and Theorem~\ref{thm:final} applies.
\end{theorem}
\begin{remark}[Comparison with multivariate regular variation II]
We continue the comparison of Remark~\ref{rem:classical_I} under the assumptions of Theorem~\ref{thm:example}.
  Write $\mathfrak{a}_0 = 1$ and let $\sigma:\{0,1\}\to \{0,1\}$ be a permutation. Then 
  \begin{align*}
    &
    \PP[X_{\sigma(1)}>\mathfrak{a}_{\sigma(1)}\cdot u \mid X_{\sigma(0)}> \mathfrak{a}_{\sigma(0)}\cdot u]
    \ = \ 
    \frac{\PP[X_0 > u\,, X_1 > \mathfrak{a}_1 \cdot u]}
    {\PP[X_0 > \mathfrak{a}_{\sigma(0)}\cdot u]}
    \\&
    \ = \ 
    2
    \frac{
\PP[X_0 > u]
    }{
\PP[X_0 > \mathfrak{a}_{\sigma(0)}\cdot u]
    }
    \frac{\PP[X_0 > u\,, X_1 > \mathfrak{a}_1 \cdot u]}
    {
    \PP[|X_0|>u]
    }
    \\&
    \ \to \ 
    2 \cdot \mathfrak{a}_{\sigma(0)}^{\nu}
    \cdot
    \mu_2[(1,\infty)\times (\mathfrak{a}_1,\infty)]
    \,,
  \end{align*}
  where the tail measure is given by
  \begin{align*}
    &
    \mu_2[(1,\infty)\times (\mathfrak{a}_1,\infty)]
    \ = \
    \frac{1}{\sum_{i=0}^{\infty}a_i^{\nu}}
    \sum_{j=-1}^{\infty}
    \int_{\RR}
    \frac{\nu}{2}
    |z|^{-1-\nu}
    \ind\{z \cdot a_j > 1, z\cdot a_{j+1}
    > 
    \mathfrak{a}_1 
    \}
    \,\mathrm{d}z
    \ = \ 
    \frac{\mathfrak{a}_1^{-\nu}}{2}
    \frac{
    \sum_{i=1}^{\infty}a_i^{\nu}
    }{
    \sum_{i=0}^{\infty}a_i^{\nu}
    }
    \,,
  \end{align*}
  so that
\begin{align*}
  \lim_{u\to\infty}
  \PP[X_1 > \mathfrak{a}_1 \cdot u \mid X_0 > u]
  \ = \ 
  \mathfrak{a}_1^{-\nu}
  \frac{
  \sum_{i=1}^{\infty} a_i^{\nu}
  }{\sum_{i=0}^{\infty}a_i^{\nu}}
  \quad\text{and}\quad
  \lim_{u\to\infty}
  \PP[X_0 > u \mid X_1 > \mathfrak{a}_1\cdot u]
  \ = \ 
  \frac{
  \sum_{i=1}^{\infty} a_i^{\nu}
  }{\sum_{i=0}^{\infty}a_i^{\nu}}
  \ = \ p_{2,A}
  \,.
\end{align*}
\end{remark}

\subsection{Cancellation in the asymptotic scale}
\label{sec:rect:cancellation}
In Remark~\ref{rem:markedly} we point out that the asymptotic scale in Theorem~\ref{thm:det_centering} may vanish. We now study this for rectangular sets.
 In the setting of Proposition~\ref{prop:imply}, we have
	\begin{align*}
		H_{A}(1)^{\top}
		\mathbf{1}_T
		\ + \
		p_A(1)
		\ = \
		\sum_{j\in J_A}
		\mathfrak{a}_j^{-\nu}
		\left(
    \frac{1}{\mathfrak{a}_j}
		\ - \ 1
		\right)
		p_{j,A}
		\,.
	\end{align*}
  This is the asymptotic scale in Theorem~\ref{thm:det_centering} multiplied by $2/\nu$.
    In the setting of Theorem~\ref{thm:example}, this equals
		\begin{align*}
			\mathfrak{a}_1^{-\nu}
			\left(
			\frac{1}{\mathfrak{a}_1}
			\ - \ 1
			\right)
			\left(
			1
			-
			\frac{1}{\sum_{i=0}^{\infty}a_i^{\nu}}
			\right)
      \,.
		\end{align*}
    Note that, if $\mathfrak{a}_j = 1$ for all $j\in J_A$, this vanishes and therefore the limit in Theorem~\ref{thm:det_centering} is degenerate. 
For $T=1$, this is 
expected,
since
\begin{align*}
  \widehat{\PP}_n(X_{n-k:n}/u_n,A)
\ = \ 
\frac{1}{n}
\frac{1}{\PP[|X_0|>u_n]}
\sum_{t=1}^{n} 
\ind\{X_{t}> X_{n-k:n}\}
\ = \ 
\frac{k}{n}
\frac{1}{\PP[|X_0|>u_n]}
\end{align*}
is deterministic.
For $T>1$, this is remarkable and genuinely new, because
\begin{align*}
  \widehat{\PP}_n(X_{n-k:n}/u_n,A)
\ = \ 
\frac{1}{n-T+1}
\frac{1}{\PP[|X_0|>u_n]}
\sum_{t=1}^{n-T+1} \ind\left\{\min_{j\in J_A}X_{t+j-1}> X_{n-k:n}\right\}
\end{align*}
is generally random. 
Under the speed of Theorem~\ref{thm:conv_multi_simple}, this randomness, however, is not sufficient
to yield
a non-degenerate limit in the setting of Theorem~\ref{thm:det_centering}.
It remains an open question what the correct rate of convergence is in this setting.

\subsection{Univariate theory for random thresholds}
\label{sec:rect:1d}

Since our theory is also valid in the univariate case $T=1$, we now compare our results with those in the i.i.d.~and weakly dependent setting which are readily derived from existing theory.
For $\mathfrak{a}_1 = \mathfrak{a}\neq 1$, it follows from
Theorem~\ref{thm:det_centering} that
	\begin{align*}
    n^{1-d-1/\alpha}
    q_{X_0}
    \left(
    1
    -
    \frac{k}{n}
    \right)
		\left(
		\frac{1}{n}
		\sum_{t=1}^{n}
		\left(
			\frac{\ind\{X_t>\mathfrak{a}\cdot X_{n-k:n}\}}{\PP[X_0>u_n]}
			\ - \
			\frac{\PP[X_0>\mathfrak{a}\cdot u_n]}{\PP[X_0 > u_n]}
			\right)
		\right)
		\ \to_d \
        \nu
		\mathfrak{a}^{-\nu}
		\left|
		\frac{1}{\mathfrak{a}}
		- 1
		\right|
		Z_\alpha,
	\end{align*}
  as $n\to\infty$.
    The next theorem provides a comparable result in the i.i.d. setting.
\begin{theorem}[Univariate i.i.d.~setting]
	\label{thm:comp:iid}
	Let $X_1,\ldots, X_n$ be i.i.d. copies of a random variable $X$ having a continuous distribution. Assume that there are $\nu>0$, $\rho<0$, and a function $\mathfrak{A}$ that is either positive or negative, such that, for all $x>0$,
	\begin{align*}
		\frac{
			\dfrac{\PP[X>x\cdot u]}{\PP[X>u]}
			\ - \ x^{-\nu}
		}{\mathfrak{A}(1/\PP[X>u])}
		\ \to \
		x^{-\nu}
		\frac{x^{\rho\nu}- 1}{\rho/\nu}
		\qquad\text{as}\ u\to\infty\,.
	\end{align*}
	Let 
	$u_n\to \infty$ satisfy $n\PP[X>u_n]\to \infty$ and $\mathfrak{A}(1/\PP[X>u_n])=O(1/\sqrt{n\PP[X>u_n]})$, and write $k=\lfloor n\cdot \PP[X>u_n]\rfloor$. Then, for any $\mathfrak{a}>0$,
	\begin{align*}
		\sqrt{k}
		\left(
		\frac{1}{n}
		\sum_{i=1}^{n}
		\left(
			\frac{\ind\{X_i>\mathfrak{a}\cdot X_{n-k:n}\}}{\PP[X>u_n]}
			\ - \
			\frac{\PP[X>\mathfrak{a}\cdot u_n]}{\PP[X > u_n]}
			\right)
		\right)
		\ \to_d \
    \left(
		\mathfrak{a}^{-\nu}
		\left|
		\frac{1}{\mathfrak{a}^{\nu}}
		- 1
		\right|
    \right)^{1/2}
		Z\,,
	\end{align*}
	as $n\to\infty$, where $Z$ is a standard normal random variable.
\end{theorem}
\begin{remark}[Weakly dependent time series]
  The asymptotic variance in Theorem~\ref{thm:comp:iid} comes from
\begin{align*}
		 &
	    \Var[W(\mathfrak{a}^{-\nu})]
    \ + \
    \mathfrak{a}^{-2\nu}
    \Var[W(1)]
    \ - \ 
    2
    \mathfrak{a}^{-\nu}
    \Cov[W(\mathfrak{a}^{-\nu}),W(1)]
  		  \ = \
		  \mathfrak{a}^{-\nu}
		\left|
		  \frac{1}{\mathfrak{a}^{\nu}}
		\ - \ 1
		  \right|
		  \,,
	\end{align*}
  where $W$ is a Brownian motion. 
  We compare this with the short memory setting.
Let $(X_t)$ be a weakly dependent linear time series as in Section~3.2 in \cite{dreesExtremeQuantileEstimation2003}
with non-negative, decreasing coefficients $(a_i)$. 
  Under stronger assumptions, it follows from Theorem~2.1 in \cite{dreesExtremeQuantileEstimation2003},
by an argument similar to that used in the
  proof of Theorem~\ref{thm:comp:iid}, that
  	\begin{align*}
    &
		\sqrt{k}
		\left(
		\frac{1}{n}
		\sum_{t=1}^{n}
		\left(
			\frac{\ind\{X_t>\mathfrak{a}\cdot X_{n-k:n}\}}{\PP[X_0>u_n]}
			\ - \
			\frac{\PP[X_0>\mathfrak{a}\cdot u_n]}{\PP[X_0 > u_n]}
			\right)
		\right)
        \\&
		\ \to_d \
        \left(
		c(\mathfrak{a}^{-\nu},\mathfrak{a}^{-\nu})
		\ + \
		\mathfrak{a}^{-2\nu}
		\cdot
		c(1,1)
		\ - \
		2\cdot \mathfrak{a}^{-\nu}
		\cdot
		c(\mathfrak{a}^{-\nu},1)
        \right)^{1/2}
		Z
	\end{align*}
  as $n\to\infty$, where $Z$ is a standard normal random variable, and 
  where
	\begin{align*}
		c(x,y)
		\ = \
		(x\land y)
		\ + \
		\sum_{m=1}^{\infty}
		\left(
		c_m(x,y) \ + \
		c_m(y,x)
		\right)
	\end{align*}
	with
	\begin{align*}
		c_m(x,y)
		\ = \
		\frac{1}{\sum_{i=0}^{\infty}a_i^{\nu}}
		\sum_{j=0}^{\infty}
		\left(
		(x\cdot a_j^{\nu})\land
		(y\cdot a_{j+m}^{\nu})
		\right)
    \,,
    \qquad m\in \NN
		\,.
	\end{align*}
    Thus, the Brownian motion $W$ is replaced by a Gaussian process with covariance function $c$.
	The term $x\land y$ yields the variance in the i.i.d. setting. The additional contribution resulting from weak dependence is
	\begin{align*}
		 &
		\sum_{m=1}^{\infty}
		2\cdot
		\left(
		c_m(\mathfrak{a}^{-\nu},\mathfrak{a}^{-\nu})
		\ + \
		\mathfrak{a}^{-2\nu} c_m(1,1)
		\ - \
		\mathfrak{a}^{-\nu}
		\left(
		c_m(\mathfrak{a}^{-\nu},1)
		\ + \
		c_m(1,\mathfrak{a}^{-\nu})
		\right)
		\right)
		\\ &
		  \ = \
      \frac{2}{\sum_{i=0}^{\infty}a_i^{\nu}}
		\sum_{m=1}^{\infty}
		\sum_{j=0}^{\infty}
		\left(
		  \mathfrak{a}^{-\nu}
		a_{j+m}^{\nu}
		\ + \
		  \mathfrak{a}^{-2\nu}
		a_{j+m}^{\nu}
		\ - \
		  \mathfrak{a}^{-\nu}
		\left(
		  (\mathfrak{a}^{-\nu}a_j^{\nu})
		  \land
		  a_{j+m}^{\nu}
		\ + \
		  a_j^{\nu}
		\land
		  (\mathfrak{a}^{-\nu}a_{j+m}^{\nu})
		\right)
		  \right)
		\\ &
		  \ = \
		  \mathfrak{a}^{-\nu}
    \frac{2}{\sum_{i=0}^{\infty}a_i^{\nu}}
		\sum_{m=1}^{\infty}
		\sum_{j=0}^{\infty}
		\left(
		  a_{j+m}^{\nu}
		\ + \
    \mathfrak{a}^{-\nu}
		  a_{j+m}^{\nu}
		\ - \
		  \left(
		  (\mathfrak{a}^{-\nu}a_j^{\nu})
		  \land
		  a_{j+m}^{\nu}
		\ + \
		  a_j^{\nu}
		\land
		  (\mathfrak{a}^{-\nu}a_{j+m}^{\nu})
		\right)
		  \right)
          \,.
	\end{align*}
\end{remark}
\begin{remark}[Comparison of univariate results]
  While both the i.i.d. and weakly dependent settings yield the speed $\sqrt{k}$, our long memory setting gives the rate $n^{1-d-1/\alpha}q_{X_0}(1-k/n)$.
A similar difference in speed has also been observed in the univariate setting for the Hill estimator \cite[Section~1.2]{scheffelCentralLimitTheory2025}.
  While we see a Gaussian limit with modified variance under weak dependence, the long memory setting with $\nu\in (1,2)$ changes the class of limiting distributions to $S\alpha S$ where $\alpha=(\nu\land 2)\in(1,2)$. We also note that the dependence of the limit on $\mathfrak{a}$ 
  generally differs among the three settings.
\end{remark}

\section*{Acknowledgments}

Financial support from the French CNRS within the project ``Extreme value analysis of time series through de-randomization techniques'', funded through the IEA program, is gratefully acknowledged. G.~Stupfler acknowledges further financial support from the French \textit{Agence Nationale de la Recherche} under the grants ANR-23-CE40-0009 (EXSTA project) and ANR-11-LABX-0020-01 (Centre Henri Lebesgue), as well as from the Chair Stress Test, RISK Management and Financial Steering of the Foundation Ecole Polytechnique.

\clearpage

\appendix

\renewcommand{\thesection}{\Alph{section}}

\begin{center}
	{\Large Supplementary Material to the article \\[1ex] Central limit theory for serial tail dependence estimators in heavy-tailed long memory linear time series} \\

	Ioan Scheffel$^{a}$, Marco Oesting$^{a,b}$, Gilles Stupfler$^{c}$ \\

	$^{a}$ {\small Institute for Stochastics and Applications, University of Stuttgart, D-70563 Stuttgart, Germany} \\[1ex]
	$^{b}$ {\small Stuttgart Center for Simulation Science (SC SimTech), University of Stuttgart, D-70569 Stuttgart, Germany} \\[1ex]
	$^{c}$ {\small Univ Angers, CNRS, LAREMA, SFR MATHSTIC, F-49000 Angers, France} %
\end{center}
\noindent

Fix $T\ge 1$ and $0<\underline{s}<1<\overline{s}<\infty$. The first lemma reduces the multivariate problem to a series of univariate problems, which motivates us to 
distinguish it from the rest of the supplementary material.
\begin{lemma*}
	Let either of the following conditions hold:
	\begin{enumerate}
		\item
		      $A$ is bounded away from 0, $s_1\in (\underline{s},\infty)$ and $s_2 = \infty$
		\item
		      $A\in \mathcal{C}_T$ and $s_1,s_2\in (\underline{s},\overline{s}]$ satisfy $s_1<s_2$
	\end{enumerate}
	Then there exists $\{\mathfrak{a}_j\}_{j\in \{1,\ldots,T\}}\subset (0,\infty)$ independent of $s_1$ and $s_2$, such that
	\begin{align}
    \label{lem:two_sets}
		\ind\{
    \mathbf{z}
    \in u_n\cdot (s_1\cdot A)\setminus (s_2\cdot A)\}
		\ \le \
		\sum_{j=1}^T
		\ind
		\left\{
		|z_{j-1}|
		\in u_n \cdot \mathfrak{a}_j \cdot [s_1,s_2]
		\right\}
    \,,
    \qquad
    \mathbf{z}\in\RR^T
		\,.
	\end{align}
\end{lemma*}
\begin{proof}
  Under condition~(ii), the statement follows immediately from the definition of $\mathcal{C}_T$. It remains to show the statement under condition~(i).
  \\
  \textbf{Proof under condition~(i):}
	A set $A$ is bounded away from $0$ if and only if there exists $\{\mathfrak{a}_j\}_{j\in \{1,\ldots,T\}}\subset (0,\infty)$ such that
	\begin{align*}
		A
		\ \subset\
		\bigcup_{j=1}^T
		\{ | x_{j-1} | \ge \mathfrak{a}_j\}
		\,.
	\end{align*}
	For $s_1\in (\underline{s},\infty)$ and $s_2=\infty$, it holds
	\begin{align*}
		(s_1\cdot A) \setminus (s_2\cdot A)
		\ =\
		s_1\cdot A
		\ \subset\
		s_1\cdot
		\bigcup_{j=1}^T
		\{ |x_{j-1}| \ge \mathfrak{a}_j\}
		\ = \
		\bigcup_{j=1}^T
		\{ |x_{j-1}| \in  \mathfrak{a}_j\cdot [s_1,\infty)\}
		\,.
	\end{align*}
	Taking indicators of these sets proves the claimed inequality.
\end{proof}
\begin{assumption*}
  Throughout the supplementary material, we assume that the conditions of the lemma leading to \eqref{lem:two_sets} are satisfied with $\{\mathfrak{a}_j\}$, $s_1$ and $s_2$ fixed.
\end{assumption*}
\section{Extensions of results from~\cite{scheffelCentralLimitTheory2025}}
\begin{lemma}[Lipschitz bound for $f_{\varepsilon}$]
	\label{lem:lip_f}
	Let Assumption~\ref{asu:f} hold. Then, for all $x,y\in\RR$ satisfying $|x-y|<1$,
	\begin{align*}
		|
		f_\varepsilon(x)
		-
		f_\varepsilon(y)
		|
		 &
		\ \lesssim \
		|x-y|
		\cdot
		g_{\alpha-1}(x)
		\,,
	\end{align*}
	where $\lesssim$ means inequality up to a multiplicative constant that depends only on the distribution of $\varepsilon$.
\end{lemma}
\begin{proof}
	Applying the mean value theorem together with the triangle inequality and Assumption~\ref{asu:f} it follows that
	\begin{align*}
		|
		f_\varepsilon(x)
		-
		f_\varepsilon(y)
		|
		 &
		\ \le \
		\int_{0\land (y-x)}^{0\lor (y-x)}
		|f_{\varepsilon}'(x + s)|
		\,\mathrm{d}s
		\lesssim \
		\int_{0\land (y-x)}^{0\lor (y-x)}
		g_{\alpha-1}(x+s)
		\,\mathrm{d}s
		\\ &
		  \lesssim \
		  g_{\alpha-1}(x)
		  \cdot
		  \int_{0\land (y-x)}^{0\lor (y-x)}
		  (1\lor |s|)^{\alpha}
		\,\mathrm{d}s
		  \,,
	\end{align*}
	where the last step follows from
	Lemma~\ref{lem:5.1}.
	To conclude the proof, note that $|y-x|<1$ by assumption, so that
	\begin{align*}
		\int_{0\land (y-x)}^{0\lor (y-x)}
		(1\lor |s|)^{\alpha}
		\,\mathrm{d}s
		\ = \
		|y-x|
		\,.
	\end{align*}
	This proves the asserted bound for $|f_{\varepsilon}(x)-f_{\varepsilon}(y)|$.
\end{proof}
For $\gamma>0$ we define
\[
	g_{\gamma}(x)
	\ := \
	\frac{1}{(1+|x|)^{1+\gamma}}
	\qquad\text{for}\ x\in\RR
	\,.
\]

\begin{lemma}[Bounds for $g_{\gamma}$]
	\label{lem:5.1}
	Let $\gamma\in (0,\infty)$. Then the following holds for $g_\gamma$:
	\begin{enumerate}[label=(\roman*)]
		\item
		      For all $y,z\in\mathbb{R}$, and all $\gamma\in (0,1)$,
		      \begin{align*}
			      g_{\gamma}
			      (z+y)
			      \ \lesssim \
			      g_\gamma(z)
			      \cdot
			      (1\lor |y|)^{1+\gamma}\,,
		      \end{align*}
		      where $\lesssim$ means inequality up to a multiplicative constant that depends only on $\gamma$.
		\item
		      Let $a,b\in\mathbb{R}$ with $a<b$. Then, for all $z\in\mathbb{R}$, and all $\gamma\in (0,1)$,
		      \begin{align*}
			      \int_a^b g_\gamma(z-w)\,\mathrm{d}w
			      \ \lesssim \
			      g_\gamma(z)
			      \cdot
			      ((b-a) \lor 1)^{1+\gamma}
			      \,,
		      \end{align*}
		      where $\lesssim$ means inequality up to a multiplicative constant that depends only on $\gamma$.
		\item
		      For $c\neq 0$, 
		      \begin{align*}
			      g_{\gamma}(z\cdot c)
			      \ \le \
			      \left(
			      1
			      \lor
			      \frac{1}{|c|}
			      \right)^{1+\gamma}
			      g_{\gamma}(z)
			      \ \lesssim\
			      g_{\gamma}(z)
			      \,,
		      \end{align*}
		      where $\lesssim$ means inequality up to a multiplicative constant that depends only on $\gamma$ and $c$.
		\item
		      Let $\mathfrak{x},\mathfrak{y},\mathfrak{z}\in \RR$ and $\gamma\in (0,1)$. Then
		      \begin{align*}
			       &
			      \int_{\mathfrak{b}\land \mathfrak{c}}^{\mathfrak{b}\lor\mathfrak{c}}
			      \left(
			      1
			      \land |s\cdot \mathfrak{x}|
			      \right)
			      \cdot
			      g_{\gamma}
			      \left(
			      \mathfrak{z}
			      +
			      s\cdot \mathfrak{y}
			      \right)
			      \,\mathrm{d}s
			      \ \lesssim \
			      \left(
			      |\mathfrak{y}|^{\gamma}
			      +
			      |
			      \mathfrak{x}|^{\gamma}
			      \right)
			      \cdot
			      g_{\gamma}(\mathfrak{z})
			      \cdot
			      \left(
			      |\mathfrak{b}|^{1+\gamma}
			      +
			      |\mathfrak{c}|^{1+\gamma}
			      \right)
			      \,,
		      \end{align*}
		      where $\lesssim$ means inequality up to a multiplicative constant that depends only on $\gamma$.
	\end{enumerate}
\end{lemma}
\begin{proof}
	For a proof of (i) and (ii)
	see the
proof of Lemma~A.3 in~\citet{scheffelCentralLimitTheory2025}.
	\\
	\textbf{Proof of (iii):}
	We distinguish $|c|\in (0,1)$ and $|c|\ge 1$. For $|c|\in (0,1)$ it holds $1/|c| >1$, and therefore
	\begin{align*}
		g_{\gamma}(z\cdot c)
		 &
		\ = \
		\frac{1}{(1+|z\cdot c|)^{1+\gamma}}
		\ = \
		\left(
		\frac{1}{|c|}
		\right)^{1+\gamma}
		\frac{1}{(1/|c|+|z|)^{1+\gamma}}
		\\ &
		  \ \le \
		  \left(
		  1\ \lor \
		  \frac{1}{|c|}
		  \right)^{1+\gamma}
		\frac{1}{(1+|z|)^{1+\gamma}}
		\ = \
		  \left(
		  1\ \lor \
		  \frac{1}{|c|}
		  \right)^{1+\gamma}
		g_{\gamma}(z)
		  \,.
	\end{align*}
	If $|c|\ge 1$, then $1/|c|\le 1$, and $|z\cdot c|\ge |z|$, so that
	\begin{align*}
		g_{\gamma}(z\cdot c)
		 &
		\ = \
		\frac{1}{(1+|z\cdot c|)^{1+\gamma}}
		\ \le \
		\frac{1}{(1+|z|)^{1+\gamma}}
		\ = \
		\left(
		1\ \lor \
		\frac{1}{|c|}
		\right)^{1+\gamma}
		g_{\gamma}(z)
		\,.
	\end{align*}
	Combining the two cases proves the asserted bound for $g_{\gamma}(z\cdot c)$.
	\\
	\textbf{Proof of (iv):}
	If
	$\mathfrak{x}=0$ there is nothing to prove. Therefore we assume $\mathfrak{x}\neq 0$.
	\\
  \textit{Case $\mathfrak{y}=0$:} Since $\gamma\in (0,1)$,
  \begin{align}
    \label{eq:1990}
    \left(
    1 \land |s\cdot \mathfrak{x}|
    \right)
    \ \le \ 
    \left(
    1 \land
    |s\cdot \mathfrak{x}|
    \right)^{\gamma}
    \ \le \ 
    |s\cdot \mathfrak{x}|^{\gamma}
    \,,
  \end{align}
  so that
	\begin{align*}
		 &
		\int_{\mathfrak{b}\land \mathfrak{c}}^{\mathfrak{b}\lor\mathfrak{c}}
		\left(
		1 \land |s\cdot \mathfrak{x}|
		\right)
		\cdot
		g_{\gamma}
		\left(
		\mathfrak{z}
		+
		s\cdot \mathfrak{y}
		\right)
		\,\mathrm{d}s
		\ = \
		g_{\gamma}(\mathfrak{z})
		\int_{\mathfrak{b}\land \mathfrak{c}}^{\mathfrak{b}\lor\mathfrak{c}}
		\left(
		1 \land |s\cdot \mathfrak{x}|
		\right)
		\,\mathrm{d}s
		\ \le \
		g_{\gamma}(\mathfrak{z})
		\int_{\mathfrak{b}\land \mathfrak{c}}^{\mathfrak{b}\lor\mathfrak{c}}
		|s\cdot \mathfrak{x}|^{\gamma}
		\,\mathrm{d}s
		\\ &
		  \ \lesssim \
		  g_{\gamma}(\mathfrak{z})
		  \cdot
		  |\mathfrak{x}|^{\gamma}
      \cdot
		\left(|\mathfrak{b}|^{1+\gamma}+|\mathfrak{c}|^{1+\gamma}
		  \right)
    \ = \ 
		  \left(
		  |\mathfrak{y}|^{\gamma}
		  +
		  |\mathfrak{x}|^{\gamma}
		  \right)
		\cdot
		  g_{\gamma}(\mathfrak{z})
		  \cdot
		  \left(
		  |\mathfrak{b}|^{1+\gamma}
		  +
		  |\mathfrak{c}|^{1+\gamma}
		  \right)
		\,,
	\end{align*}
	where the multiplicative constant in $\lesssim$ depends only on $\gamma$.
	\\
	\textit{Case $\mathfrak{y}\neq 0$:}
	We make the split
	\begin{align*}
    &
		\int_{\mathfrak{b}\land \mathfrak{c}}^{\mathfrak{b}\lor\mathfrak{c}}
		\left(
		1 \land |s\cdot \mathfrak{x}|
		\right)
		\cdot
		g_{\gamma}
		\left(
		\mathfrak{z}
		+
		s\cdot \mathfrak{y}
		\right)
		\,\mathrm{d}s
    \\&
		\ \le \
		\int_{\mathfrak{b}\land \mathfrak{c}}^{\mathfrak{b}\lor\mathfrak{c}}
		\left(
			1 \land |s\cdot \mathfrak{x}|
			\right)
		\cdot
		\ind\{|s\cdot \mathfrak{y}|<1\}
    \cdot
		g_{\gamma}
		\left(
		\mathfrak{z}
		+
		s\cdot \mathfrak{y}
		\right)
    \,\mathrm{d}s
		+
		\int_{\mathfrak{b}\land \mathfrak{c}}^{\mathfrak{b}\lor\mathfrak{c}}
		\ind\{|s\cdot \mathfrak{y}|\ge 1\}
		\cdot
		g_{\gamma}
		\left(
		\mathfrak{z}
		+
		s\cdot \mathfrak{y}
		\right)
		\,\mathrm{d}s
		\,.
	\end{align*}
	\\
	\textit{Analysis of $
			\int_{\mathfrak{b}\land \mathfrak{c}}^{\mathfrak{b}\lor\mathfrak{c}}
			\left(
			1 \land |s\cdot \mathfrak{x}|
			\right)
			\ind\{|s\cdot \mathfrak{y}|<1\}
			\cdot
			g_{\gamma}
			\left(
			\mathfrak{z}
			+
			s\cdot \mathfrak{y}
			\right)
			\,\mathrm{d}s
		$:}
  By (i), $\gamma\in (0,1)$, and \eqref{eq:1990},
	\begin{align*}
		 &
		\ind\{|s\cdot \mathfrak{y}|<1\}
		\cdot
		\left(1 \land |s\cdot \mathfrak{x}|
		\right)
		\cdot
		g_{\gamma}
		\left(
		\mathfrak{z}
		+
		s\cdot \mathfrak{y}
		\right)
		\\ &
		  \ \lesssim \
		  \ind\{|s\cdot \mathfrak{y}|<1\}
		  \cdot
		  \left(1 \land |s\cdot \mathfrak{x}|
		  \right)
		\cdot
		  \left(
		  1
		  \lor |s\cdot \mathfrak{y}|
		  \right)^{1+\gamma}
		\cdot
		  g_{\gamma}(\mathfrak{z})
		\\ &
		  \ = \
		  \ind\{|s\cdot \mathfrak{y}|<1\}
		  \cdot
		  \left(1 \land |s\cdot \mathfrak{x}|
		  \right)
		\cdot
		  g_{\gamma}(\mathfrak{z})
		\\ &
		  \ \le \
		  \left(
		  1\land
		  |s\cdot \mathfrak{x}|
		  \right)
		\cdot
		  g_{\gamma}(\mathfrak{z})
		\\ &
		  \ \le \
		  |s\cdot \mathfrak{x}|^{\gamma}
		\cdot
		  g_{\gamma}(\mathfrak{z})
		  \,,
	\end{align*}
	so that
	\begin{align*}
		 &
		\int_{\mathfrak{b}\land \mathfrak{c}}^{\mathfrak{b}\lor\mathfrak{c}}
		\ind\{|s\cdot \mathfrak{y}|<1\}
		\cdot
		\left(
		1 \land |s\cdot \mathfrak{x}|
		\right)
		\cdot
		g_{\gamma}
		\left(
		\mathfrak{z}
		+
		s\cdot \mathfrak{y}
		\right)
		\,\mathrm{d}s
		\\ &
		  \ \lesssim \
		  |\mathfrak{x}|^{\gamma}
		\cdot
		  g_{\gamma}(\mathfrak{z})
		  \cdot
		  \int_{\mathfrak{b}\land \mathfrak{c}}^{\mathfrak{b}\lor\mathfrak{c}}
		|s|^{\gamma}
		\,\mathrm{d}s
		\\ &
		  \ \lesssim \
		  |\mathfrak{x}|^{\gamma}
		\cdot
		  g_{\gamma}(\mathfrak{z})
		  \cdot
		  \left(
		  |\mathfrak{b}|^{1+\gamma}+|\mathfrak{c}|^{1+\gamma}
		  \right)
		\\ &
		  \ \le \
		  \left(
		  |\mathfrak{y}|^{\gamma}
		  +
		  |\mathfrak{x}|^{\gamma}
		  \right)
		\cdot
		  g_{\gamma}(\mathfrak{z})
		  \cdot
		  \left(
		  |\mathfrak{b}|^{1+\gamma}
		  +
		  |\mathfrak{c}|^{1+\gamma}
		  \right)
		\,,
	\end{align*}
	where the constant in $\lesssim$ depends only on $\gamma$.
	\\
	\textit{Analysis of $
			\int_{\mathfrak{b}\land \mathfrak{c}}^{\mathfrak{b}\lor\mathfrak{c}}
			\ind\{|s\cdot \mathfrak{y}|\ge 1\}
			\cdot
			g_{\gamma}
			\left(
			\mathfrak{z}
			+
			s\cdot \mathfrak{y}
			\right)
			\,\mathrm{d}s
		$:}
	Note that $s\in (\mathfrak{b}\land \mathfrak{c}, \mathfrak{b}\lor \mathfrak{c})$ implies $|s| < |\mathfrak{b}|\lor |\mathfrak{c}|$. Moreover, we have
  \begin{align*}
    1\ \ge\  |\mathfrak{y}|\cdot(|\mathfrak{b}|\lor |\mathfrak{c}|)\,,\qquad 
    |\mathfrak{y}||\mathfrak{b}|\ \ge \ (1\lor |\mathfrak{y}||\mathfrak{c}|)
    \qquad
    \text{or}
    \qquad
    |\mathfrak{y}||\mathfrak{c}|\ \ge\  (1\lor |\mathfrak{y}||\mathfrak{b}|)\,.
  \end{align*}
  We consider the three cases separately.
	\\
	\textit{Case
		$1\ge |\mathfrak{y}|\cdot(|\mathfrak{b}|\lor |\mathfrak{c}|)$:}
  Since $|s\cdot \mathfrak{y}|<|\mathfrak{y}|\cdot(|\mathfrak{b}|\lor |\mathfrak{c}|)\leq 1$ for $|s|< (|\mathfrak{b}|\lor |\mathfrak{c}|)$,
	\begin{align*}
		\ind\{1\ge |\mathfrak{y}|(|\mathfrak{b}|\lor |\mathfrak{c}|)\}
		\int_{\mathfrak{b}\land \mathfrak{c}}^{\mathfrak{b}\lor \mathfrak{c}}
		\ind\{|s\cdot \mathfrak{y}|\ge 1\}
		\cdot g_{\gamma}(\mathfrak{z}+s\cdot \mathfrak{y})
		\,\mathrm{d}s
		\ = \ 0
		\,
	\end{align*}
	\\
	\textit{Case
		$|\mathfrak{y}||\mathfrak{b}|\ge (1\lor |\mathfrak{y}||\mathfrak{c}|)$:
	}
	Since $\mathfrak{y}\neq 0$, by the change of variables $s\mapsto s/\mathfrak{y}$,
	\begin{align*}
		 &
		\ind\{
		|\mathfrak{y}||\mathfrak{b}|\ge (1\lor |\mathfrak{y}||\mathfrak{c}|)
		\}
		\int_{\mathfrak{b}\land \mathfrak{c}}^{\mathfrak{b}\lor \mathfrak{c}}
		\ind\{|s\cdot \mathfrak{y}|\ge 1\}
		\cdot g_{\gamma}(\mathfrak{z}+s\cdot \mathfrak{y})
		\,\mathrm{d}s
		\\ &
		  \ \le \
		  \ind\{|\mathfrak{y}||\mathfrak{b}|\ge 1\}
		  \cdot
		  \frac{1}{|\mathfrak{y}|}
		\cdot
		  \int_{|s|\le |\mathfrak{y}||\mathfrak{b}|}
		g_{\gamma}(\mathfrak{z}+s)
		  \,\mathrm{d}s
		\\ &
		  \ \lesssim \
		  \ind\{|\mathfrak{y}||\mathfrak{b}|\ge 1\}
      \cdot
		  g_{\gamma}(\mathfrak{z})
      \cdot
		  \frac{1}{|\mathfrak{y}|}
      \cdot
      \left(
     1 \lor|\mathfrak{b}\cdot\mathfrak{y}|
      \right)^{1+\gamma}
		\\ &
		  \ \le \
		  g_{\gamma}(\mathfrak{z})
		  \cdot
		  |\mathfrak{b}|^{1+\gamma}
		\cdot
		  |\mathfrak{y}|^{\gamma}
		\\ &
		  \ \le \
		  \left(
		  |\mathfrak{y}|^{\gamma}
		  +
		  |\mathfrak{x}|^{\gamma}
		  \right)
		\cdot
		  g_{\gamma}(\mathfrak{z})
		  \cdot
		  \left(
		  |\mathfrak{b}|^{1+\gamma}
		  +
		  |\mathfrak{c}|^{1+\gamma}
		  \right)
		\,,
	\end{align*}
	where the third to last inequality follows from~(ii), and therefore the multiplicative constant in $\lesssim$ depends only on $\gamma$. 
	\\
	\textit{Case
		$|\mathfrak{y}||\mathfrak{c}|\ge (1\lor |\mathfrak{y}||\mathfrak{b}|)$:
	}
	The same argument yields
	\begin{align*}
		 &
		\ind\{|\mathfrak{c}||\mathfrak{y}|\ge (1\lor |\mathfrak{b}||\mathfrak{y}|)\}
		\int_{\mathfrak{b}\land \mathfrak{c}}^{\mathfrak{b}\lor \mathfrak{c}}
		\ind\{|s\cdot \mathfrak{y}|\ge 1\}
		\cdot g_{\gamma}(\mathfrak{z}+s\cdot \mathfrak{y})
		\,\mathrm{d}s
		\\
		 &
		\ \lesssim \
		g_{\gamma}(\mathfrak{z})
		\cdot
		|\mathfrak{c}|^{1+\gamma}
		\cdot
		|\mathfrak{y}|^{\gamma}
		\\ &
		  \ \le \
		  \left(
		  |\mathfrak{y}|^{\gamma}
		  +
		  |\mathfrak{x}|^{\gamma}
		  \right)
		\cdot
		  g_{\gamma}(\mathfrak{z})
		  \cdot
		  \left(
		  |\mathfrak{b}|^{1+\gamma}
		  +
		  |\mathfrak{c}|^{1+\gamma}
		  \right)
		\,,
	\end{align*}
	where the constant in $\lesssim$ depends only on $\gamma$.
\end{proof}
\begin{lemma}[Martingale difference inequality]
	\label{lem:bahr}
	\
	Let $p\in[1,2]$, and
	let $(Y_{m,i}\,, i=1, \ldots,m\,, m\in\NN )\subset L^p(\PP)$ be an
	array of random variables satisfying
	\begin{align}
		\label{eq:bahr:cond:1}
		\EE
		\Bigg[
			Y_{m,\ell + 1} \, \Bigg{|} \, \sum_{i=1}^{\ell}Y_{m,i}
			\Bigg]
		\ = \
		0
		\qquad\text{for all}\ 1\le \ell \le m-1
		\ \text{and all}\  m\in\NN\,.
	\end{align}
	\begin{enumerate}
		\item
		      It holds that
		      \begin{align}
			      \label{eq:bahr:fi}
			      \EE
			      \left|
			      \sum_{i=1}^{m}
			      Y_{m,i}
			      \right|^p
			      \ \le \
			      2
			      \sum_{i=1}^{m}
			      \EE
			      \left|
			      Y_{m,i}
			      \right|^p
			      \qquad\text{for all}\ m\in\NN\,.
		      \end{align}
		\item
		      If there exists a sequence $(Y_i)_{i\in\NN}\subset L^p(\PP)$ such that
		      \begin{align}
			      \label{eq:bahr:cond:2}
			      \lim_{m\to\infty}
			      \sum_{i=1}^m
			      Y_{m,i}
			      =
			      \sum_{i=1}^{\infty}
			      Y_{i}
			      \quad\text{$\PP$-almost surely, and }\quad
			      \liminf_{m\to\infty}
			      \sum_{i=1}^m
			      \EE|
			      Y_{m,i}
			      |^p
			      \le
			      \sum_{i=1}^\infty
			      \EE|
			      Y_{i}
			      |^p
			      \,,
		      \end{align}
		      then
		      \begin{align}
			      \label{eq:bahr:infty}
			      \EE
			      \left|
			      \sum_{i=1}^{\infty}
			      Y_i
			      \right|^p
			      \ \le \
			      2 \sum_{i=1}^{\infty}
			      \EE
			      \left|
			      Y_i
			      \right|^p
			      \,.
		      \end{align}
	\end{enumerate}
\end{lemma}
\begin{proof}
	See the proof of Lemma~A.1 in~
	\citet{scheffelCentralLimitTheory2025}.
\end{proof}
\begin{lemma}[Moments of the marginal distribution]
	\label{lem:fm}
	Let Assumptions~\ref{asu:f} and~\ref{asu:coef} hold.
	Then
	$X_0\in L^r(\PP)$
	for all $r\in [1,\alpha)$.
\end{lemma}
\begin{proof}
	Since $L^q(\PP)\subset L^p(\PP)$ for $1\le p <  q$, it suffices to show the statement for $r\in (1/(1-d),\alpha)$.
	To this end, we apply Lemma~\ref{lem:bahr} to the array $(Y_{m,i})=(a_i\varepsilon_{-i}: i=0,\ldots,m,m\in\NN)$ and the sequence $(Y_i)=(a_i\varepsilon_{-i})_{i\ge 0}$. Since $r<\alpha$ by assumption, $(Y_{m,i}),(Y_i)\subset L^r(\PP)$. Since the innovations are independent and centered, Condition~\eqref{eq:bahr:cond:1} is satisfied. The other conditions of Lemma~\ref{lem:bahr} are clearly met, so that
	\begin{align*}
		\EE[|X_0|^r]
		\ = \
		\EE
		\left[
			\left|
			\sum_{i=0}^\infty
			Y_i
			\right|^r
			\right]
		\ \lesssim \
		\sum_{i=0}^\infty
		\EE[|Y_i|^r]
		\ = \
		\EE[|\varepsilon|^r]
		\sum_{i=0}^\infty
		|a_i|^r
		\ \lesssim \
		\sum_{i=0}^\infty
		i^{-(1-d)r}
		\,.
	\end{align*}
	The latter expression is finite by the assumption $1/(1-d)<r$, so that $(1-d)r > 1$.
	This shows that $X_0\in L^r(\PP)$.
\end{proof}
\begin{lemma}[Decay rate of the moment of the remainder]
	\label{lem:x0k}
	Let Assumptions~\ref{asu:f} and~\ref{asu:coef} hold.
	Then, for all $k\ge 1$ and $r\in(1/(1-d),\alpha)$,
	\begin{align*}
    \EE
    \left[
    \left|
    \sum_{i=k+1}^{\infty}a_i\varepsilon_{-i}
    \right|^r
    \right]
		\ \lesssim \
		k^{1-(1-d)r}
		\,,
	\end{align*}
	where $\lesssim$ means inequality up to a multiplicative constant that depends on the coefficients $(a_i)$, the distribution of $\varepsilon$, and $r$, and is independent of $k$.
\end{lemma}
\begin{proof}
	We define the array and limiting sequence
	\begin{align*}
		(Y_{m,i})
		\ = \
		(a_{k+i}\varepsilon_{-(k+i)}\,,i=1,\ldots,m\,,m\in\NN)
		\qquad\text{and}\qquad
		(Y_i)
		\ = \
		(a_{k+i}\varepsilon_{-(k+i)})_{i\in\NN}
		\,.
	\end{align*}
	Note that \eqref{eq:bahr:cond:1} and \eqref{eq:bahr:cond:2} are satisfied for $(Y_{m,i})$ and $(Y_i)$.
	Then we may write
	\begin{align*}
		\sum_{i=k+1}^\infty
		a_i\varepsilon_{-i}
		\ = \
		\sum_{i=1}^\infty
		a_{k+i}\varepsilon_{-(k+i)}
		\ = \
		\sum_{i=1}^\infty
		Y_i
		\,.
	\end{align*}
	The $\varepsilon_j$ are independent, centered and belong to $L^r(\PP)$, so applying Inequality~\eqref{eq:bahr:infty}
	in Lemma~\ref{lem:bahr}(ii), we obtain
	\begin{align*}
    \EE
    \left[
    \left|
    \sum_{i=k+1}^{\infty}a_i\varepsilon_{-i}
    \right|^r
    \right]
		\ \lesssim \
		\sum_{i=k+1}^\infty
		|a_i|^{r}
		\EE|\varepsilon|^r
		\ \lesssim \
		\sum_{i=k+1}^\infty
		i^{-(1-d)r}
		\ \lesssim \
		k^{1-(1-d)r}
		\ \le \
		1
		\,,
	\end{align*}
	where we used $a_i i^{1-d}\sim c_a$ and $(1-d)r>1$.
	The multiplicative constant in $\lesssim$ depends on the coefficients $(a_i)$, the distribution of $\varepsilon$, and $r$, and is independent of $k$.
	This proves the asserted upper bound.
\end{proof}

\section{Multivariate techniques}
We define, for $k\ge 0$, the $k$-truncated time series $(X_{t,k})$ by
$
	X_{t,k}
	:=
	\sum_{i=0}^k
	a_i \varepsilon_{t-i}
$
and make the convention $X_{t,\infty}=X_t$ so that $X_{t,k}=X_{t,\infty}-\sum_{i=k+1}^{\infty} a_i \varepsilon_{t-i}$. 
For $i\in\{1,\ldots,T\}$,
let
$\mathbf{e}_{i}\in\RR^T$ be the $i$-th unit vector of the canonical basis, and define the coefficient matrix $\mathbf{A}_T \in \RR^{T\times T}$ by
\begin{align}
	\label{def:A_T}
	\mathbf{A}_T
	\ := \
	\begin{bmatrix}
		a_0       & 0      & \cdots & 0      \\
		a_1     &  a_0     & \ddots & \vdots \\
		\vdots  & \ddots & \ddots & 0      \\
		a_{T-1} & \ldots & a_1    & a_0
	\end{bmatrix}
  \ = \ 
	\begin{bmatrix}
		1       & 0      & \cdots & 0      \\
		a_1     & 1      & \ddots & \vdots \\
		\vdots  & \ddots & \ddots & 0      \\
		a_{T-1} & \ldots & a_1    & 1
	\end{bmatrix}
	\,.
\end{align}
\subsection{Density of the joint distribution and its gradient}
\begin{lemma}
	\label{lem:multi_density}
	Let Assumptions~\ref{asu:f} and~\ref{asu:coef} hold.
	Then, for all $k\in \NN_0 \cup\{\infty\}$, and all $\mathbf{x}_{}\in\RR^T$, the following holds for the density of $[X_{\ell,k+\ell}]_\ell$ and its gradient:
	\begin{enumerate}
		\item
		      \begin{align*}
			      f_{[X_{\ell,k+\ell}]_\ell}(\mathbf{x})
			       &
			       \ = \
			      \int_{\RR^T}
			      \prod_{i=1}^{T}
			      f_\varepsilon(\mathbf{e}_{i}^{\top}\cdot\mathbf{A}_{T}^{-1}
			      \cdot
				      (\mathbf{x}_{}-\widetilde{\mathbf{x}_{}}_{})
			      )
			      \,\mathrm{d}\PP_{[X_{\ell,k+\ell}-X_{\ell,\ell}]_\ell}(\widetilde{\mathbf{x}})
            \\&
			      \ = \
			      \EE
			      \left[
				      \prod_{i=1}^{T}
				      f_\varepsilon(\mathbf{e}_{i}^{\top}\cdot\mathbf{A}_{T}^{-1}\cdot[x_\ell - (X_{\ell,k+\ell}-X_{\ell,\ell})]_\ell)
				      \right]
			      \,.
		      \end{align*}
		\item
		      \begin{align*}
			       &
			      \nabla
			      f_{[X_{\ell,k+\ell}]_\ell}
			      (\mathbf{x}_{})
			      \\
			       &
			      \ = \
			      \left(
			      \mathbf{A}_{T}^{-1}
			      \right)^{\top}
			      \left[
				      \int_{\RR^T}
				      f'_\varepsilon(\mathbf{e}_{i}^{\top}\cdot\mathbf{A}_{T}^{-1}\cdot
					      (\mathbf{x}_{}-\widetilde{\mathbf{x}_{}}_{})
				      )
				      \prod_{j\neq i}
				      f_\varepsilon(\mathbf{e}_{j}^{\top}\cdot\mathbf{A}_{T}^{-1}\cdot
					      (\mathbf{x}_{}-\widetilde{\mathbf{x}_{}}_{})
				      )
				      \,\mathrm{d}
				      \PP_{[X_{\ell,k+\ell}-X_{\ell,\ell}]_\ell}
				      (
				      \widetilde{\mathbf{x}_{}}_{}
				      )
				      \right]_{i\in\{1,\ldots,T\}}
			      \\ &
			        \ = \
			        \left(
			        \mathbf{A}_{T}^{-1}
			        \right)^{\top}
			      \\ & \qquad
			        \times
			        \left[
				        \EE
				        \left[
					        f'_\varepsilon(\mathbf{e}_{i}^{\top}\cdot\mathbf{A}_{T}^{-1}\cdot
					        (\mathbf{x}_{}-
					        [X_{\ell,k+\ell}-X_{\ell,\ell}]_\ell
					        )
					      )
					        \prod_{j\neq i}
					      f_\varepsilon(\mathbf{e}_{j}^{\top}\cdot\mathbf{A}_{T}^{-1}\cdot
					        (\mathbf{x}_{}-
					        [X_{\ell,k+\ell}-X_{\ell,\ell}]_\ell
					        )
					      )
					        \right]
				      \right]_{i\in\{1,\ldots,T\}}
		      \end{align*}
		\item
    For fixed $\mathbf{z}\in\RR^T$,
    it holds
    $\mathbf{z}^{\top}\cdot\nabla f_{[X_{\ell,k+\ell}]_{\ell}}\in L^1(\RR^T)\cap C^{0}(\RR^T)$.
    \item
            If Assumption~\ref{asu:f_more} holds, then
            \begin{align*}
            \left(
              \mathbf{x}\ \mapsto\ 
              x_{i-1} \frac{\partial}{\partial x_{i-1}} f_{[X_{\ell,k+\ell}]_{\ell}}(\mathbf{x})
            \right)
              \ \in\  L^1(\RR^T)\cap C^{0}(\RR^T)
              \qquad\text{for all $i\in\{1,\ldots,T\}$.}
            \end{align*}
            In particular,
            $(\mathbf{x}\mapsto \mathbf{x}^{\top}\cdot \nabla f_{[X_{\ell,k+\ell}]_{\ell}}(\mathbf{x}))\in L^1(\RR^T)\cap C^{0}(\RR^T)$.
	\end{enumerate}
\end{lemma}
\begin{proof}[\textbf{Proof of (i):}]
	\ \\
	\textbf{Case $k<\infty$:}
	By the stationarity of the (truncated) time series it holds
	\begin{align*}
		f_{
				[X_{\ell,k+\ell}]_\ell
			}(\mathbf{x})
		\ = \
		f_{[X_{k+\ell,k+\ell}]_\ell}(\mathbf{x})
		\ = \
		\int_{\RR^k}
		f_{[X_{\ell,\ell}]_{\ell\in\{0,\ldots,T-1+k\}}}
		\left(
		\begin{bmatrix}
				\mathbf{z} \\
				\mathbf{x}
			\end{bmatrix}
		\right)
		\,\mathrm{d}\mathbf{z}
		\,.
	\end{align*}
	We use that
	\begin{align*}
		[X_{\ell,\ell}]_{\ell\in\{0,\ldots, T-1+k\}}
		\ = \
		\mathbf{A}_{T+k}
		\cdot
		[\varepsilon_\ell]_{\ell\in\{0,\ldots,T-1+k\}}
		\qquad\text{and}\qquad
		\det \mathbf{A}_{T+k}
		\ =\ 1
		\,,
	\end{align*}
	to make a change of variables
	\begin{align*}
		 &
		f_{[X_{\ell,\ell}]_{\ell\in\{0,\ldots,T-1+k\}}}
		\left(
		\begin{bmatrix}
				\mathbf{z} \\
				\mathbf{x}
			\end{bmatrix}
		\right)
		\ = \
		f_{[\varepsilon_{\ell}]_{\ell\in\{0,\ldots,T-1+k\}}}
		\left(
		\mathbf{A}_{T+k}^{-1}
		\cdot
		\begin{bmatrix}
			\mathbf{z} \\
			\mathbf{x}
		\end{bmatrix}
		\right)
		\,.
	\end{align*}
	Our goal is to simplify the expression $
		\mathbf{A}_{T+k}^{-1}
		\cdot
		\begin{bmatrix}
			\mathbf{z} \\
			\mathbf{x}
		\end{bmatrix}
	$.
	Consider $\mathbf{A}_{T+k}$ as a block matrix, that is
	\[
		\mathbf{A}_{T+k}
		\ = \
		\begin{bmatrix}
			\mathbf{A}_k     & \mathbf{0}_{k,T} \\
			\mathbf{A}_{T,k} & \mathbf{A}_{T}
		\end{bmatrix}
		\ \mbox{ where } \ \mathbf{A}_{T,k}
		\ := \
		\begin{bmatrix}
			a_k       & a_{k-1}   & \cdots & a_1    \\
			a_{k+1}   & a_{k}     & \cdots & a_2    \\
			\vdots    & \vdots    & \vdots & \vdots \\
			a_{T-1+k} & a_{T-2+k} & \cdots & a_{T}
		\end{bmatrix}
		\in \RR^{T\times k}
	\]
	and $\mathbf{0}_{k,T}$ is the zero matrix in $\RR^{k\times T}$.
	With Schur's formula for inversion of block matrices we get
	\begin{align*}
		\mathbf{A}^{-1}_{T+k}
		\ = \
		\begin{bmatrix}
			\mathbf{A}^{-1}_k                                     & \mathbf{0}_{k, T}   \\
			-\mathbf{A}^{-1}_{T}\mathbf{A}_{T,k}\mathbf{A}_k^{-1} & \mathbf{A}^{-1}_{T}
		\end{bmatrix}
		\,.
	\end{align*}
	For all $\mathbf{z}\in \mathbb{R}^k$ and $\mathbf{x}\in \mathbb{R}^T$ it then holds
	\begin{align*}
		\mathbf{A}^{-1}_{T+k}
		\begin{bmatrix}
			\mathbf{z} \\
			\mathbf{x}
		\end{bmatrix}
		\ = \
		\begin{bmatrix}
			\mathbf{A}_k^{-1}
			\mathbf{z} \\
			\mathbf{A}^{-1}_{T}
			\left(
			\mathbf{x}
			-
			\mathbf{A}_{T,k}
			\mathbf{A}_{k}^{-1}
			\mathbf{z}
			\right)
		\end{bmatrix}
		\,.
	\end{align*}
	Making the change of variables $\mathbf{A}^{-1}_k \mathbf{z}\mapsto \mathbf{z}$  and recalling that $\det \mathbf{A}^{-1}_k=1$ we get
	\begin{align*}
		 &
		\int_{\RR^k}
		f_{[\varepsilon_{\ell}]_{\ell\in\{0,\ldots,T-1+k\}}}
		\left(
		\mathbf{A}^{-1}_{T+k}
		\begin{bmatrix}
			\mathbf{z} \\
			\mathbf{x}
		\end{bmatrix}
		\right)
		\,\mathrm{d}\mathbf{z}
		\\ &
		  \ = \
		  \int_{\RR^k}
		f_{[\varepsilon_{\ell}]_{\ell\in\{0,\ldots,T-1+k\}}}
		\left(
		  \begin{bmatrix}
				\mathbf{z} \\
				\mathbf{A}^{-1}_{T}
				\left(
				\mathbf{x}
				-
				\mathbf{A}_{T,k}
				\mathbf{z}
				\right)
			\end{bmatrix}
		  \right)
		\,\mathrm{d}\mathbf{z}
		\,.
	\end{align*}
	Using that the innovations are i.i.d., this equals
	\begin{align*}
		 &
		\int_{\RR^k}
		f_{[\varepsilon_{\ell}]_{\ell\in\{0,\ldots,T-1+k\}}}
		\left(
		\begin{bmatrix}
				\mathbf{z} \\
				\mathbf{A}^{-1}_{T}
				\left(
				\mathbf{x}
				-
				\mathbf{A}_{T,k}
				\mathbf{z}
				\right)
			\end{bmatrix}
		\right)
		\,\mathrm{d}\mathbf{z}
		\ = \
		\EE
		\left[
			\prod_{i=1}^{T}
			f_{\varepsilon}
			\left(
			\mathbf{e}_i^\top \mathbf{A}_T^{-1}
			\left(
			\mathbf{x}
			-
			\mathbf{A}_{T,k}
			\cdot
			[\varepsilon_\ell]_{\ell\in\{0,\ldots,k-1\}}
			\right)
			\right)
			\right]
		\\ &
		  \ = \
		  \EE
		  \left[
			  \prod_{i=1}^{T}
			f_{\varepsilon}
			\left(
			  \mathbf{e}_i^\top \mathbf{A}_T^{-1}
			  \left(
			  \mathbf{x}
			  -
			  \mathbf{A}_{T,k}
			  \cdot
			  [\varepsilon_{-k+\ell}]_{\ell\in\{0,\ldots,k-1\}}
			  \right)
			\right)
			\right]
		\,.
	\end{align*}
	We have the simplification
	\begin{align*}
		\mathbf{A}_{T,k}
		\cdot
		[\varepsilon_{-k+\ell}]_{\ell\in\{0,\ldots, k-1\}}
		 &
		\ = \
		\left[
			\sum_{i=0}^{k-1}
			a_{k+\ell-i}
			\varepsilon_{-k+ i}
			\right]_{\ell\in\{0,\ldots,T-1\}}
		\\ &
		  \ = \
		  \left[
			  \sum_{i=1+\ell}^{k+\ell}
			a_{i}
			\varepsilon_{\ell-i}
			\right]_{\ell\in\{0,\ldots,T-1\}}
		\\ &
		  \ = \
		  \left[
			  \sum_{i=0}^{k+\ell}
			a_{i}
			\varepsilon_{\ell-i}
			\ - \
			  \sum_{i=0}^{\ell}
			a_{i}
			\varepsilon_{\ell-i}
			\right]_{\ell\in\{0,\ldots,T-1\}}
		\\ &
		  \ = \
		  \left[
			  X_{\ell, k+\ell}
			\ - \
			  X_{\ell, \ell}
			\right]_{\ell\in\{0,\ldots,T-1\}}
		\,,
	\end{align*}
	so that
	\begin{align*}
		 &
		\EE
		\left[
			\prod_{i=1}^{T}
			f_{\varepsilon}
			\left(
			\mathbf{e}_i^\top \mathbf{A}_T^{-1}
			\left(
			\mathbf{x}
			-
			\mathbf{A}_{T,k}
			\cdot
			[\varepsilon_{-k+\ell}]_{\ell\in\{0,\ldots,k-1\}}
			\right)
			\right)
			\right]
		\\ &
		  \ = \
		  \EE
		  \left[
			  \prod_{i=1}^{T}
			f_{\varepsilon}
			\left(
			  \mathbf{e}_i^\top \cdot\mathbf{A}_T^{-1}
			  \cdot
			  [x_\ell
				  -
				  (
				  X_{\ell, k+\ell}
				\ - \
				  X_{\ell, \ell}
				)
				  ]_\ell
			  \right)
			\right]
		\,.
	\end{align*}
	This proves the claimed identity for $f_{[X_{\ell,k+\ell}]_{\ell}}$ if $k<\infty$.
	\\
	\textbf{Case $k=\infty$:}
	Using the continuity and boundedness of $f_\varepsilon$ and the weak convergence $[X_{\ell,k+\ell}-X_{\ell,\ell}]_\ell\to [X_{\ell}-X_{\ell,\ell}]_\ell$ as $k\to\infty$, we find that
	$f_{[X_{\ell,k+\ell}]_{\ell}}$ converges pointwise to the candidate density
	\begin{align*}
		\mathbf{x}_{}
		\ \mapsto\
		\EE
		\left[
			\prod_{i=1}^T
			f_{\varepsilon}(\mathbf{e}_{i}^{\top}\cdot \mathbf{A}_{T}^{-1}\cdot
			(\mathbf{x}_{}-[X_{\ell}-X_{\ell,\ell}]_{\ell}))
			\right]
		\,.
	\end{align*}
	To show that this is $f_{[X_{\ell}]_{\ell}}$, that is, the density of $[X_{\ell}]_{\ell}$, we also need to show that, for any bounded and continuous function $h:\RR^T \to \RR$, it holds that
	\begin{align*}
		\EE[h([X_{\ell,k+\ell}]_{\ell})]
		\ = \
		\int_{\RR^T}
		h(\mathbf{x}_{})
		f_{[X_{\ell,k+\ell}]_{\ell}}(\mathbf{x}_{})
		\,\mathrm{d}\mathbf{x}_{}
		\ \to \
		\int_{\RR^T}
		h(\mathbf{x}_{})
		\EE
		\left[
			\prod_{i=1}^T
			f_{\varepsilon}(\mathbf{e}_{i}^{\top}\cdot \mathbf{A}_{T}^{-1}\cdot
			(\mathbf{x}_{}-[X_{\ell}-X_{\ell,\ell}]_{\ell}))
			\right]
		\,\mathrm{d}\mathbf{x}_{}
		\,,
	\end{align*}
	as $k\to\infty$.
	To this end, we apply Pratt's version of dominated convergence
	\citep{prattInterchangingLimitsIntegrals1960}.
  First, note that, by $\det \mathbf{A}_{T}^{-1}=1/(\det \mathbf{A}_T) = 1$ and the result for $k<\infty$, it holds
	\begin{align*}
		 &
		\int_{\RR^T}
		h(\mathbf{x}_{})
		f_{[X_{\ell,k+\ell}]_{\ell}}(\mathbf{x}_{})
		\,\mathrm{d}\mathbf{x}_{}
		\ = \
		\int_{\RR^T}
		h(\mathbf{A}_{T}\cdot\mathbf{x}_{})
		f_{[X_{\ell,k+\ell}]_{\ell}}(\mathbf{A}_{T}\cdot\mathbf{x}_{})
		\,\mathrm{d}\mathbf{x}_{}
		\\ &
		  \ = \
		  \int_{\RR^T}
		h(\mathbf{A}_{T}\cdot\mathbf{x}_{})
		  \EE
		  \left[
			  \prod_{i=1}^T
			  f_{\varepsilon}(x_{i-1}-\mathbf{e}_{i}^{\top}\cdot \mathbf{A}_{T}^{-1}\cdot
			  [X_{\ell,k+\ell}-X_{\ell,\ell}]_{\ell})
			  \right]
		\,\mathrm{d}\mathbf{x}_{}
		\,.
	\end{align*}
	By Assumption~\ref{asu:f} and the boundedness of $h$, it holds
	\begin{align*}
		 &
		\left|
		h(\mathbf{x}_{})
		\EE
		\left[
			\prod_{i=1}^T
			f_{\varepsilon}(x_{i-1}-\mathbf{e}_{i}^{\top}\cdot \mathbf{A}_{T}^{-1}\cdot
			[X_{\ell,k+\ell}-X_{\ell,\ell}]_{\ell})
			\right]
		\right|
		\\ &
		  \ \lesssim \
		  \EE
		  \left[
			  \prod_{i=1}^T
			  g_{\alpha-1}(x_{i-1}-\mathbf{e}_{i}^{\top}\cdot \mathbf{A}_{T}^{-1}\cdot
			  [X_{\ell,k+\ell}-X_{\ell,\ell}]_{\ell})
			  \right]
		\ =: \
		  \varphi_{k}(\mathbf{x}_{})
		\,.
	\end{align*}
	Since $g_{\alpha-1}$ is bounded and continuous, it follows from the weak convergence of
	$[X_{\ell,k+\ell}-X_{\ell,\ell}]_{\ell}$
	to
	$[X_{\ell}-X_{\ell,\ell}]_{\ell}$ as $k\to\infty$ that
	\begin{align*}
		\varphi_{k}(\mathbf{x}_{})
		\ \to \
		\varphi_{\infty}(\mathbf{x}_{})
		\ := \
		\EE
		\left[
			\prod_{i=1}^T
			g_{\alpha-1}(x_{i-1}-
			\mathbf{e}_{i}^{\top}\cdot \mathbf{A}_{T}^{-1}\cdot
			[X_{\ell}-X_{\ell,\ell}]_{\ell})
			\right]
		\qquad\text{pointwise for all $\mathbf{x}_{}\in\RR^T$.}\qquad
	\end{align*}
	Making the shift of variables
	$
		x_{i-1}-
		\mathbf{e}_{i}^{\top}\cdot \mathbf{A}_{T}^{-1}\cdot
		[X_{\ell,k+\ell}-X_{\ell,\ell}]_{\ell})
		\mapsto
		x_{i-1}
	$
	for each component $i\in\{1,\ldots,T\}$, we get that
	\begin{align*}
		\int_{\RR^T}
		\varphi_{k}(\mathbf{x}_{})
		\,\mathrm{d}\mathbf{x}_{}
		\ = \
		\int_{\RR^T}
		\varphi_{\infty}(\mathbf{x}_{})
		\,\mathrm{d}\mathbf{x}_{}
		\ = \
		\prod_{i=1}^{T}
		\int_{\RR}
		g_{\alpha-1}(x_{i-1})
		\,\mathrm{d}x_{i-1}
		\ = \
		\norm{g_{\alpha-1}}_{L^1(\RR)}^{T}
		\ < \
		\infty
		\,,
	\end{align*}
	so that applying Pratt's version of dominated convergence
  \cite[]{prattInterchangingLimitsIntegrals1960}
	allows us to conclude the proof.
	\\
	\textbf{Proof of (ii):}
	We show this for the $l$-th component of the gradient, that is $\partial/\partial x_{l-1}$.
	To this end, we use the identity
	\begin{align*}
		\prod_{i=1}^T a_i \ -\  \prod_{i=1}^T b_i
		\  =\
		\sum_{i=1}^T (\prod_{j<i} a_j)(a_i-b_i)(\prod_{j>i} b_j)
		\,,
	\end{align*}
	with
	\begin{align*}
		a_j
		\ :=\
		f_\varepsilon
		(
		\mathbf{e}_{j}^{\top}
		\mathbf{A}_{T}^{-1}
			(\mathbf{x}-[X_{\ell,k+\ell}-X_{\ell,\ell}]_\ell+h\cdot\mathbf{e}_{l})
		)
		\qquad\text{and}\qquad
		b_j
		\ :=\
		f_\varepsilon
		(
		\mathbf{e}_{j}^{\top}
		\mathbf{A}_{T}^{-1}
			(\mathbf{x}-[X_{\ell,k+\ell}-X_{\ell,\ell}]_\ell)
		)
	\end{align*}
	for a fixed $h\neq 0$, and write, using (i),
	\begin{align*}
		 & \frac{f_{
					[X_{\ell,k+\ell}]_\ell}(\mathbf{x}+h\cdot\mathbf{e}_{l})
			-
			f_{[X_{\ell,k+\ell}]_\ell}(\mathbf{x})
		}{h}
		\\ &
		  \ = \
		  \sum_{i=1}^T \EE\left[
			  \frac{
				  f_\varepsilon
				  (
				  \mathbf{e}_{i}^{\top}
				\mathbf{A}_{T}^{-1}
				  (\mathbf{x}-[X_{\ell,k+\ell}-X_{\ell,\ell}]_\ell+h\cdot\mathbf{e}_{l})
				)
				  -
				  f_\varepsilon
				  (
				  \mathbf{e}_{i}^{\top}
				\mathbf{A}_{T}^{-1}
				  (\mathbf{x}-[X_{\ell,k+\ell}-X_{\ell,\ell}]_\ell)
				)
				  }{h} \right.
		\\ &
			  \qquad \times \left.
			  \prod_{j<i}
			f_\varepsilon
			  (
			  \mathbf{e}_{j}^{\top}
			\mathbf{A}_{T}^{-1}
				  (\mathbf{x}-[X_{\ell,k+\ell}-X_{\ell,\ell}]_\ell+h\cdot\mathbf{e}_{l})
			)
			\prod_{j>i}
			f_\varepsilon
			  (
			  \mathbf{e}_{j}^{\top}
			\mathbf{A}_{T}^{-1}
				  (\mathbf{x}-[X_{\ell,k+\ell}-X_{\ell,\ell}]_\ell)
			) \right]
		\displaybreak[0]\\ &
		  \ = \
		  \sum_{i=1}^T
		  \int_{\RR^T}
    \,\mathrm{d}\PP_{
    [X_{\ell,k+\ell}-X_{\ell,\ell}]_{\ell}}(\widetilde{\mathbf{x}_{}}_{})
		\
		  \frac{
			  f_\varepsilon
			  (
			  \mathbf{e}_{i}^{\top}
			\mathbf{A}_{T}^{-1}
			  (\mathbf{x}-\widetilde{\mathbf{x}}+h\cdot\mathbf{e}_{l})
			)
			  -
			  f_\varepsilon
			  (
			  \mathbf{e}_{i}^{\top}
			\mathbf{A}_{T}^{-1}
			  (\mathbf{x}-\widetilde{\mathbf{x}})
			)
			  }{h}
		\\ &
		  \qquad
		  \qquad
		  \times
		  \prod_{j<i}
		f_\varepsilon
		  (
		  \mathbf{e}_{j}^{\top}
		\mathbf{A}_{T}^{-1}
			  (\mathbf{x}-\widetilde{\mathbf{x}}+h\cdot\mathbf{e}_{l})
		)
		\prod_{j>i}
		f_\varepsilon
		  (
		  \mathbf{e}_{j}^{\top}
		\mathbf{A}_{T}^{-1}
			  (\mathbf{x}-\widetilde{\mathbf{x}})
		)
		\,.
	\end{align*}
	We show that we can apply dominated convergence to move
	the limit as $h\to 0$ inside the integral. To this end, note that the terms in the product $j<i$ and $j>i$ are bounded by Assumption~\ref{asu:f}. Also note that
	\begin{align*}
		 &
		\left|
		\frac{
			f_\varepsilon
			(
			\mathbf{e}_{i}^{\top}
			\cdot
			\mathbf{A}_{T}^{-1}
			\cdot
			(\mathbf{x} - \widetilde{\mathbf{x}}+ h\cdot \mathbf{e}_{l})
			)
			-
			f_\varepsilon
			(
			\mathbf{e}_{i}^{\top}
			\cdot
			\mathbf{A}_{T}^{-1}
			\cdot
			(\mathbf{x} - \widetilde{\mathbf{x}})
			)
		}{h}
		\right|
		\\ &
		  \ \le \
		  \frac{1}{|h|}
		\int_{0\land h}^{0\lor h}
		\left|
		  f'_\varepsilon
		  (
		  \mathbf{e}_{i}^{\top}
		\cdot
		  \mathbf{A}_{T}^{-1}
		\cdot
			  (\mathbf{x} - \widetilde{\mathbf{x}}+s\cdot\mathbf{e}_{l})
		)
		\right|
		  \,\mathrm{d}s
	\end{align*}
	and
	$
		\left|
		f'_\varepsilon
		(
		\mathbf{e}_{i}^{\top}
		\cdot
		\mathbf{A}_{T}^{-1}
		\cdot
			(\mathbf{x} - \widetilde{\mathbf{x}}+s\cdot\mathbf{e}_{l})
		)
		\right|
		\ \lesssim \
		(1\lor (|s||\mathbf{e}_{i}^{\top}\cdot \mathbf{A}_{T}^{-1}\cdot \mathbf{e}_{l}|))^{1+\gamma}
	$, by
	Assumption~\ref{asu:f}
	and
	Lemma~\ref{lem:5.1}(i). Therefore,
	\begin{align*}
		 &
		\left|
		\frac{
			f_\varepsilon
			(
			\mathbf{e}_{i}^{\top}
			\cdot
			\mathbf{A}_{T}^{-1}
			\cdot
			(\mathbf{x} - \widetilde{\mathbf{x}}+ h\cdot \mathbf{e}_{l})
			)
			-
			f_\varepsilon
			(
			\mathbf{e}_{i}^{\top}
			\cdot
			\mathbf{A}_{T}^{-1}
			\cdot
			(\mathbf{x} - \widetilde{\mathbf{x}})
			)
		}{h}
		\right|
		\\ &
		  \ \le \
		  \frac{1}{|h|}
		\int_{0\land h}^{0\lor h}
		  (1\lor |s||\mathbf{e}_{i}^{\top}\cdot \mathbf{A}_{T}^{-1}\cdot \mathbf{e}_{l}|)^{1+\gamma}
		\,\mathrm{d}s
		  \ \le \
		  (1\lor |h||\mathbf{e}_{i}^{\top}\cdot \mathbf{A}_{T}^{-1}\cdot \mathbf{e}_{l}|)^{1+\gamma}
		\ \le \
		  2
	\end{align*}
	for $h$ small enough, whatever the values of $\mathbf{x}$ and $\widetilde{\mathbf{x}}$.
	In order to identify 
	the limit as $h\to 0$ inside the integral, we note that
	\begin{align*}
		 &
		\lim_{h\to 0}
		\frac{
			f_\varepsilon
			(
			\mathbf{e}_{i}^{\top}
			\mathbf{A}_{T}^{-1}
			(\mathbf{x}_{}-\widetilde{\mathbf{x}_{}}_{}
			+h\cdot\mathbf{e}_{l})
			)
			-
			f_\varepsilon
			(
			\mathbf{e}_{i}^{\top}
			\mathbf{A}_{T}^{-1}
			(\mathbf{x}_{}-\widetilde{\mathbf{x}_{}}_{})
			)
		}{h}
		\\
		 &
		\ = \
		\frac{\partial}{\partial x_{l-1}}
		f_\varepsilon
		(
		\mathbf{e}_i^{\top}
		\cdot
		\mathbf{A}_{T}^{-1}
		\cdot
		(\mathbf{x} - \widetilde{\mathbf{x}})
		)
		\ = \
		\mathbf{e}_i^{\top}
		\cdot
		\mathbf{A}_{T}^{-1}
		\cdot
		\mathbf{e}_{l}
		\cdot
		f'_\varepsilon
		(
		\mathbf{e}_i^{\top}
		\cdot
		\mathbf{A}_{T}^{-1}
		\cdot
		(\mathbf{x} - \widetilde{\mathbf{x}})
		)
		\,.
	\end{align*}
	Besides, by
	the continuity of $f_{\varepsilon}$
	\begin{align*}
		\lim_{h\to 0}
		\prod_{j<i}
		f_\varepsilon
		(
		\mathbf{e}_{j}^{\top}
		\mathbf{A}_{T}^{-1}
			(\mathbf{x} - \widetilde{\mathbf{x}}
		+h\cdot\mathbf{e}_{l})
		)
		\prod_{j>i}
		f_\varepsilon
		(
		\mathbf{e}_{j}^{\top}
		\mathbf{A}_{T}^{-1}
			(\mathbf{x} - \widetilde{\mathbf{x}})
		)
		\ = \
		\prod_{j\neq i}
		f_\varepsilon
		(
		\mathbf{e}_{j}^{\top}
		\mathbf{A}_{T}^{-1}
			(\mathbf{x} - \widetilde{\mathbf{x}})
		)
		\,.
	\end{align*}
	It follows that
	\begin{align*}
		 &
		\lim_{h\to 0}
		\frac{f_{
					[X_{\ell,k+\ell}]_\ell}(\mathbf{x}+h\cdot\mathbf{e}_{l})
			-
			f_{[X_{\ell,k+\ell}]_\ell}(\mathbf{x})
		}{h}
		\\ &
		  \ = \
		  \sum_{i=1}^T
		  \mathbf{e}_i^{\top}
		\cdot
		  \mathbf{A}_{T}^{-1}
		\cdot
		  \mathbf{e}_{l}
		\int_{\RR^T}
		f'_\varepsilon
		  (
		  \mathbf{e}_i^{\top}
		  \cdot
		  \mathbf{A}_{T}^{-1}
		\cdot
		  (\mathbf{x} - \widetilde{\mathbf{x}})
		)
		\
		  \cdot
		  \
		  \prod_{j\neq i}
		f_\varepsilon
		  (
		  \mathbf{e}_{j}^{\top}
		\mathbf{A}_{T}^{-1}
			  (\mathbf{x} - \widetilde{\mathbf{x}})
		)
		\,\mathrm{d}\PP_{[X_{\ell,k+\ell}-X_{\ell,\ell}]_{\ell}}(\widetilde{\mathbf{x}_{}}_{})
		\,.
	\end{align*}
	Having identified the $l$-th component of the gradient, we now rewrite the whole expression and make it match the announced result.
	It holds
	\begin{align*}
		\nabla f_{[X_{\ell,k+\ell}]_{\ell}}(\mathbf{x}_{})
		 &
		\ = \
		\sum_{l=1}^T
		\mathbf{e}_{l}
		\cdot
		\frac{\partial}{\partial x_{l-1}}
		f_{[X_{\ell,k+\ell}]_{\ell}}(\mathbf{x}_{})
		\\ &
		  \ = \
		  \sum_{l=1}^T
		  \mathbf{e}_{l}
		\cdot
		  \lim_{h\to 0}
		\frac{f_{
					  [X_{\ell,k+\ell}]_\ell}(\mathbf{x}+h\cdot\mathbf{e}_{l})
			  -
			  f_{[X_{\ell,k+\ell}]_\ell}(\mathbf{x})
			  }{h}
		\\ &
		  \ = \
		  \sum_{l=1}^T
		  \mathbf{e}_{l}
		\sum_{i=1}^T
		  \mathbf{e}_{i}^\top
		  \cdot
		  \mathbf{A}_{T}^{-1}
		\cdot
		  \mathbf{e}_{l}
		\\ &
		  \qquad
		  \times \int_{\RR^T}
		f'_\varepsilon
		  (
		  \mathbf{e}_i^{\top}
		  \cdot
		  \mathbf{A}_{T}^{-1}
		\cdot
		  (\mathbf{x} - \widetilde{\mathbf{x}})
		)
		\
		  \cdot
		  \
		  \prod_{j\neq i}
		f_\varepsilon
		  (
		  \mathbf{e}_{j}^{\top}
		\mathbf{A}_{T}^{-1}
			  (\mathbf{x} - \widetilde{\mathbf{x}})
		)
		\,\mathrm{d}\PP_{[X_{\ell,k+\ell}-X_{\ell,\ell}]_{\ell}}(\widetilde{\mathbf{x}_{}}_{})
		\,.
	\end{align*}
	Note that,
	for any sequence $[v_i]_{i\in\{1,\ldots,T\}}$,
	\begin{align*}
		\sum_{l=1}^T
		\mathbf{e}_{l}
		\sum_{i=1}^T
		\mathbf{e}_{i}^\top
		\cdot
		\mathbf{A}_{T}^{-1}
		\cdot
		\mathbf{e}_{l}
		\cdot
		v_i
		 &
		\ = \
		\left(
		\sum_{l=1}^T
		\mathbf{e}_{l}
		\cdot
		\mathbf{e}_{l}^{\top}
		\right)
		\left(
		\sum_{i=1}^T
		\left(
			\mathbf{e}_{i}^\top
			\cdot
			\mathbf{A}_{T}^{-1}
			\right)^{\top}
		\cdot
		v_i
		\right)
		\\ &
		  \ = \
		  \mathbf{I}_{T}
		\cdot
		  \left(
		  \mathbf{A}_{T}^{-1}
		  \right)^{\top}
		\sum_{i=1}^T
		  \mathbf{e}_{i}
		\cdot v_i
		\\
		 &
		\ = \
		\left(
		\mathbf{A}_{T}^{-1}
		\right)^{\top}
		\cdot
		[v_i]_{i\in\{1,\ldots,T\}}
		\,.
	\end{align*}
	Therefore, the
	result follows
	with
	\begin{align*}
		v_i
		\ = \
		\int_{\RR^T}
		f'_\varepsilon
		(
		\mathbf{e}_i^{\top}
		\cdot
		\mathbf{A}_{T}^{-1}
		\cdot
		(\mathbf{x} - \widetilde{\mathbf{x}})
		)
		\
		\cdot
		\
		\prod_{j\neq i}
		f_\varepsilon
		(
		\mathbf{e}_{j}^{\top}
		\mathbf{A}_{T}^{-1}
			(\mathbf{x} - \widetilde{\mathbf{x}})
		)
		\,\mathrm{d}\PP_{[X_{\ell,k+\ell}-X_{\ell,\ell}]_{\ell}}(\widetilde{\mathbf{x}_{}}_{})
		\,.
	\end{align*}
  \\
  \textbf{Proof of (iii):}
	Since $\det \mathbf{A}_T^{-1}=1/(\det \mathbf{A}_T) = 1$, it follows from (ii) that
		\begin{align*}
			 &
			\int_{\RR^T}
			\left|
			\mathbf{z}^{\top}
			\cdot
      \nabla f_{[X_{\ell,k+\ell}]_\ell}
			\left(\mathbf{x}
			\right)
			\right|
			\,\mathrm{d} \mathbf{x}
			\ = \
			\int_{\RR^T}
			\left|
			\mathbf{z}
			^{\top}
			\cdot
      \nabla f_{[X_{\ell,k+\ell}]_\ell}
			\left(\mathbf{A}_T\cdot\mathbf{x}
			\right)
			\right|
			\,\mathrm{d}\mathbf{x}
			\\ &
			  \ = \
			  \int_{\RR^T}
			  \,\mathrm{d}\mathbf{x}
        \\&\  
			\left|
      \left(
      \mathbf{A}_T^{-1}
			  \mathbf{z}
      \right)
			  ^{\top}
			\left[
				  \EE
				  \left[
					  f'_\varepsilon(
					  x_{i-1}-
					  \mathbf{e}_{i}^{\top}\cdot\mathbf{A}_{T}^{-1}\cdot
					  [X_{\ell,k+\ell}-X_{\ell,\ell}]_\ell
					  )
					\prod_{j\neq i}
					f_\varepsilon(x_{j-1}
					  - \mathbf{e}_{j}^{\top}\cdot\mathbf{A}_{T}^{-1}\cdot
					  [X_{\ell,k+\ell}-X_{\ell,\ell}]_\ell
					  )
					\right]
				\right]_{i=1,\ldots,T}
			\right|
			\\ &
			  \ \lesssim \
			  \sum_{i=1}^T
			  \int_{\RR^T}
			\EE
			  \left[
				  \left|
				  f'_\varepsilon(
				  x_{i-1}-
				  \mathbf{e}_{i}^{\top}\cdot\mathbf{A}_{T}^{-1}\cdot
				  [X_{\ell,k+\ell}-X_{\ell,\ell}]_\ell
				  )
				\right|
				  \prod_{j\neq i}
				f_\varepsilon(x_{j-1}
				  - \mathbf{e}_{j}^{\top}\cdot\mathbf{A}_{T}^{-1}\cdot
				  [X_{\ell,k+\ell}-X_{\ell,\ell}]_\ell
				  )
				\right]
			\,\mathrm{d}\mathbf{x}
      \,,
      \end{align*}
      where the multiplicative constant in $\lesssim$ depends only on $\mathbf{z}$ and the coefficients $(a_i)$. 
      Integrating out the density terms, together with Assumption~\ref{asu:f} and Lemma~\ref{lem:5.1}(i) gives,
      for $\gamma\in(0,\alpha-1)$, the upper bound
      \begin{align*}
        &
			  \sum_{i=1}^T
			  \int_{\RR}
			\EE
			  \left[
				  \left|
				  f'_\varepsilon(
				  x_{i-1}-
				  \mathbf{e}_{i}^{\top}\cdot\mathbf{A}_{T}^{-1}\cdot
				  [X_{\ell,k+\ell}-X_{\ell,\ell}]_\ell
				  )
				\right|
				  \right]
			\,\mathrm{d}x_{i-1}
			\\ &
			  \ \le \
			  \sum_{i=1}^T
			  \int_{\RR}
        \EE[
        g_{\gamma}(
				  x_{i-1}-
				  \mathbf{e}_{i}^{\top}\cdot\mathbf{A}_{T}^{-1}\cdot
				  [X_{\ell,k+\ell}-X_{\ell,\ell}]_\ell
        )
        ]
        \,\mathrm{d}x_{i-1}
        \\&
        \ \le \ 
\sum_{i=1}^T
			  \int_{\RR}
        \EE\left[
        \left(
        1\lor
        \left|
				  \mathbf{e}_{i}^{\top}\cdot\mathbf{A}_{T}^{-1}\cdot
				  [X_{\ell,k+\ell}-X_{\ell,\ell}]_\ell
        \right|
        \right)^{1+\gamma}
        \right]
        g_{\gamma}(x_{i-1})
        \,\mathrm{d}x_{i-1}
        \\&
        \ < \ \infty 
			\,,
		\end{align*}
    since
    $
X_{\ell,k+\ell}-X_{\ell,\ell}\in L^{1+\gamma}(\PP)
    $ for all $\ell\in \{0,\ldots,T-1\}$.
    Therefore $\mathbf{z}^{\top}\cdot \nabla f_{[X_{\ell,k+\ell}]_\ell}\in L^1(\RR^T)$. Continuity follows immediately from $f'_{\varepsilon},f_{\varepsilon}\in C^{0}(\RR)\cap L^{\infty}(\RR)$ and dominated convergence. 
	\\
	\textbf{Proof of (iv):}
  Fix $i\in\{1,\ldots,T\}$.
		Since $\det \mathbf{A}_T^{-1}=1/(\det \mathbf{A}_T) = 1$, 
		\begin{align}
      \label{eqq:split}
      \begin{split}
			 &
			\int_{\RR^T}
			\left|
      x_{i-1}
      \frac{\partial}{\partial x_{i-1}}
      f_{[X_{\ell,k+\ell}]_\ell}
			\left(\mathbf{x}
			\right)
			\right|
      \,\mathrm{d}\mathbf{x}
			\ = \
			\int_{\RR^T}
			\left|
      \left(
      \mathbf{e}_i^{\top}
      \cdot
      \mathbf{A}_T
      \cdot
      \mathbf{x}
      \right)
      \cdot
        \mathbf{e}_i^{\top}
      \nabla
      f_{[X_{\ell,k+\ell}]_\ell}
			\left(\mathbf{A}_T \cdot\mathbf{x}
			\right)
			\right|
      \,\mathrm{d}\mathbf{x}
      \\&
      \ = \ 
			\int_{\RR^T}
			\left|
      \left(
        \sum_{l=1}^i
      a_{i-l} x_{l-1}
      \right)
        \mathbf{e}_i^{\top}
      \nabla
      f_{[X_{\ell,k+\ell}]_\ell}
			\left(\mathbf{A}_T\cdot\mathbf{x}
			\right)
			\right|
			\,\mathrm{d}\mathbf{x}
      \ \lesssim \ 
      \sum_{l=1}^i
	\int_{\RR^T}
			\left|
      x_{l-1}
        \frac{\partial}{\partial x_{i-1}}
      f_{[X_{\ell,k+\ell}]_\ell}
			\left(\mathbf{A}_T\cdot\mathbf{x}
			\right)
			\right|
      \,\mathrm{d}\mathbf{x}
      \,,
      \end{split}
      \end{align}
      where the multiplicative constant in $\lesssim$ depends only on the coefficients $(a_i)$. 
      If $l\in \{1,\ldots,i-1\}$, 
    it follows from (ii) that
      \begin{align*}
        &
\int_{\RR^T}
			\left|
      x_{l-1}
        \frac{\partial}{\partial x_{i-1}}
      f_{[X_{\ell,k+\ell}]_\ell}
			\left(\mathbf{A}_T\cdot\mathbf{x}
			\right)
			\right|
      \,\mathrm{d}\mathbf{x}
        \\&
        \ = \ 
						  \int_{\RR^T}
        \,\mathrm{d}\mathbf{x}
        \\&
        \qquad
			\left|
        x_{l-1}
				  \EE
				  \left[
					  f'_\varepsilon(
					  x_{i-1}-
					  \mathbf{e}_{i}^{\top}\cdot\mathbf{A}_{T}^{-1}\cdot
					  [X_{\ell,k+\ell}-X_{\ell,\ell}]_\ell
					  )
					\prod_{j\neq i}
					f_\varepsilon(x_{j-1}
					  - \mathbf{e}_{j}^{\top}\cdot\mathbf{A}_{T}^{-1}\cdot
					  [X_{\ell,k+\ell}-X_{\ell,\ell}]_\ell
					  )
					\right]
			\right|
        \\&
      			  \ \le \
			  \int_{\RR}
			\EE
			  \left[
				  \left|
				  x_{l-1}\cdot
				 f_\varepsilon(
				  x_{l-1}-
				  \mathbf{e}_{l}^{\top}\cdot\mathbf{A}_{T}^{-1}\cdot
				  [X_{\ell,k+\ell}-X_{\ell,\ell}]_\ell
				  )
				\right|
				  \right]
        \,\mathrm{d}x_{l-1} 
        \int_{\RR}
			|f_{\varepsilon}'(x)|
			\,\mathrm{d}x
			\\ &
			  \ \le \
        \left( \EE[|\varepsilon|]
			  +
			  \EE
			  \left[
				  |
				  \mathbf{e}_{l}^{\top}\cdot\mathbf{A}_{T}^{-1}\cdot
				  [X_{\ell,k+\ell}-X_{\ell,\ell}]_\ell
			  |
			  \right] \right)
        \int_{\RR}
			|f_{\varepsilon}'(x)|
			\,\mathrm{d}x
			  			\,.
		\end{align*}
    If $l=i$, then
    \begin{align*}
      &
\int_{\RR^T}
			\left|
      x_{i-1}
      \frac{\partial}{\partial x_{i-1}}
      f_{[X_{\ell,k+\ell}]_\ell}
			\left(\mathbf{A}_T\cdot\mathbf{x}
			\right)
			\right|
      \,\mathrm{d}\mathbf{x}
        \\&
        \ = \ 
       			  \int_{\RR}
			\EE
			  \left[
				  \left|
				  x_{i-1}\cdot
				 f'_\varepsilon(
				  x_{i-1}-
				  \mathbf{e}_{i}^{\top}\cdot\mathbf{A}_{T}^{-1}\cdot
				  [X_{\ell,k+\ell}-X_{\ell,\ell}]_\ell
				  )
				\right|
				  \right]
        \,\mathrm{d}x_{i-1}
			\\ &
			  \ \le \
        \int_{\RR}
        |x_{i-1}||f'_{\varepsilon}(x_{i-1})|
        \,\mathrm{d}x_{i-1}
			  +
			  \EE
			  \left[
				  |
				  \mathbf{e}_{i}^{\top}\cdot\mathbf{A}_{T}^{-1}\cdot
				  [X_{\ell,k+\ell}-X_{\ell,\ell}]_\ell
			  |
			  \right]
        \int_{\RR}
        |f'_{\varepsilon}(x_{i-1})|
        \,\mathrm{d}x_{i-1}
			  			\,.
		\end{align*}
    By Assumption~\ref{asu:f_more},
		\begin{align*}
			\int_{\RR}|x||f'_{\varepsilon}(x)|\,\mathrm{d}x \ < \  \infty\,,
		\end{align*}
    by Assumption~\ref{asu:f},
		\begin{align*}
			\int_{\RR}
			|f_{\varepsilon}'(x)|
			\,\mathrm{d}x
			\ \lesssim \
			\norm{g_{\alpha - 1}(x)}_{L^1(\RR)}
			\ <\  \infty
			\,,
		\end{align*}
    and
    since
    $\varepsilon \in L^1(\PP)$ and 
    $X_{\ell,k+\ell}-X_{\ell,\ell}\in L^{1}(\PP)
    $ for all $\ell\in \{0,\ldots,T-1\}$, we obtain 
		\begin{align*}
            \EE[|\varepsilon|] < \infty \quad \text{and} \quad
			\EE
			\left[
				|
				\mathbf{e}_{i}^{\top}\cdot\mathbf{A}_{T}^{-1}\cdot
				[X_{\ell,k+\ell}-X_{\ell,\ell}]_\ell
				|
				\right]
			\ \lesssim\
			\EE[\norm{
					[X_{\ell,k+\ell}-X_{\ell,\ell}]_\ell
        }_{1}]
			\ <\ \infty\,.
		\end{align*}
    Therefore,
    \begin{align*}
\int_{\RR^T}
			\left|
      x_{l-1}
             \frac{\partial}{\partial x_{i-1}}
      f_{[X_{\ell,k+\ell}]_\ell}
			\left(\mathbf{A}_T\cdot\mathbf{x}
			\right)
			\right|
      \,\mathrm{d}\mathbf{x}
      \ < \ \infty 
      \qquad\text{for all $l\in \{1,\ldots,i\}$,}
    \end{align*}
    so that the statement follows from~\eqref{eqq:split}.
\end{proof}
\subsection{Gradient of $G_{k,n}$}
We first require an auxiliary result about $L^1$-continuity with respect to scaling and translation. 
\begin{lemma}[$L^1$-continuity of scaling and translation]
  \label{lem:alt}
  Let $\mathfrak{G}\in L^1(\RR^T)$, $\mathbf{y}\in\RR^T$, and $s\in\RR\setminus \{0\}$. Suppose that
  $(\mathbf{y}_n)\subset \RR^T$, and  $(s_n)\subset \RR\setminus\{0\}$ satisfy $\mathbf{y}_n \to \mathbf{y}$ and $s_n \to s$ as $n\to\infty$. Then 
  \begin{align*}
    \norm{
    \mathfrak{G}(s_n(\cdot)-\mathbf{y}_n)
    \ - \ 
    \mathfrak{G}(s(\cdot)-\mathbf{y})
    }_{L^1(\RR^T)}
    \ \to \ 0
    \qquad\text{as}\ n\to\infty\,.
  \end{align*}
\end{lemma}
\begin{proof}
 By Lemma~3.26(2) in~\cite{altLinearFunctionalAnalysis2016a}, there exists a sequence of continuous functions $(\mathfrak{G}_m)$ with compact support such that 
        \begin{align*}
          &
          \norm{
          \mathfrak{G} - \mathfrak{G}_m
          }_{L^1(\RR^T)}
          \ \to \ 0
          \qquad\text{as}\ m\to\infty\,.
        \end{align*}
        We make the decomposition
        \begin{align*}
          &
          \norm{
    \mathfrak{G}(s_n(\cdot)-\mathbf{y}_n)
    \ - \ 
    \mathfrak{G}(s(\cdot)-\mathbf{y})
          }_{L^1(\RR^T)}
          \\&
          \ \le \ 
          \norm{
    \mathfrak{G}(s_n(\cdot)-\mathbf{y}_n)
    \ - \ 
    \mathfrak{G}(s_n(\cdot)-\mathbf{y})
          }_{L^1(\RR^T)}
                   \\&
          \qquad + \
          \norm{
    \mathfrak{G}(s_n(\cdot)-\mathbf{y})
    \ - \ 
    \mathfrak{G}_m(s_n(\cdot)-\mathbf{y})
          }_{L^1(\RR^T)}
                   \ + \ 
          \norm{
    \mathfrak{G}_m(s_n(\cdot)-\mathbf{y})
    \ - \ 
    \mathfrak{G}_m(s(\cdot)-\mathbf{y})
          }_{L^1(\RR^T)}
          \\&
          \qquad
          + \ 
          \norm{
    \mathfrak{G}_m(s(\cdot)-\mathbf{y})
    \ - \ 
    \mathfrak{G}(s(\cdot)-\mathbf{y})
          }_{L^1(\RR^T)}
                    \,.
        \end{align*}
        By change of variables,
        \begin{align*}
          \norm{
    \mathfrak{G}(s_n(\cdot)-\mathbf{y}_n)
    \ - \ 
    \mathfrak{G}(s_n(\cdot)-\mathbf{y})
          }_{L^1(\RR^T)}
          &
          \ = \ 
          |s_n|^{-T}
          \norm{
          \mathfrak{G}((\cdot)+(\mathbf{y}-\mathbf{y}_n))
    \ - \ 
    \mathfrak{G}
          }_{L^1(\RR^T)}
          \,,
          \\
 \norm{
    \mathfrak{G}(s_n(\cdot)-\mathbf{y})
    \ - \ 
    \mathfrak{G}_m(s_n(\cdot)-\mathbf{y})
          }_{L^1(\RR^T)}
          &
          \ = \ 
          |s_n|^{-T}
 \norm{
    \mathfrak{G}
    \ - \ 
    \mathfrak{G}_m
          }_{L^1(\RR^T)}
          \,,
          \\
          \norm{
    \mathfrak{G}_m(s(\cdot)-\mathbf{y})
    \ - \ 
    \mathfrak{G}(s(\cdot)-\mathbf{y})
          }_{L^1(\RR^T)}
          &
          \ = \ 
          |s|^{-T}
          \norm{
    \mathfrak{G}_m
    \ - \ 
    \mathfrak{G}
          }_{L^1(\RR^T)}
          \,.
        \end{align*}
                By the continuity and compact support of $\mathfrak{G}_m$, the dominated convergence theorem provides for any fixed~$m$
        \begin{align*}
   \norm{
    \mathfrak{G}_m(s_n(\cdot)-\mathbf{y})
    \ - \ 
    \mathfrak{G}_m(s(\cdot)-\mathbf{y})
          }_{L^1(\RR^T)}
          \ \to \ 0
          \qquad\text{as}\  n\to \infty
          \,,
        \end{align*}
        and by Lemma~4.15(1) in~\cite{altLinearFunctionalAnalysis2016a}
        \begin{align*}
    \norm{
          \mathfrak{G}((\cdot)+(\mathbf{y}-\mathbf{y}_n))
    \ - \ 
    \mathfrak{G}
          }_{L^1(\RR^T)}
          \ \to \ 
          0
          \qquad\text{as}\ n\to\infty\,.
        \end{align*}
        It follows that
        \begin{align*}
          \limsup_{n\to \infty}
   \norm{
    \mathfrak{G}(s_n(\cdot)-\mathbf{y}_n)
    \ - \ 
    \mathfrak{G}(s(\cdot)-\mathbf{y})
          }_{L^1(\RR^T)}
          \ \le \ 
          2|s|^{-T}
          \norm{\mathfrak{G}-\mathfrak{G}_m}_{L^1(\RR^T)}
          \ \to \ 0
          \qquad\text{as}\ m\to\infty\,.
        \end{align*}
This proves the claimed convergence.
\end{proof}
\begin{lemma}[Gradient of $G_{k,n}$]
	\label{lem:multi_swap}
	Let Assumptions~\ref{asu:f} and~\ref{asu:coef} hold.
  Then, for all $k\in\NN_0\cup\{\infty\}$, all sets $A\subset\RR^T$, the following hold:
	\begin{enumerate}
		\item For all $s_1,s_2\in\RR\cup\{\infty\}$ satisfying $s_1<s_2$, 
		      \begin{align*}
			      \mathbf{y}_{}
			      \ \mapsto\
			      G_{k,n}(\mathbf{y},s_1,s_2,A)
			      \ := \
			      \EE[G_n([X_{\ell,k+\ell}]_\ell+\mathbf{y},s_1,s_2,A)]
			      \ = \
			      \int_{\mathbb{R}^T} G_n(\mathbf{x},s_1,s_2,A) f_{[X_{\ell,k+\ell}]_\ell}(\mathbf{x}-\mathbf{y})\ \mathrm{d}\mathbf{x}
		      \end{align*}
		      is continuously differentiable, with gradient
		      \begin{align*}
			      \mathbf{y}_{}
			      \ \mapsto\
			      \nabla_{\mathbf{y}}
			      G_{k,n}^{}(\mathbf{y}, s_1,s_2,A)
			      \ = \
			      -
			      \int_{\RR^T}
			      G_n(\mathbf{x},s_1,s_2,A)
			      \nabla
			      f_{[X_{\ell,k+\ell}]_\ell}
			      (\mathbf{x}-\mathbf{y})
			      \,\mathrm{d}\mathbf{x}
			      \,.
		      \end{align*}
   		\item
			      If Assumption~\ref{asu:f_more} holds,
            then
			      \begin{align*}
				      \RR^T\times(0,\infty)
				      \ \to\  \RR\,,
				      \qquad
              \left(
              \mathbf{y},
				      s
              \right)
				      \ \mapsto\
              G_{k,n}(\mathbf{y},s,\infty,A)
			      \end{align*}
            is continuously differentiable, with partial derivatives with respect to $\mathbf{y}$ given in (i), and 
            with partial derivative with respect to $s$ given as 
			      \begin{align*}
              &
				      \frac{\partial}{\partial s}
				      G_{k,n}^{}(\mathbf{y}, s,\infty,A)
              \\&
				        \ = \
				        \frac{1}{s}
				      \left(
				        T \, 
                \PP[[X_{\ell,k+\ell}]_\ell+\mathbf{y}\in u_n\cdot s\cdot A]
				      \ + \
				        \int_{\RR^T}
				      \ind\{\mathbf{x}\in u_n \cdot s\cdot A\}
				        \cdot
				        \mathbf{x}^{\top}
				      \cdot
                \nabla f_{[X_{\ell,k+\ell}]_\ell}(\mathbf{x}-\mathbf{y})
				        \,\mathrm{d}\mathbf{x}
				      \right)
				        \,.
			      \end{align*}
 	\end{enumerate}
\end{lemma}
\begin{proof} We prove the two statements separately. 
\vskip1ex
\noindent
\textbf{Proof of (i):}
	We calculate each partial derivative $\partial/\partial y_{i-1}$ for $i\in \{1,\ldots,T\}$.
	To streamline the analysis, we
	write throughout $G_n(  \mathbf{x})$ for $G_n(\mathbf{x},s_1,s_2,A)$. For all $i\in\{1,\ldots,T\}$ and $h\neq 0$,
	\begin{align*}
		\begin{split}
			 &
			\frac{
				\EE[G_n([X_{\ell,k+\ell}]_\ell+\mathbf{y}+h \mathbf{e}_i)]
				-
				\EE[G_n([X_{\ell,k+\ell}]_\ell+\mathbf{y})]
			}{h}
			\\ &
			  \ = \
			  \int_{\RR^T}
			G_n(\mathbf{x})
			  \frac{
				  f_{[X_{\ell,k+\ell}]_\ell}(\mathbf{x}-\mathbf{y} - h \mathbf{e}_i)
				  -
				  f_{[X_{\ell,k+\ell}]_\ell}(\mathbf{x}-\mathbf{y})
				  }{h}
			\,\mathrm{d}\mathbf{x}
			\,,
		\end{split}
	\end{align*}
	so that, if we may interchange
	the limit $h\to 0$ with the integral, we find
	\begin{align*}
		 &
		\frac{\partial}{\partial y_{i-1}}
		\EE[G_n([X_{\ell,k+\ell}]_\ell+\mathbf{y})]
		\\ &
		  \ = \
		  -
		  \int_{\RR^T}
		G_n(\mathbf{x})
		  \lim_{h\to 0}
		\frac{
			  f_{[X_{\ell,k+\ell}]_\ell}(\mathbf{x}-\mathbf{y} - h \mathbf{e}_i)
			  -
			  f_{[X_{\ell,k+\ell}]_\ell}(\mathbf{x}-\mathbf{y})
			  }{-h}
		\,\mathrm{d}\mathbf{x}
		\\ &
		  \ = \
		  -
		  \int_{\RR^T}
		G_n(\mathbf{x})
		  \frac{\partial}{\partial z_{i-1}}
    f_{[X_{\ell,k+\ell}]_\ell}(\mathbf{z})\Big|_{\mathbf{z}=\mathbf{x}-\mathbf{y}}
		  \,\mathrm{d}\mathbf{x}
		\,,
	\end{align*}
	as required. We now show that we may interchange limit and integral. Since
\begin{align*}
  &
		\frac{
			f_{[X_{\ell,k+\ell}]_\ell}(\mathbf{x}-\mathbf{y})
			-
			f_{[X_{\ell,k+\ell}]_\ell}(\mathbf{x}-\mathbf{y} - h \mathbf{e}_i)
		}{h}
    \ = \
		\int_{0}^{1}
		  \frac{\partial}{\partial z_{i-1}}
		f_{[X_{\ell,k+\ell}]_\ell}
    (\mathbf{z})
    \Big|_{
    \mathbf{z}
    =
		\mathbf{x}-\mathbf{y} - \xi h \mathbf{e}_i
    }
		  \,\mathrm{d}\xi
      \,,
\end{align*}
it holds, by Tonelli's theorem, 
\begin{align*}
  &
  \left|
			\frac{
				\EE[G_n([X_{\ell,k+\ell}]_\ell+\mathbf{y}+h \mathbf{e}_i)]
				-
				\EE[G_n([X_{\ell,k+\ell}]_\ell+\mathbf{y})]
			}{h}
      \ + \
		  \int_{\RR^T}
		G_n(\mathbf{x})
		  \frac{\partial}{\partial z_{i-1}}
		f_{[X_{\ell,k+\ell}]_\ell}(\mathbf{z})
    \Big|_{\mathbf{z}=\mathbf{x}-\mathbf{y}}
		  \,\mathrm{d}\mathbf{x}
  \right|
  \\&
  \ \le \ 
  \int_{\RR^T}G_n(\mathbf{x})
  \left(
  \int_{0}^{1}
  \left|
		  \frac{\partial}{\partial z_{i-1}}
		f_{[X_{\ell,k+\ell}]_\ell}(\mathbf{z})
    \Big|_{\mathbf{z}=\mathbf{x}-\mathbf{y}
    }
    \ - \ 
		  \frac{\partial}{\partial z_{i-1}}
    f_{[X_{\ell,k+\ell}]_\ell}(\mathbf{z})
    \Big|_{\mathbf{z}=\mathbf{x}-\mathbf{y}-\xi h  \mathbf{e}_i
    }
  \right|
  \,\mathrm{d}\xi
  \right)
  \,\mathrm{d}\mathbf{x}
  \\&
  \ \le \ 
  \int_{0}^{1}
  \norm{
		  \frac{\partial}{\partial z_{i-1}}
		f_{[X_{\ell,k+\ell}]_\ell}
    \Big|_{\mathbf{z}=(\cdot)
    }
    \ - \ 
		  \frac{\partial}{\partial z_{i-1}}
    f_{[X_{\ell,k+\ell}]_\ell}
    \Big|_{\mathbf{z}=(\cdot)- \xi h\mathbf{e}_i
    }
  }_{L^1(\RR^T)}
  \,\mathrm{d}\xi
  \,.
\end{align*}
Since $\partial f_{[X_{\ell,k+\ell}]_{\ell}}/\partial z_{i-1}\in L^1(\RR^T)$ by Lemma~\ref{lem:multi_density}(iii), it follows from Lemma~\ref{lem:alt}
that
\begin{align*}
  \int_0^1 \norm{
		  \frac{\partial}{\partial z_{i-1}}
		f_{[X_{\ell,k+\ell}]_\ell}
    \Big|_{\mathbf{z}=(\cdot)
    }
    \ - \ 
		  \frac{\partial}{\partial z_{i-1}}
    f_{[X_{\ell,k+\ell}]_\ell}
    \Big|_{\mathbf{z}=(\cdot)- \xi h\mathbf{e}_i
    }
  }_{L^1(\RR^T)} \, \mathrm{d}\xi 
  \ \to\ 
  0
\end{align*}
as $h\to 0$.
\\
\textbf{Proof of continuity in part (i):} 
Let $\mathbf{h}\in \RR^T$. 
Since $\partial f_{[X_{\ell,k+\ell}]_{\ell}}
/\partial z_{i-1}
\in L^1(\RR^T)$ by Lemma~\ref{lem:multi_density}(iii), it follows from Lemma~\ref{lem:alt} that
  \begin{align*}
    &
    \norm{
		  \frac{\partial}{\partial z_{i-1}}
      G_{k,n}
    \Big|_{\mathbf{z}=(\cdot)
    }
      \ - \ 
		  \frac{\partial}{\partial z_{i-1}}
      G_{k,n}
    \Big|_{\mathbf{z}=(\cdot)+\mathbf{h}
    }
    }_{L^{\infty}(\RR^T)}
    \\&
    \ \le \ 
    \norm{
		  \frac{\partial}{\partial z_{i-1}}
      f_{[X_{\ell,k+\ell}]_\ell}\Big|_{\mathbf{z}=(\cdot)}
      \ - \
		  \frac{\partial}{\partial z_{i-1}}
      f_{[X_{\ell,k+\ell}]_\ell}\Big|_{\mathbf{z}= (\cdot)- \mathbf{h}}
    }_{L^1(\RR^T)}
    \ \to \ 
    0
    \,,
  \end{align*}
  as $\mathbf{h}\to 0$.
This proves (uniform) continuity.
	\\
		\textbf{Proof of (ii):} To streamline the proof, we write $G_{k,n}(\mathbf{y},s)$ for $G_{k,n}(\mathbf{y},s,\infty,A)$.
		By the change of variables $\mathbf{x}\mapsto s \mathbf{x}$,
\[
      G_{k,n}(\mathbf{y},s)
			\ = \
			s^T
			\int_{\RR^T}
			\ind\{\mathbf{x}\in u_n  A\}
			f_{[X_{\ell,k+\ell}]_\ell}(s \mathbf{x}-\mathbf{y})
			\,\mathrm{d}\mathbf{x}
			\,.
\]
    \textbf{Existence of $\frac{\partial}{\partial s}
      G_{k,n}(\mathbf{y},s)
    $ in part (ii):}
    The partial derivative with respect to $s>0$, if it exists, is 
		\begin{align*}
			 &
			\frac{\partial}{\partial s}
      G_{k,n}(\mathbf{y},s)
			\\ &
			  \ = \
			  T
			  s^{T-1}
			\int_{\RR^T}
			\ind\{\mathbf{x}\in u_n  A\}
			  f_{[X_{\ell,k+\ell}]_\ell}(s \mathbf{x}-\mathbf{y})
			  \,\mathrm{d}\mathbf{x}
			\ + \
			  s^T
			  \frac{\partial}{\partial s}
        \left(
			\int_{\RR^T}
			\ind\{\mathbf{x}\in u_n  A\}
			  f_{[X_{\ell,k+\ell}]_\ell}(s \mathbf{x}-\mathbf{y})
			  \,\mathrm{d}\mathbf{x}
        \right)
			\,.
		\end{align*}
    It therefore remains to show, for the second term, that
    we can swap differentiation and integration.
    Since
		\begin{align*}
			\frac{\partial}{\partial s}
      \left(
			f_{[X_{\ell,k+\ell}]_\ell}(s \mathbf{x}-\mathbf{y})
      \right)
			\ = \
			\mathbf{x}^{\top}
      \nabla f_{[X_{\ell,k+\ell}]_\ell}(s \mathbf{x}-\mathbf{y})
					\,
		\end{align*}
    if we are allowed to swap differentiation and integration,
		we will then get, as required,
		\begin{align*}
			 &
			\frac{\partial}{\partial s}
      G_{k,n}(\mathbf{y},s)
			\\ &
			  \ = \
			  T
			  s^{T-1}
			\int_{\RR^T}
			\ind\{\mathbf{x}\in u_n  A\}
			  f_{[X_{\ell,k+\ell}]_\ell}(s \mathbf{x}-\mathbf{y})
			  \,\mathrm{d}\mathbf{x}
			\ + \
			  s^T
			  \int_{\RR^T}
			\ind\{\mathbf{x}\in u_n  A\}
			  \mathbf{x}^{\top}
        \nabla f_{[X_{\ell,k+\ell}]_\ell}(s \mathbf{x}-\mathbf{y})
			  \,\mathrm{d}\mathbf{x}
			\\ &
			  \ = \
			  \frac{1}{s}
			\left(
			  T \, 
			  \PP[[X_{\ell,k+\ell}]_\ell +\mathbf{y}\in u_n  s A]
			\ + \
			  \int_{\RR^T}
			\ind\{\mathbf{x}\in u_n  s  A\}
			  \mathbf{x}^{\top}
        \nabla f_{[X_{\ell,k+\ell}]_\ell}(\mathbf{x}-\mathbf{y})
			  \,\mathrm{d}\mathbf{x}
			\right)
			  \,.
		\end{align*}
		Let us focus on the swap. Let $h\in \RR$ satisfy $|h|<s/2$. Then 
\begin{align}
  &
	\frac{1}{h}
			\left(
			\int_{\RR^T}
			\ind\{\mathbf{x}\in u_n  A\}
			f_{[X_{\ell,k+\ell}]_\ell}((s+h)\mathbf{x}-\mathbf{y})
			\,\mathrm{d}\mathbf{x}
			\ - \
			\int_{\RR^T}
			\ind\{\mathbf{x}\in u_n  A\}
			f_{[X_{\ell,k+\ell}]_\ell}(s \mathbf{x}-\mathbf{y})
			\,\mathrm{d}\mathbf{x}
			\right)
			\nonumber\\ &
			  \ = \
			  \int_{\RR^T}
			\ind\{\mathbf{x}\in u_n  A\}
			  \left(
			\int_{0}^{1}
			\mathbf{x}^{\top}
        \nabla f_{[X_{\ell,k+\ell}]_\ell}((s + h\xi) \mathbf{x}-\mathbf{y})
			  \,\mathrm{d}\xi
			  \right)
			  \,\mathrm{d}\mathbf{x}
        \nonumber\\&
        \ = \ 
\int_{\RR^T}
			\ind\{\mathbf{x}\in u_n  A\}
			  \left(
			\int_{0}^{1}
      \frac{
      \left(
(s + h\xi) \mathbf{x}
      \right)^{\top}
        \nabla f_{[X_{\ell,k+\ell}]_\ell}((s + h\xi) \mathbf{x}-\mathbf{y})
      }{
  s + h\xi
      }
			  \,\mathrm{d}\xi
			  \right)
			  \,\mathrm{d}\mathbf{x}
        \nonumber\\&
        \ = \ 
\int_{\RR^T}
			\ind\{\mathbf{x}\in u_n  A\}
			  \left(
			\int_{0}^{1}
      \frac{
      \mathfrak{H}((s + h\xi) \mathbf{x}-\mathbf{y})
      }{s+h\xi}
      \,\mathrm{d}\xi
      \right)
      \,\mathrm{d}\mathbf{x}
      \nonumber\\&
      \qquad + \ 
      \int_{\RR^T}
      \ind\{\mathbf{x}\in u_n A\}
      \left(
      \int_0^1
      \frac{
      \mathbf{y}^{\top}
        \nabla f_{[X_{\ell,k+\ell}]_\ell}((s + h\xi) \mathbf{x}-\mathbf{y})
      }{
  s + h\xi
      }
			  \,\mathrm{d}\xi
			  \right)
			  \,\mathrm{d}\mathbf{x}
        \,, \label{eq:swap1}
			  \end{align}
        where
        \begin{align*}
          \mathfrak{H}(\mathbf{x})
          \ := \ 
          \mathbf{x}^{\top}
          \nabla f_{[X_{\ell,k+\ell}]_\ell}(\mathbf{x})
          \,,\qquad \mathbf{x}\in \RR^T
          \,.
        \end{align*}
        The function $\mathfrak{H}$ is integrable and continuous, by Lemma~\ref{lem:multi_density}(iv). We expand 
        \begin{align}
          &
			  \int_{\RR^T}
			\ind\{\mathbf{x}\in u_n  A\}
			  \mathbf{x}^{\top}
			  \nabla f_{[X_{\ell,k+\ell}]_\ell}(s \mathbf{x}-\mathbf{y})
			  \,\mathrm{d}\mathbf{x}
          \nonumber\\&
          \ = \ 
          \int_{\RR^T}
          \ind\{\mathbf{x}\in u_n A\}
          \left(
          \int_0^1
          \frac{\mathfrak{H}(s \mathbf{x}-\mathbf{y})}{s}
          \,\mathrm{d}\xi
          \right)
          \,\mathrm{d}\mathbf{x}
          \  + \ 
\int_{\RR^T}
			\ind\{\mathbf{x}\in u_n  A\}
      \int_0^1
      \left(
      \frac{
			  \mathbf{y}^{\top}
			  \nabla f_{[X_{\ell,k+\ell}]_\ell}(s \mathbf{x}-\mathbf{y})
      }{s}
      \,\mathrm{d}\xi
      \right)
			  \,\mathrm{d}\mathbf{x}
        \,. \label{eq:swap2}
        \end{align}
        It then holds, by~\eqref{eq:swap1} and~\eqref{eq:swap2},
        \begin{align*}
          &
          \Bigg|
\frac{1}{h}
			\left(
			\int_{\RR^T}
			\ind\{\mathbf{x}\in u_n  A\}
			f_{[X_{\ell,k+\ell}]_\ell}((s+h)\mathbf{x}-\mathbf{y})
			\,\mathrm{d}\mathbf{x}
			\ - \
			\int_{\RR^T}
			\ind\{\mathbf{x}\in u_n  A\}
			f_{[X_{\ell,k+\ell}]_\ell}(s \mathbf{x}-\mathbf{y})
			\,\mathrm{d}\mathbf{x}
			\right)
          \\&
\qquad - \ 
          \int_{\RR^T}
          \ind\{\mathbf{x}\in u_n A\}
          \mathbf{x}^{\top}\nabla f_{[X_{\ell,k+\ell}]_\ell}(s \mathbf{x}-\mathbf{y})
          \,\mathrm{d}\mathbf{x}
          \Bigg|
          \\&
          \ \le \ 
          \int_{\RR^T}
          \ind\{\mathbf{x}\in u_n A\}
          \Bigg(
          \int_{0}^1
          \left|
          \frac{\mathfrak{H}(s \mathbf{x}-\mathbf{y})}{s}
          \ - \ 
          \frac{\mathfrak{H}((s+h\xi) \mathbf{x} -\mathbf{y})}{s+h\xi}
          \right|
          \\&
          \qquad
          \qquad
          + \ 
          \left|
      \frac{
			  \mathbf{y}^{\top}
			  \nabla f_{[X_{\ell,k+\ell}]_\ell}(s \mathbf{x}-\mathbf{y})
      }{s}
      \ - \ 
      \frac{
			  \mathbf{y}^{\top}
          \nabla f_{[X_{\ell,k+\ell}]_\ell}((s+h\xi) \mathbf{x} -\mathbf{y})
      }{s+h\xi}
          \right|
          \,\mathrm{d}\xi
          \Bigg)
          \,\mathrm{d}\mathbf{x}
                   \,.
        \end{align*}
        Note that 
        \begin{align*}
          &
   \left|
          \frac{\mathfrak{H}(s \mathbf{x}-\mathbf{y})}{s}
          \ - \ 
          \frac{\mathfrak{H}((s+h\xi) \mathbf{x} -\mathbf{y})}{s+h\xi}
          \right|
          \ + \ 
          \left|
      \frac{
			  \mathbf{y}^{\top}
			  \nabla f_{[X_{\ell,k+\ell}]_\ell}(s \mathbf{x}-\mathbf{y})
      }{s}
      \ - \ 
      \frac{
			  \mathbf{y}^{\top}
          \nabla f_{[X_{\ell,k+\ell}]_\ell}((s+h\xi) \mathbf{x} -\mathbf{y})
      }{s+h\xi}
          \right|
          \\&
          \ \le \ 
          \frac{1}{s}
          \left(
\left|
          \mathfrak{H}(s \mathbf{x}-\mathbf{y})
          \ - \ 
          \mathfrak{H}((s+h\xi) \mathbf{x} -\mathbf{y})
          \right|
          \ + \ 
          \left|
			  \mathbf{y}^{\top}
			  \nabla f_{[X_{\ell,k+\ell}]_\ell}(s \mathbf{x}-\mathbf{y})
      \ - \ 
			  \mathbf{y}^{\top}
          \nabla f_{[X_{\ell,k+\ell}]_\ell}((s+h\xi) \mathbf{x} -\mathbf{y})
          \right|
          \right)
          \\&
          \qquad + \ 
          \left|
          \frac{1}{s}
          \ - \ 
          \frac{1}{s+h\xi}
          \right|
          \left(
          \left|
          \mathfrak{H}((s+h\xi) \mathbf{x} -\mathbf{y})
          \right|
          \ + \ 
          \left|
			  \mathbf{y}^{\top}
          \nabla f_{[X_{\ell,k+\ell}]_\ell}((s+h\xi) \mathbf{x} -\mathbf{y})
          \right|
          \right)
          \,.
        \end{align*}
        Therefore, by Tonelli's theorem, 
        \begin{align}
          \label{eqn:tonelli}
          \begin{split}
          &
       \int_{\RR^T}
          \ind\{\mathbf{x}\in u_n A\}
            \Bigg(
          \int_{0}^1
          \left|
          \frac{\mathfrak{H}(s \mathbf{x}-\mathbf{y})}{s}
          \ - \ 
          \frac{\mathfrak{H}((s+h\xi) \mathbf{x} -\mathbf{y})}{s+h\xi}
          \right|
            \\&
          \qquad \qquad+ \ 
          \left|
      \frac{
			  \mathbf{y}^{\top}
			  \nabla f_{[X_{\ell,k+\ell}]_\ell}(s \mathbf{x}-\mathbf{y})
      }{s}
      \ - \ 
      \frac{
			  \mathbf{y}^{\top}
          \nabla f_{[X_{\ell,k+\ell}]_\ell}((s+h\xi) \mathbf{x} -\mathbf{y})
      }{s+h\xi}
          \right|
          \,\mathrm{d}\xi
            \Bigg)
          \,\mathrm{d}\mathbf{x}
          \\
          &
          \le
          \int_0^1
          \Bigg(
   \frac{1}{s}
            \Big(
          \norm{
          \mathfrak{H}(s(\cdot)-\mathbf{y})
          \ - \ 
          \mathfrak{H}((s+h\xi)(\cdot)-\mathbf{y})
          }_{L^1(\RR^T)}
            \\&
          \qquad
          \qquad
            +\  
          \norm{
			  \mathbf{y}^{\top}
			  \nabla f_{[X_{\ell,k+\ell}]_\ell}(s(\cdot)-\mathbf{y})
      \ - \ 
			  \mathbf{y}^{\top}
          \nabla f_{[X_{\ell,k+\ell}]_\ell}((s+h\xi)(\cdot)-\mathbf{y})
          }_{L^1(\RR^T)}
            \Big)
          \\
          &
          \qquad\qquad + \ 
          \left|
          \frac{1}{s}
          \ - \ 
          \frac{1}{s+h\xi}
          \right|
          \Big(
          \norm{
          \mathfrak{H}((s+h\xi)(\cdot)-\mathbf{y})
          }_{L^1(\RR^T)}
          \\&
          \qquad\qquad\qquad\ + \ 
          \norm{
			  \mathbf{y}^{\top}
          \nabla f_{[X_{\ell,k+\ell}]_\ell}((s+h\xi)(\cdot)-\mathbf{y})
          }_{L^1(\RR^T)}
          \Big)
          \Bigg)
          \,\mathrm{d}\xi.
 \end{split}
        \end{align}
        Note that
        \begin{align*}
          &
\norm{
          \mathfrak{H}((s+h\xi)(\cdot)-\mathbf{y})
          }_{L^1(\RR^T)}
          \ + \ 
          \norm{
			  \mathbf{y}^{\top}
          \nabla f_{[X_{\ell,k+\ell}]_\ell}((s+h\xi)(\cdot)-\mathbf{y})
          }_{L^1(\RR^T)}
          \\&
          \ = \ 
          |s+h\xi|^{-T}
          \left(
          \norm{\mathfrak{H}
          }_{L^1(\RR^T)}
          \ + \ 
          \norm{\mathbf{y}^{\top}\nabla f_{[X_{\ell,k+\ell}]_\ell}
          }_{L^1(\RR^T)}
          \right)
          \,,
        \end{align*}
        so that
        \begin{align*}
          &
          \int_0^1
\left|
          \frac{1}{s}
          \ - \ 
          \frac{1}{s+h\xi}
          \right|
          \left(
          \norm{
          \mathfrak{H}((s+h\xi)(\cdot)-\mathbf{y})
          }_{L^1(\RR^T)}
          \ + \ 
          \norm{
			  \mathbf{y}^{\top}
          \nabla f_{[X_{\ell,k+\ell}]_\ell}((s+h\xi)(\cdot)-\mathbf{y})
          }_{L^1(\RR^T)}
          \right)
          \,\mathrm{d}\xi
          \\&
          \ \le \ 
  \left(
          \norm{\mathfrak{H}
          }_{L^1(\RR^T)}
          \ + \ 
          \norm{\mathbf{y}^{\top}\nabla f_{[X_{\ell,k+\ell}]_\ell}
          }_{L^1(\RR^T)}
          \right)
          \sup_{\xi\in (0,1)}
          \left(
          \left|
          \frac{1}{s}
          \ - \ 
          \frac{1}{s+h\xi}
          \right|
          |s+h\xi|^{-T}
          \right)
          \ \to \ 0
        \end{align*}
        as $h\to 0$.
        Let us deal with the first pair of integrals in the right-hand side of~\eqref{eqn:tonelli}.
               Since $\mathfrak{H}$ and $\mathbf{y}^{\top}\nabla f_{[X_{\ell,k+\ell}]_\ell}\in L^1(\RR^T)\cap C^0(\RR^T)$, it follows from Lemma~\ref{lem:alt} that
        \begin{align*}
          &
   \int_0^1
   \frac{1}{s}
          \Bigg(
          \norm{
          \mathfrak{H}(s(\cdot)-\mathbf{y})
          \ - \ 
          \mathfrak{H}((s+h\xi)(\cdot)-\mathbf{y})
          }_{L^1(\RR^T)}
          \\&
           \qquad\qquad+\  
          \norm{
			  \mathbf{y}^{\top}
			  \nabla f_{[X_{\ell,k+\ell}]_\ell}(s(\cdot)-\mathbf{y})
      \ - \ 
			  \mathbf{y}^{\top}
          \nabla f_{[X_{\ell,k+\ell}]_\ell}((s+h\xi)(\cdot)-\mathbf{y})
          }_{L^1(\RR^T)}
          \Bigg)
          \,\mathrm{d}\xi
          \\&
          \ \to \ 0
        \end{align*}
        as $h\to 0$, which proves the claim.
       			\\
    \textbf{Proof of continuity of $\nabla_{\mathbf{y}} G_{k,n}$ in part (ii):}
 Let $(\mathbf{y}_{},s)\in\RR^T\times (0,\infty)$ and $(\mathbf{y}_{m},s_m)_{m\in\NN}\subset \RR^T\times (s/2,3s/2)$ with $(\mathbf{y}_{m},s_m)\to (\mathbf{y}_{},s)$ as $m\to\infty$. By part (i) and the change of variables $\mathbf{x}\mapsto s\mathbf{x}$, 
 \[
      \nabla_{\mathbf{y}} G_{k,n}(\mathbf{y},s)
			\ = \
			- s^T
			\int_{\RR^T}
			\ind\{\mathbf{x}\in u_n  A\}
			\nabla f_{[X_{\ell,k+\ell}]_\ell}(s \mathbf{x}-\mathbf{y})
			\,\mathrm{d}\mathbf{x}
			\,.
\]
 Then, for all $\mathbf{z}\in \RR^T$,
    \begin{align*}
      &
      \left|
      \mathbf{z}^{\top}
      \nabla_{\mathbf{y}}G_{k,n}(\mathbf{y},s)
      \ - \ 
      \mathbf{z}^{\top}
      \nabla_{\mathbf{y}}G_{k,n}(\mathbf{y}_m,s_m)
      \right|
      \\&
      \ \le \ 
      \left|
      s^T - s_m^T
      \right|
      \norm{
      \mathbf{z}^{\top}
      \nabla
      f_{[X_{\ell,k+\ell}]_\ell}(s(\cdot)-\mathbf{y})
          }_{L^1(\RR^T)}
          \\&
          \qquad + \ 
          |s_m|^T
          \norm
{
      \mathbf{z}^{\top}
      \nabla
      f_{[X_{\ell,k+\ell}]_\ell}(s_m(\cdot)-\mathbf{y}_m)
       \ - \ 
      \mathbf{z}^{\top}
      \nabla
      f_{[X_{\ell,k+\ell}]_\ell}(s(\cdot)-\mathbf{y})
          }_{L^1(\RR^T)}
          \\&
          \ \to \ 0
          \,,
    \end{align*}
    as $m\to\infty$, with similar arguments as before.
    \\ 
    \textbf{Proof of continuity of $\frac{\partial}{\partial s} G_{k,n}$ in part (ii):}
     		To prove continuity of the partial derivative in $s$, 
    note that 
    \begin{align*}
      &
      \frac{\partial}{\partial s}
      G_{k,n}(\mathbf{y}_m,s_m)
      \\&
      \ = \ 
      \frac{1}{s_m}
      \left(
      T \, 
      \PP[[X_{\ell,k+\ell}]_\ell + \mathbf{y}_m \in u_n s_m A]
      \ + \ 
      \int_{\RR^T}
      \ind\{\mathbf{x}\in u_n s_mA\} \mathbf{x}^{\top}\nabla f_{[X_{\ell,k+\ell}]_{\ell}}(\mathbf{x}-\mathbf{y}_m)
      \,\mathrm{d}\mathbf{x}
      \right)
            \,.
    \end{align*}
    For the first term, we have
    \begin{align*}
      \PP[[X_{\ell,k+\ell}]_\ell + \mathbf{y}_m \in u_n s_m A]
      \ = \ 
      s_m^{T}
      \int_{\RR^T}
      \ind\{\mathbf{x} \in u_n A\}
      f_{[X_{\ell,k+\ell}]_\ell}(s_m\mathbf{x}-\mathbf{y}_m)
      \,\mathrm{d}\mathbf{x}
    \end{align*}
    and, by Lemma~\ref{lem:alt},
    \begin{align*}
      &
      \left|
      \int_{\RR^T}
      \ind\{\mathbf{x} \in u_n A\}
      f_{[X_{\ell,k+\ell}]_{\ell}}(s_m\mathbf{x}-\mathbf{y}_m)
      \,\mathrm{d}\mathbf{x}
      \ - \ 
      \int_{\RR^T}
      \ind\{\mathbf{x} \in u_n A\}
      f_{[X_{\ell,k+\ell}]_{\ell}}(s\mathbf{x}-\mathbf{y})
      \,\mathrm{d}\mathbf{x}
      \right|
      \\&
      \ \le \ 
      \norm{
      f_{[X_{\ell,k+\ell}]_\ell}(s_m(\cdot)-\mathbf{y}_m)
      \ - \
      f_{[X_{\ell,k+\ell}]_\ell}(s(\cdot)-\mathbf{y})
      }_{L^1(\RR^T)}
      \ \to \ 0
      \,,
    \end{align*}
    so that
    \begin{align*}
      \frac{1}{s_m}
      T\, \PP[[X_{\ell,k+\ell}]_{\ell}+ \mathbf{y}_m \in u_n s_m A] 
      \ \to\  
      \frac{1}{s}
      T\, \PP[[X_{\ell,k+\ell}]_{\ell}+ \mathbf{y} \in u_n s A] 
    \end{align*}
    as $m\to\infty$.
    For the second term, 
    note that
    \begin{align*}
      \int_{\RR^T}
      \ind\{\mathbf{x}\in u_n s_mA\} \mathbf{x}^{\top}\nabla f_{[X_{\ell,k+\ell}]_{\ell}}(\mathbf{x}-\mathbf{y}_m)
      \,\mathrm{d}\mathbf{x}
      \ = \ 
      s_m^{T}
      \int_{\RR^T}
      \ind\{\mathbf{x}\in u_n A\} (s_m\mathbf{x})^{\top}\nabla f_{[X_{\ell,k+\ell}]_{\ell}}(s_m\mathbf{x}-\mathbf{y}_m)
      \,\mathrm{d}\mathbf{x}
      \,,
    \end{align*}
    so that
    it remains to prove that
    \begin{align*}
      \left|
 \int_{\RR^T}
      \ind\{\mathbf{x}\in u_n  A\} (s_m\mathbf{x})^{\top}\nabla f_{[X_{\ell,k+\ell}]_{\ell}}(s_m\mathbf{x}-\mathbf{y}_m)
      \,\mathrm{d}\mathbf{x}
\ - \ 
 \int_{\RR^T}
      \ind\{\mathbf{x}\in u_n  A\} (s\mathbf{x})^{\top}\nabla f_{[X_{\ell,k+\ell}]_{\ell}}(s\mathbf{x}-\mathbf{y})
      \,\mathrm{d}\mathbf{x}
      \right|
      \ \to \ 0
    \end{align*}
    as $m\to\infty$.
    Note that for $m$ sufficiently large $\norm{\mathbf{y}_m-\mathbf{y}}_{\infty}\le 1$, so that 
    \begin{align*}
      &
\left|
 \int_{\RR^T}
      \ind\{\mathbf{x}\in u_n  A\} (s_m\mathbf{x})^{\top}\nabla f_{[X_{\ell,k+\ell}]_{\ell}}(s_m\mathbf{x}-\mathbf{y}_m)
      \,\mathrm{d}\mathbf{x}
\ - \ 
 \int_{\RR^T}
      \ind\{\mathbf{x}\in u_n  A\} (s\mathbf{x})^{\top}\nabla f_{[X_{\ell,k+\ell}]_{\ell}}(s\mathbf{x}-\mathbf{y})
      \,\mathrm{d}\mathbf{x}
      \right|
      \displaybreak[0] \\ &
      \ \le \ 
      \norm{
      s_m
      (\cdot)^{\top}\nabla f_{[X_{\ell,k+\ell}]_{\ell}}(s_m(\cdot)-\mathbf{y}_m)
      \ - \ 
      s
      (\cdot)^{\top}\nabla f_{[X_{\ell,k+\ell}]_{\ell}}(s(\cdot)-\mathbf{y})
      }_{L^1(\RR^T)}
      \\&
      \ \le \ 
      \norm{\mathfrak{H}(s_m(\cdot)-\mathbf{y}_m)
      \ - \ 
      \mathfrak{H}(s(\cdot)-\mathbf{y})}_{L^1(\RR^T)}
      \\&
      \qquad + \ 
      \norm{\mathbf{y}^{\top}
      \left(
      \nabla f_{[X_{\ell,k+\ell}]_{\ell}}(s_m(\cdot)-\mathbf{y}_m)
      \ - \
      \nabla f_{[X_{\ell,k+\ell}]_{\ell}}(s(\cdot)-\mathbf{y})
      \right)
      }_{L^1(\RR^T)}
      \\&
      \qquad + \
\norm{
      (\mathbf{y}_m-\mathbf{y})^{\top}
      \left(
      \nabla f_{[X_{\ell,k+\ell}]_{\ell}}(s_m(\cdot)-\mathbf{y}_m)
      \ - \
      \nabla f_{[X_{\ell,k+\ell}]_{\ell}}(s(\cdot)-\mathbf{y})
      \right)
      }_{L^1(\RR^T)}
      \\&
      \qquad + \ 
      \norm{
      (\mathbf{y}_m-\mathbf{y})^{\top}
      \nabla f_{[X_{\ell,k+\ell}]_{\ell}}(s(\cdot)-\mathbf{y})
      }_{L^1(\RR^T)}
      \displaybreak[0]\\&
      \ \le \ 
   \norm{\mathfrak{H}(s_m(\cdot)-\mathbf{y}_m)
      \ - \ 
      \mathfrak{H}(s(\cdot)-\mathbf{y})}_{L^1(\RR^T)}
      \\&
      \qquad + \ 
      \norm{\mathbf{y}^{\top}
      \left(
      \nabla f_{[X_{\ell,k+\ell}]_{\ell}}(s_m(\cdot)-\mathbf{y}_m)
      \ - \
      \nabla f_{[X_{\ell,k+\ell}]_{\ell}}(s(\cdot)-\mathbf{y})
      \right)
      }_{L^1(\RR^T)}
      \\&
      \qquad + \
      \sum_{j=1}^T
\norm{
\mathbf{e}_j^{\top}
      \left(
      \nabla f_{[X_{\ell,k+\ell}]_{\ell}}(s_m(\cdot)-\mathbf{y}_m)
      \ - \
      \nabla f_{[X_{\ell,k+\ell}]_{\ell}}(s(\cdot)-\mathbf{y})
      \right)
      }_{L^1(\RR^T)}
      \\&
      \qquad + \ 
      \sum_{j=1}^T
      \norm{
\mathbf{e}_j^{\top}
      \nabla f_{[X_{\ell,k+\ell}]_{\ell}}(s(\cdot)-\mathbf{y})
      }_{L^1(\RR^T)}
      \norm{\mathbf{y}_m-\mathbf{y}}_{\infty}
      \,.
    \end{align*}
        By Lemma~\ref{lem:alt}, this upper bound tends to 0 as $m\to\infty$.
\end{proof}
\subsection{Central limit theorem for multivariate partial sums}
Next we provide a proof of Theorem~\ref{thm:linear_clt_multi}, which extends a well known central limit theorem for univariate partial sums of long memory linear time series to the multivariate setting relevant in this work.
\begin{proof}[\textbf{Proof of Theorem~\ref{thm:linear_clt_multi}:}]
\label{sec:linear_clt_multi}
  Since
  \begin{align*}
    \EE
    \left[
    \norm{
    n^{-d-1/\alpha}
    \left(
    \sum_{t=1}^n [X_{t+\ell}]_\ell
    \ - \ 
    \sum_{t=1}^{n-T+1} [X_{t+\ell}]_\ell
    \right)
    }_1
    \right]
    \ \le \ 
    n^{-d-1/\alpha}
    \cdot
    T^2
    \cdot
    \EE[|X_0|]
    \ \to \ 0\qquad\text{as}\ n\to\infty\,,
  \end{align*}
  we show the result without loss of generality for
  \begin{align*}
    \sum_{t=1}^n [X_{t+\ell}]_\ell
    \qquad\text{instead of}\qquad
    \sum_{t=1}^{n-T+1} [X_{t+\ell}]_\ell
    \,.
  \end{align*}
	Define
	\begin{align*}
		[\Delta_{t,t+\ell}]_{\ell}
		\ := \
		[X_{t+\ell}-X_t]_{\ell}
		\ = \
		[X_{t+\ell}]_\ell \ - \ 
		X_{t}\cdot \mathbf{1}_{T}
		\qquad\text{for}\ t\ge 1\,.
	\end{align*}
	Then
	\begin{align*}
		\sum_{t=1}^n
		[X_{t+\ell}]_{\ell}
		\ = \
		\left(\sum_{t=1}^n X_t\right)\cdot \mathbf{1}_{T}
		\ + \
		\sum_{t=1}^n [\Delta_{t,t+\ell}]_{\ell}
		\,.
	\end{align*}
	We show that, as $n\to\infty$,
	\begin{align}
		\label{eq:conv_delta}
		n^{-d-1/\alpha}
		\sum_{t=1}^n [\Delta_{t,t+\ell}]_{\ell}
		\ \to_{\PP} \ \mathbf{0}_{T}\,.
	\end{align}
	Note that by a telescoping sum argument
	\begin{align*}
		\sum_{t=1}^n
		[\Delta_{t,t+\ell}]_{\ell}
		\ = \
		\left[
			\sum_{t=1}^n
			(X_{t+\ell}
			\ - \
			X_{t})
			\right]_{\ell}
		\ = \
		\left[
			\sum_{t=1}^\ell
			(X_{n+t}
			\ - \
			X_{t})
			\right]_{\ell}
		\,.
	\end{align*}
	Since for all $\ell\in \{0,\ldots,T-1\}$
	\begin{align*}
		\EE
		\left[
			\left|
			n^{-d-1/\alpha}
			\sum_{t=1}^\ell
			(X_{n+t}
			\ - \
			X_{t})
			\right|
			\right]
		\ \le \
		n^{-d-1/\alpha}
		\cdot
		\left(
		2\cdot T
		\cdot \EE[|X_0|]
		\right)
		\ \to \ 0
		\qquad\text{as}\ n\to\infty\,,
	\end{align*}
	it readily follows \eqref{eq:conv_delta}.
	The result hence follows from Slutsky's theorem combined with the standard asymptotic theory for
	$n^{-d-1/\alpha}\sum_{t=1}^n X_t$, see, for example, Theorem~3.1 in~\cite{scheffelCentralLimitTheory2025}.
\end{proof}
\section{Separable product bounds featuring $u_n$ and $(s_1,s_2)$}
By \eqref{lem:two_sets}, there exist $\{\mathfrak{a}_{j}\}\subset(0,\infty)$ such that
\begin{align*}
	 &
	G_n(\mathbf{A}_{T}\cdot\mathbf{z}_{}, s_1,s_2,A)
	\ = \
	\ind\{\mathbf{A}_{T}\cdot \mathbf{z}_{}\in u_n \cdot ((s_1\cdot A)\setminus (s_2\cdot A))\}
	\\ &
	  \ \le \
	  \sum_{j=1}^T
	  \ind\left\{
	  \left|
	  \sum_{m=1}^j a_{j-m}z_{m-1}
	\right|
	  \in u_n \cdot
	  \mathfrak{a}_j \cdot [s_1,s_2]
	\right\}
	  \,.
\end{align*}
The main purpose of this section is to separate
the
threshold $u_n$ from the uniformity interval $[s_1,s_2]$
in a way that
preserves the structure required for the later
estimates.
We formulate the argument for general index sets
\(I\subset\{1,\ldots,T\}\) and coefficients $(c_m)_{m\in I}$,
where in the later sections we mostly have $I=\{1,\ldots,T\}$ and
$c_m = a_{j-m}$.
The first lemma contains the basic separation argument.
In the rest of the section we then build the complexity needed for the
subsequent analysis around this.
\begin{lemma}[Separable product bound I]
  \label{lem:gist_separation}
	Let $\emptyset\neq I\subset \{1,\ldots,T\}$, $(c_m)_{m\in I}\subset \RR$, $\mathfrak{a}\in(0,\infty)$, $\gamma\in(0,\infty)$.
	Then
	\begin{align*}
		\int_{\RR^{|I|}}
		\ind
		\left\{
		\left|
		\sum_{m\in I}
		c_{m}
		z_{m-1}
		\right|
		\in u_n \cdot \mathfrak{a} \cdot [s_1,s_2]
		\right\}
		\cdot
		\prod_{\substack{m\in I}}
		g_{\gamma}
		\left(
		z_{m-1}
		\right)
		\,\mathrm{d}\mathbf{z}_{}
		\ \lesssim \
		u_n^{-\gamma}
		\cdot
		\left(
		(s_2-s_1)
		\land 1
		\right)
		\,,
	\end{align*}
	where $\lesssim$ means inequality up to a constant that only depends on $T$, $(c_m)$, $\mathfrak{a}$, $\gamma$, and $\underline{s}$,
	and is independent of $n$ and $s_1,s_2$.
\end{lemma}
\begin{proof}
	Since
	\begin{align*}
		 &
		\int_{\RR^{|I|}}
		\ind
		\left\{
		\left|
		\sum_{m\in I}
		c_{m}
		z_{m-1}
		\right|
		\in u_n \cdot \mathfrak{a} \cdot [s_1,s_2]
		\right\}
		\cdot
		\prod_{\substack{m\in I}}
		g_{\gamma}
		\left(
		z_{m-1}
		\right)
		\,\mathrm{d}\mathbf{z}_{}
		\\ &
		  \ \le \
		  \int_{\RR^{|I|}}
		\left(
		  \ind
		  \left\{
		  \sum_{m\in I}
		c_{m}
		  z_{m-1}
		  \in u_n \cdot \mathfrak{a} \cdot [s_1,s_2]
		\right\}
		  \ + \
		  \ind
		  \left\{
		  \sum_{m\in I}
		  (-c_{m})
		z_{m-1}
		  \in u_n \cdot \mathfrak{a} \cdot [s_1,s_2]
		\right\}
		  \right)
		\cdot
		  \prod_{\substack{m\in I}}
		g_{\gamma}
		\left(
		  z_{m-1}
		  \right)
		\,\mathrm{d}\mathbf{z}_{}
		\,,
	\end{align*}
	it is sufficient to prove
	\begin{align*}
		\int_{\RR^{|I|}}
		\ind
		\left\{
		\sum_{m\in I}
		c_{m}
		z_{m-1}
		\in u_n \cdot \mathfrak{a} \cdot [s_1,s_2]
		\right\}
		\cdot
		\prod_{\substack{m\in I}}
		g_{\gamma}
		\left(
		z_{m-1}
		\right)
		\,\mathrm{d}\mathbf{z}_{}
		\ \lesssim \
		u_n^{-\gamma}
		\cdot
		\left(
		(s_2-s_1)
		\land 1
		\right)
		\,,
	\end{align*}
	for $(c_m)_{m\in I}\subset \RR$, since this automatically covers $(-c_m)_{m\in I}\subset \RR$.
	\begin{align*}
	\end{align*}
	First consider $I\subset\{1,\ldots,T\}$ with $|I|=1$. Without loss of generality, we assume $1\in I$, $c_1\neq 0$.
	Making the change of variables
	\begin{align*}
		z_0 \ \mapsto\
		\frac{u_n \cdot \mathfrak{a}}{c_1}
		z_0
	\end{align*}
	and using Lemma~\ref{lem:5.1}(iii), we get
	\begin{align*}
		 &
		\int_{\RR}
		\ind\{c_1z_0 \in u_n \cdot \mathfrak{a}\cdot [s_1,s_2]\}
		g_{\gamma}(z_0)
		\,\mathrm{d}z_0
		\ = \
		u_n
		\frac{\mathfrak{a}}{|c_1|}
		\int_{\RR}
		\ind\{z_0 \in [s_1,s_2]\}
		g_{\gamma}
		\left(
		u_n \cdot z_0
		\frac{\mathfrak{a}}{c_1}
		\right)
		\,\mathrm{d}z_0
		\\ &
		  \ \lesssim \
		  u_n
		  \int_{s_1}^{s_2}
		\ind\{z_0 \in [s_1,s_2]\}
		  g_{\gamma}(u_n \cdot z_0)
		  \,\mathrm{d}z_0
		  \,,
	\end{align*}
	by the symmetry of $g_{\gamma}$.
	The multiplicative constant in $\lesssim$ depends only on $\gamma$, $\mathfrak{a}$, and the coefficient $c_1$.
	We bound this in two different ways. First,
	\begin{align*}
		 &
		u_n
		\int_{s_1}^{s_2}
		g_{\gamma}(u_n \cdot z_0)
		\,\mathrm{d}z_0
		\ \le \
		u_n
		\cdot
		g_{\gamma}(u_n\cdot s_1)
		\cdot
		(s_2-s_1)
		\ = \
		\frac{u_n}{(1+s_1\cdot u_n)^{1+\gamma}}
		\left(
		s_2-s_1
		\right)
		\\ &
		  \ \lesssim \
		  u_n^{-\gamma}(s_2-s_1)
		  \,,
	\end{align*}
	since $s_1>\underline{s}>0$.
	Here the constant in $\lesssim$ depends only on $c_1$, $\mathfrak{a}$, and $\underline{s}$.
	The second bound is
	\begin{align*}
		u_n
		\int_{s_1}^{s_2}
		g_{\gamma}(u_n \cdot z_0)
		\,\mathrm{d}z_0
		\ \lesssim\
		u_n
		\int_{\underline{s}}^{\infty}
		g_{\gamma}(u_n\cdot z_0)\,\mathrm{d}z_0
		\ \lesssim\
		(1+u_n\cdot \underline{s})^{-\gamma}
		\ \lesssim\
		u_n^{-\gamma}\,,
	\end{align*}
	where the multiplicative constant in $\lesssim$ depends only on $\underline{s}$.
	This proves the case with $|I|=1$. We proceed by mathematical induction: suppose that
	there exists $j\in \{1,\ldots,T\}$ such that the statement holds for all $I\subset\{1,\ldots,T\}$ with $|I|=j$, that is,
	\begin{align}
		\label{eq:ind_asu}
		\int_{\RR^{|I|}}
		\ind
		\left\{
		\sum_{m\in I}
		c_{m}
		z_{m-1}
		\in u_n \cdot \mathfrak{a} \cdot [s_1,s_2]
		\right\}
		\cdot
		\prod_{\substack{m\in I}}
		g_{\gamma}
		\left(
		z_{m-1}
		\right)
		\,\mathrm{d}\mathbf{z}_{}
		\ \lesssim \
		u_n^{-\gamma}
		\cdot
		\left(
		(s_2-s_1)
		\land 1
		\right)
	\end{align}
  for $
		I\subset \{1,\ldots,T\}
		\,, |I|=j\,,
  $
	where the multiplicative constant in $\lesssim$ depends only on $T$, $\gamma$, $\mathfrak{a}$, $\underline{s}$, and the coefficients $(c_m)$.
	Let now $I\subset\{1,\ldots,T\}$ with $|I|=j+1$. 
  If $c_m = 0$ for some $m\in I$, then this coordinate is not active in the indicator, so that flushing that inactive variable out of the integral using $\norm{g_{\gamma}}_{L^1(\RR)}<\infty$, we have only $j$ active variables left, so~\eqref{eq:ind_asu} applies. Therefore, we assume that $c_m\neq 0$ for all $m\in I$.  
  Write
	\begin{align*}
		 &
		\left\{
		\sum_{m\in I}
		c_m z_{m-1}
		\in
		u_n \cdot \mathfrak{a}\cdot [s_1,s_2]
		\right\}
		\\ &
		  \ \subset\
		  \bigcup_{\mathfrak{m}\in I}
		\left\{
		  c_{\mathfrak{m}}z_{\mathfrak{m}-1}
		\ \ge \
		  \frac{
			  u_n \cdot \mathfrak{a}\cdot s_1
			  }{j+1}
		\right\}
		  \ \cap\
		  \bigcup_{\mathfrak{n}\in I}
		\left\{
		  c_{\mathfrak{n}}z_{\mathfrak{n}-1}
		\ \le \
		  \frac{
			  u_n \cdot \mathfrak{a}\cdot s_2
			  }{j+1}
		\right\}
		\\ &
		  \ = \
		\ \bigcup_{\substack{\mathfrak{m},\mathfrak{n}\in I \\\mathfrak{m}\neq \mathfrak{n}}}
		\left\{
		c_{\mathfrak{m}}z_{\mathfrak{m}-1}
		\ \ge \
		\frac{
			u_n \cdot \mathfrak{a}\cdot s_1
		}{j+1}
		\qquad\text{and}\qquad
		c_{\mathfrak{n}}z_{\mathfrak{n}-1}
		\ \le \
		\frac{
			u_n \cdot \mathfrak{a}\cdot s_2
		}{j+1}
		\right\}
		\setminus
		\left(
		\bigcup_{\mathfrak{m}\in I}
		\left\{
		c_\mathfrak{m} z_{\mathfrak{m}-1}\in
		\frac{
			u_n \cdot \mathfrak{a}
		}{j+1}
		\cdot [s_1,s_2]
		\right\}
		\right)
		\\ &
		  \qquad \cup\ \bigcup_{\mathfrak{m}\in I}
		\left\{
		  c_\mathfrak{m} z_{\mathfrak{m}-1}\in
		  \frac{
			  u_n \cdot \mathfrak{a}
		}{j+1}
		\cdot [s_1,s_2]
		\right\}
		  \,.
	\end{align*}
	Since
	\begin{align*}
		\left(
		\bigcup_{\mathfrak{m}\in I}
		\left\{
		c_\mathfrak{m} z_{\mathfrak{m}-1}\in
		\frac{
			u_n \cdot \mathfrak{a}
		}{j+1}
		\cdot [s_1,s_2]
		\right\}
		\right)^c
		\ = \
		\bigcap_{\mathfrak{m}\in I}
		\left\{
		c_\mathfrak{m} z_{\mathfrak{m}-1}
		\ > \
		\frac{
			u_n \cdot \mathfrak{a}
		}{j+1}
		s_2
		\qquad \text{or}\qquad
		c_\mathfrak{m} z_{\mathfrak{m}-1}
		\ < \
		\frac{
			u_n \cdot \mathfrak{a}
		}{j+1}
		s_1
		\right\}
		\,,
	\end{align*}
	it holds on the intersection with this event that
	\begin{align*}
		 &
		\ \bigcup_{\substack{\mathfrak{m},\mathfrak{n}\in I \\\mathfrak{m}\neq \mathfrak{n}}}
		\left\{
		c_{\mathfrak{m}}z_{\mathfrak{m}-1}
		\ \ge \
		\frac{
			u_n \cdot \mathfrak{a}\cdot s_1
		}{j+1}
		\qquad\text{and}\qquad
		c_{\mathfrak{n}}z_{\mathfrak{n}-1}
		\ \le \
		\frac{
			u_n \cdot \mathfrak{a}\cdot s_2
		}{j+1}
		\right\}
		\\ &
		  \ = \
		\ \bigcup_{\substack{\mathfrak{m},\mathfrak{n}\in I \\\mathfrak{m}\neq \mathfrak{n}}}
		\left\{
		c_{\mathfrak{m}}z_{\mathfrak{m}-1}
		\ > \
		\frac{
			u_n \cdot \mathfrak{a}\cdot s_2
		}{j+1}
		\qquad\text{and}\qquad
		c_{\mathfrak{n}}z_{\mathfrak{n}-1}
		<
		\frac{
			u_n \cdot \mathfrak{a}\cdot s_1
		}{j+1}
		\right\}
		\,.
	\end{align*}
	It follows that
	\begin{align*}
		 &
		\left\{
		\sum_{m\in I}
		c_m z_{m-1}
		\in
		u_n \cdot \mathfrak{a}\cdot [s_1,s_2]
		\right\}
		\\ &
		  \ \subset\
		\ \bigcup_{\substack{\mathfrak{m},\mathfrak{n}\in I \\\mathfrak{m}\neq \mathfrak{n}}}
		\left\{
		c_{\mathfrak{m}}z_{\mathfrak{m}-1}
		\ > \
		\frac{
			u_n \cdot \mathfrak{a}\cdot s_2
		}{j+1}
		\qquad\text{and}\qquad
		c_{\mathfrak{n}}z_{\mathfrak{n}-1}
		\ < \
		\frac{
			u_n \cdot \mathfrak{a}\cdot s_1
		}{j+1}
		\right\}
		\\ &
		  \qquad\cup \
		  \bigcup_{\mathfrak{m}\in I}
		\left\{
		  c_\mathfrak{m} z_{\mathfrak{m}-1}\in
		  \frac{
			  u_n \cdot \mathfrak{a}
		}{j+1}
		\cdot [s_1,s_2]
		\right\}
		  \,.
	\end{align*}
	Then
	\begin{align*}
		 & \int_{\RR^{|I|}}
		\ind
		\left\{
		\sum_{m\in I}
		c_{m}
		z_{m-1}
		\in u_n \cdot \mathfrak{a} \cdot [s_1,s_2]
		\right\}
		\cdot
		\prod_{\substack{m\in I}}
		g_{\gamma}
		\left(
		z_{m-1}
		\right)
		\,\mathrm{d}\mathbf{z}_{}                        \\
		 &
		\leq \sum_{\mathfrak{m}\in I} \int_{\RR^{|I|}}
		\ind
		\left\{
		c_{\mathfrak{m}} z_{\mathfrak{m}-1}\in
		\frac{
			u_n \cdot \mathfrak{a}
		}{j+1}
		\cdot [s_1,s_2]
		\right\}
		\cdot
		\prod_{\substack{m\in I}}
		g_{\gamma}
		\left(
		z_{m-1}
		\right)
		\,\mathrm{d}\mathbf{z}_{}                        \\
		 &
		+ \sum_{\substack{\mathfrak{m},\mathfrak{n}\in I \\\mathfrak{m}\neq \mathfrak{n}}} \int_{\RR^{|I|}}
		\ind
		\left\{
		c_{\mathfrak{m}}z_{\mathfrak{m}-1}
		\ > \
		\frac{
			u_n \cdot \mathfrak{a}\cdot s_2
		}{j+1}
		\qquad\text{and}\qquad
		c_{\mathfrak{n}}z_{\mathfrak{n}-1}
		\ < \
		\frac{
			u_n \cdot \mathfrak{a}\cdot s_1
		}{j+1}
		\right\}
		\\ & \qquad\qquad\times\
		  \ind
		  \left\{
		  \sum_{m\in I}
		c_{m}
		z_{m-1}
		\in u_n \cdot \mathfrak{a} \cdot [s_1,s_2]
		\right\}
		  \cdot
		  \prod_{\substack{m\in I}}
		g_{\gamma}
		\left(
		  z_{m-1}
		  \right)
		\,\mathrm{d}\mathbf{z}_{}
	\end{align*}
	and we control each term. On the diagonal $\mathfrak{m}=\mathfrak{n}$, the bound follows immediately by \eqref{eq:ind_asu} for $j=1$ by setting
	\begin{align*}
		\widetilde{s}_{1}
		\ = \
		\frac{s_1}{j+1}
		\qquad\text{and}\qquad
		\widetilde{s}_{2}
		\ = \
		\frac{s_2}{j+1}
		\,, \qquad\text{so that}\qquad
		\left( (\widetilde{s}_{2} - \widetilde{s}_{1}) \wedge 1 \right) \leq \left( (s_{2} - s_{1}) \wedge 1 \right),
	\end{align*}
	and integrating out all other variables.
	Next we lay out how we tackle the off-diagonal part.
	If $s_2=\infty$, there is nothing to do. Therefore, assume that $s_2<\infty$.
	For the part off the diagonal we
	draw the smaller variable into the $g_{\gamma}$ of the larger variable. Then we use monotonicity on the half-line of $g_{\gamma}$ to fix the smaller variable to a constant. We then transform back and distribute the constant to the interval $[s_1,s_2]$. Finally, we use the assumption of the mathematical induction.
	Let us execute this. Let $\mathfrak{m},\mathfrak{n}\in I$ satisfying $\mathfrak{m}\neq \mathfrak{n}$ such that
	\begin{align*}
		c_{\mathfrak{m}}z_{\mathfrak{m}-1}
		\ > \
		\frac{
			u_n \cdot \mathfrak{a}\cdot s_2
		}{j+1}
		\qquad\text{and}\qquad
		c_{\mathfrak{n}}z_{\mathfrak{n}-1}
		\ <\
		\frac{
			u_n \cdot \mathfrak{a}\cdot s_1
		}{j+1}
		\,.
	\end{align*}
	Then
	\begin{align*}
		c_{\mathfrak{m}}z_{\mathfrak{m}-1}
		\ - \
		c_{\mathfrak{n}}z_{\mathfrak{n}-1}
		\ > \
		c_{\mathfrak{m}}z_{\mathfrak{m}-1}
		\ - \
		\frac{
			u_n \cdot \mathfrak{a}\cdot s_1
		}{j+1}
		\ >\
		\frac{
			u_n \cdot \mathfrak{a}
		}{j+1}
		\left(s_2 - s_1
		\right)
		\ \ge \
		0
		\,,
	\end{align*}
	so that by the monotonicity of $g_\gamma$ on $[0,\infty)$,
	\begin{align*}
		g_{\gamma}
		\left(
		c_{\mathfrak{m}}z_{\mathfrak{m}-1}
		\ - \
		c_{\mathfrak{n}}z_{\mathfrak{n}-1}
		\right)
		\ \le \
		g_\gamma
		\left(
		c_{\mathfrak{m}}z_{\mathfrak{m}-1}
		\ - \
		\frac{
			u_n \cdot \mathfrak{a}\cdot s_1}{j+1}
		\right)
		\,.
	\end{align*}
	By
	\begin{align}
		\label{eq:g_sandwitch}
		g_{\gamma}(z_{\mathfrak{m}-1})
		\
    \lesssim\  g_{\gamma}(c_{\mathfrak{m}}z_{\mathfrak{m}-1})
		\ \lesssim\
		g_{\gamma}(z_{\mathfrak{m}-1})
		\,,
	\end{align}
	where the multiplicative constant in $\lesssim$ depends only on $c_{\mathfrak{m}}$ and $\gamma$.
	With the change of variables
	\begin{align*}
		z_{\mathfrak{m}-1}
		\ \mapsto\
		z_{\mathfrak{m}-1}
		\ - \
		\frac{c_{\mathfrak{n}}}{c_{\mathfrak{m}}}
		z_{\mathfrak{n}-1}
		\,,
	\end{align*}
	we get
	\begin{align*}
		 &
		\int_{\RR^{|I|}}
		\ind
		\left\{
		c_{\mathfrak{m}}z_{\mathfrak{m}-1}
		\ > \
		\frac{
			u_n \cdot \mathfrak{a}\cdot s_2
		}{j+1}
		\qquad\text{and}\qquad
		c_{\mathfrak{n}}z_{\mathfrak{n}-1}
		<
		\frac{
			u_n \cdot \mathfrak{a}\cdot s_1
		}{j+1}
		\right\}
		\\ & \qquad\times\
		  \ind
		  \left\{
		  \sum_{m\in I}
		c_{m}
		z_{m-1}
		\in u_n \cdot \mathfrak{a} \cdot [s_1,s_2]
		\right\}
		  \cdot
		  \prod_{\substack{m\in I}}
		g_{\gamma}
		\left(
		  z_{m-1}
		  \right)
		\,\mathrm{d}\mathbf{z}_{}
		\\ &
		  \ \lesssim \
		  \int_{\RR^{|I|}}
		\ind
		  \left\{
		  c_{\mathfrak{m}}z_{\mathfrak{m}-1}
		\ > \
		  \frac{
			  u_n \cdot \mathfrak{a}\cdot s_2
			  }{j+1}
		\qquad\text{and}\qquad
		  c_{\mathfrak{n}}z_{\mathfrak{n}-1}
		<
		  \frac{
			  u_n \cdot \mathfrak{a}\cdot s_1
			  }{j+1}
		\right\}
		\\ & \qquad\times\
		  \ind
		  \left\{
		  \sum_{m\in I}
		c_{m}
		z_{m-1}
		\in u_n \cdot \mathfrak{a} \cdot [s_1,s_2]
		\right\}
		  g_{\gamma}(c_{\mathfrak{m}}z_{\mathfrak{m}-1})
		  \cdot
		\prod_{\substack{m\in I \\m\neq \mathfrak{m}}}
		g_{\gamma}
		\left(
		z_{m-1}
		\right)
		\,\mathrm{d}\mathbf{z}_{}
		\\ &
		  \ = \
		  \int_{\RR^{|I|}}
		\ind
		  \left\{
		  c_{\mathfrak{m}}z_{\mathfrak{m}-1}
		\ > \
		  \frac{
			  u_n \cdot \mathfrak{a}\cdot s_2
			  }{j+1}
		\qquad\text{and}\qquad
		  c_{\mathfrak{n}}z_{\mathfrak{n}-1}
		<
		  \frac{
			  u_n \cdot \mathfrak{a}\cdot s_1
			  }{j+1}
		\right\}
		\\ & \qquad\times\
		  \ind
		  \left\{
		\sum_{\substack{m\in I  \\ m\neq \mathfrak{n}}}
		c_{m}
		z_{m-1}
		\in u_n \cdot \mathfrak{a} \cdot [s_1,s_2]
		\right\}
		g_{\gamma}
		\left(
		c_{\mathfrak{m}}z_{\mathfrak{m}-1}
		\ - \
		c_{\mathfrak{n}}z_{\mathfrak{n}-1}
		\right)
		\cdot
		\prod_{\substack{m\in I \\ m\neq \mathfrak{m}}}
		g_{\gamma}
		\left(
		z_{m-1}
		\right)
		\,\mathrm{d}\mathbf{z}_{}
		\\ &
		  \ \le \
		  \int_{\RR^{|I|}}
		\ind
		  \left\{
		\sum_{\substack{m\in I  \\ m\neq \mathfrak{n}}}
		c_{m}
		z_{m-1}
		\in u_n \cdot \mathfrak{a} \cdot [s_1,s_2]
		\right\}
		g_{\gamma}
		\left(
		c_{\mathfrak{m}}z_{\mathfrak{m}-1}
		\ - \
		\frac{u_n \cdot \mathfrak{a}}{j+1}s_1
		\right)
		\cdot
		\prod_{\substack{m\in I \\ m\neq \mathfrak{m}}}
		g_{\gamma}
		\left(
		z_{m-1}
		\right)
		\,\mathrm{d}\mathbf{z}_{}
		\,,
	\end{align*}
	where the multiplicative constant in $\lesssim$ depends only on $c_{\mathfrak{m}}$ and $\gamma$.
	By the change of variables
	\begin{align*}
		z_{\mathfrak{m}-1}
		\ \mapsto\
		z_{\mathfrak{m}-1}
		\ + \
		\frac{1}{c_{\mathfrak{m}}}
		\frac{u_n \cdot \mathfrak{a}}{j+1}
		\cdot
		s_1
	\end{align*}
	and Lemma~\ref{lem:5.1}(iii) again, this equals
	\begin{align*}
		 &
		\int_{\RR^{|I|}}
		\ind
		\left\{
		\sum_{\substack{m\in I   \\m\neq \mathfrak{n}}}
		c_{m}
		z_{m-1}
		\in u_n \cdot \mathfrak{a} \cdot
		\left[
			s_{1}
			-
			\frac{s_1}{j+1}
			\,,
			s_{2}
			-
			\frac{s_1}{j+1}
			\right]
		\right\}
		g_{\gamma}(c_{\mathfrak{m}}z_{\mathfrak{m}-1})
		\prod_{\substack{m\in I  \\ m\neq \mathfrak{m}}}
		g_{\gamma}
		\left(
		z_{m-1}
		\right)
		\,\mathrm{d}\mathbf{z}_{}
		\\ &
		  \ \lesssim \
		  \int_{\RR^{|I|}}
		\ind
		  \left\{
		\sum_{\substack{m\in I   \\m\neq \mathfrak{n}}}
		c_{m}
		z_{m-1}
		\in u_n \cdot \mathfrak{a} \cdot
		\left[
			s_{1}
			-
			\frac{s_1}{j+1}
			\,,
			s_{2}
			-
			\frac{s_1}{j+1}
			\right]
		\right\}
		\prod_{m\in I}
		g_{\gamma}
		\left(
		z_{m-1}
		\right) \,\mathrm{d}\mathbf{z}_{}
		\\ &
		  \ \lesssim \
		  \norm{g_{\gamma}}_{L^1(\RR)}
		\int_{\RR^{|I|-1}}
		\ind
		  \left\{
		\sum_{\substack{m\in I   \\m\neq \mathfrak{n}}}
		c_{m}
		z_{m-1}
		\in u_n \cdot \mathfrak{a} \cdot
		\left[
			s_{1}
			-
			\frac{s_1}{j+1}
			\,,
			s_{2}
			-
			\frac{s_1}{j+1}
			\right]
		\right\}
		\prod_{{\substack{m\in I \\ m\neq \mathfrak{n}}}}
		g_{\gamma}
		\left(
		z_{m-1}
		\right) \,\mathrm{d}\mathbf{z}_{}
		\\ &
		  \ \lesssim \
		  u_n^{-\gamma}
		\left( \left(
			  s_{2}
			  -
			  \frac{s_1}{j+1}
			\ - \
			  \left(
			  s_{1}
			  -
			  \frac{s_1}{j+1}
			  \right)
			\right) \wedge 1 \right)
		\\ &
		  \ = \
		  u_n^{-\gamma}\cdot ((s_2-s_1) \wedge 1)
		  \,,
	\end{align*}
	where the second to last inequality follows from \eqref{eq:g_sandwitch}, and
	the last inequality follows from $|I\setminus \{\mathfrak{n}\}|=j$ and \eqref{eq:ind_asu} with
	\begin{align*}
		\widetilde{s}_{1}
		\ = \
		s_1
		\ - \
		\frac{s_1}{j+1}
		\qquad\text{and}\qquad
		\widetilde{s}_{2}
		\ = \
		s_2
		\ - \
		\frac{s_1}{j+1}
		\,.
	\end{align*}
	The multiplicative constant in $\lesssim$ depends only on $T$, $(c_m)$, $\mathfrak{a},$ $\gamma$,
	and $\underline{s}$.
	Applying mathematical induction finishes the proof.
\end{proof}
Next we provide the variants of the bounds needed for the subsequent analysis.
\begin{lemma}[Separable product bound II]
	\label{lem:comprehensive}
	Let $\emptyset\neq I\subset \{1,\ldots,T\}$, $(c_m)_{m\in I}\subset \RR$, $\mathfrak{a}\in(0,\infty)$, $\gamma\in(0,\alpha - 1)$,
	and let $\varphi,\psi:\{1,\ldots,T\}\to\RR$ and $\mathfrak{x}\in\RR$. Then
	\begin{enumerate}
		\item
		      \begin{align*}
			       &
			      \left(
			      1\land |\mathfrak{x}|
			      \right)
			      \int_{\RR^{|I|}}
			      \ind
			      \left\{
			      \left| \sum_{m\in I}
			      c_{m}z_{m-1} \right|
			      \in u_n \cdot \mathfrak{a} \cdot [s_1,s_2]
			      \right\}
			      \cdot
			      \prod_{m\in I}
			      g_{\gamma}
			      \left(
			      z_{m-1}+
			      \varphi(m)+
			      \psi(m)
			      \right)
			      \,\mathrm{d}\mathbf{z}_{}
			      \\ &
			        \ \lesssim \
			        u_n^{-\gamma}\cdot ((s_2-s_1)\land 1)
			        \cdot
			        \left(
			        1
			        +
			        \left|
			        \sum_{m\in I}
			      c_m
			        \varphi(m)
			      \right|^{1+\gamma}
			        \right)
			      \cdot
			        \left(
			        |\mathfrak{x}|+
			        \left|
			        \sum_{m\in I}
			      c_m \psi(m)
			      \right|^{1+\gamma}
			        \right)
			      \,,
		      \end{align*}
		\item
		      \begin{align*}
			       &
			      \int_{\mathfrak{b}\land \mathfrak{c}}^{\mathfrak{b}\lor\mathfrak{c}}
			      \left(
			      1 \land |s\cdot \mathfrak{x} |
			      \right)
			      \left(
			      \int_{\RR^{|I|}}
			      \ind
			      \left\{
			      \left| \sum_{m\in I}
			      c_{m}z_{m-1} \right|
			      \in u_n \cdot \mathfrak{a} \cdot [s_1,s_2]
			      \right\}
			      \cdot
			      \prod_{m\in I}
			      g_{\gamma}
			      \left(
			      z_{m-1}+
			      \varphi(m)
			      +
			      s\cdot\psi(m)
			      \right)
			      \,\mathrm{d}\mathbf{z}_{}
			      \right)
			      \,\mathrm{d}s
			      \\ &
			        \ \lesssim \
			        u_n^{-\gamma}\cdot ((s_2-s_1)\land 1)
			      \\ &
			        \qquad \times\
			        \left(
			        \left|
			        \sum_{m\in I}
			      c_m
			        \psi(m)
			      \right|^{\gamma}
			        \ + \
			        |\mathfrak{x}|^{\gamma}
			        \right)
			      \ \cdot \
			        \left(
			        1 +
			        \left|
			        \sum_{m\in I}
			      c_m
			        \varphi(m)
			      \right|^{1+\gamma}
			        \right)
			      \cdot
			        \left(
			        |\mathfrak{b}|^{1+\gamma}
			        +
			        |\mathfrak{c}|^{1+\gamma}
			        \right)
		      \end{align*}
	\end{enumerate}
	where $\lesssim$ means inequality up to a multiplicative constant that depends only on $T$, $(c_m)$, $\mathfrak{a}$, $\gamma$, and $\underline{s}$, and is independent of $\varphi$, $\psi$, $\mathfrak{x}$, $n$, $s_1$ and $s_2$.
\end{lemma}
\begin{proof}
	If $c_m=0$ for all $m\in I$,
	\begin{align*}
		\ind
		\left\{
		\left| \sum_{m\in I}
		c_{m}z_{m-1} \right|
		\in u_n \cdot \mathfrak{a}\cdot [s_1,s_2]
		\right\}
		\ = \ 0
		\,,
	\end{align*}
	so there is nothing to show. Therefore we assume $|\{m\in I\mid c_m \neq 0\}|>0$. Since
	\begin{align*}
		\sum_{m\in I}c_mz_{m-1}
		\
		=
		\
		\sum_{\substack{m\in I \\ c_m \neq 0}} c_m z_{m-1}
		\,,
	\end{align*}
	and
	\begin{align*}
		\int_{\RR^{|\{m\in I\mid c_m =0\}|}}
		\prod_{\substack{m\in I \\ c_m = 0}}
		g_{\gamma}(z_{m-1}+\varphi(m))
		\,\mathrm{d}\mathbf{z}_{}
		\ = \
		\norm{g_{\gamma}}_{L^1(\RR)}^{
			|\{m\in I\mid c_m =0\}|
		}
		\ < \infty\,,
	\end{align*}
	we may assume without loss of generality that $c_m\neq 0$ for all $m\in I$.
	\\
	\textbf{Proof of (i):}
	For all $m\in I$ we make the change of variables
	\begin{align*}
		z_{m-1}
		\ \mapsto\
		z_{m-1}
		-(\varphi(m)+\psi(m))
	\end{align*}
	to get
	\begin{align*}
		 &
		\int_{\RR^{|I|}}
		\ind
		\left\{
		\left| \sum_{m\in I}
		c_{m}z_{m-1} \right|
		\in u_n \cdot \mathfrak{a} \cdot [s_1,s_2]
		\right\}
		\cdot
		\prod_{m\in I}
		g_{\gamma}
		\left(
		z_{m-1}+
		\varphi(m)
		+\psi(m)
		\right)
		\,\mathrm{d}\mathbf{z}_{}
		\\ &
		  \ = \
		  \int_{\RR^{|I|}}
		\ind
		  \left\{
		  \left| \sum_{m\in I}
		c_{m}
		\left(
		  z_{m-1}
		  -
		  \varphi(m)
		-\psi(m)
		\right) \right|
		  \in u_n \cdot \mathfrak{a} \cdot [s_1,s_2]
		\right\}
		  \cdot
		  \prod_{m\in I}
		g_{\gamma}
		\left(
		  z_{m-1}
		  \right)
		\,\mathrm{d}\mathbf{z}_{}
		\,.
	\end{align*}
	Next we invoke Lemma~\ref{lem:5.1}(i) to separate the $\mathbf{z}{}$ variable from $\varphi$ and $\psi$ in $g_\gamma$. We cannot do this separately for each $m\in I$, because this would yield a product that is unsuitable for the subsequent analysis. We therefore fix $\mathfrak{m}\in I$, shift all the relevant terms to the corresponding $g_\gamma$, and invoke Lemma~\ref{lem:5.1}(i) once.
	Making the change of variables
	\begin{align*}
		z_{\mathfrak{m}-1}
		\ \mapsto\
		z_{\mathfrak{m}-1}
		\ + \
		\frac{1}{c_{\mathfrak{m}}}
		\sum_{m \in I}
		c_m
		\left(
		\varphi(m)
		+\psi(m)
		\right)
	\end{align*}
	this equals
	\begin{align*}
		 &
		\int_{\RR^{|I|}}
		\ind
		\left\{
		\left| \sum_{m\in I}
		c_{m}z_{m-1} \right|
		\in u_n \cdot \mathfrak{a} \cdot [s_1,s_2]
		\right\}
		\cdot
		\prod_{\substack{m\in I \\ m\neq \mathfrak{m}}}
		g_{\gamma}
		\left(
		z_{m-1}
		\right)
		\cdot
		g_{\gamma}
		\left(
		z_{\mathfrak{m}-1}
		\ + \
		\frac{1}{
			c_{\mathfrak{m}}
		}
		\sum_{m\in I}
		c_m 
    \left( \varphi(m)
		\ + \
		\psi(m)
		\right) 
    \right)
		\,\mathrm{d}\mathbf{z}_{}
		\,.
	\end{align*}
	Applying Lemma~\ref{lem:5.1}(i) twice,
	\begin{align*}
		 &
		g_{\gamma}
		\left(
		z_{\mathfrak{m}-1}
		\ + \
		\frac{1}{
			c_{\mathfrak{m}}
		}
		\sum_{m\in I}
		c_m
		\varphi(m)
		\ + \
		\frac{1}{
			c_{\mathfrak{m}}
		}
		\sum_{m\in I}
		c_m
		\psi(m)
		\right)
		\\ &
		  \ \lesssim \
		  g_{\gamma}(z_{\mathfrak{m}-1})
		  \left(
		  1
		  \ + \
		  \left|
		  \frac{1}{c_{\mathfrak{m}}}
		  \sum_{m\in I}
		c_m
		  \varphi(m)
		\right|^{1+\gamma}
		  \right)
		\left(
		  1
		  \ + \
		  \left|
		  \frac{1}{c_{\mathfrak{m}}}
		  \sum_{m\in I}
		c_m
		  \psi(m)
		\right|
		  ^{1+\gamma}
		  \right)
		\\ &
		  \ \lesssim \
		  g_{\gamma}(z_{\mathfrak{m}-1})
		  \left(
		  1
		  \ + \
		  \left|
		  \sum_{m\in I}
		c_m
		  \varphi(m)
		\right|^{1+\gamma}
		  \right)
		\left(
		  1
		  \ + \
		  \left|
		  \sum_{m\in I}
		c_m
		  \psi(m)
		\right|^{1+\gamma}
		  \right)
		\,,
	\end{align*}
	where the multiplicative constant in $\lesssim$ depends only on $\gamma$.
	The inequality
	\begin{align*}
		\left(1\land|\mathfrak{x}|
		\right)
		\cdot
		\left(
		1
		\ + \
		\left|
		\sum_{m\in I}
		c_m
		\psi(m)
		\right|^{1+\gamma}
		\right)
		\ \le \
		\left(
		|\mathfrak{x}|
		\ + \
		\left|
		\sum_{m\in I}
		c_m
		\psi(m)
		\right|^{1+\gamma}
		\right)
		\,,
	\end{align*}
	together with Lemma~\ref{lem:gist_separation} yields the claimed inequality.
	\\
	\textbf{Proof of (ii):}
	For all $m\in I$ we make the change of variables
	\begin{align*}
		z_{m-1}
		\ \mapsto\
		z_{m-1}
		-(\varphi(m)+s\cdot \psi(m))
	\end{align*}
	to get
	\begin{align*}
		 &
		\int_{\RR^{|I|}}
		\ind
		\left\{
		\left| \sum_{m\in I}
		c_{m}z_{m-1} \right|
		\in u_n \cdot \mathfrak{a} \cdot [s_1,s_2]
		\right\}
		\cdot
		\prod_{m\in I}
		g_{\gamma}
		\left(
		z_{m-1}+
		\varphi(m)
		+s\cdot \psi(m)
		\right)
		\,\mathrm{d}\mathbf{z}_{}
		\\ &
		  \ = \
		  \int_{\RR^{|I|}}
		\ind
		  \left\{
		  \left| \sum_{m\in I}
		c_{m}
		\left(
		  z_{m-1}
		  -
		  \varphi(m)
		-s\cdot\psi(m)
		\right) \right|
		  \in u_n \cdot \mathfrak{a} \cdot [s_1,s_2]
		\right\}
		  \cdot
		  \prod_{m\in I}
		g_{\gamma}
		\left(
		  z_{m-1}
		  \right)
		\,\mathrm{d}\mathbf{z}_{}
		\,.
	\end{align*}
	Fix $\mathfrak{m}\in I$. Making the change of variables
	\begin{align*}
		z_{\mathfrak{m}-1}
		\ \mapsto\
		z_{\mathfrak{m}-1}
		\ + \
		\frac{1}{c_{\mathfrak{m}}}
		\sum_{m \in I}
		c_m
		\left(
		\varphi(m)
		+s\cdot \psi(m)
		\right)
	\end{align*}
	this equals
	\begin{align*}
		 &
		\int_{\RR^{|I|}}
		\ind
		\left\{
		\left| \sum_{m\in I}
		c_{m}z_{m-1} \right|
		\in u_n \cdot \mathfrak{a} \cdot [s_1,s_2]
		\right\}
		\cdot
		\prod_{\substack{m\in I \\ m\neq \mathfrak{m}}}
		g_{\gamma}
		\left(
		z_{m-1}
		\right)
		\cdot
		g_{\gamma}
		\left(
		z_{\mathfrak{m}-1}
		\ + \
		\frac{1}{
			c_{\mathfrak{m}}
		}
		\sum_{m\in I}
		c_m
    \left(
		\varphi(m)
		\ + \
	s
		\psi(m)
    \right)
		\right)
		\,\mathrm{d}\mathbf{z}_{}
		\,.
	\end{align*}

	By Lemma~\ref{lem:5.1}(iv),
	with
	\begin{align*}
		\mathfrak{x}
		= \mathfrak{x}
		\,,
		\qquad \mathfrak{y}
		\ = \
		\frac{1}{c_{\mathfrak{m}}}
		\sum_{m\in I}
		c_m \psi(m)
		\,,
		\qquad
		\mathfrak{z}
		\ = \
		z_{\mathfrak{m}-1}
		\ + \
		\frac{1}{c_{\mathfrak{m}}}
		\sum_{m\in I}
		c_m
		\varphi(m)
		\,,
	\end{align*}
	and then
	Lemma~\ref{lem:5.1}(i),
	\begin{align*}
		 &
		\int_{\mathfrak{b}\land \mathfrak{c}}^{\mathfrak{b}\lor \mathfrak{c}}
		(1\land |s\cdot \mathfrak{x}|)
		\cdot
		g_{\gamma}
		\left(
		z_{\mathfrak{m}-1}
		\ + \
		\frac{1}{
			c_{\mathfrak{m}}
		}
		\sum_{m\in I}
		c_m
		\left(
		\varphi(m)
		+
		s\cdot \psi(m)
		\right)
		\right) \,\mathrm{d}s
		\\ &
		  \ \lesssim \
		  \left(
		  \left|
		  \sum_{m\in I}
		c_m
		  \psi(m)
		\right|^{\gamma}
		  \ + \
		  |\mathfrak{x}|^{\gamma}
		  \right)
		\ \cdot \
		  \left(
		  1 +
		  \left|
		  \sum_{m\in I}
		c_m
		  \varphi(m)
		\right|^{1+\gamma}
		  \right)
		\cdot
		  \left(
		  |\mathfrak{b}|^{1+\gamma}
		  +
		  |\mathfrak{c}|^{1+\gamma}
		  \right)
		\cdot g_{\gamma}(z_{\mathfrak{m}-1})
		  \,,
	\end{align*}
	where the multiplicative constant in $\lesssim$ depends only on $\gamma$.
	Applying Lemma~\ref{lem:gist_separation} yields the result.
\end{proof}

The next lemma flushes the inactive variables out of the analysis, while keeping enough degrees of freedom for the subsequent analysis.
\begin{lemma}[Flushing inactive variables]
	\label{lem:flush_inactive}
	Let $\gamma\in (0,\infty)$ and
	$\mathfrak{y}_1,\ldots,\mathfrak{y}_T\in \RR$. Then
	\begin{align*}
		 &
		\int_{\RR^T}
		G_{n}(\mathbf{A}_{T}\cdot\mathbf{z}_{},s_1,s_2,A)
		\prod_{m=1}^T g_{\gamma}(z_{m-1}+\mathfrak{y}_m) \,\mathrm{d}\mathbf{z}_{}
		\\ &
		  \ \lesssim \
		  \sum_{j=1}^T
		  \int_{\RR^{j}}
		\ind
		  \left\{
		  \left|
		  \sum_{m=1}^j
		  a_{j-m} z_{m-1}
		\right|
		  \in u_n \cdot \mathfrak{a}_j\cdot [s_1,s_2]
		\right\}
		  \prod_{m=1}^j
		  g_{\gamma}(z_{m-1}+\mathfrak{y}_m)
		  \,\mathrm{d}\mathbf{z}_{}
		\,,
	\end{align*}
	where $\lesssim$ means inequality up to a multiplicative constant that depends only on $T$ and $\gamma$.
\end{lemma}
\begin{proof}
	By \eqref{lem:two_sets},
	\begin{align*}
		G_{n}(\mathbf{A}_{T}\cdot\mathbf{z}_{},s_1,s_2,A)
		\ \le \
		\sum_{j=1}^T
		\ind
		\left\{
		\left|
		\sum_{m=1}^j
		a_{j-m} z_{m-1}
		\right|
		\in u_n \cdot \mathfrak{a}_j\cdot [s_1,s_2]
		\right\}
		\,,
	\end{align*}
	so that
	\begin{align*}
		 &
		\int_{\RR^T}
		G_{n}(\mathbf{A}_{T}\cdot\mathbf{z}_{},s_1,s_2,A)
		\prod_{m=1}^T
		g_{\gamma}(z_{m-1}+\mathfrak{y}_m)
		\,\mathrm{d}\mathbf{z}_{}
		\\ &
		  \ \le \
		  \sum_{j=1}^T
		  \left(
		  \int_{\RR^{j}}
		\ind
		  \left\{
		  \left|
		  \sum_{m=1}^j
		  a_{j-m} z_{m-1}
		\right|
		  \in u_n \cdot \mathfrak{a}_j\cdot [s_1,s_2]
		\right\}
		  \prod_{m=1}^j
		  g_{\gamma}(z_{m-1}+\mathfrak{y}_m)
		  \,\mathrm{d}\mathbf{z}_{}
		\right.
		\\ &
		  \left.
		  \qquad\times \
		  \int_{\RR^{T-j}}
		\prod_{m=j+1}^T
		  g_{\gamma}(z_{m-1}+\mathfrak{y}_m)
		  \,\mathrm{d}\mathbf{z}_{}
		\right)
		  \,.
	\end{align*}
	For $m>j$, we make the changes of variables
	\begin{align*}
		z_{m-1}
		\ \mapsto\ z_{m-1} - \mathfrak{y}_m
		\,,
	\end{align*}
	so that
	\begin{align*}
		\int_{\RR^{T-j}}
		\prod_{m=j+1}^T
		g_{\gamma}(z_{m-1}+\mathfrak{y}_m)
		\,\mathrm{d}\mathbf{z}_{}
		\ = \
		\left(
		\norm{g_{\gamma}}_{L^1(\RR)}
		\right)^{T-j}
		\ < \ \infty
		\,.
	\end{align*}
	This gives the claimed upper bound.
\end{proof}

We finish the section with a lemma that bounds two single probabilistic terms
which will appear in the subsequent analysis. Since the bounds developed above
are analytic in nature, the lemma also illustrates how they apply to the probabilistic problem at hand.
\begin{lemma}[Separable product bound III]
	\label{lem:single}
	Let Assumptions~\ref{asu:f} and~\ref{asu:coef}
	hold, let $T\ge 1$, let $\gamma\in(0,\alpha - 1)$.
	Then,
	for $n\in\NN$, we have the following upper bounds:
	\begin{enumerate}
		\item
		      \begin{align*}
			      \norm{
				      \nabla G_{\infty,n}(\mathbf{0}_{T},
				      s_1,s_2
				      ,A
				      )
			      }_{\infty}
			      \ \lesssim \
			      u_n^{-\gamma}\cdot ((s_2-s_1)\land 1)
			      \,,
		      \end{align*}
		\item
		      \begin{align*}
			      \PP[[X_{\ell}]_\ell \in u_n \cdot ((s_1\cdot A) \setminus (s_2 \cdot A))]
			      \ \lesssim \
			      u_n^{-\gamma}\cdot
			      ((s_2-s_1)\land 1)
			      \,.
		      \end{align*}
	\end{enumerate}
	In both (i) and (ii)
	$\lesssim$ means inequality up to a multiplicative constant that
	only depends on $A$, $T$, $\underline{s}$,
	$\gamma$, the coefficients $(a_i)$, and the distribution of $\varepsilon$, and
	is independent of $s_1$, $s_2$, and $n$.
\end{lemma}
\begin{proof}[\textbf{Proof of (i):}]
	By Lemma~\ref{lem:multi_swap},
	\begin{align*}
		\norm{
			\nabla G_{\infty,n}(\mathbf{0}_{T},
			s_1,s_2
			,A
			)
		}_{\infty}
		 &
		\ \le \
		\int_{\RR^T}
		G_{n}(\mathbf{z}_{},
		s_1,s_2
		,A
		)
		\norm{ \nabla f_{[X_{\ell}]_{\ell}}(\mathbf{z}_{})}_{\infty}
		\,\mathrm{d}\mathbf{z}_{}
		\\ &
		  \ = \
		  \int_{\RR^T}
		G_{n}(\mathbf{A}_{T}\cdot \mathbf{z}_{},
		  s_1,s_2
		  ,A
		  )
		  \norm{ \nabla f_{[X_{\ell}]_{\ell}}(\mathbf{A}_{T}\cdot\mathbf{z}_{})}_{\infty}
		\,\mathrm{d}\mathbf{z}_{}
		\,,
	\end{align*}
	where the last equality follows from the change of variables $\mathbf{z}_{}\mapsto \mathbf{A}_{T}\cdot \mathbf{z}_{}$, with $\det \mathbf{A}_{T}=1$.
	By Lemma~\ref{lem:multi_density} 
	and Assumption~\ref{asu:f},
	\begin{align*}
		\norm{\nabla f_{[X_{\ell}]_\ell}(\mathbf{A}_{T}\cdot\mathbf{z}_{})}_{\infty}
		 &
		\ \lesssim \
		\sum_{m=1}^T
		\EE
		\left[
			\left|
			f'_\varepsilon(z_{m-1} - \mathbf{e}_{m}^{\top}\cdot\mathbf{A}_{T}^{-1}\cdot
			[X_{\ell}-X_{\ell,\ell}]_\ell
			)
			\right|
			\prod_{h\neq m}
			f_\varepsilon(z_{h-1}- \mathbf{e}_{h}^{\top}\cdot\mathbf{A}_{T}^{-1}\cdot
			[X_{\ell}-X_{\ell,\ell}]_\ell
			)
			\right]
		\\ &
		  \ \lesssim \
		  \EE
		  \left[
			  \prod_{m=1}^T
			  g_{\gamma}(z_{m-1}- \mathbf{e}_{m}^{\top}\cdot\mathbf{A}_{T}^{-1}\cdot
			  [X_{\ell}-X_{\ell,\ell}]_\ell
			  )
			  \right]
		\,,
	\end{align*}
	where the multiplicative constant in $\lesssim$ depends only on $T$, the coefficients $(\mathfrak{a}_i)$, and the distribution of $\varepsilon$.
	By Lemma~\ref{lem:flush_inactive},
	\begin{align*}
		 &
		\int_{\RR^T}
		G_n(\mathbf{A}_{T}\cdot \mathbf{z}_{},s_1,s_2,A)
		\prod_{m=1}^T
		g_{\gamma}(z_{m-1}- \mathbf{e}_{m}^{\top}\cdot\mathbf{A}_{T}^{-1}\cdot
		[X_{\ell}-X_{\ell,\ell}]_\ell
		)
		\,\mathrm{d}\mathbf{z}_{}
		\\ &
		  \ \lesssim \
		  \sum_{j=1}^T
		  \int_{\RR^j}
		\ind
		  \left\{
		  \left| \sum_{m=1}^j
		  a_{j-m} z_{m-1} \right|
		  \in u_n \cdot \mathfrak{a}_j \cdot [s_1,s_2]
		\right\}
		  \prod_{m=1}^j
		  g_{\gamma}(z_{m-1}- \mathbf{e}_{m}^{\top}\cdot\mathbf{A}_{T}^{-1}\cdot
		  [X_{\ell}-X_{\ell,\ell}]_\ell
		  )
		  \,\mathrm{d}\mathbf{z}_{}
		\,,
	\end{align*}
	Since
	\begin{align*}
		\sum_{m=1}^j
		a_{j-m}
		\mathbf{e}_{m}^{\top}
		\ = \
		\mathbf{e}_{j}^{\top}\mathbf{A}_{T}
	\end{align*}
	it holds
	\begin{align*}
		\sum_{m=1}^j
		a_{j-m}
		\cdot
		\mathbf{e}_{m}^{\top}
		\cdot
		\mathbf{A}_{T}^{-1}
		\cdot
		[X_{\ell}-X_{\ell,\ell}]_{\ell}
		\ = \
		\mathbf{e}_{j}^{\top}
		\cdot
		[X_{\ell}-X_{\ell,\ell}]_{\ell}
		\ = \
		X_{j-1}
		-
		X_{j-1,j-1}
		\,.
	\end{align*}
	Applying Lemma~\ref{lem:comprehensive}(i) with
	\begin{align*}
		\varphi(m)
		\ = \
		\ - \
		\mathbf{e}_{m}^{\top}
		\cdot
		\mathbf{A}_{T}^{-1}
		\cdot
		[X_{\ell}-X_{\ell,\ell}]_{\ell}
		\,,\qquad
		\psi
		\ \equiv \ 0
		\,,
		\qquad m\in \{1,\ldots,j\},
	\end{align*}
	we get
	\begin{align*}
		 &
		\sum_{j=1}^T
		\int_{\RR^j}
		\ind
		\left\{
		\left| \sum_{m=1}^j
		a_{j-m} z_{m-1} \right|
		\in u_n \cdot \mathfrak{a}_j \cdot [s_1,s_2]
		\right\}
		\prod_{m=1}^j
		g_{\gamma}(z_{m-1}- \mathbf{e}_{m}^{\top}\cdot\mathbf{A}_{T}^{-1}\cdot
		[X_{\ell}-X_{\ell,\ell}]_\ell
		)
		\,\mathrm{d}\mathbf{z}_{}
		\\ &
		  \ \lesssim \
		  u_n^{-\gamma}\cdot ((s_2-s_1)\land 1) \cdot
		  \sum_{j=1}^T
		  \left(
		  1
		  +
		  \left|
		  X_{j-1}-X_{j-1,j-1}
		  \right|^{1+\gamma}
		  \right)
		\,.
	\end{align*}
	To obtain the claimed upper bound $u_n^{-\gamma}((s_2-s_1)\land 1)$, note that $1+\gamma <\alpha$, so that
	\begin{align*}
		\EE
		\left[
			\sum_{j=1}^T
			\left(
			1
			+
			\left|
			X_{j-1}-X_{j-1,j-1}
			\right|^{1+\gamma}
			\right)
			\right]
		\ = \
		T
		\left(
		1
		\ + \
		\EE[
			\left|
			X_{0}-X_{0,0}
			\right|^{1+\gamma}
		]
		\right)
		\ < \ \infty\,.
	\end{align*}
	\\
	\textbf{Proof of (ii):}
	It holds $u_n\cdot ((s_1\cdot A)\setminus (s_2\cdot A))\subset \bigcup_{
			j=1}^T\{|x_{j-1}|\in \mathfrak{a}_j\cdot u_n\cdot[s_1,s_2]\}$, so that, by stationarity of $(X_t)$,
	\[
		\PP[[X_{\ell}]_\ell \in u_n \cdot (s_{1}\cdot A \setminus s_2\cdot A)]
		\ \le \
		\sum_{j=1}^T
		\PP[|X_0| \in \mathfrak{a}_j\cdot u_n\cdot [s_1,s_2]]
		\,.
	\]
	Then, for 
  $j\in \{1,\ldots,T\}$,
	\begin{align*}
		 &
		\PP[|X_0| \in \mathfrak{a}_j\cdot u_n\cdot [s_1,s_2]]
		\ = \
		\int_{\RR}
		\ind\{|z_0| \in \mathfrak{a}_j\cdot u_n \cdot[s_1,s_2]\}
		\cdot
		f_{X_0}(z_0)
		\,\mathrm{d}z_0
		\\ &
		  \ = \
		  \int_{\RR}
		\ind\{|z_0| \in \mathfrak{a}_j\cdot u_n\cdot [s_1,s_2]\}
		  \cdot
		  \EE[f_{\varepsilon}(z_0 - (X_0-X_{0,0}))]
		\,\mathrm{d}z_0
		\\ &
		  \ \lesssim \
		  \EE[1\lor |X_0-X_{0,0}|^{1+\gamma}]
		\int_{\RR}
		\ind\{|z_0| \in \mathfrak{a}_j\cdot u_n \cdot[s_1,s_2]\}
		  \cdot
		  g_{\gamma}(z_0)
		  \,\mathrm{d}z_0
		\\ &
		  \ \lesssim \
		  u_n^{-\gamma}
		\cdot
		  ((s_2-s_1)\land 1)
		  \,,
	\end{align*}
	since $1+\gamma<\alpha$, by Lemma~\ref{lem:5.1}(i), and
	\begin{align*}
		 &
		\int_{\RR}
		\ind\{|z_0| \in \mathfrak{a}_j\cdot u_n \cdot[s_1,s_2]\}
		\cdot
		g_{\gamma}(z_0)
		\,\mathrm{d}z_0
		\ \lesssim \
		u_n^{-\gamma}\cdot ((s_2-s_1)\land 1)
		\,,
	\end{align*}
	by Lemma~\ref{lem:gist_separation},
	where the constant in $\lesssim$ depends only on $A$, $\underline{s}$, $\gamma$, the coefficients $(\mathfrak{a}_i)$, and the distribution of $\varepsilon$, and is independent of $s_1$, $s_2$, and $n$.
\end{proof}
\section{Analysis of gradient differences}
In the subsequent analysis we have to control differences of gradients
of the form
\[
	\nabla_{\mathbf{z}} G_{k,n}(\mathbf{z},s_1,s_2,A)
	\Big|_{\mathbf{z}=\mathbf{x}_{}+\mathbf{y}_{}}
	\ - \
	\nabla_{\mathbf{z}} G_{k,n}(\mathbf{z},s_1,s_2,A)
	\Big|_{\mathbf{z}=\mathbf{x}_{}}
	\quad\text{or}\quad
	\nabla_{\mathbf{x}} G_{\infty,n}(\mathbf{x},s_1,s_2,A)
	\ - \
	\nabla_{\mathbf{x}} G_{k,n}(\mathbf{x},s_1,s_2,A)
	\,,
\]
where $\nabla G_{k,n}$ is defined in 
Lemma~\ref{lem:multi_swap} for $k\in \NN\cup \{\infty\}$.
These contain differences of
products that create algebraic complexity.
To organize this, we introduce the
\(\widetilde R\)-notation.
The following lemma then explains how this notation will be used. Part (i) rewrites the
relevant differences of products in terms of \(\widetilde R\)-terms.
Part (ii) gives a bound for these terms that is useful in the broader analysis.
The final two lemmas of this section
express the gradient differences in \(\widetilde R\)-notation,
and apply the bounds of the last section to yield useful bounds for the final analysis.

\begin{definition}[$\widetilde{R}$-notation]
	\label{def:Rtilde}
	We define, for functions
	$\varphi,\psi:\{1,\ldots,T\}\to \RR$
	and $i,l\in \{1,\ldots,T\}$, 
	\begin{align*}
		 &
		\widetilde{R}_{i,l}
		(\varphi,\psi)
		\\
		 &
		\ := \
		\begin{cases}
			\left(
			\prod_{\substack{j<l}} f_{\varepsilon}(\varphi(j))
			\right)
			\cdot
			(f_{\varepsilon}(\varphi(l))-f_{\varepsilon}(\varphi(l)+\psi(l)))\cdot f'_{\varepsilon}(\varphi(i) + \psi(i))\cdot
			\left(
			\prod_{\substack{j>l  \\ j\neq i}} f_{\varepsilon}(\varphi(j)+\psi(j))
			\right)
			\,, & \quad l < i\,, \\
			\left(
			\prod_{\substack{j<i}} f_{\varepsilon}(\varphi(j))
			\right)
			\cdot
			(f'_{\varepsilon}(\varphi(i))-f'_{\varepsilon}(\varphi(i)+\psi(i)))\cdot
			\left(
			\prod_{\substack{j>i }} f_{\varepsilon}(\varphi(j)+\psi(j))
			\right)
			\,, & \quad l = i\,, \\
			f'_{\varepsilon}(\varphi(i))\cdot
			\left(
			\prod_{\substack{j<l  \\ j\neq i}} f_{\varepsilon}(\varphi(j))
			\right)
			\cdot
			(f_{\varepsilon}(\varphi(l))-f_{\varepsilon}(\varphi(l)+\psi(l)))\cdot
			\left(
			\prod_{\substack{j>l}} f_{\varepsilon}(\varphi(j)+\psi(j))
			\right)
			\,, & \quad l > i\,.
		\end{cases}
	\end{align*}
\end{definition}
\begin{lemma}[Properties of $\widetilde{R}_{i,l}$]
	\label{lem:aux_R}
	Let Assumption~\ref{asu:f} hold.
	Then, for all functions $\varphi,\psi:\{1,\ldots,T\}\to \RR$, and all $i,l\in \{1,\ldots,T\}$, the following statements about $\widetilde{R}_{i,l}(\varphi,\psi)$ are true:
	\begin{enumerate}
		\item
		      \begin{align*}
			      \sum_{l=1}^T
			      \widetilde{R}_{i,l}(\varphi,\psi)
			      \ = \
			      f_{\varepsilon}'(\varphi(i))
			      \cdot
			      \prod^T_{\begin{smallmatrix}
						               j=1\\ j\neq i
					               \end{smallmatrix}}
			      f_{\varepsilon}(\varphi(j))
			      \ - \
			      f_{\varepsilon}'(\varphi(i)+\psi(i))
			      \cdot
			      \prod^T_{\begin{smallmatrix}
						               j=1\\j\neq i
					               \end{smallmatrix}}
			      f_{\varepsilon}(\varphi(j)+\psi(j))
		      \end{align*}
		\item
		      For all $\gamma\in(0,\alpha - 1)$,
		      \begin{align*}
			      \left|
			      \widetilde{R}_{i,l}(
			      \varphi,\psi
			      )
			      \right|
			       &
			      \ \lesssim \
			      \left(
			      1 \land
			      |\psi(l)|
			      \right)
			      \cdot
			      \prod_{m= 1}^{T}
			      \left(
			      g_{\gamma}(\varphi(m))
			      \ + \
			      g_{\gamma}(\varphi(m)+\psi(m))
			      \right)
			      \\ &
			        \ = \
			        \left(
			        1 \land
			        |\psi(l)|
			        \right)
			      \sum_{I\subset\{1,\ldots,T\}}
			      \left(
			        \prod_{m\notin I}
			      g_{\gamma}(\varphi(m))
			        \right)
			        \left(
			        \prod_{m\in I}
			      g_{\gamma}(\varphi(m)+\psi(m))
			        \right)
			        \,.
		      \end{align*}
	\end{enumerate}
	In (ii) $\lesssim$ means inequality up to a multiplicative constant
	that only depends on
	the distribution of $\varepsilon$, and is independent of
	$T$, $\gamma$,
	$\varphi$, $\psi$, $i$ and $\ell$.
\end{lemma}
\begin{proof}
	Fix $i,l\in\{1,\ldots,T\}$.
	\\
	\textbf{Proof of (i):}
	The statement follows from the definition of $\widetilde{R}_{i,l}$ and
	\begin{align*}
		 &
		c_i\cdot \prod_{\substack{j=1                                \\ j\neq i}}^T a_j
		\
		-
		\
		d_i \cdot \prod_{\substack{j=1                               \\ j\neq i}}^T b_j
		\\
		 & \ =\  \sum_{\substack{l=1}}^{i-1}
		\left(
		\prod_{\substack{j<l}} a_j
		\right)
		(a_{l}-b_{l})\cdot d_i\cdot
		\left(
		\prod_{\substack{j>l                                         \\ j\neq i}} b_j
		\right)
		\\ &
		  \qquad
		  +
		  \
		  \left(\prod_{j<i} a_j\right)(c_i-d_i)\left(\prod_{j>i} b_j\right)
		\\ &
		  \qquad
		  +
		  \
		\sum_{\substack{l=i+1}}^T c_i\cdot\left(\prod_{\substack{j<l \\ j\neq i}} a_j\right)(a_{l}-b_{l})\left(\prod_{\substack{j>l \\ j\neq i}} b_j\right)\,,
	\end{align*}
	with
	\begin{align*}
		a_j
		\ = \
		f_{\varepsilon}(\varphi(j))
		\,,
		\qquad
		b_j
		\ = \
		f_{\varepsilon}(\varphi(j)+\psi(j))
		\,,
		\qquad
		c_i
		\ = \
		f'_{\varepsilon}(\varphi(i))
		\,,
		\qquad
		d_i
		\ = \
		f'_{\varepsilon}(\varphi(i)+\psi(i))
		\,.
	\end{align*}
	\\
	\textbf{Proof of (ii):}
	Expanding the product
	\begin{align}
		\label{eq:lhs}
		\prod_{m=1}^T (a_m + b_m)
		\ = \
		\sum_{I\subset\{1,\ldots,T\}}
		\prod_{m\notin I} a_m
		\prod_{m\in I} b_m
		\,,
	\end{align}
	with
	\begin{align*}
		a_m \ = \ g_{\gamma}(\varphi(m))
		\,,
		\qquad
		b_m \ = \ g_{\gamma}(\varphi(m)  +  \psi(m))
		\,,\qquad m\in \{1,\ldots,T\}
		\,,
	\end{align*}
	it remains to prove that the left-hand side of \eqref{eq:lhs} is a valid upper bound for $|\widetilde{R}_{i,l}(\varphi,\psi)|$.
	\\
	\textit{Case $|\psi(l)|<1$:}
	In the product structure of $\widetilde{R}_{i,l}$ there appears exactly one factor that is a difference term, namely the $l$-th term.
	We begin by bounding the absolute value of this difference term. If this features $f'_{\varepsilon}$, we use Assumption~\ref{asu:f}, and if this features $f_{\varepsilon}$, we use Lemma~\ref{lem:lip_f}. The difference term is then bounded above by
	\begin{align*}
		|\psi(l)|
		\cdot
		g_{\gamma}(\varphi(l))
		\ \lesssim \
		\left(
		1\land
		|\psi(l)|
		\right)
		\cdot
		\left(
		g_{\gamma}(\varphi(l))
		\ + \
		g_{\gamma}(\varphi(l)+\psi(l))
		\right)
		\,,
	\end{align*}
	where the multiplicative constant in $\lesssim$ depends only on
	the distribution of $\varepsilon$.
	By Assumption~\ref{asu:f},
	any remaining factor
	$j\in \{1,\ldots,T\}\setminus\{l\}$ is bounded above by
	\begin{align*}
		g_{\gamma}(\varphi(j))
		\ + \
		g_{\gamma}(\varphi(j)+\psi(j))
		\,.
	\end{align*}
	up to a  multiplicative constant that depends only on the distribution of $\varepsilon$.
	\\
	\textit{Case $|\psi(l)|\ge 1$:}
	We use the triangle inequality and Assumption~\ref{asu:f} to bound the absolute value of the 
	difference term above by
	\begin{align*}
		g_{\gamma}(\varphi(l))
		\ + \
		g_{\gamma}(\varphi(l)+\psi(l))
		\,,
	\end{align*}
	up to a  multiplicative constant that depends only on the distribution of $\varepsilon$.
	All remaining terms are bounded as in the previous case. Since $|\psi(l)|\ge 1$, it holds $(1\land|\psi(l)|)=1$, so we multiply this factor to the bound for free.
\end{proof}
The next lemma expresses gradient differences in terms of $\widetilde{R}_{}$-notation.
\begin{lemma}[Gradient differences and $\widetilde{R}$-notation]
	\label{lem:grad:lip}
	Let Assumptions~\ref{asu:f} and~\ref{asu:coef} hold.
	Then, for all $\mathbf{x}_{},\mathbf{y}_{}\in \RR^{T}$ and all $k\in\NN_0\cup\{\infty\}$ we have the following identities:
	\begin{enumerate}
		\item
		      \begin{align*}
			       &
			      \nabla_{\mathbf{z}} G_{k,n}
			      (\mathbf{z}_{},s_1,s_2,A)
			      \Big|_{\mathbf{z}=\mathbf{x}+\mathbf{y}}
			      \ - \
			      \nabla_{\mathbf{z}} G_{k,n}
			      (\mathbf{z}_{},s_1,s_2,A)
			      \Big|_{\mathbf{z}=\mathbf{x}}
			      \\ &
			        \ = \
			        \left(
			        \mathbf{A}_{T}^{-1}
			        \right)^{\top}
			      \\ &
			        \qquad \times\
			        \left[
				        \int_{\RR^T}
				      G_n
				        (\mathbf{A}_{T}\cdot \mathbf{z}_{},s_1,s_2,A)
				        \,\mathrm{d}\mathbf{z}_{}
				      \int_{\RR^T}
				      \,\mathrm{d}
				      \PP_{[X_{\ell,k+\ell}-X_{\ell,\ell}]_\ell}
				        (\widetilde{\mathbf{x}_{}}_{})
				      \right.
			      \\ &
				        \qquad
				        \qquad
				        \left.
				        \widetilde{R}_{i,l}
				      \left(
				        \left(
					        j
					        \
					        \mapsto
					        \
					        z_{j-1} - \mathbf{e}_{j}^{\top}
					      \cdot
					        \mathbf{A}_{T}^{-1}
					      \cdot
					        (\mathbf{x}_{}+\widetilde{\mathbf{x}}_{})
					      \right)
				      \,,
				        \left(
					        j
					        \
					        \mapsto
					        \
					        - \mathbf{e}_{j}^{\top}
					        \cdot
					        \mathbf{A}_{T}^{-1}
					        \cdot
					        \mathbf{y}_{}
					        \right)
				      \right)
				      \right]_{i,l\in \{1,\ldots,T\}}
			      \cdot
			        \mathbf{1}_{T}
			      \,.
		      \end{align*}
		\item
		      \begin{align*}
			      \begin{split}
				       &
				      \nabla_{\mathbf{x}} G_{\infty,n}(\mathbf{x}_{},s_1,s_2,A)
				      \ - \
				      \nabla_{\mathbf{x}} G_{k,n}(\mathbf{x}_{},s_1,s_2,A)
				      \\ &
				        \ = \
				        \left(
				        \mathbf{A}_{T}^{-1}
				        \right)^{\top}
				      \\ &
				        \qquad
				        \times
				        \left[
					        \int_{\RR^T}
					      G_n
					        (\mathbf{A}_{T}\cdot \mathbf{z}_{},s_1,s_2,A)
					        \,\mathrm{d}\mathbf{z}_{}
					      \int_{\RR^T}
					      \,\mathrm{d}\PP_{[X_{\ell}-X_{\ell,k+\ell}]_{\ell}}(\widetilde{\mathbf{x}_{}}_{})
					      \int_{\RR^T}
					      \,\mathrm{d}\PP_{[X_{\ell,k+\ell}-X_{\ell,\ell}]_{\ell}}(\widetilde{\mathbf{y}_{}}_{})
					      \right.
				      \\ &
					        \qquad
					        \qquad
					        \widetilde{R}_{i,l}
					      \left(
					        \left(
						        j\ \mapsto \
						        z_{j-1}
						        \ - \
						        \mathbf{e}_{j}^{\top}
						      \cdot
						        \mathbf{A}_{T}^{-1}
						      \cdot
						        \left(
						        \mathbf{x}_{}
						        \ + \
						        \widetilde{\mathbf{y}_{}}_{}
						        \right)
						      \right)
					      \left. ,
					        \left(
						        j\ \mapsto \
						        -
						        \mathbf{e}_{j}^{\top}
						        \cdot
						        \mathbf{A}_{T}^{-1}
						        \cdot
						        \widetilde{\mathbf{x}_{}}_{}
						        \right)
					      \right)
					      \right]_{i,l\in\{1,\ldots,T\}}
				      \cdot
				        \mathbf{1}_{T}
			      \end{split}
		      \end{align*}
	\end{enumerate}
\end{lemma}
\begin{proof}[\textbf{Proof of (i):}]
	Note that
	for arbitrary $\mathbf{x}_{}\in \RR^T$ it holds that
	\begin{align*}
		 &
		\nabla_{\mathbf{x}} G_{k,n}
		(\mathbf{x}_{},s_1,s_2,A)
		\\ &
		  \ = \
		  -
		  \int_{\RR^T}
		G_n
		  (\mathbf{z}_{},s_1,s_2,A)
		  \nabla f_{[X_{\ell,k+\ell}]_\ell}
		(\mathbf{z}_{}-\mathbf{x}_{})
		  \,\mathrm{d}\mathbf{z}_{}
		\\ &
		  \ = \
		  -
		  \sum_{i=1}^{T}
		\left(
		  \mathbf{A}^{-1}_T
		  \right)^{\top}
		\mathbf{e}_i
		\\ &
		  \qquad\quad
		  \int_{\RR^T}
		G_n
		  (\mathbf{z}_{},s_1,s_2,A)
		  \,\mathrm{d}\mathbf{z}_{}
		\int_{\RR^T}
		\,\mathrm{d}
		\PP_{[X_{\ell,k+\ell}-X_{\ell,\ell}]_\ell}
		  (\widetilde{\mathbf{x}_{}}_{})
		\\ &
		  \qquad\qquad
		  f'_\varepsilon
		  \left(
		  \mathbf{e}_{i}^\top\cdot \mathbf{A}_{T}^{-1}
		\cdot
			  (\mathbf{z}_{}-\mathbf{x}_{}-\widetilde{\mathbf{x}_{}}_{})
		\right)
		\prod_{j\neq i}
		f_\varepsilon
		  \left(
		  \mathbf{e}_{j}^\top\cdot \mathbf{A}_{T}^{-1}
		\cdot
			  (\mathbf{z}_{}-\mathbf{x}_{}-\widetilde{\mathbf{x}_{}}_{})
		\right)
      \displaybreak[0]
		\\ &
		  \ = \
		  -
		  \sum_{i=1}^{T}
		\left(
		  \mathbf{A}^{-1}_T
		  \right)^{\top}
		\mathbf{e}_i
		\\ &
		  \qquad\quad
		  \int_{\RR^T}
		G_n
		  (\mathbf{A}\cdot\mathbf{z}_{},s_1,s_2,A)
		  \,\mathrm{d}\mathbf{z}_{}
		\int_{\RR^T}
		\,\mathrm{d}
		\PP_{[X_{\ell,k+\ell}-X_{\ell,\ell}]_\ell}
		  (\widetilde{\mathbf{x}_{}}_{})
		\\ &
		  \qquad\qquad
		  f'_\varepsilon
		  \left(
		  z_{i-1}-\mathbf{e}_{i}^\top\cdot \mathbf{A}_{T}^{-1}
		\cdot
		  (\mathbf{x}_{}+\widetilde{\mathbf{x}_{}}_{})
		\right)
		\prod_{j\neq i}
		f_\varepsilon
		  \left(
		  z_{j-1}-\mathbf{e}_{j}^\top\cdot \mathbf{A}_{T}^{-1}
		\cdot
		  (\mathbf{x}_{}+\widetilde{\mathbf{x}_{}}_{})
		\right)
		\,,
	\end{align*}
	where the
	first equality follows from Lemma~\ref{lem:multi_swap},
	the second equality follows from Lemma~\ref{lem:multi_density}(ii) and the third equality follows from the change of variables $\mathbf{A}_{T}^{-1}\mathbf{z}_{}\mapsto \mathbf{z}_{}$, where $\det \mathbf{A}_T^{-1}= (\det \mathbf{A}_T)^{-1} =1$.
	Therefore, for arbitrary $\mathbf{y}_{}\in \RR^T$,
	\begin{align*}
		 &
		\nabla_{\mathbf{z}} G_{k,n}
		(\mathbf{z},s_1,s_2,A)
		\Big|_{\mathbf{z}=\mathbf{x}+\mathbf{y}}
		\ - \
		\nabla_{\mathbf{z}} G_{k,n}
		(\mathbf{z},s_1,s_2,A)
		\Big|_{\mathbf{z}=\mathbf{x}}
		\\ &
		  \ = \
		  \sum_{i=1}^{T}
		\left(
		  \mathbf{A}^{-1}_T
		  \right)^{\top}
		\mathbf{e}_i
		\\ &
		  \qquad\quad
		  \int_{\RR^T}
		G_n
		  (\mathbf{A}_{T}\cdot\mathbf{z}_{},s_1,s_2,A)
		  \,\mathrm{d}\mathbf{z}_{}
		\int_{\RR^T}
		\,\mathrm{d}
		\PP_{[X_{\ell,k+\ell}-X_{\ell,\ell}]_\ell}
		  (\widetilde{\mathbf{x}_{}}_{})
		\\ &
		  \qquad\qquad
		  \bigg[
			  f'_\varepsilon
			  \left(
			  z_{i-1}-\mathbf{e}_{i}^\top\cdot \mathbf{A}_{T}^{-1}
			\cdot
			  (\mathbf{x}_{}+\widetilde{\mathbf{x}_{}}_{})
			\right)
			\prod_{j\neq i}
			f_\varepsilon
			  \left(
			  z_{j-1}-\mathbf{e}_{j}^\top\cdot \mathbf{A}_{T}^{-1}
			\cdot
			  (\mathbf{x}_{}+\widetilde{\mathbf{x}_{}}_{})
			\right)
		\\ &
			  \qquad\qquad
			  \ - \
			  f'_\varepsilon
			  \left(
			  z_{i-1}-\mathbf{e}_{i}^\top\cdot \mathbf{A}_{T}^{-1}
			\cdot
			  (\mathbf{x}+\mathbf{y}+\widetilde{\mathbf{x}_{}}_{})
			\right)
			\prod_{j\neq i}
			f_\varepsilon
			  \left(
			  z_{j-1}-\mathbf{e}_{j}^\top\cdot \mathbf{A}_{T}^{-1}
			\cdot
			  (\mathbf{x}+\mathbf{y}+\widetilde{\mathbf{x}_{}}_{})
			\right)
			\bigg]
		\,,
	\end{align*}
	which, by Lemma~\ref{lem:aux_R}(i), is equal to
	\begin{align*}
		 &
		\sum_{i=1}^{T}
		\left(
		\mathbf{A}^{-1}_T
		\right)^{\top}
		\mathbf{e}_i
		\\ &
		  \quad
		  \int_{\RR^T}
		G_n
		  (\mathbf{A}_{T}\cdot \mathbf{z}_{},s_1,s_2,A)
		  \,\mathrm{d}\mathbf{z}_{}
		\int_{\RR^T}
		\,\mathrm{d}
		\PP_{[X_{\ell,k+\ell}-X_{\ell,\ell}]_\ell}
		  (\widetilde{\mathbf{x}_{}}_{})
		\\ &
		  \qquad \sum_{l=1}^T
		  \widetilde{R}_{i,l}
		\left(
		  \left(
			  j
			  \
			  \mapsto
			  \
			  z_{j-1} - \mathbf{e}_{j}^{\top}
			\cdot
			  \mathbf{A}_{T}^{-1}
			\cdot
			  (\mathbf{x}_{}+\widetilde{\mathbf{x}}_{})
			\right)
		\,,
		  \left(
			  j
			  \
			  \mapsto
			  \
			  - \mathbf{e}_{j}^{\top}
			  \cdot
			  \mathbf{A}_{T}^{-1}
			  \cdot
			  \mathbf{y}_{}
			  \right)
		\right)
		\\ &
		  \ = \
		  \left(
		  \mathbf{A}_{T}^{-1}
		  \right)^{\top}
		\\ &
		  \qquad
		  \times\
		  \left[
			  \int_{\RR^T}
			G_n
			  (\mathbf{A}_{T}\cdot \mathbf{z}_{},s_1,s_2,A)
			  \,\mathrm{d}\mathbf{z}_{}
			\int_{\RR^T}
			\,\mathrm{d}
			\PP_{[X_{\ell,k+\ell}-X_{\ell,\ell}]_\ell}
			  (\widetilde{\mathbf{x}_{}}_{})
			\right.
		\\ &
			  \left.
			  \qquad
			  \qquad
			  \widetilde{R}_{i,l}
			\left(
			  \left(
				  j
				  \
				  \mapsto
				  \
				  z_{j-1} - \mathbf{e}_{j}^{\top}
				\cdot
				  \mathbf{A}_{T}^{-1}
				\cdot
				  (\mathbf{x}_{}+\widetilde{\mathbf{x}}_{})
				\right)
			\,,
			  \left(
				  j
				  \
				  \mapsto
				  \
				  - \mathbf{e}_{j}^{\top}
				  \cdot
				  \mathbf{A}_{T}^{-1}
				  \cdot
				  \mathbf{y}_{}
				  \right)
			\right)
			\right]_{i,l\in \{1,\ldots,T\}}
		\cdot
		  \mathbf{1}_{T}
		\,.
	\end{align*}
	\\
	\textbf{Proof of (ii):}
	From Lemma~\ref{lem:multi_swap} it follows for all $\mathbf{x}_{}\in\RR^T$ that 
	\begin{align}
		\label{eq:gf1}
		\begin{split}
			 &
			\nabla_{\mathbf{x}} G_{\infty,n}(\mathbf{x}_{},s_1,s_2,A)
			\ - \
			\nabla_{\mathbf{x}} G_{k,n}(\mathbf{x}_{},s_1,s_2,A)
			\\ &
			  \ = \
			  \int_{\RR^{T}}
			G_{n}(\mathbf{z}_{},s_1,s_2,A)
			  \left(
			  \nabla f_{[X_{\ell,k+\ell}]_{\ell}}(\mathbf{z}_{}-\mathbf{x}_{})
			  \ - \
			  \nabla f_{[X_{\ell}]_{\ell}}(\mathbf{z}_{}-\mathbf{x}_{})
			  \right)
			  \,\mathrm{d}\mathbf{z}_{}
			\\ &
			  \ = \
			  \int_{\RR^{T}}
			G_{n}(\mathbf{A}_{T}\cdot\mathbf{z}_{},s_1,s_2,A)
			  \left(
			  \nabla f_{[X_{\ell,k+\ell}]_{\ell}}(\mathbf{A}_{T}\cdot\mathbf{z}_{}-\mathbf{x}_{})
			  \ - \
			  \nabla f_{[X_{\ell}]_{\ell}}(\mathbf{A}_{T}\cdot\mathbf{z}_{}-\mathbf{x}_{})
			  \right)
			  \,\mathrm{d}\mathbf{z}_{}
			\,,
		\end{split}
	\end{align}
	where for the change of variables in the last step we used again $\det \mathbf{A}_T=1$.
	Note that
	by Lemma~\ref{lem:multi_density}(ii) again,
	\begin{align*}
		 &
		\nabla f_{[X_{\ell,k+\ell}]_{\ell}}(\mathbf{A}_{T}\cdot\mathbf{z}_{}-\mathbf{x}_{})
		\ - \
		\nabla f_{[X_{\ell}]_{\ell}}(\mathbf{A}_{T}\cdot\mathbf{z}_{}-\mathbf{x}_{})
		\\ &
		  \ = \
		  \left(
		  \mathbf{A}_{T}^{-1}
		  \right)^{\top}
		\\ &
		  \times
		  \left[
			  \EE
			  \left[
				  f'_\varepsilon(
				  z_{i-1}
				  \ - \
          \mathbf{e}_i^{\top}
          \cdot
          \mathbf{A}_T^{-1}
          \cdot
				  (\mathbf{x}_{}+
				  [X_{\ell,k+\ell}-X_{\ell,\ell}]_\ell
				  )
				)
				\prod_{j\neq i}
				f_\varepsilon(
				  z_{j-1}
				  \ - \
          \mathbf{e}_j^{\top}
          \cdot
          \mathbf{A}_T^{-1}
          \cdot
				  (\mathbf{x}_{}+
				  [X_{\ell,k+\ell}-X_{\ell,\ell}]_\ell
				  )
				)
				\right.
				  \right.
		\\ &
				  \left.
				  \left.
				  \qquad
				  - \
				  f'_\varepsilon(
				  z_{i-1}
				  \ - \
          \mathbf{e}_i^{\top}
          \cdot
          \mathbf{A}_T^{-1}
          \cdot
				  (\mathbf{x}_{}+
				  [X_{\ell}-X_{\ell,\ell}]_\ell
				  )
				)
				\prod_{j\neq i}
				f_\varepsilon(
				  z_{j-1}
				  \ - \
          \mathbf{e}_j^{\top}
          \cdot
          \mathbf{A}_T^{-1}
          \cdot
				  (
				  \mathbf{x}_{}+
				  [X_{\ell}-X_{\ell,\ell}]_\ell
				  )
				)
				\right]
			\right]_{i\in\{1,\ldots,T\}}
		\,.
	\end{align*}
	By Lemma~\ref{lem:aux_R}(i) and independence of the random vectors $[X_{\ell,k+\ell}-X_{\ell,\ell}]_\ell$ and $[X_{\ell}-X_{\ell,k+\ell}]_\ell$, this equals
	\begin{align*}
		 &
		\left(
		\mathbf{A}_{T}^{-1}
		\right)^{\top}
		\\ &
		  \times
		  \Bigg[
			  \EE
			  \Bigg[
				  \sum_{l=1}^{T}
				\widetilde{R}_{i,l}
				\big(
				  \left(
				  j\ \mapsto \
				  z_{j-1}
				  \ - \
          \mathbf{e}_j^{\top}
          \cdot
          \mathbf{A}_T^{-1}
          \cdot
				  \left(
				  \mathbf{x}_{}
				  \ + \
				  [X_{\ell,k+\ell}-X_{\ell,\ell}]_{\ell}
				  \right)
				\right)
				\,, \\
                & \qquad \qquad \qquad \quad 
				  \left(
				  j\ \mapsto \
				  -
          \mathbf{e}_j^{\top}
          \cdot
          \mathbf{A}_T^{-1}
          \cdot
				  [X_{\ell}-X_{\ell,k+\ell}]_{\ell}
				  \right)
				\big)
				  \Bigg]
			\Bigg]_{i\in\{1,\ldots,T\}}
		\\ &
		  \ = \
		  \left(
		  \mathbf{A}_{T}^{-1}
		  \right)^{\top}
		\\ &
		  \qquad
		  \times
		  \left[
			  \int_{\RR^T}
			\,\mathrm{d}\PP_{[X_{\ell}-X_{\ell,k+\ell}]_\ell}(\widetilde{\mathbf{x}_{}}_{})
			\int_{\RR^T}
			\,\mathrm{d}\PP_{[X_{\ell,k+\ell}-X_{\ell,\ell}]_\ell}(\widetilde{\mathbf{y}_{}}_{})
			\right.
		\\ &
			  \left.
			  \qquad
			  \qquad
			  \widetilde{R}_{i,l}
			\left(
			  \left(
				  j\ \mapsto \
				  z_{j-1}
				  \ - \
				  \mathbf{e}_{j}^{\top}
				\cdot
				  \mathbf{A}_{T}^{-1}
				\cdot
				  \left(
				  \mathbf{x}_{}
				  \ + \
				  \widetilde{\mathbf{y}_{}}_{}
				  \right)
				\right)
			\,,
			  \left(
				  j\ \mapsto \
				  -
				  \mathbf{e}_{j}^{\top}
				  \cdot
				  \mathbf{A}_{T}^{-1}
				  \cdot
				  \widetilde{\mathbf{x}_{}}_{}
				  \right)
			\right)
			\right]_{i\in\{1,\ldots,T\}}
		\cdot
		  \mathbf{1}_{T}
	\end{align*}
    The proof is complete.
\end{proof}
The final result of this section bounds the right-hand sides of the integrands in the gradient
difference identities 
of Lemma~\ref{lem:grad:lip}, and their variants needed later.
It thereby shows how the analytic estimates developed so far are brought into
the probabilistic setting considered in the next section.
\begin{lemma}[Separable product bounds IV]
	\label{lem:ready}
	Let Assumptions~\ref{asu:f} and~\ref{asu:coef} hold, and let
	$\mathfrak{b},\mathfrak{c}\in\RR$, $\varphi,\psi:\{1,\ldots,T\}\to\RR$, $\gamma\in(0,\alpha - 1)$, and $i,l\in\{1,\ldots,T\}$. Then:
	\begin{enumerate}
		\item
		      \begin{align*}
			       &
			      \int_{\RR^T}
			      G_n(\mathbf{A}_{T}\cdot \mathbf{z}_{},s_1,s_2,A)
			      \left|
			      \widetilde{R}_{i,l}\left(
			      j\mapsto z_{j-1} + \varphi(j), \psi
			      \right)
			      \right|
			      \,\mathrm{d}\mathbf{z}_{}
			      \\ &
			        \ \lesssim \
			        u_n^{-\gamma}\cdot ((s_2-s_1)\land 1)
			        \sum_{j=1}^T
			        \sum_{I\subset \{1,\ldots,j\}}
			      \left(
			        1+
			        \left|
			        \sum_{m=1}^j
			        a_{j-m}
			        \varphi(m)
			      \right|^{1+\gamma}
			        \right)
			      \left(
			        |\psi(l)|
			        +
			        \left|
			        \sum_{m\in I}
			      a_{j-m}
			        \psi(m)
			      \right|^{1+\gamma}
			        \right)
		      \end{align*}
		\item
		      \begin{align*}
			       &
			      \int_{\mathfrak{b}\land \mathfrak{c}}^{\mathfrak{b}\lor \mathfrak{c}}
			      \left(
			      \int_{\RR^T}
			      G_n(\mathbf{A}_{T}\cdot \mathbf{z}_{},s_1,s_2,A)
			      \left|
			      \widetilde{R}_{i,l}(
			      j\mapsto z_{j-1} + \varphi(j), s\cdot \psi)
			      \right|
			      \,\mathrm{d}\mathbf{z}_{}
			      \right)
			      \,\mathrm{d}s
			      \\ &
			        \ \lesssim \
			        u_n^{-\gamma}\cdot ((s_2-s_1)\land 1)
			      \\ &
			       \times\
			        \sum_{j=1}^T
			        \sum_{I\subset\{1,\ldots,j\}}
			      \left(
			        \left|
			        \sum_{m\in I}
			      a_{j-m} \psi(m)
			      \right|^{\gamma}
			        \ + \
			        |\psi(l)|^{\gamma}
			        \right)
			      \ \times \
			        \left(
			        1 \ + \
			        \left|
			        \sum_{m=1}^j
			        a_{j-m} \varphi(m)
			      \right|^{1+\gamma}
			        \right)
			      \left(
			        |\mathfrak{b}|^{1+\gamma}
			        +
			        |\mathfrak{c}|^{1+\gamma}
			        \right)
			      \,.
		      \end{align*}
		      In both (i) and (ii), $\lesssim$ means inequality up to a multiplicative constant
		      that depends only on $T$, $(a_m)$, $(\mathfrak{a}_j)$, $\gamma$, $\underline{s}$ and the distribution of $\varepsilon$.
	\end{enumerate}
\end{lemma}

\begin{proof}[\textbf{Proof of (i):}]
	By Lemma~\ref{lem:aux_R}(ii),
	\begin{align*}
		 &
		\left|
		\widetilde{R}_{i,l}\left(
		j\mapsto z_{j-1} + \varphi(j), \psi
		\right)
		\right|
    \\&
		\ \lesssim \
		(1\land |\psi(l)|)
		\sum_{I\subset\{1,\ldots,T\}}
		\left(
		\prod_{m\notin I}
		g_{\gamma}(z_{m-1}+\varphi(m))
		\right)
		\left(
		\prod_{m\in I}
		g_{\gamma}(z_{m-1}+\varphi(m)+\psi(m))
		\right)
		\,,
	\end{align*}
	where the multiplicative constant in $\lesssim$ depends only on the distribution of $\varepsilon$.
	We invoke this bound
	inside the integral to get
	\begin{align*}
		 &
		\int_{\RR^T}
		G_n(\mathbf{A}_{T}\cdot \mathbf{z}_{},s_1,s_2,A)
		\left|
		\widetilde{R}_{i,l}\left(
		j\mapsto z_{j-1} + \varphi(j), \psi
		\right)
		\right|
		\,\mathrm{d}\mathbf{z}_{}
		\\ &
		  \ \lesssim \
		  (1\land |\psi(l)|)
      \\&\qquad\times\ 
		  \sum_{I\subset\{1,\ldots,T\}}
		\int_{\RR^T}
		G_n(\mathbf{A}_{T}\cdot \mathbf{z}_{},s_1,s_2,A)
		  \left(
		  \prod_{m\notin I}
		g_{\gamma}(z_{m-1}+\varphi(m))
		  \right)
		  \left(
		  \prod_{m\in I}
		g_{\gamma}(z_{m-1}+\varphi(m)+\psi(m))
		  \right)
		  \,\mathrm{d}\mathbf{z}_{}
		\,,
	\end{align*}
	where the multiplicative constant in $\lesssim$ depends only on the distribution of $\varepsilon$.
	Applying Lemma~\ref{lem:flush_inactive}, for all $I\subset \{1,\ldots,T\}$, with
	\begin{align*}
		\mathfrak{y}_m
		\ = \
		\varphi(m)
		\ + \
		\ind\{m\in I\}
		\cdot
		\psi(m)
		\,,
		\qquad
		m\in \{1,\ldots,T\}
		\,,
	\end{align*}
	this is bounded above by
	\begin{align*}
		 &
		(1\land |\psi(l)|)
		\sum_{I\subset\{1,\ldots,T\}}
		\sum_{j=1}^T
		\\ &
		  \
		  \int_{\RR^j}
		\ind
		  \left\{
		  \left| \sum_{m=1}^j
		  a_{j-m}z_{m-1} \right|
		  \cdot u_n \mathfrak{a}_j \cdot [s_1,s_2]
		\right\}
      \Bigg(
		\prod_{\substack{m\notin I \\ m\le j}}
		g_{\gamma}(z_{m-1}+\varphi(m))
    \Bigg)
    \Bigg(
		\prod_{\substack{m\in I    \\ m\le j}}
		g_{\gamma}(z_{m-1}+\varphi(m)+\psi(m))
    \Bigg)
		\,\mathrm{d}\mathbf{z}_{}
		\,,
	\end{align*}
	up to a multiplicative constant that depends only on $T$, $\gamma$, and the distribution of $\varepsilon$.
	Since sets $I,I'\subset \{1,\ldots,T\}$ with $I\cap\{1,\ldots,j\}=I'\cap\{1,\ldots,j\}$ yield the same summand, this is bounded above by a multiple of
	\begin{align*}
		 &
		(1\land |\psi(l)|)
		\sum_{j=1}^T
		\sum_{I\subset\{1,\ldots,j\}}
		\\ &
		  \
		  \int_{\RR^j}
		\ind
		  \left\{
		  \left| \sum_{m=1}^j
		  a_{j-m}z_{m-1} \right|
		  \cdot u_n \cdot \mathfrak{a}_j \cdot [s_1,s_2]
		\right\}
		  \left(
		  \prod_{\substack{m\notin I}}
		g_{\gamma}(z_{m-1}+\varphi(m))
		  \right)
		  \left(
		  \prod_{\substack{m\in I}}
		g_{\gamma}(z_{m-1}+\varphi(m)+\psi(m))
		  \right)
		  \,\mathrm{d}\mathbf{z}_{}
		\,.
	\end{align*}
	By Lemma~\ref{lem:comprehensive}(i), this is then bounded above by
	\begin{align*}
		u_n^{-\gamma}\cdot ((s_2-s_1)\land 1)
		\sum_{j=1}^T
		\sum_{I\subset\{1,\ldots,j\}}
		\left(
		1+
		\left|
		\sum_{m=1}^j
		a_{j-m}
		\varphi(m)
		\right|^{1+\gamma}
		\right)
		\left(
		|\psi(l)|
		+
		\left|
		\sum_{m\in I}
		a_{j-m}
		\psi(m)
		\right|^{1+\gamma}
		\right)
		\,,
	\end{align*}
	up to a multiplicative constant that depends only on $T$, $(a_m)$, $(\mathfrak{a}_j)$, $\gamma$, $\underline{s}$ and the distribution of $\varepsilon$.
	\\
	\textbf{Proof of (ii):}
	By Lemma~\ref{lem:aux_R}(ii),
	\begin{align*}
		 &
		\left|
		\widetilde{R}_{i,l}(
		j\mapsto z_{j-1} + \varphi(j), s\cdot \psi)
		\right|
		\\ &
		  \ \lesssim \
		  \left(
		  1\land |s\cdot \psi(l)|
		  \right)
		\sum_{I\subset\{1,\ldots,T\}}
		\left(
		  \prod_{m\notin I}
		g_{\gamma}(z_{m-1}+\varphi(m))
		  \right)
		  \left(
		  \prod_{m\in I}
		g_{\gamma}(z_{m-1}+\varphi(m)+s\cdot \psi(m))
		  \right)
		  \,,
	\end{align*}
	where the multiplicative constant in $\lesssim$ depends only on the distribution of $\varepsilon$.
	Applying Lemma~\ref{lem:flush_inactive}, for all $I\subset\{1,\ldots,T\}$, with
	\begin{align*}
		\mathfrak{y}_m
		\ = \
		\varphi(m)
		\ + \
		\ind\{m\in I\}
		\cdot
		s
		\cdot \psi(m)
		\,,
		\qquad
		m\in \{1,\ldots,T\}
		\,,
	\end{align*}
	and proceeding as in the proof of (i), we get
	\begin{align*}
		 &
		\int_{\mathfrak{b}\land \mathfrak{c}}^{\mathfrak{b}\lor \mathfrak{c}}
		\left(
		\int_{\RR^T}
		G_n(\mathbf{A}_{T}\cdot \mathbf{z}_{},s_1,s_2,A)
		\left|
		\widetilde{R}_{i,l}(
		j\mapsto z_{j-1} + \varphi(j), s\cdot \psi)
		\right|
		\,\mathrm{d}\mathbf{z}_{}
		\right)
		\,\mathrm{d}s
		\\ &
		  { \ \lesssim \ }
		\sum_{j=1}^T
		  \sum_{I\subset \{1,\ldots,j\}}
		\int_{\mathfrak{b}\land \mathfrak{c}}^{\mathfrak{b}\lor \mathfrak{c}}
		\left(
		  1\land |s\cdot \psi(l)|
		  \right)
		\,\mathrm{d}s
		\\ & \qquad\times\
		  \left(
		  \int_{\RR^j}
		\ind
		  \left\{
		  \left| \sum_{m=1}^j
		  a_{j-m}z_{m-1} \right|
		  \in u_n \cdot \mathfrak{a}_j \cdot [s_1,s_2]
		\right\}
		  \right.
		\\ &
		  \left.
		  \qquad
		  \qquad
		  \qquad
		  \times\
		  \prod_{m\notin I}
		g_{\gamma}(z_{m-1}+\varphi(m))
		  \cdot
		  \prod_{m\in I}
		g_{\gamma}(z_{m-1}+\varphi(m)+s\cdot \psi(m))
		  \,\mathrm{d}\mathbf{z}_{}
		\right)
		\\
		 &
		\ \lesssim \
		u_n^{-\gamma}\cdot ((s_2-s_1)\land 1)
		\\ & \qquad\times\
		  \sum_{j=1}^T
		  \sum_{I\subset\{1,\ldots,j\}}
		\left(
		  \left|
		  \sum_{m\in I}
		a_{j-m} \psi(m)
		\right|^{\gamma}
		  \ + \
		  |\psi(l)|^{\gamma}
		  \right)
		\ \times \
		  \left(
		  1 \ + \
		  \left|
		  \sum_{m= 1}^j
		  a_{j-m} \varphi(m)
		\right|^{1+\gamma}
		  \right)
		\left(
		  |\mathfrak{b}|^{1+\gamma}
		  +
		  |\mathfrak{c}|^{1+\gamma}
		  \right)
		\,,
	\end{align*}
	where the last step follows from Lemma~\ref{lem:comprehensive}(ii).
	The multiplicative constant in $\lesssim$ depends only on
	$T$, $(a_m)$, $(\mathfrak{a}_j)$, $\gamma$, $\underline{s}$ and the distribution of $\varepsilon$.
\end{proof}
\section{Reduction Principle}
We are ready to
analyze the probabilistic part of the problem. We first formulate the main
result of the section, which is Lemma~\ref{lem:reduction_principle}. We
next provide a short proof of
Theorem~\ref{lem:reduction_principle_optimal}, assuming that Lemma~\ref{lem:reduction_principle} holds, and we then proceed to prove Lemma~\ref{lem:reduction_principle} in several steps.
\begin{lemma}[$L^r$-reduction principle]
	\label{lem:reduction_principle}
	Let Assumptions~\ref{asu:f} and~\ref{asu:coef} hold.
	For all
	$\gamma\in(0,1)$ and $r\in[1,2]$ satisfying
	\begin{align}
    \label{eq:paramrange}
		\gamma
		\ < \
		\min
		\left\{
		\frac{d}{1-d}
		\,,
		1
		-
		\frac{1}{r(1-d)}
		\,,
		\frac{\alpha}{r}
		-
		1
		\right\}
		\qquad\text{and}\qquad
		\frac{1}{1-d}
		\ < \
		r
		\ < \
		\alpha
		\,,
	\end{align}
	it holds for $n\in \NN$,
	\begin{align*}
		 &
		\EE
		\left[
			\left|
      \sum_{t=1}^{n-T+1}
			G_n([X_{t+\ell}]_\ell, s_1,s_2,A)
			\ - \
			\EE[G_n([X_{t+\ell}]_\ell,s_1,s_2,A)]
			\ - \
			\left(
			\nabla_{\mathbf{y}} G_{\infty,n}(\mathbf{y}_{},s_1,s_2,A)\Big|_{\mathbf{y}=\mathbf{0}_T}
			\right)^{\top}
			\cdot
			[X_{t+\ell}]_{\ell}
			\right|^{r}
			\right]
		\\ &
		  \ \lesssim \
		  ((s_2-s_1)\land 1)
		  \cdot
		  n^{r+1-(1-d)(1+\gamma)r}
		\,,
	\end{align*}
	where $\lesssim$ means inequality up to a multiplicative constant that depends only on $A$, $T$, $\underline{s}$, $\gamma$, $r$, the coefficients $(a_i)$, and the distribution of $\varepsilon$, and is independent of $s_1$, $s_2$, and $n$.
\end{lemma}
\begin{proof}[\textbf{Proof of Theorem~\ref{lem:reduction_principle_optimal}:}]
Following Remark~3.4 in~\cite{scheffelCentralLimitTheory2025},
  a version of the upper bound in Lemma~\ref{lem:reduction_principle} with optimal parameter values $(\gamma_0,r_0)$ is given in Theorem~\ref{lem:reduction_principle_optimal}, where
    \begin{equation}
 (\gamma_0, r_0) := \left\{ \begin{array}{l}
    \left( \dfrac{d}{1-d}, \ \alpha(1-d) \right)
    \ \text{ if } \
    \dfrac{1}{(1-d)(1-2d)}
    < \alpha\,,\\[10pt]
    \left( \dfrac{\alpha(1-d)-1}{\alpha(1-d)+1}, \ \dfrac{1}{2}
    \left(
    \alpha + \dfrac{1}{1-d}
    \right) \right)
     \ \text{ otherwise,}
    \end{array} \right.
  \end{equation}
  yielding the optimal value $\kappa_0 := \kappa(\gamma_0,r_0)=1 + 1/r_0 - (1-d)(1+\gamma_0)$
  as
  \begin{align}
  \kappa_0
    =
  \left\{ \begin{array}{l}
    \dfrac{1}{\alpha(1-d)}
    \ \text{ if } \
    \dfrac{1}{(1-d)(1-2d)}
    < \alpha\,,\\[10pt]
    \dfrac{2(1-d) + (1-\alpha(1-d)(1-2d))}{\alpha(1-d)+1} = d
    +
    (1-d)
    \dfrac{3-\alpha(1-d)}{\alpha(1-d)+1}
    \ \text{ otherwise.}
  \end{array} \right.
\end{align}
We remark that, while $\gamma_0$ lies on the boundary of the admissible parameter range~\eqref{eq:paramrange}, the value $\gamma_0-\delta$ lies inside this range for each $\delta>0$ sufficiently small. We may therefore apply Lemma~\ref{lem:reduction_principle} with $(\gamma,r)=(\gamma_0-\delta,r_0)$. 
The multiplicative constant in $\lesssim$ then also depends on $\delta$. 
\end{proof}
\subsection{Martingale decomposition}
The derivation of the martingale decomposition follows the argument that leads to
Lemma~B.3 in~\cite{scheffelCentralLimitTheory2025}, where its conceptual motivation is explained in detail.
Define, for $n\in\NN$, $t\in \mathbb{Z}$,
\begin{align*}
	 &
	\mathcal{U}_n([X_{t+\ell}]_\ell,s_1,s_2,A)
	\\ &
	:= \
	  G_n([X_{t+\ell}]_\ell,s_1,s_2,A) \ -\  \EE[G_n([X_{t+\ell}]_\ell,s_1,s_2,A)] \ -\
	  \left(
	  \nabla_{\mathbf{y}} G_{\infty,n}(\mathbf{y},s_1,s_2,A)
	  \Big|_{\mathbf{y}=\mathbf{0}_T}
	\right)^\top\cdot [X_{t+\ell}]_\ell
\end{align*}
where the gradient
$
	\nabla_{\mathbf{y}} G_{\infty,n}(\mathbf{y},s_1,s_2,A)
	\Big|_{\mathbf{y}=\mathbf{0}_T}
$ is well-defined under the assumptions of Lemma~\ref{lem:multi_swap}.
Let
$\mathcal{F}_{k}=\sigma(\varepsilon_{k},\varepsilon_{k-1},\ldots)$ 
be the past $\sigma$-algebra generated by the sequence $(\varepsilon_{t})$ up to index $k$. For $Y\in L^1(\PP)$, we define the projection
\begin{align*}
	P_k\,Y
	\ := \
	\EE[Y\mid \mathcal{F}_{k}]
	\ - \
	\EE[Y\mid \mathcal{F}_{k-1}]
	\qquad\text{for all}\ k \in\mathbb{Z}\,.
\end{align*}
\begin{lemma}[Moment bounds for the centered process $\mathcal{U}_n$ and its projection]
	\label{lem:2.2}
	Let Assumptions~\ref{asu:f} and~\ref{asu:coef} hold.
	Then, the following holds for all $r\in [1,\alpha)$:
	\begin{enumerate}
		\item
		      $\norm{\mathcal{U}_n([X_\ell]_\ell,s_1,s_2,A)}_{L^r(\PP)}\lesssim ((s_2-s_1)\land 1)^{1/r}$, and in particular
		      $\mathcal{U}_n([X_{\ell}]_{\ell},s_1,s_2,A)\in L^r(\PP)$.
		\item
		      $\norm{
				      P_k\,\mathcal{U}_n([X_{\ell}]_{\ell},s_1,s_2,A)
			      }_{L^r(\PP)}\lesssim \norm{\mathcal{U}_n([X_{\ell}]_{\ell},s_1,s_2,A)}_{L^r(\PP)}$,
		      and  $P_k\,\mathcal{U}_n([X_{\ell}]_{\ell},s_1,s_2,A)\in L^r(\PP)$ for all $k\in\mathbb{Z}$.
	\end{enumerate}
	In both (i) and (ii) $\lesssim$ means inequality up to a multiplicative constant that
	depends only on $A$, $T$, $\underline{s}$, $\gamma$, $r$, the coefficients $(a_i)$, and the distribution of $\varepsilon$, and is independent of $s_1$, $s_2$, $n$, and $k$.
\end{lemma}
\begin{remark}[Dropping $u_n^{-\gamma/r}$ rate]
	\label{rem:dropping}
	A sharper bound with a
	factor $u_n^{-\gamma/r}\in(0,1)$ is possible. Since the final analysis does not crucially depend on this factor and the gains of keeping it are small, we drop it.
\end{remark}
\begin{proof}[\textbf{Proof of (i):}]
	Since, for all $\mathbf{x}_{}\in\RR^T$,
	\begin{align*}
		\norm{\mathbf{x}_{}^\top\cdot [X_\ell]_\ell}_{L^r(\PP)}
		\ \le \
		\sum_{\ell=0}^{T-1}|x_\ell|
		\norm{X_\ell}_{L^r(\PP)}
		\ \le \
		T
		\norm{\mathbf{x}_{}}_{\infty}
		\norm{X_0}_{L^r(\PP)}
		\,,
	\end{align*}
	by stationarity of $(X_t)$, it follows from the definition of $\mathcal{U}_n$ that
	\begin{align*}
		 &
		\norm{\mathcal{U}_n([X_{\ell}]_\ell,s_1,s_2,A)}_{L^r(\PP)}
		\\
		 &
		\ \le \
		\left(
		\PP[[X_{\ell}]_{\ell}\in u_n \cdot ((s_1\cdot A)\setminus(s_2\cdot A))]
		\right)^{1/r}
		\ + \
		\PP[[X_{\ell}]_{\ell}\in u_n \cdot ((s_1\cdot A)\setminus(s_2\cdot A))]
		\\ &
		  \qquad + \
		  \norm{
			  \left(
			  \nabla_{\mathbf{y}} G_{\infty,n}(\mathbf{y},s_1,s_2,A)
			  \Big|_{\mathbf{y}=\mathbf{0}_T}
			\right)
			  ^\top\cdot [X_{\ell}]_\ell
			  }_{L^r(\PP)}
		\\ &
		  \ \le \
		  2
		  \left(
		  \PP[[X_{\ell}]_{\ell}\in u_n \cdot ((s_1\cdot A)\setminus(s_2\cdot A))]
		\right)^{1/r}
		\ + \
		  T
		  \norm{
			  \nabla_{\mathbf{y}} G_{\infty,n}(\mathbf{y},s_1,s_2,A)
			  \Big|_{\mathbf{y}=\mathbf{0}_T}
		}_{\infty}
		\norm{
			  X_{0}
		}_{L^r(\PP)}
		\\ &
		  \ \lesssim \
		  \left(
		  \PP[[X_{\ell}]_{\ell}\in u_n \cdot ((s_1\cdot A)\setminus(s_2\cdot A))]
		\right)^{1/r}
		\ + \
		  \norm{
			  \nabla_{\mathbf{y}} G_{\infty,n}(\mathbf{y},s_1,s_2,A)
			  \Big|_{\mathbf{y}=\mathbf{0}_T}
		}_{\infty}
	\end{align*}
	since $X_0\in L^r(\PP)$ by Lemma~\ref{lem:fm}. The constant in $\lesssim$ depends only on $T$, $r$, the coefficients $(a_i)$, and the distribution of $\varepsilon$.
	By Lemma~\ref{lem:single}, dropping the $u_n^{-\gamma}$ rate (see Remark~\ref{rem:dropping}), this is bounded above by
	\begin{align*}
		\left(
		(s_2-s_1)\land 1
		\right)^{1/r}
		\ + \
		((s_2-s_1)\land 1)
		\ \lesssim \
		\left(
		(s_2-s_1)\land 1
		\right)^{1/r}
		\,,
	\end{align*}
	up to a multiplicative constant that depends only on $A$, $T$, $\underline{s}$, $r$, $\gamma$, the coefficients $(a_i)$, and the distribution of $\varepsilon$.
	\\
	\textbf{Proof of (ii):}
	By Jensen's inequality and the properties of conditional expectation,
	\begin{align*}
		\norm{
			P_k\,Y
		}_{L^r(\PP)}
		\  \le   \
		\norm{\EE[Y\mid \mathcal{F}_k]}_{L^r(\PP)}
		\ + \
		\norm{\EE[Y\mid \mathcal{F}_{k-1}]}_{L^r(\PP)}
		\ \le \
		2
		\norm{Y}_{L^r(\PP)}
		\ < \ \infty
	\end{align*}
	for $Y\in L^r(\PP)$. The statement then follows from part~(i).
\end{proof}
We define
\begin{align*}
	 &
	\mathcal{S}_n(s_1,s_2,A)
  \ := \ \sum_{t=1}^{n-T+1}
	\mathcal{U}_n([X_{t+\ell}]_\ell,s_1,s_2,A)
	\\ &
	= \
    \sum_{t=1}^{n-T+1}
    \left(
	G_n([X_{t+\ell}]_\ell,s_1,s_2,A) \ -\  \EE[G_n([X_{t+\ell}]_\ell,s_1,s_2,A)] \ -\
	  \left(
	  \nabla_{\mathbf{y}} G_{\infty,n}(\mathbf{y},s_1,s_2,A)
	  \Big|_{\mathbf{y}=\mathbf{0}_T}
	\right)^\top [X_{t+\ell}]_\ell
    \right)
	  \,.
\end{align*}
\begin{lemma}[Moment bounds after martingale decomposition]
	\label{lem:cond_decomp}
	Let Assumptions~\ref{asu:f} and~\ref{asu:coef} hold. Then, for all $r\in[1,\alpha)$,
	\begin{align*}
		\mathbf{E}\left|
		\mathcal{S}_n(s_1,s_2,A)
		\right|^r
		\ \lesssim \
		\sum_{k\le n}^{}
		\left(
    \sum_{t=(k-T+1)\lor 1}^{n-T+1}
		\norm{
				P_{k-t}\, \mathcal{U}_n([X_\ell]_\ell,s_1,s_2,A)
			}_{L^r(\PP)}
		\right)^r
		\,,
	\end{align*}
	where $\lesssim$ means inequality up to a multiplicative constant that is independent of all other quantities in this work, in particular, independent of $r$ and $n$.
\end{lemma}
\begin{proof}
	For a proof with $T=1$ see Lemma~B.2 in~\cite{scheffelCentralLimitTheory2025}. The proof extends readily
	to the case $T>1$ by considering
	$[X_0,\ldots,X_{T-1}]$ instead of $X_0$
	and noting that
	$\mathcal{S}_n(s_1,s_2,A)$ is $\mathcal{F}_{n}$-measurable, and
  $P_k\ \mathcal{U}_n([X_{t+\ell}]_\ell,s_1,s_2,A)= 0$ for $k-T+1>t$.
  We omit the details.
\end{proof}
\begin{lemma}[Error decomposition]
	\label{lem:core_r}
	Let Assumptions~\ref{asu:f} and~\ref{asu:coef} hold. Then, for all $k\in\NN$, it holds
	$
		P_{-k}\,\mathcal{U}_n([X_{\ell}]_{\ell},s_1,s_2,A)
		=
		R_1
		+
		R_2
		+
		R_3
	$,
	with
	\begin{align*}
		R_1
		    & \ := \
		G_{k-1,n}(
		\mathbf{y},s_1,s_2,A
		)
		\Big|_{\mathbf{y}=[X_{\ell}-X_{\ell,k-1+\ell}]_\ell}
		\ - \
		\int_\RR
		G_{k-1,n}(
		\mathbf{y},s_1,s_2,A
		)
		\Big|_{\mathbf{y}=
					[X_{\ell}-X_{\ell,k+\ell}+z\cdot a_{k+\ell}]_\ell
			}
		\,\mathrm{d}\PP_\varepsilon(z)
		\\ &
		  \qquad\ - \
		  \varepsilon_{-k}\cdot
		  [a_{k+\ell}]_\ell^{\top}
		\cdot
		  \nabla_{\mathbf{y}}
		G_{k-1,n}(
		  \mathbf{y},s_1,s_2,A
		  )
		  \Big|_{\mathbf{y}=
					  [X_{\ell}-X_{\ell,k+\ell}]_\ell
				  }
		\,,
		\\
		R_2
		    & \ := \
		\varepsilon_{-k}
		\cdot
		[a_{k+\ell}]_\ell^{\top}
		\cdot
		\left(
		\nabla_{\mathbf{y}}
		G_{\infty,n}(
		\mathbf{y},s_1,s_2,A
		)
		\Big|_{\mathbf{y}=
					[X_{\ell}-X_{\ell,k+\ell}]_\ell
			}
		\ - \
		\nabla_{\mathbf{y}}
		G_{\infty,n}(
		\mathbf{y},s_1,s_2,A
		)
		\Big|_{\mathbf{y}=
				\mathbf{0}_T
			}
		\right)
		\,,
		\\
		R_3 & \ := \
		\varepsilon_{-k}
		\cdot
		[a_{k+\ell}]_\ell^{\top}
		\cdot
		\left(
		\nabla_{\mathbf{y}}
		G_{k-1,n}(
		\mathbf{y},s_1,s_2,A
		)
		\Big|_{\mathbf{y}=
					[X_{\ell}-X_{\ell,k+\ell}]_\ell
			}
		\ - \
		\nabla_{\mathbf{y}}
		G_{\infty,n}(
		\mathbf{y},s_1,s_2,A
		)
		\Big|_{\mathbf{y}=
					[X_{\ell}-X_{\ell,k+\ell}]_\ell
			}
		\right)
		\,.
	\end{align*}
\end{lemma}
\begin{proof}
	By definition,
	\begin{align*}
		 &
		R_1
		\ + \
		R_2
		\ + \
		R_3
		\\ &
		  \ = \
		  G_{k-1,n}(
		  \mathbf{y},s_1,s_2,A
		  )
		  \Big|_{\mathbf{y}=
					  [X_{\ell}-X_{\ell,k-1+\ell}]_\ell
				  }
		\ - \
		  \int_{\RR}
		G_{k-1,n}(
		  \mathbf{y},s_1,s_2,A
		  )
		  \Big|_{\mathbf{y}=
					  [X_{\ell}-X_{\ell,k+\ell}+z\cdot a_{k+\ell}]_\ell
				  }
		\,\mathrm{d}\PP_\varepsilon(z)
		\\ &
		  \qquad - \
		  \varepsilon_{-k}
		\cdot
		  [a_{k+\ell}]_\ell^{\top}
		\cdot
		  \nabla_{\mathbf{y}}
		G_{\infty,n}(
		  \mathbf{y},s_1,s_2,A
		  )
		  \Big|_{\mathbf{y}=
				  \mathbf{0}_T
				  }
		\,.
	\end{align*}
	For fixed $j,k\in\mathbb{Z}$, observe that
	\begin{align*}
		P_{-k}
		\varepsilon_{-j}
		 &
		\ = \
		\begin{cases}
			0\,\quad                & \text{if} \ j\neq k\,, \\
			\varepsilon_{-k}\,\quad & \text{if} \ j= k\,,    \\
		\end{cases}
	\end{align*}
	so, by dominated convergence, it follows that
	\begin{align*}
		P_{k}\, X_\ell
		\ = \
		\sum_{i=0}^\infty
		a_i
		P_{k}\,\varepsilon_{\ell-i}
		\ = \
		\begin{cases}
			0\,\quad                         & \text{if} \ k>\ell\,,    \\
			a_{\ell-k}\varepsilon_{k}\,\quad & \text{if} \ k\le \ell\,. \\
		\end{cases}
	\end{align*}
	Thus,
	\begin{align}
		\label{eq:PkU}
		\begin{split}
			 &
			P_{k}\,\mathcal{U}_n([X_{\ell}]_{\ell},s_1,s_2,A)
			\\
			 &
			\ = \
			P_{k}\,
			\Big(G_n([X_{\ell}]_{\ell},s_1,s_2,A) \ -\  \EE[G_n([X_{\ell}]_\ell,s_1,s_2,A)] \ -\
			\nabla_{\mathbf{y}}
			\left(
			G_{\infty,n}(
			\mathbf{y},s_1,s_2,A
			)
			\Big|_{\mathbf{y}=
					\mathbf{0}_T
				}
			\right)^{\top}
			\cdot [X_\ell]_{\ell}\Big)
			\\
			 &
			\ = \
			\EE[G_n([X_\ell]_{\ell},s_1,s_2,A)\, | \, \mathcal{F}_{k}]
			\ - \
			\EE[G_n([X_{\ell}]_{\ell},s_1,s_2,A)\, | \, \mathcal{F}_{k-1}]
			\\ &
			  \qquad - \
			  \varepsilon_{k}
			\cdot
			  \left(
			  \left(
			  \nabla_{\mathbf{y}}
			G_{\infty,n}(
			  \mathbf{y},s_1,s_2,A
			  )
			  \Big|_{\mathbf{y}=
					  \mathbf{0}_T
					  }
			\right)^{\top}
			\cdot
			  [
				  a_{\ell-k}
				\cdot
				  \ind\{k\le \ell\}
				  ]_{\ell}
			\right)
			  \,.
		\end{split}
	\end{align}
	Since we assume $k\ge 1$ and $\ell\in\{0,\ldots,T-1\}$, it holds $\ind\{-k\le \ell\}=1$, so that by \eqref{eq:PkU},
	\begin{align*}
		 &
		P_{-k}\,\mathcal{U}_n([X_{\ell}]_\ell,s_1,s_2,A)
		\\
		 &
		\ = \
		\EE[G_n([X_\ell]_\ell,s_1,s_2,A)\, | \, \mathcal{F}_{-k}]
		\ - \
		\EE[G_n([X_\ell]_\ell,s_1,s_2,A)\, | \, \mathcal{F}_{-(k+1)}]
		\\ &
		  \qquad - \
		  \varepsilon_{-k}
		\cdot
		  [a_{k+\ell}
		]_\ell^{\top}
		\cdot
		  \nabla_{\mathbf{y}}
		G_{\infty,n}(
		  \mathbf{y},s_1,s_2,A
		  )
		  \Big|_{\mathbf{y}=
				  \mathbf{0}_T
				  }
		\,.
	\end{align*}
	Note that
	$X_{t,t+k-1}=\sum_{j=0}^{t+k-1}a_j\varepsilon_{t-j}$ is independent of $\mathcal{F}_{-k}$, and $X_t-X_{t,t+k-1}=\sum_{j=t+k}^{\infty}a_j\varepsilon_{t-j}$ is $\mathcal{F}_{-k}$-measurable.
	Therefore,
	\begin{align*}
		\EE[G_n([X_\ell]_\ell,s_1,s_2,A)\, | \, \mathcal{F}_{-k}] & \ = \ \EE[G_n([X_{\ell,k-1+\ell}+(X_\ell - X_{\ell,k-1+\ell})]_\ell,s_1,s_2,A
			\mid \mathcal{F}_{-k}
		]
		\\ &
		  \ = \
		  \int_{\RR^T}
		G_n(\mathbf{x}+[X_{\ell}-X_{\ell,k-1+\ell}]_\ell,s_1,s_2,A)
		  \,\mathrm{d}
		\PP_{
			  [X_{\ell,k-1+\ell}]_{\ell}
		}(\mathbf{x})
		\\ &
		  \ = \
		  G_{k-1,n}(
		  \mathbf{y},s_1,s_2,A
		  )
		  \Big|_{\mathbf{y}=
					  [X_\ell-X_{\ell,k-1+\ell}]_\ell
				  }
		\,.
	\end{align*}
	Similarly,
	\begin{align*}
		 &
		\EE[G_n([X_\ell]_\ell,s_1,s_2,A)\, | \, \mathcal{F}_{-(k+1)}]
		\\
		 &
		\ = \
		\int_{\RR^T}
		G_n(\mathbf{x}+[X_{\ell}-X_{\ell,k+\ell}]_\ell,s_1,s_2,A)
		\,\mathrm{d}
		\PP_{
			[X_{\ell,k+\ell}]_{\ell}
		}(\mathbf{x}_{})
		\\ &
		  \ = \
		  \int_{\RR}
		\left(
		  \int_{\RR^T}
		G_n(\mathbf{x}+[X_{\ell}-X_{\ell,k+\ell}+z\cdot a_{k+\ell}]_\ell,s_1,s_2,A)
		  \,\mathrm{d}
		\PP_{
			  [X_{\ell,k-1+\ell}]_{\ell}
		}(\mathbf{x})
		\right)
		  \,\mathrm{d}\PP_{\varepsilon}(z)
		\\ &
		  \ = \
		  \int_{\RR}
		G_{k-1,n}(
		  \mathbf{y},s_1,s_2,A
		  )
		  \Big|_{\mathbf{y}=
					  [X_\ell - X_{\ell,k+\ell} + z \cdot a_{k+\ell}]_\ell
				  }
		\,\mathrm{d}\PP_\varepsilon(z)
		\,.
	\end{align*}
	Consequently, combining the expressions gives
	\begin{align*}
		 & P_{-k}\,\mathcal{U}_n([X_\ell]_\ell,s_1,s_2,A) \\
		 & \ = \
		G_{k-1,n}(
		\mathbf{y},s_1,s_2,A
		)
		\Big|_{\mathbf{y}=
					[X_\ell-X_{\ell,k-1+\ell}]_\ell
			}
		\ - \
		\int_\RR
		G_{k-1,n}(
		\mathbf{y},s_1,s_2,A
		)
		\Big|_{\mathbf{y}=
					[X_\ell - X_{\ell,k+\ell}+z\cdot a_{k+\ell}]_\ell
			}
		\,\mathrm{d}\PP_\varepsilon(z)
		\\ &
		  \qquad - \
		  \varepsilon_{-k}
		\cdot
		  [a_{k+\ell}]_\ell^{\top}
		\cdot
		  \nabla_{\mathbf{y}}
		G_{\infty,n}(
		  \mathbf{y},s_1,s_2,A
		  )
		  \Big|_{\mathbf{y}=
				  \mathbf{0}_T
				  }
		\\
		 & \ = \  R_1
		\ + \
		R_2
		\ + \
		R_3\,,
	\end{align*}
	as claimed.
\end{proof}
\subsection{Analysis of the remainder terms $R_j$}
In this subsection we use the tools developed in the previous sections. In the next lemma we provide intermediate bounds that are closely linked to the bounds of gradient differences provided in the last section.
\begin{lemma}[Intermediate bounds of the remainder terms]
	\label{lem:boundsRj}
	Let Assumptions~\ref{asu:f} and~\ref{asu:coef} hold, and let $k\in \NN$. Then the following statements are true:
	\begin{enumerate}
		\item
		      \begin{align*}
			       &
			      |R_1|
			      \cdot
			      k^{1-d}
			      \\ &
			        \ \lesssim \
			        \max_{i,l\in \{1,\ldots,T\}}
			      \left[
				        \int_{\RR}\,\mathrm{d}\PP_{\varepsilon}(z)
				      \int_{\RR^T}
				      \,\mathrm{d}\PP_{[X_{\ell,k-1+\ell}-X_{\ell,\ell}]}(\widetilde{\mathbf{x}_{}}_{})
				      \int_{\varepsilon_{-k}\land z}^{\varepsilon_{-k}\lor z}
				      \,\mathrm{d}s
				        \int_{\RR^T}
				      G_{n}(\mathbf{A}_{T}\cdot \mathbf{z}_{}, s_1,s_2,A)
				        \,\mathrm{d}\mathbf{z}_{}
				      \right.
			      \\ &
				        \qquad
				        \qquad
				        \qquad
				        \left.
				        \left|
				        \widetilde{R}_{i,l}
				      \left(
				        j\mapsto z_{j-1}- \mathbf{e}_{j}^{\top}\cdot \mathbf{A}_{T}^{-1}
				      \cdot
				        \left(
				        [X_\ell-X_{\ell,k+\ell}]_\ell + \widetilde{\mathbf{x}_{}}_{}
				        \right)
				      \,,
				        j\mapsto -s\cdot \mathbf{e}_{j}^{\top}
				      \cdot \mathbf{A}_{T}^{-1}\cdot [a_{k+\ell}]_{\ell}
				        \right)
				      \right|
				        \right]
			      \,.
		      \end{align*}
		\item
		      \begin{align*}
			       &
			      |R_2|
			      \cdot
			      k^{1-d}
			      \\ &
			        \ \lesssim \
			        \left|
			        \varepsilon_{-k}
			      \right|
			      \\ &
			        \qquad\times\ \max_{i,l\in \{1,\ldots,T\}}
			      \left[
				        \int_{\RR^T}
				      \,\mathrm{d}\PP_{[X_{\ell }-X_{\ell,\ell}]_{\ell}}(\widetilde{\mathbf{x}_{}}_{})
				      \int_{\RR^T}
				      G_n(\mathbf{A}_{T}\cdot \mathbf{z}_{},s_1,s_2,A)
				        \,\mathrm{d}\mathbf{z}_{}
				      \right.
			      \\ &
				        \qquad
				        \qquad
				        \qquad
				        \qquad
				        \left|
				        \widetilde{R}_{i,l}
				      \left(
				        \left(
					        j
					        \
					        \mapsto
					        \
					        z_{j-1}
					        - \mathbf{e}_{j}^{\top}
					        \cdot
					        \mathbf{A}_{T}^{-1}
					        \cdot
					        \widetilde{\mathbf{x}}_{}
					        \right)
				      \,,
				        \left(
					        \left.
					        j
					        \
					        \mapsto
					        \
					        -
					        \mathbf{e}_{j}^{\top}
					        \cdot
					        \mathbf{A}_{T}^{-1}
					        \cdot
					        [X_{\ell}-X_{\ell,k+\ell}]_\ell
					        \right)
				      \right)
				      \right|
				        \right]
			      \,.
		      \end{align*}
		\item
		      \begin{align*}
			       &
			      |R_3|\cdot k^{1-d}
			      \\ &
			        \ \lesssim\
			        |\varepsilon_{-k}|
			      \\ &
			        \times\max_{i,l\in \{1,\ldots T\}}
			      \left[
				        \int_{\RR^T}
				      \,\mathrm{d}\PP_{[X_{\ell}-X_{\ell,k-1+\ell}]_{\ell}}(\widetilde{\mathbf{x}_{}}_{})
				      \int_{\RR^T}
				      \,\mathrm{d}\PP_{[X_{\ell,k-1+\ell}-X_{\ell,\ell}]_{\ell}}(\widetilde{\mathbf{y}_{}}_{})
				      \int_{\RR^T}
				      G_n
				        (\mathbf{A}_{T}\cdot \mathbf{z}_{})
				        \,\mathrm{d}\mathbf{z}_{}
				      \right.
			      \\ &
				        \left.
				        \qquad
				        \qquad
				        \qquad
				        \left|
				        \widetilde{R}_{i,l}
				      \left(
				        \left(
					        j\ \mapsto \
					        z_{j-1}
					        \ - \
					        \mathbf{e}_{j}^{\top}
					      \cdot
					        \mathbf{A}_{T}^{-1}
					      \cdot
					        \left(
					        [X_{\ell}-X_{\ell,k+\ell}]_{\ell}
					        \ + \
					        \widetilde{\mathbf{y}_{}}_{}
					        \right)
					      \right)
				      \,,
				        \left(
					        j\ \mapsto \
					        -
					        \mathbf{e}_{j}^{\top}
					        \cdot
					        \mathbf{A}_{T}^{-1}
					        \cdot \widetilde{\mathbf{x}_{}}
					        \right)
				      \right)
				      \right|
				        \right]
		      \end{align*}
	\end{enumerate}
	In particular, in (ii) and (iii), $\varepsilon_{-k}$ is independent of the second factor.
	The multiplicative constant in $\lesssim$ depends throughout (i)-(iii) only on
	$T$ and the coefficients $(a_i)$.
\end{lemma}
\begin{proof} 
	Throughout the proof we use that
	\begin{align}
		\label{eq:factor}
		\norm{
			\mathbf{A}_{T}^{-1}\cdot [a_{k+\ell}]_{\ell}
		}_{\infty}
		\ \lesssim \
		k^{-(1-d)}
		\,,
	\end{align}
	where the multiplicative constant in $\lesssim$ depends only on $T$ and the coefficients $(a_i)$.
	\\
	\textbf{Proof of (i):}
	By definition,
	\begin{align}
		\label{eq:r1:0}
		\begin{split}
			R_1
			 &
			\ = \
			G_{k-1,n}(
			\mathbf{y},s_1,s_2,A
			)
			\Big|_{\mathbf{y}=[X_{\ell}-X_{\ell,k-1+\ell}]_\ell}
			\ - \
			\int_\RR
			G_{k-1,n}(
			\mathbf{y},s_1,s_2,A
			)
			\Big|_{\mathbf{y}=
						[X_{\ell}-X_{\ell,k+\ell}+z\cdot a_{k+\ell}]_\ell
				}
			\,\mathrm{d}\PP_\varepsilon(z)
			\\ &
			  \qquad\ - \
			  \varepsilon_{-k}\cdot
			  [a_{k+\ell}]_\ell^{\top}
			\cdot
			  \nabla_{\mathbf{y}}
			G_{k-1,n}(
			  \mathbf{y},s_1,s_2,A
			  )
			  \Big|_{\mathbf{y}=
						  [X_{\ell}-X_{\ell,k+\ell}]_\ell
					  }
			\\
			 &
			\ = \
			\int_\RR
			\bigg[
				G_{k-1,n}(
				\mathbf{y},s_1,s_2,A
				)
				\Big|_{\mathbf{y}=
							[X_{\ell}-X_{\ell,k+\ell}+\varepsilon_{-k}\cdot a_{k+\ell}]_\ell
					}
				\ -\
				G_{k-1,n}(
				\mathbf{y},s_1,s_2,A
				)
				\Big|_{\mathbf{y}=
							[X_{\ell}-X_{\ell,k+\ell}+z\cdot a_{k+\ell}]_\ell
					}
			\\ &
				  \qquad\ - \
				  \varepsilon_{-k}\cdot
				  [a_{k+\ell}]_\ell^{\top}
				\cdot
				  \nabla_{\mathbf{y}}
				G_{k-1,n}(
				  \mathbf{y},s_1,s_2,A
				  )
				  \Big|_{\mathbf{y}=
							  [X_{\ell}-X_{\ell,k+\ell}]_\ell
						  }
				\bigg]
			\,\mathrm{d}\PP_\varepsilon(z)
			\\ &
			  \ = \
			  \int_\RR
			  \bigg[
				  \left(
				  \int_{z}^{\varepsilon_{-k}}
					  [a_{k+\ell}]_\ell^\top
				  \cdot
				  \nabla_{\mathbf{y}}
				G_{k-1,n}(
				  \mathbf{y},s_1,s_2,A
				  )
				  \Big|_{\mathbf{y}=
							  [X_{\ell}-X_{\ell,k+\ell}+s\cdot a_{k+\ell}]_\ell
						  }
				\,\mathrm{d}s
				  \right)
			\\ &
				  \qquad\ - \
				  (
				  \varepsilon_{-k}
				-
				  z
				  )
				  \cdot
				  [a_{k+\ell}]_\ell^\top
				  \cdot
				  \nabla_{\mathbf{y}}
				G_{k-1,n}(
				  \mathbf{y},s_1,s_2,A
				  )
				  \Big|_{\mathbf{y}=
							  [X_{\ell}-X_{\ell,k+\ell}]_\ell
						  }
				\bigg]
			\,\mathrm{d}\PP_{\varepsilon}(z)
			\\ &
			  \ = \
			  \int_\RR
			  \,\mathrm{d}\PP_{\varepsilon}(z)
			\int_{z}^{\varepsilon_{-k}}
			\,\mathrm{d}s
			\\ & \qquad
			  [a_{k+\ell}]_\ell^\top
			  \bigg[
				  \nabla_{\mathbf{y}}
				G_{k-1,n}(
				  \mathbf{y},s_1,s_2,A
				  )
				  \Big|_{\mathbf{y}=
							  [X_{\ell}-X_{\ell,k+\ell}+s\cdot a_{k+\ell}]_\ell
						  }
				\ - \
				  \nabla_{\mathbf{y}}
				G_{k-1,n}(
				  \mathbf{y},s_1,s_2,A
				  )
				  \Big|_{\mathbf{y}=
							  [X_{\ell}-X_{\ell,k+\ell}]_\ell
						  }
				\bigg]
			\,.
		\end{split}
	\end{align}
	Note that the third
	equality follows by applying the fundamental theorem of calculus to $G_{k-1,n}$, which is justified by
	Lemma~\ref{lem:multi_swap}.
	Note also that adding the term
	\begin{align*}
		z
		\cdot
		[a_{k+\ell}]_\ell^\top
		\cdot
		\nabla_{\mathbf{y}}
		G_{k-1,n}(
		\mathbf{y},s_1,s_2,A
		)
		\Big|_{\mathbf{y}=
					[X_{\ell}-X_{\ell,k+\ell}]_\ell
			}
	\end{align*}
	inside the integral is valid because $\varepsilon$ is centered.
	Then we use Lemma~\ref{lem:grad:lip}(i) for $k-1$
	with
	\begin{align*}
		\mathbf{x}_{}
		\ = \
		[X_{\ell}-X_{\ell,k+\ell}]_{\ell}
		\,,\qquad
		\mathbf{y}_{}
		\ = \
		s\cdot [a_{k+\ell}]_\ell
		\,,
	\end{align*}
	to get
	\begin{align*}
		 &
		\nabla_{\mathbf{y}}
		G_{k-1,n}(
		\mathbf{y},s_1,s_2,A
		)
		\Big|_{\mathbf{y}=
					[X_{\ell}-X_{\ell,k+\ell}+s\cdot a_{k+\ell}]_\ell
			}
		\ - \
		\nabla_{\mathbf{y}}
		G_{k-1,n}(
		\mathbf{y},s_1,s_2,A
		)
		\Big|_{\mathbf{y}=
					[X_{\ell}-X_{\ell,k+\ell}]_\ell
			}
		\\ &
		  \ = \
		  \left(
		  \mathbf{A}_{T}^{-1}
		  \right)^{\top}
		\\ &
		  \qquad
		  \times \
		  \left[
			  \int_{\RR^T}
			G_{n}(\mathbf{A}_{T}\cdot \mathbf{z}_{}, s_1,s_2,A)
			  \,\mathrm{d}\mathbf{z}_{}
			\int_{\RR^T}
			\,\mathrm{d}\PP_{[X_{\ell,k-1+\ell}-X_{\ell,\ell}]}(\widetilde{\mathbf{x}_{}}_{})
			\right.
		\\ &
			  \qquad
			  \qquad
			  \left.
			  \widetilde{R}_{i,l}
			\left(
			  j\mapsto z_{j-1}- \mathbf{e}_{j}^{\top}\cdot \mathbf{A}_{T}^{-1}
			\cdot
			  \left(
			  [X_\ell-X_{\ell,k+\ell}]_\ell + \widetilde{\mathbf{x}_{}}_{}
			  \right)
			\,,
			  j\mapsto -s\cdot \mathbf{e}_{j}^{\top}
			\cdot \mathbf{A}_{T}^{-1}\cdot [a_{k+\ell}]_{\ell}
			  \right)
			\right]_{i,l\in \{1,\ldots,T\}}
		\cdot
		  \mathbf{1}_{T}
		\,.
	\end{align*}
	Therefore, taking absolute values and using \eqref{eq:factor} and Fubini's theorem,
	\begin{align*}
		 &
		|R_1|
		\\ &
		  \ \lesssim \
		  k^{-(1-d)}
		\\ &
		  \qquad\times
		  \max_{i,l\in \{1,\ldots,T\}}
		\left[
			  \int_{\RR}\,\mathrm{d}\PP_{\varepsilon}(z)
			\int_{\RR^T}
			\,\mathrm{d}\PP_{[X_{\ell,k-1+\ell}-X_{\ell,\ell}]}(\widetilde{\mathbf{x}_{}}_{})
			\int_{\varepsilon_{-k}\land z}^{\varepsilon_{-k}\lor z}
			\,\mathrm{d}s
			  \int_{\RR^T}
			G_{n}(\mathbf{A}_{T}\cdot \mathbf{z}_{}, s_1,s_2,A)
			  \,\mathrm{d}\mathbf{z}_{}
			\right.
		\\ &
			  \qquad
			  \qquad
			  \qquad
			  \left.
			  \left|
			  \widetilde{R}_{i,l}
			\left(
			  j\mapsto z_{j-1}- \mathbf{e}_{j}^{\top}\cdot \mathbf{A}_{T}^{-1}
			\cdot
			  \left(
			  [X_\ell-X_{\ell,k+\ell}]_\ell + \widetilde{\mathbf{x}_{}}_{}
			  \right)
			\,,
			  j\mapsto -s\cdot \mathbf{e}_{j}^{\top}
			\cdot \mathbf{A}_{T}^{-1}\cdot [a_{k+\ell}]_{\ell}
			  \right)
			\right|
			  \right]
		\,,
	\end{align*}
	where the multiplicative constant in $\lesssim$ depends only on $T$ and the coefficients $(a_i)$.
	\\
	\textbf{Proof of (ii):}
	By definition
	\begin{align*}
		R_2
		\ = \
		\varepsilon_{-k}
		\cdot
		[a_{k+\ell}]_\ell^{\top}
		\cdot
		\left(
		\nabla_{\mathbf{y}}
		G_{\infty,n}(
		\mathbf{y},s_1,s_2,A
		)
		\Big|_{\mathbf{y}=
					[X_{\ell}-X_{\ell,k+\ell}]_\ell
			}
		\ - \
		\nabla_{\mathbf{y}}
		G_{\infty,n}(
		\mathbf{y},s_1,s_2,A
		)
		\Big|_{\mathbf{y}=
				\mathbf{0}_T
			}
		\right)
		\,,
	\end{align*}
	and by Lemma~\ref{lem:grad:lip}(i) for $k=\infty$  with
	\begin{align*}
		\mathbf{x}_{}
		\ = \ \mathbf{0}_{T}
		\qquad\text{and}\qquad
		\mathbf{y}_{}
		\ = \
		[X_{\ell}-X_{\ell,k+\ell}]_{\ell}
		\,,
	\end{align*}
	one has
	\begin{align*}
		 &
		\nabla_{\mathbf{y}}
		G_{\infty,n}(
		\mathbf{y},s_1,s_2,A
		)
		\Big|_{\mathbf{y}=
					[X_{\ell}-X_{\ell,k+\ell}]_\ell
			}
		\ - \
		\nabla_{\mathbf{y}}
		G_{\infty,n}(
		\mathbf{y},s_1,s_2,A
		)
		\Big|_{\mathbf{y}=
				\mathbf{0}_T
			}
		\\ &
		  \ = \
		  \int_{\RR^T}
		G_n
		  (\mathbf{A}_{T}\cdot \mathbf{z}_{},s_1,s_2,A)
		  \,\mathrm{d}\mathbf{z}_{}
		\int_{\RR^T}
		\,\mathrm{d}
		\PP_{[X_{\ell}-X_{\ell,\ell}]_\ell}
		  (\widetilde{\mathbf{x}_{}}_{})
		\\ &
		  \qquad
		  \left(
		  \mathbf{A}^{-1}_T
		  \right)^{\top}
		\left[
			  \widetilde{R}_{i,l}
			\left(
			  \left(
			  j
			  \
			  \mapsto
			  \
			  z_{j-1} - \mathbf{e}_{j}^{\top}
			  \cdot
			  \mathbf{A}_{T}^{-1}
			  \cdot
			  \widetilde{\mathbf{x}}_{}
			  \right)
			\,,
			  \left(
			  j
			  \
			  \mapsto
			  \
			  - \mathbf{e}_{j}^{\top}
			  \cdot
			  \mathbf{A}_{T}^{-1}
			  \cdot
			  [X_{\ell}-X_{\ell,k+\ell}]_{\ell}
			  \right)
			\right)
			  \right]_{i,l\in \{1,\ldots,T\}}
		\cdot
		  \mathbf{1}_{T}
		\,.
	\end{align*}
	Taking absolute values and using \eqref{eq:factor} and Fubini's theorem yields the result.
	\\
	\textbf{Proof of (iii):}
	By definition
	\begin{align*}
		R_3 & \ := \
		\varepsilon_{-k}
		\cdot
		[a_{k+\ell}]_\ell^{\top}
		\cdot
		\left(
		\nabla_{\mathbf{y}}
		G_{k-1,n}(
		\mathbf{y},s_1,s_2,A
		)
		\Big|_{\mathbf{y}=
					[X_{\ell}-X_{\ell,k+\ell}]_\ell
			}
		\ - \
		\nabla_{\mathbf{y}}
		G_{\infty,n}(
		\mathbf{y},s_1,s_2,A
		)
		\Big|_{\mathbf{y}=
					[X_{\ell}-X_{\ell,k+\ell}]_\ell
			}
		\right)
		\,,
	\end{align*}
	and by Lemma~\ref{lem:grad:lip}(ii) for $k-1$ with
	\begin{align*}
		\mathbf{x}_{}
		\ = \
		[X_{\ell}-X_{\ell,k+\ell}]_{\ell}
		\,,
	\end{align*}
	one has
	\begin{align*}
		 &
		\nabla_{\mathbf{y}}
		G_{\infty,n}(
		\mathbf{y},s_1,s_2,A
		)
		\Big|_{\mathbf{y}=
					[X_{\ell}-X_{\ell,k+\ell}]_{\ell}
			}
		\ - \
		\nabla_{\mathbf{y}}
		G_{k-1,n}(
		\mathbf{y},s_1,s_2,A
		)
		\Big|_{\mathbf{y}=
					[X_{\ell}-X_{\ell,k+\ell}]_{\ell}
			}
		\\ &
		  \ = \
		  \left(
		  \mathbf{A}_{T}^{-1}
		  \right)^{\top}
		\\ &
		  \qquad
		  \times
		  \left[
			  \int_{\RR^T}
			\,\mathrm{d}\PP_{[X_{\ell}-X_{\ell,k-1+\ell}]_{\ell}}(\widetilde{\mathbf{x}_{}}_{})
			\int_{\RR^T}
			\,\mathrm{d}\PP_{[X_{\ell,k-1+\ell}-X_{\ell,\ell}]_{\ell}}(\widetilde{\mathbf{y}_{}}_{})
			\int_{\RR^T}
			G_n
			  (\mathbf{A}_{T}\cdot \mathbf{z}_{},s_1,s_2,A)
			  \,\mathrm{d}\mathbf{z}_{}
			\right.
		\\ &
			  \left.
			  \left.
			  \qquad
			  \qquad
			  \widetilde{R}_{i,l}
			\left(
			  \left(
			  j\ \mapsto \
			  z_{j-1}
			  \ - \
			  \mathbf{e}_{j}^{\top}
			\cdot
			  \mathbf{A}_{T}^{-1}
			\cdot
			  \left(
			  [X_{\ell}-X_{\ell,k+\ell}]_{\ell}
			  \ + \
			  \widetilde{\mathbf{y}_{}}_{}
			  \right)
			\right)
			\,,
			  \left(
			  j\ \mapsto \
			  -
			  \mathbf{e}_{j}^{\top}
			  \cdot
			  \mathbf{A}_{T}^{-1}
			  \cdot
			  \widetilde{\mathbf{x}_{}}_{}
			  \right)
			\right.
			  \right]
		\right]_{i,l\in\{1,\ldots,T\}}
		\cdot
		  \mathbf{1}_{T}
		\,.
	\end{align*}
	Taking absolute values and using \eqref{eq:factor} and Fubini's theorem yields the result.
\end{proof}
In the next lemma we apply the results of the second to last section to separate the threshold $u_n$ and the uniformity interval $[s_1,s_2]$ and simplify the intermediate bounds.
\begin{lemma}[Separable product bounds V]
	\label{lem:boundsRj_2}
	Let Assumptions~\ref{asu:f} and~\ref{asu:coef} hold.
	Then the following bounds hold for $\gamma\in (0,\alpha-1)$:
	\begin{enumerate}
		\item
		      \begin{align*}
			       &
			      \int_{\RR}\,\mathrm{d}\PP_{\varepsilon}(z)
			      \int_{\RR^T}
			      \,\mathrm{d}\PP_{[X_{\ell,k-1+\ell}-X_{\ell,\ell}]_{\ell}}(\widetilde{\mathbf{x}_{}}_{})
			      \int_{\varepsilon_{-k}\land z}^{\varepsilon_{-k}\lor z}
			      \,\mathrm{d}s
			      \int_{\RR^T}
			      G_{n}(\mathbf{A}_{T}\cdot \mathbf{z}_{}, s_1,s_2,A)
			      \,\mathrm{d}\mathbf{z}_{}
			      \\ &
			        \qquad
			        \qquad
			        \qquad
			        \left|
			        \widetilde{R}_{i,l}
			      \left(
			        j\mapsto z_{j-1}- \mathbf{e}_{j}^{\top}\cdot \mathbf{A}_{T}^{-1}
			      \cdot
			        \left(
			        [X_\ell-X_{\ell,k+\ell}]_\ell + \widetilde{\mathbf{x}_{}}_{}
			        \right)
			      \,,
			        j\mapsto -s\cdot \mathbf{e}_{j}^{\top}\cdot \mathbf{A}_{T}^{-1}\cdot [a_{k+\ell}]_{\ell}
			        \right)
			      \right|
			      \\ &
			        \ \lesssim \
			        u_n^{-\gamma}\cdot ((s_2-s_1)\land 1)
			        \cdot
			        k^{-\gamma(1-d)}
			      \cdot
			        \left(
			        1+
			        |\varepsilon_{-k}|^{1+\gamma}
			        \right)
			      \cdot
			        \sum_{j=1}^T
			        \left(
			        1 \ + \
			        \left|
			        X_{j-1}-X_{j-1,k+j-1}
			        \right|^{1+\gamma}
			        \right)
			      \,.
		      \end{align*}
		\item
		      \begin{align*}
			       &
			      \int_{\RR^T}
			      \,\mathrm{d}\PP_{[X_{\ell}-X_{\ell,\ell}]_{\ell}}(\widetilde{\mathbf{x}_{}}_{})
			      \int_{\RR^T}
			      G_n(\mathbf{A}_{T}\cdot \mathbf{z}_{},s_1,s_2,A)
			      \,\mathrm{d}\mathbf{z}_{}
			      \\ &
			        \qquad
			        \left|
			        \widetilde{R}_{i,l}
			      \left(
			        \left(
				        j
				        \
				        \mapsto
				        \
				        z_{j-1}
				        - \mathbf{e}_{j}^{\top}
				        \cdot
				        \mathbf{A}_{T}^{-1}
				        \cdot
				        \widetilde{\mathbf{x}}_{}
				        \right)
			      \,,
			        \left(
				        \left.
				        j
				        \
				        \mapsto
				        \
				        -
				        \mathbf{e}_{j}^{\top}
				        \cdot
				        \mathbf{A}_{T}^{-1}
				        \cdot
				        [X_{\ell}-X_{\ell,k+\ell}]_\ell
				        \right)
			      \right)
			      \right|
			        \right.
			      \\
			       &
			      \ \lesssim\
			      u_n^{-\gamma}\cdot ((s_2-s_1)\land 1)
			      \cdot
			      \left(
			      \norm{[X_{\ell}-X_{\ell,k+\ell}]_{\ell}}_{1}
			      \ + \
			      \norm{[X_{\ell}-X_{\ell,k+\ell}]_{\ell}}_{1}
			      ^{1+\gamma}
			      \right)
		      \end{align*}
		\item
		      \begin{align*}
			       &
			      \int_{\RR^T}
			      \,\mathrm{d}\PP_{[X_{\ell}-X_{\ell,k-1+\ell}]_{\ell}}(\widetilde{\mathbf{x}_{}}_{})
			      \int_{\RR^T}
			      \,\mathrm{d}\PP_{[X_{\ell,k-1+\ell}-X_{\ell,\ell}]_{\ell}}(\widetilde{\mathbf{y}_{}}_{})
			      \int_{\RR^T}
			      G_n
			      (\mathbf{A}_{T}\cdot \mathbf{z}_{},s_1,s_2,A)
			      \,\mathrm{d}\mathbf{z}_{}
			      \\ &
			        \qquad
			        \left|
			        \widetilde{R}_{i,l}
			      \left(
			        \left(
				        j\ \mapsto \
				        z_{j-1}
				        \ - \
				        \mathbf{e}_{j}^{\top}
				      \cdot
				        \mathbf{A}_{T}^{-1}
				      \cdot
				        \left(
				        [X_{\ell}-X_{\ell,k+\ell}]_{\ell}
				        \ + \
				        \widetilde{\mathbf{y}_{}}_{}
				        \right)
				      \right)
			      \,,
			        \left(
				        j\ \mapsto \
				        -
				        \mathbf{e}_{j}^{\top}
				        \cdot
				        \mathbf{A}_{T}^{-1}
				        \cdot \widetilde{\mathbf{x}_{}}
				        \right)
			      \right)
			      \right|
			      \\
			       &
			      \ \lesssim\
			      u_n^{-\gamma}\cdot ((s_2-s_1)\land 1)
			      \sum_{j=1}^T
			      \bigg[
				      \left(
				      \EE[\norm{[X_{\ell}-X_{\ell,k-1+\ell}]}_1]
				      +
				      \EE[\norm{[X_{\ell}-X_{\ell,k-1+\ell}]}_1^{1+\gamma}]
				      \right)
			      \\ & \qquad
				        \times
				        \left(
				        1+
				        |X_{j-1}-X_{j-1,k+j-1}|^{1+\gamma}
				        \right)
				      \bigg]
		      \end{align*}
	\end{enumerate}
	Throughout (i)-(iii), the multiplicative constant in $\lesssim$
	depends only on $T$, $(a_i)$, $(\mathfrak{a}_j)$, $\gamma$, $\underline{s}$, and the distribution of $\varepsilon$.
\end{lemma}
\begin{proof}[\textbf{Proof of (i):}]
	We want to apply Lemma~\ref{lem:ready}(ii) with
	\begin{align*}
		\varphi(j) \ = \ - \mathbf{e}_{j}^{\top} \cdot \mathbf{A}_{T}^{-1}
		\cdot
		\left(
		[X_\ell-X_{\ell,k+\ell}]_\ell + \widetilde{\mathbf{x}_{}}_{}
		\right)
		\qquad\text{and}\qquad
		\psi(j) \ = \ -\mathbf{e}_{j}^{\top} \cdot \mathbf{A}_{T}^{-1}\cdot [a_{k+\ell}]_{\ell}
		\,.
	\end{align*}
	Note that
	\begin{align*}
		\sum_{m=1}^j
		a_{j-m} \mathbf{e}_{m}^{\top}
		\ = \
		\mathbf{e}_{j}^{\top}
		\cdot
		\mathbf{A}_{T}
		\,,
	\end{align*}
	so that
	\begin{align*}
		\sum_{m=1}^j
		a_{j-m} \varphi(m)
		\ = \
		-([X_{j-1}-X_{j-1,k+j-1}]+\widetilde{x}_{j-1})
		\,.
	\end{align*}
	Also
	\begin{align*}
		\left|
		\mathbf{e}_{l}^{\top}
		\cdot \mathbf{A}_{T}^{-1}\cdot [a_{k+\ell}]_{\ell}
		\right|
		\ \lesssim \
		k^{-(1-d)}
		\,,
	\end{align*}
	and for $I\subset\{1,\ldots,j\}$,
	\begin{align*}
		\left|
		\sum_{m\in I}
		a_{j-m}
		\psi(m)
		\right|
		\ \lesssim \
		\max_{m\in I}
		\left|
		\psi(m)
		\right|
		\ \lesssim \
		\norm{[a_{k+\ell}]_\ell}_{\infty}
		\ \lesssim \
		k^{-(1-d)},
	\end{align*}
	where the multiplicative constant in $\lesssim$ depends only on $T$, and the coefficients $(a_i)$.
	Therefore, by Lemma~\ref{lem:ready}(ii),
	\begin{align}
		\label{eq:expands_to}
		\begin{split}
			 &
			\int_{\varepsilon_{-k}\land z}^{\varepsilon_{-k}\lor z}
			\,\mathrm{d}s
			\int_{\RR^T}
			G_{n}(\mathbf{A}_{T}\cdot \mathbf{z}_{}, s_1,s_2,A)
			\,\mathrm{d}\mathbf{z}_{}
			\\ &
			  \left|
			  \widetilde{R}_{i,l}
			\left(
			  j\mapsto z_{j-1}- \mathbf{e}_{j}^{\top} \cdot \mathbf{A}_{T}^{-1}
			\cdot
			  \left(
			  [X_\ell-X_{\ell,k+\ell}]_\ell + \widetilde{\mathbf{x}_{}}_{}
			  \right)
			\,,
			  j\mapsto -s\cdot \mathbf{e}_{j}^{\top} \cdot \mathbf{A}_{T}^{-1}\cdot [a_{k+\ell}]_{\ell}
			  \right)
			\right|
			\\ &
			  \ \lesssim \
			  u_n^{-\gamma}\cdot ((s_2-s_1)\land 1)
			\\ &
			  \qquad\times\
			  \sum_{j=1}^T
			  \sum_{I\subset \{1,\ldots,j\}}
			\left(
			  \left|
			  \sum_{m\in I}
			a_{j-m}
			  \psi(m)
			\right|^{\gamma}
			  \ + \
			  |
			  \mathbf{e}_{l}^{\top}
			\cdot \mathbf{A}_{T}^{-1}\cdot [a_{k+\ell}]_{\ell}
			  |^{\gamma}
			  \right)
			\\ &
			  \qquad\times \
			  \left(
			  1 \ + \
			  \left|
			  [X_{j-1}-X_{j-1,k+j-1}]+\widetilde{x}_{j-1}
			  \right|^{1+\gamma}
			  \right)
			\left(
			  |\varepsilon_{-k}|^{1+\gamma}
			  +
			  |z|^{1+\gamma}
			  \right)
			\,.
		\end{split}
	\end{align}
	where the multiplicative constant in $\lesssim$ depends only on $T$, $(a_i)$, $(\mathfrak{a}_j)$, $\gamma$, $\underline{s}$, and the distribution of $\varepsilon$.
	Note that
	\begin{align*}
		\left|
		\sum_{m\in I}
		a_{j-m}
		\psi(m)
		\right|^{\gamma}
		\ + \
		|
		\mathbf{e}_{l}^{\top}
		\cdot \mathbf{A}_{T}^{-1}\cdot [a_{k+\ell}]_{\ell}
		|^{\gamma}
		\ \lesssim \
		k^{-\gamma(1-d)}
		\,,
	\end{align*}
	and
	\begin{align*}
		1 \ + \
		\left|
		[X_{j-1}-X_{j-1,k+j-1}]+\widetilde{x}_{j-1}
		\right|^{1+\gamma}
		\ \lesssim \
		1
		\ + \
		\left|X_{j-1}-X_{j-1,k+j-1}\right|^{1+\gamma}
		\ + \
		|\widetilde{x}_{j-1}|^{1+\gamma}
		\,,
	\end{align*}
	where the multiplicative constant in $\lesssim$ depends only on $T$, $(a_i)$, and $\gamma$.
	The bound in \eqref{eq:expands_to}
	expands to
	\begin{align*}
		u_n^{-\gamma}\cdot ((s_2-s_1)\land 1)
		\cdot
		k^{-\gamma(1-d)}
		\cdot
		\left(
		|\varepsilon_{-k}|^{1+\gamma}
		+
		|z|^{1+\gamma}
		\right)
		\cdot
		\sum_{j=1}^T
		\left(
		1
		\ + \
		\left|X_{j-1}-X_{j-1,k+j-1}\right|^{1+\gamma}
		\ + \
		|\widetilde{x}_{j-1}|^{1+\gamma}
		\right)
		\,,
	\end{align*}
	up to a multiplicative constant that depends only on $T$, $(a_i)$, $\gamma$.
	Note that, by $1+\gamma< \alpha$, 
	\begin{align*}
		 &
		\sum_{j=1}^T
		\int_{\RR}\,\mathrm{d}\PP_{\varepsilon}(z)
		\int_{\RR^T}
		\,\mathrm{d}\PP_{[X_{\ell,k-1+\ell}-X_{\ell,\ell}]_{\ell}}(\widetilde{\mathbf{x}_{}}_{})
		\left(
		1
		\ + \
		\left|X_{j-1}-X_{j-1,k+j-1}\right|^{1+\gamma}
		\ + \
		|\widetilde{x}_{j-1}|^{1+\gamma}
		\right)
		\left(
		|\varepsilon_{-k}|^{1+\gamma}
		+
		|z|^{1+\gamma}
		\right)
		\\ &
		  \ \lesssim \
		  \sum_{j=1}^T
		  \left(
		  1
		  +
		  \left|X_{j-1}-X_{j-1,k+j-1}\right|^{1+\gamma}
		  \right)
		\left(
		  1
		  + |\varepsilon_{-k}|^{1+\gamma}
		  \right)
		\,,
	\end{align*}
	where the multiplicative constant in $\lesssim$ depends only on $(a_i)$, $\gamma$, and the distribution of $\varepsilon$.
	The result follows.
	\\
	\textbf{Proof of (ii):}
	We want to apply Lemma~\ref{lem:ready}(i) with
	\begin{align*}
		\varphi(m)
		\ = \
		-\mathbf{e}_{m}^{\top}
		\cdot \mathbf{A}_{T}^{-1}\cdot \widetilde{\mathbf{x}_{}}_{}
		\qquad\text{and}\qquad
		\psi(m)
		\ = \
		-\mathbf{e}_{m}^{\top}
		\cdot \mathbf{A}_{T}^{-1}\cdot [X_{\ell}-X_{\ell,k+\ell}]_{\ell}
		\,,
	\end{align*}
	so that
	\begin{align*}
		\sum_{m=1}^j
		a_{j-m}
		\varphi(m)
		\ = \
		-\widetilde{x}_{j-1}
		\,.
	\end{align*}
	By Lemma~\ref{lem:ready}(i),
	\begin{align*}
		 &
		\int_{\RR^T}
		G_n(\mathbf{A}_{T}\cdot \mathbf{z}_{},s_1,s_2,A)
		\left|
		\widetilde{R}_{i,l}
		\left(
		\left(
			j
			\
			\mapsto
			\
			z_{j-1}
			- \mathbf{e}_{j}^{\top}
			\cdot
			\mathbf{A}_{T}^{-1}
			\cdot
			\widetilde{\mathbf{x}_{}}_{}
			\,,
			j
			\
			\mapsto
			\
			-
			\mathbf{e}_{j}^{\top}
			\cdot
			\mathbf{A}_{T}^{-1}
			\cdot
			[X_{\ell}-X_{\ell,k+\ell}]_\ell
			\right)
		\right)
		\right|
		\,\mathrm{d}\mathbf{z}_{}
		\\ &
		  \ \lesssim \
		  u_n^{-\gamma}\cdot ((s_2-s_1)\land 1)
		  \sum_{j=1}^T
		  \sum_{I\subset \{1,\ldots,j\}}
		\left(
		  1+|\widetilde{x}_{j-1}|^{1+\gamma}
		  \right)
		\left(
		  |\mathbf{e}_{l}^{\top} \mathbf{A}_{T}^{-1}\cdot [X_{\ell}-X_{\ell,k+\ell}]_{\ell}|
		  +
		  \left|
		  \sum_{m\in I}
		a_{j-m}\psi(m)
		\right|
		  ^{1+\gamma}
		  \right)
		\,,
	\end{align*}
	where the multiplicative constant in $\lesssim$ depends only on $T$, $(a_i)$, $(\mathfrak{a}_j)$, $\gamma$, $\underline{s}$, and the distribution of $\varepsilon$.
	Using
	\begin{align*}
		|\mathbf{e}_{l}^{\top} \mathbf{A}_{T}^{-1}\cdot [X_{\ell}-X_{\ell,k+\ell}]_{\ell}|
		\ \lesssim \
		\norm{
			[X_{\ell}-X_{\ell,k+\ell}]_{\ell}
		}_{1}
		\,,
	\end{align*}
  one finds, by Lemma~\ref{lem:x0k},
	\begin{align*}
		\left|
		\sum_{m\in I}
		a_{j-m}\psi(m)
		\right|
		\ \lesssim \
		\max_{m\in I}
		|\psi(m)|
		\ \lesssim \
		\norm{[X_{\ell}-X_{\ell,k+\ell}]_\ell}_1
		\ \in \ L^{1+\gamma}(\PP)
		\,,
	\end{align*}
	and
	\begin{align*}
		\int_{\RR^T}
		\left(
		1+|\widetilde{x}_{j-1}|^{1+\gamma}
		\right)
		\,\mathrm{d}\PP_{[X_{\ell}-X_{\ell,\ell}]_{\ell}}(\widetilde{\mathbf{x}_{}}_{})
		\ = \
		1+
		\EE[|
			X_{j-1}-X_{j-1,j-1}
			|^{1+\gamma}]
		\ < \ \infty
		\,.
	\end{align*}
	This yields the result.
	\\
	\textbf{Proof of (iii):}
	We want to apply Lemma~\ref{lem:ready}(i) with
	\begin{align*}
		\varphi(j) \ = \ -\mathbf{e}_{j}^{\top} \cdot \mathbf{A}_{T}^{-1}\cdot
		\left(
		\widetilde{\mathbf{y}_{}}_{}
		\ + \
		[X_{\ell}-X_{\ell,k+\ell}]_{\ell}
		\right)
		\qquad\text{and}\qquad
		\psi(j) \ = \ -\mathbf{e}_{j}^{\top}
		\cdot \mathbf{A}_{T}^{-1}\cdot \widetilde{\mathbf{x}_{}}_{}
		\,,
	\end{align*}
	so that
	\begin{align*}
		\sum_{m=1}^j
		a_{j-m}
		\varphi(m)
		\ = \
		-(\widetilde{y}_{j-1}+ (X_{j-1}-X_{j-1,k+j-1}))
		\,.
	\end{align*}
	By Lemma~\ref{lem:ready}(i),
	\begin{align*}
		 &
		\int_{\RR^T}
		G_n(\mathbf{A}_{T}\cdot \mathbf{z}_{},s_1,s_2,A)
		\left|
		\widetilde{R}_{i,l}
		\left(
		\left(
			j
			\
			\mapsto
			\
			z_{j-1}
			- \mathbf{e}_{j}^{\top}
			\cdot
			\mathbf{A}_{T}^{-1}
			\cdot
			\left(
			\widetilde{\mathbf{y}_{}}_{}
			+
			[X_{\ell}-X_{\ell,k+\ell}]_{\ell}
			\right)
			\,,
			j
			\
			\mapsto
			\
			-
			\mathbf{e}_{j}^{\top}
			\cdot
			\mathbf{A}_{T}^{-1}
			\cdot
			\widetilde{\mathbf{x}_{}}_{}
			\right)
		\right)
		\right|
		\,\mathrm{d}\mathbf{z}_{}
		\\ &
		  \ \lesssim \
		  u_n^{-\gamma}\cdot ((s_2-s_1)\land 1)
		  \sum_{j=1}^T
		  \sum_{I\subset \{1,\ldots,j\}}
		\left(
		  1+ |\widetilde{y}_{j-1}|^{1+\gamma}
		  + |X_{j-1}-X_{j-1,k+j-1}|^{1+\gamma}
		  \right)
		\\ & \qquad\times\
		  \left(
		  |\mathbf{e}_{l}^{\top} \cdot \mathbf{A}_{T}^{-1}\cdot \widetilde{\mathbf{x}_{}}_{}|
		  +
		  \left|
		  \sum_{m\in I}
		a_{j-m}
		  \psi(m)
		\right|^{1+\gamma}
		  \right)
		\,,
	\end{align*}
	where the multiplicative constant in $\lesssim$ depends only on $T$, $(a_i)$, $(\mathfrak{a}_j)$, $\gamma$, $\underline{s}$, and the distribution of $\varepsilon$.
	Using
	\begin{align*}
		|\mathbf{e}_{l}^{\top} \cdot \mathbf{A}_{T}^{-1}\cdot \widetilde{\mathbf{x}_{}}_{}|
		\ \lesssim \
		\norm{
			\widetilde{\mathbf{x}_{}}_{}
		}_{1}
		\,,
	\end{align*}
	\begin{align*}
		\left|
		\sum_{m\in I}
		a_{j-m}
		\psi(m)
		\right|
		\ \lesssim \
		\max_{m\in I}
		|\psi(m)|
		\ \lesssim \
		\norm{\widetilde{\mathbf{x}_{}}_{}}_1
	\end{align*}
  and, by Lemma~\ref{lem:x0k},
	\begin{align*}
		\int_{\RR^T}
		\left(
		\norm{
			\widetilde{\mathbf{x}_{}}_{}
		}_{1}
		+
		\norm{
			\widetilde{\mathbf{x}_{}}_{}
		}_{1}^{1+\gamma}
		\right)
		\,\mathrm{d}\PP_{[X_{\ell}-X_{\ell,k-1+\ell}]}(\widetilde{\mathbf{x}_{}}_{})
		\ = \
		\EE[\norm{[X_{\ell}-X_{\ell,k-1+\ell}]}_1]
		+
		\EE[\norm{[X_{\ell}-X_{\ell,k-1+\ell}]}_1^{1+\gamma}]
    \ <\ \infty\,,
	\end{align*}
	\begin{align*}
		\int_{\RR^T}
		(1+|\widetilde{y}_{j-1}|^{1+\gamma})
		\,\mathrm{d}\PP_{[X_{\ell,k-1+\ell}-X_{\ell,\ell}]}(\widetilde{\mathbf{y}_{}}_{})
		\ = \
		1+
		\EE[|X_{j-1,k-1+j-1}-X_{j-1,j-1}|^{1+\gamma}]
		\ < \ \infty
	\end{align*}
	yields the result.
\end{proof}
This is the main result of the subsection, where we complete the analysis of the remainder terms $R_j$ by taking the $L^r(\PP)$-norm of the intermediate bounds provided so far.
\begin{lemma}[$L^r$-bounds of the remainder terms $R_j$]
	\label{lem:boundsRj_3}
	Let Assumptions~\ref{asu:f} and~\ref{asu:coef} hold, and
	let
	$\gamma\in(0,\alpha-1)$ and $r\in(1,2)$
	satisfy
	\begin{align*}
		\gamma
		\ < \
		\min
		\left\{
		\frac{d}{1-d}
		\,,
		1
		-
		\frac{1}{r(1-d)}
		\,,
		\frac{\alpha}{r}
		-
		1
		\right\}
		\qquad\text{and}\qquad
		\frac{1}{1-d}
		\ < \
		r
		\ < \
		\alpha
		\,.
	\end{align*}
	Then, for all $j=1,2,3$,
	\begin{align*}
		\norm{R_j}_{L^r(\PP)}
		\ \lesssim \
		u_n^{-\gamma}
		\cdot ((s_2-s_1)\land 1)
		\cdot k^{-(1+\gamma)(1-d)}
		\,,
	\end{align*}
	where the multiplicative constant in $\lesssim$ depends only on $T$, $(a_i)$, $(\mathfrak{a}_j)$, $r$, $\gamma$, $\underline{s}$, and the distribution of $\varepsilon$.
\end{lemma}
\begin{proof} We use first Lemmas~\ref{lem:boundsRj} and~\ref{lem:boundsRj_2} and treat each $R_j$ separately.
	\\
	\textbf{Analysis of $R_1$:}
	It suffices to bound the $L^r(\PP)$ norm of
	\begin{align*}
		\left(
		1+
		|\varepsilon_{-k}|^{1+\gamma}
		\right)
		\cdot
		\left(
		1 \ + \
		\left|
		X_{j-1}-X_{j-1,k+j-1}
		\right|^{1+\gamma}
		\right)
	\end{align*}
  for each $j$. The two factors have a finite $r$th moment, due to the assumption $(1+\gamma)r<\alpha$ and Lemma~\ref{lem:x0k}, and they are independent, so their product has a finite $r$th moment as well.
	\\
	\textbf{Analysis of $R_2$:}
	First use the independence property and the fact that $\varepsilon_{-k}$ has a finite $L^r(\PP)$ norm, to find that it is enough to show that
	\begin{align*}
		k^{-(1-d)}
		\left(
		\norm{[X_{\ell}-X_{\ell,k+\ell}]_{\ell}}_{1}
		\ + \
		\norm{[X_{\ell}-X_{\ell,k+\ell}]_{\ell}}_{1}^{1+\gamma}
		\right)
	\end{align*}
	is bounded in $L^r(\PP)$ by a multiple of $k^{-(1+\gamma)(1-d)}$.
	By stationarity of $(X_t)$ and Lemma~\ref{lem:x0k},
	\begin{align*}
		\norm{
			\norm{[X_{\ell}-X_{\ell,k+\ell}]_{\ell}}_1
		}_{L^r(\PP)}
		\ \lesssim \
		\sum_{\ell=0}^{T-1}
		\norm{
			X_{\ell}-X_{\ell,k+\ell}
		}_{L^r(\PP)}
		\ \lesssim \
		\norm{
			X_{0}-X_{0,k}
		}_{L^r(\PP)}
		\ \lesssim \
		k^{1/r-(1-d)}
		\,,
	\end{align*}
	where the multiplicative constant in $\lesssim$ depends only on $T$, $(a_i)$,
	and the distribution of $\varepsilon$.
	By
	convexity of $x\mapsto x^{1+\gamma}$,
	\begin{align*}
		\norm{[X_{\ell}-X_{\ell,k+\ell}]_{\ell}}_1^{1+\gamma}
		\ \lesssim \
		\sum_{\ell=0}^{T-1}
		\left|
		X_{\ell}-X_{\ell,k+\ell}
		\right|^{1+\gamma}
		\,,
	\end{align*}
	and therefore, by Lemma~\ref{lem:x0k} again,
	\begin{align*}
		\norm{
			\norm{[X_{\ell}-X_{\ell,k+\ell}]_{\ell}}_1^{1+\gamma}
		}_{L^r(\PP)}
		\ \lesssim \
		\sum_{\ell=0}^{T-1}
		\norm{
			\left|
			X_{\ell}-X_{\ell,k+\ell}
			\right|^{1+\gamma}
		}_{L^r(\PP)}
		\ \lesssim \
		\norm{
			|X_0-X_{0,k}|^{1+\gamma}
		}_{L^r(\PP)}
		\ \lesssim \
		k^{1/r-(1+\gamma)(1-d)}
		\,,
	\end{align*}
	where we used that by assumption $1/(1-d)<r<(1+\gamma)r<\alpha$.
	Here the multiplicative constant in $\lesssim$ depends only on $T$, $(a_i)$, $\gamma$, and the distribution of $\varepsilon$.
	Consequently
	\begin{align*}
		k^{-(1-d)} \norm{
			\norm{[X_{\ell}-X_{\ell,k+\ell}]_{\ell}}_1
		}_{L^r(\PP)}
		\ + \
		\norm{
			\norm{[X_{\ell}-X_{\ell,k+\ell}]_{\ell}}_1^{1+\gamma}
		}_{L^r(\PP)}
		\ \lesssim \
		k^{1/r-2(1-d)}
		\,.
	\end{align*}
	Now, since $\gamma< 1- 1/(r(1-d)) $ by assumption,
	\begin{align*}
		\frac{1}{r}
		\ - \
		2
		(1-d)
		\ = \
		-(1-d)
		\left(
		1
		-
		\frac{1}{r(1-d)}
		\right)
		\ - \
		(1-d)
		\ \le \
		-(1-d)\gamma
		\ - \
		(1-d)
		\ = \
		-(1-d)(1+\gamma)
		\,,
	\end{align*}
	so that
	indeed
	\begin{align*}
		k^{-(1-d)} \norm{
			\norm{[X_{\ell}-X_{\ell,k+\ell}]_{\ell}}_1
		}_{L^r(\PP)}
		\ + \
		\norm{
			\norm{[X_{\ell}-X_{\ell,k+\ell}]_{\ell}}_1^{1+\gamma}
		}_{L^r(\PP)}
		\ \lesssim \
		k^{-(1+\gamma)(1-d)}
		\,,
	\end{align*}
	where the multiplicative constant in $\lesssim$ depends only on $T$, $(a_i)$, $\gamma$, and the distribution of $\varepsilon$.
	\\
	\textbf{Analysis of $R_3$:}
	Again use the independence property and the fact that $\varepsilon_{-k}$ has a finite $L^r(\PP)$ norm, to find that
	\begin{align*}
		\norm{ |\varepsilon_{-k}| \cdot \left(
			1+|X_{j-1}-X_{j-1,k+j-1}|^{1+\gamma} \right)
		}_{L^r(\PP)}
		\ < \
		\infty
	\end{align*}
	due to $(1+\gamma)r<\alpha$.
		Then note that by convexity,
	\begin{align*}
		 &
		\EE[\norm{[X_{\ell}-X_{\ell,k-1+\ell}]_\ell}_1]
		\ + \
		\EE[\norm{[X_{\ell}-X_{\ell,k-1+\ell}]_\ell}_1^{1+\gamma}]
		\\ &
		  \ \le \
		  \norm{\norm{[X_{\ell}-X_{\ell,k-1+\ell}]_\ell}_1}_{L^r(\PP)}
		\ + \
		  \norm{\norm{[X_{\ell}-X_{\ell,k-1+\ell}]_\ell}_1^{1+\gamma}}_{L^r(\PP)}
		\\ &
		  \ \lesssim \
		  k^{1/r-(1-d)}
		\,,
	\end{align*}
	and the proof finishes as in (ii).
\end{proof}
\subsection{Proof of Lemma~\ref{lem:reduction_principle}}
We are going to apply Lemmas~\ref{lem:cond_decomp}~(moment bounds after martingale decomposition),~\ref{lem:core_r}~(error decomposition),~\ref{lem:boundsRj_3}~($L^r$-bounds of the remainder terms $R_j$), and Lemma~B.4 in~\cite{scheffelCentralLimitTheory2025},
which we are going to recall now.
\begin{lemma}[Rate of divergence of convolution-type sum; see Lemma~B.4 in~\cite{scheffelCentralLimitTheory2025}]
	\label{lem:analytic}
	Let Assumptions~\ref{asu:f} and~\ref{asu:coef} hold, and
	let
	$\gamma\in(0,\alpha-1)$ and $r\in(1,2)$
	satisfy
	\begin{align*}
		\gamma
		\ < \
		\min
		\left\{
		\frac{d}{1-d}
		\,,
		1
		-
		\frac{1}{r(1-d)}
		\,,
		\frac{\alpha}{r}
		-
		1
		\right\}
		\qquad\text{and}\qquad
		\frac{1}{1-d}
		\ < \
		r
		\ < \
		\alpha
		\,.
	\end{align*}
	Then
	\begin{align*}
		\sum_{k = -\infty}^{n}
		\left(
		\sum_{t=k\lor 1}^{n}
		\left(
			1
			\land
				(t-k)
				^{-(1-d)(1+\gamma)}
			\right)
		\right)^r
		\ \lesssim\
		n^{r+1-(1-d)(1+\gamma)r}
		\,,
	\end{align*}
	where $\lesssim$ denotes inequality up to
	a multiplicative constant that depends only on
	$\gamma$, $\alpha$, and $d$. We adopt the convention $1/0=\infty$.
\end{lemma}
\noindent
\textbf{Proof of Lemma~\ref{lem:reduction_principle}:}
By Lemma~\ref{lem:cond_decomp},
\[
	\EE
	\left|
	\mathcal{S}_{n}(s_1,s_2,A)
	\right|^r
	\ \lesssim \
	\sum_{k=-\infty}^{n}
	\left(
  \sum_{t=(k-T+1)\lor 1}^{n-T+1}
	\norm{
			P_{k-t}\, \mathcal{U}_n([X_{\ell}]_{\ell},s_1,s_2,A)
		}_{L^r(\PP)}
	\right)^r
	\,,
\]
where the multiplicative constant in $\lesssim$ is independent of all other quantities in this work.
We split the inner sum based on whether $t-k\ge 1$ and use the convexity of $x\mapsto x^r$ (due to $r>1$) to get
\begin{align*}
	\EE
	\left|
  \mathcal{S}_n(s_1,s_2,A)
	\right|^r
	 &
	\ \lesssim \
	\sum_{k=-\infty}^{n}
	\left(
	\sum_{
		\begin{smallmatrix}
      ((k-T+1) \lor 1)\le t \le n-T+1\\
			t-k \ge 1
		\end{smallmatrix}
	}
	\norm{
		P_{k-t}\, \mathcal{U}_n([X_{\ell}]_{\ell},s_1,s_2,A)
	}_{L^r(\PP)}
	\right)^{r}
	\\ &
	  \qquad + \
	  \sum_{k=-\infty}^{n}
	\left(
	  \sum_{
			  \begin{smallmatrix}
          ((k-T+1) \lor 1)\le t \le n-T+1\\
				t-k < 1
			\end{smallmatrix}
			  }
	\norm{
			  P_{k-t}\, \mathcal{U}_n([X_{\ell}]_{\ell},s_1,s_2,A)
			  }_{L^r(\PP)}
	\right)^r
	  \,.
\end{align*}
\textbf{Case $t-k\ge 1$:}
We apply Lemmas~\ref{lem:core_r},~\ref{lem:boundsRj_3}, and~\ref{lem:analytic}.
This yields
\begin{align*}
	 &
	\sum_{k=-\infty}^{n}
	\left(
	\sum_{
		\begin{smallmatrix}
      ((k-T+1) \lor 1)\le t \le n-T+1\\
			t-k \ge 1
		\end{smallmatrix}
	}
	\norm{
		P_{k-t}\, \mathcal{U}_n([X_{\ell}]_{\ell},s_1,s_2,A)
	}_{L^r(\PP)}
	\right)^{r}
	\\ &
	  \ \lesssim \
	  \left(
	  u_n^{-\gamma }
	\cdot
	  ((s_2-s_1) \land 1)
	  \right)^r
	  \sum_{k=-\infty}^{n}
	\left(
	  \sum_{
		  t=k\lor 1
		  }^n
	  \left(
	  1
	  \land
		  (t-k)^{-(1-d)(1+\gamma)}
	\right)
	\right)^{r}
	\\ &
	  \ \lesssim \
	  \left(
	  (s_2-s_1)\land 1
	  \right)
	  \cdot
	  n^{r+1-(1-d)(1+\gamma)r}
	\,,
\end{align*}
since 
$r>1$.
Here the multiplicative constant in $\lesssim$ depends only on $T$, $(a_i)$, $(\mathfrak{a}_j)$, $\gamma$, $\underline{s}$, and the distribution of $\varepsilon$.
\\
\textbf{Case $t-k < 1 $:}
Note that the constraints $t<k+1$ and $t \ge ((k-T+1) \lor 1)$ ensure that the inner sum is empty when $k\le 0$.
By Lemma~\ref{lem:2.2},
\begin{align*}
	 &
	\sum_{k=-\infty}^{n}
	\left(
	\sum_{
			\begin{smallmatrix}
        ((k-T+1) \lor 1)\le t \le n-T+1\\
				t-k < 1
			\end{smallmatrix}
		}
	\norm{
			P_{k-t}\, \mathcal{U}_n([X_{\ell}]_{\ell},s_1,s_2,A)
		}_{L^r(\PP)}
	\right)^r
	\\ &
	  \ \lesssim \
	  \left(
	  \norm{\mathcal{U}_n([X_\ell]_\ell,s_1,s_2,A)}_{L^r(\PP)}
	\right)^{r}
	\sum_{k=1}^{n}
	\left(
	  \#
	  \{
	  t\in\NN
	  \,\colon\,
	  ((k-T+1) \lor 1)\le t \le k+1
	  \}
	  \right)^r
	\\ &
	  \ \lesssim \
	  (
	  (s_2-s_1)
	  \land 1
	  )
	  \cdot
	  n \cdot T^r
	  \ \lesssim \
	  n
	  \cdot
	  (
	  (s_2-s_1)
	  \land 1
	  )
	  \,,
\end{align*}
where the multiplicative constant in $\lesssim$ depends only on $T$, $(a_i)$, $(\mathfrak{a}_j)$, $r$, $\gamma$, $\underline{s}$, and the distribution of $\varepsilon$.
We conclude that
\begin{align*}
	\EE
	\left|
	\mathcal{S}_n(s_1,s_2,A) \right|^r
	\ \lesssim \
	(
	(s_2-s_1)
	\land 1
	)
	\cdot
	\left(
	n^{r+1 - (1-d)(1+\gamma)r}
	\ + \ n
	\right)
	\,.
\end{align*}
Observe that $1-(1-d)(1+\gamma) \ge 0$  due to $\gamma \le d/(1-d)$. Therefore
\begin{align*}
	(
	(s_2-s_1)
	\land 1
	)
	\cdot
	\left(
	n^{r+1 - (1-d)(1+\gamma)r}
	\ + \ n
	\right)
	\ \le \
	2
	\cdot
	(
	(s_2-s_1)
	\land 1
	)
	\cdot
	n^{r+1 - (1-d)(1+\gamma)r}
	\,.
\end{align*}
This finishes the proof.
\qed
\section{Proof of Theorem~\ref{thm:main}}
\label{app:proofthmmain}
We adapt the chaining argument in the proof of Theorem~2.1 in~\cite{koulAsymptoticsEmpiricalProcesses2001} to our setting.
Choose $\delta>0$ sufficiently small and set
$
	u_n=o(n^{(d+1/(\nu \land 2)-\kappa_0-\delta)/(\nu+1)})
$, where $\kappa_0$ is defined in \eqref{eq:kappa_0}.
We define, for $k\in\NN_0$ and $j\in\{0,\ldots,2^k\}$ the partition
\begin{align*}
	\pi_{j,k}
	\ = \
	\underline{s}
	\ + \
	\frac{2^{k} - j}{2^k}
	\left(
	\overline{s}
	-
	\underline{s}
	\right)
	\,,
\end{align*}
so that
\begin{align*}
	\underline{s}
	\
	=
	\
	\pi_{2^k,k}
	\ < \
	\pi_{2^{k}-1,k}
	\ < \
	\cdots
	\ < \
	\pi_{1,k}
	\ < \
	\pi_{0,k}
	\ = \
	\overline{s}
	\,.
\end{align*}
For $k\in\NN_0$ and $s\in (\underline{s},\overline{s}]$, we define $j^s_k \in\{0,\ldots,2^k-1\}$ to be the unique index such that
\begin{align*}
	\pi_{j^s_k+1,k}
	\ < \
	s
	\ \le \
	\pi_{j^s_k,k}
	\,.
\end{align*}
Note that the partitions are nested with increasing $k$ so that one has $s \ \le \ \pi_{j^s_{k+1},k+1} \ \le \ \pi_{j^s_k,k}$.  We define a chain linking $\overline{s}$ to each point $s\in (\underline{s},\overline{s}]$ by
\begin{align*}
	\pi_{j^s_{K}+1,K}
	\ < \
	s
	\ \le \
	\pi_{j^s_{K},K}
	\ \le \
	\cdots
	\ \le \
	\pi_{j^s_{1},1}
	\ \le \
	\pi_{j^s_{0},0}
	\ = \
	\overline{s}
	\,,
\end{align*}
where we choose
$
  K=K_n=\lfloor\log_2(
		n^{2-d-1/\alpha}
    /
	f_{X_0}(u_n)
  )
  \rfloor
$,
so that
$
\limsup_{n\to\infty}
	2^{-K_n}
	\frac{
		n^{2-d-1/\alpha}
	}{f_{X_0}(u_n)}
  <\infty
$.
By the assumption on $u_n$ and regular variation of $f_{X_0}$, we have that $f_{X_0}(u_n)$ decays at most polynomially, so that, by $2-d-1/\alpha\in(1,2)$, $K_n\to\infty$ logarithmically.
By construction, each link in the chain comes from a different partition. Since $(\pi_{j,k})\subset (\pi_{j,k+1})$, the next result identifies the possible coordinates of adjacent links of the chain in terms of the larger partition. 
\begin{lemma}[Partition coordinates of adjacent chain links]
	\label{lem:partition}
		For all $k\in\NN$ and all $s\in (\underline{s},\overline{s}]$ it holds
	$
		\pi_{j_{k+1}^s,k+1}
		,
		\pi_{j_{k}^s,k}
		\in \{
		\pi_{2j_{k}^s,k+1},
		\pi_{2j_{k}^s+1,k+1}
		\}
	$.
\end{lemma}
\begin{proof}
		By definition,
	\begin{align*}
		\pi_{j,k}
		 &
		\ = \
		\underline{s}\ +\  \frac{2^k-j}{2^k}
		\left(
		\overline{s}
		-
		\underline{s}
		\right)
		\ = \
		\underline{s}\ +\  \frac{2^{k+1}-2j}{2^{k+1}}
		\left(
		\overline{s}
		-
		\underline{s}
		\right)
		\ = \
		\pi_{2j,k+1}
		\,,
	\end{align*}
	so that
	\begin{align*}
		\pi_{2(j_{k}^s+1),k+1}
		\ = \
		\pi_{j_{k}^s+1,k}
		<
		s
		\ \le \
		\pi_{j_{k}^s,k}
		\ = \
		\pi_{2j_{k}^s,k+1}
		\,.
	\end{align*}
	Since
	$
		\pi_{2(j_{k}^s+1),k+1}
		<
		s
		\ \le \
		\pi_{2j_{k}^s,k+1}
	$, it either holds
	\begin{align*}
		\pi_{2(j_{k}^s+1),k+1}
		\ < \
		s
		\ \le \
		\pi_{2j_{k}^s+1,k+1}
		\ < \
		\pi_{2j_{k}^s,k+1}
		\qquad\text{or}\qquad
		\pi_{2(j_{k}^s+1),k+1}
		\ < \
		\pi_{2j_{k}^s+1,k+1}
		\ < \
		s
		\ \le \
		\pi_{2j_{k}^s,k+1}
		\,,
	\end{align*}
	and therefore
	$
		j^s_{k+1}
		\in
		\{2j^s_{k}, 2j^s_{k}+1\}
	$.
	Thus
	$
		\pi_{j_{k+1}^s,k+1}
		,
		\pi_{j_{k}^s,k}
		\in \{
		\pi_{2j_{k}^s,k+1},
		\pi_{2j_{k}^s+1,k+1}
		\}
	$.
\end{proof}
Fix $A\in \mathcal{C}_T$ 
and
recall the notation
\begin{align*}
	\mathcal{S}_n(s_1,s_2,A)
	 &
	\ = \
  \sum_{t=1}^{n-T+1}
  \Bigg(
	G_n([X_{t+\ell}]_{\ell},s_1,s_2,A)
	\ - \
	\EE[
		G_n([X_{t+\ell}]_{\ell},s_1,s_2,A)
	]
	\\ &
	  \qquad - \
	  \left(
	  \nabla_{\mathbf{y}} G_{\infty,n}(\mathbf{y},s_{1},s_2,A)
	  \Big|_{\mathbf{y}=\mathbf{0}_T}
	\right)
	  ^{\top}
	\cdot [X_{t+\ell}]_{\ell}
  \Bigg)
\end{align*}
from Lemma~\ref{lem:cond_decomp}.
To streamline the argument, we will write throughout,
\begin{align*}
	\mathcal{S}_n(s_1,s_2)
	\   :=\
	\mathcal{S}_n(s_1,s_2,A)
	\,.
\end{align*}
Note that $\mathcal{S}_n(s_1,s_1)= 0$ since $(s_1\cdot A) \setminus (s_1 \cdot A) = \emptyset$.
\begin{lemma}[Decomposition of $\mathcal{S}_n$]
	\label{lem:decomp_Sn}
	Let
	Assumption~\ref{asu:f} and~\ref{asu:coef} hold. Then, for $s_1,s_2,s_3\in (\underline{s},\overline{s}]\cup\{\infty\}$ satisfying $s_1\le s_2 \le s_3$,
	\begin{align*}
		\mathcal{S}_n(s_1,s_3)
		\ = \
		\mathcal{S}_n(s_1,s_2)
		\ + \
		\mathcal{S}_n(s_2,s_3)
		\,.
	\end{align*}
	In particular, for all $s\in (\underline{s},\overline{s}]$,
	\begin{align*}
		\mathcal{S}_n(s,\overline{s})
		\ = \
		\sum_{k=1}^{K}
		\mathcal{S}_n(\pi_{j^s_{k},k},\pi_{j^s_{k-1},k-1})
		\ + \
		\mathcal{S}_n(s,\pi_{j^s_{K},K})
		\,.
	\end{align*}
\end{lemma}
\begin{proof}
	If $s_j=\infty$ for some $j\in \{1,2,3\}$, there is nothing to prove. Therefore, we assume that $s_j<\infty$ for all $j\in \{1,2,3\}$.
	Fix $s_1,s_2,s_3\in (\underline{s},\infty)$ satisfying $s_1\le s_2 \le s_3$. Since $A\in \mathcal{C}_T$ it holds $s_3\cdot A \subset s_2\cdot A \subset s_1\cdot A$, so that
	\begin{align*}
		 &
		G_n([X_{t+\ell}]_{\ell},s_1,s_3,A)
		\ = \
		\ind\{[X_{t+\ell}]_\ell \in u_n ((s_1\cdot A )\setminus (s_3\cdot A))\}
		\\ &
		  \ = \
		  \ind\{[X_{t+\ell}]_\ell \in u_n ((s_1\cdot A )\setminus (s_2\cdot A))\}
		  \ + \
		  \ind\{[X_{t+\ell}]_\ell \in u_n ((s_2\cdot A )\setminus (s_3\cdot A))\}
		\\ &
		  \ = \
		  G_n([X_{t+\ell}]_{\ell},s_1,s_2,A)
		  \ + \
		  G_n([X_{t+\ell}]_{\ell},s_2,s_3,A)
		  \,,
	\end{align*}
	and
	\begin{align*}
		 &
		\EE[G_n([X_{\ell}]_{\ell},s_1,s_3,A)]
		\ = \
		\PP[[X_{\ell}]_\ell \in u_n ((s_1\cdot A )\setminus (s_3\cdot A))]
		\\ &
		  \ = \
		  \PP[[X_{\ell}]_\ell \in u_n ((s_1\cdot A )\setminus (s_2\cdot A))]
		\ + \
		  \PP[[X_{\ell}]_\ell \in u_n ((s_2\cdot A )\setminus (s_3\cdot A))]
		\\ &
		  \ = \
		  \EE[G_n([X_{\ell}]_{\ell},s_1,s_2,A)]
		\ + \
		  \EE[G_n([X_{\ell}]_{\ell},s_2,s_3,A)]
		\,,
	\end{align*}
	and, by Lemma~\ref{lem:multi_swap},
	\begin{align*}
		 &
		\nabla_{\mathbf{y}} G_{\infty,n}(\mathbf{y},s_1,s_3,A)\Big|_{\mathbf{y}=\mathbf{0}_T}
		\ = \
		\ - \
		\int_{\RR^T}
		G_{n}(\mathbf{x}_{},s_1,s_3,A)
		\nabla
		f_{[X_{\ell}]_{\ell}}(\mathbf{x}_{})
		\,\mathrm{d}\mathbf{x}_{}
		\\ &
		  \ = \
		  \ - \
		  \int_{\RR^T}
		\left(
		  G_{n}(\mathbf{x}_{},s_1,s_2,A)
		  \ + \
		  G_{n}(\mathbf{x}_{},s_2,s_3,A)
		  \right)
		  \nabla
		  f_{[X_{\ell}]_{\ell}}(\mathbf{x}_{})
		  \,\mathrm{d}\mathbf{x}_{}
		\\ &
		  \ = \
		  \nabla_{\mathbf{y}} G_{\infty,n}(\mathbf{y},s_1,s_2,A)\Big|_{\mathbf{y}=\mathbf{0}_T}
		\ + \
		  \nabla_{\mathbf{y}} G_{\infty,n}(\mathbf{y},s_2,s_3,A)\Big|_{\mathbf{y}=\mathbf{0}_T}
		\,.
	\end{align*}
	The result now follows immediately.
\end{proof}
Now we perform the approximation step outlined in the heuristic argument before Theorem~\ref{thm:main}.
\begin{lemma}
	\label{lem:final}
	Let the assumptions of Theorem~\ref{lem:reduction_principle_optimal} hold.
	Then, for $n\in \NN$, the following holds:
	\begin{enumerate}
		\item
		      For all $s\in (\underline{s},\overline{s})$, 
		      \begin{align*}
			      \left|
			      \mathcal{S}_n(s,\pi_{j^s_{K},K})
			      \right|
			      \ \lesssim \
			      \left|
			      \mathcal{S}_n(
			      \pi_{j^s_{K}+1,K}
			      ,
			      \pi_{j^s_{K},K}
			      )
			      \right|
			      \ + \
			      B_n
			      \,,
		      \end{align*}
		      where
		      \begin{align*}
			      B_n
			       &
			      \ := \
			      2^{-K}
			      \cdot
            \sum_{t=1}^{n-T+1}
			      (1+\norm{[X_{t+\ell}]_{\ell}}_1)
			      \,.
		      \end{align*}
		\item
		      For all $s\in (\underline{s},\overline{s})$,
		      \begin{align*}
			      \EE[|B_n|]
			      \ \lesssim \
			      2^{-K}
			      \cdot n
		      \end{align*}
		      and
		      for all $k\le K$ and all $\delta>0$
		      \begin{align*}
			      \EE^{*}
			      \left[
				      \left(
				      \sup_{s\in (\underline{s},\overline{s})}
				      \left|
				      \mathcal{S}_n
				      \left(
				      \pi_{j_{k}^s,k},
				      \pi_{j_{k-1}^s,k-1}
				      \right)
				      \right|
				      \right)^{r_0}
				      \right]
			      \lor
			      \EE^{*}
			      \left[
				      \left(
				      \sup_{s\in (\underline{s},\overline{s})}
				      \left|
				      \mathcal{S}_n
				      \left(
				      \pi_{j_{K}^s+1,K},
				      \pi_{j_{K}^s,K}
				      \right)
				      \right|
				      \right)^{r_0}
				      \right]
			      \ \lesssim \
			      n^{r_0(\kappa_0+\delta/2)}
			      \,,
		      \end{align*}
		      where $\kappa_0$ and $r_0$ are defined in \eqref{eq:kappa_0}.
	\end{enumerate}
	In both (i) and (ii), $\lesssim$ means inequality up to a multiplicative constant that depends on $\delta$, $A$, $T$, $\underline{s}$, $\kappa_0$, the coefficients $(a_i)$, and the distribution of $\varepsilon$, and is independent of $k$, $K$, and $n$.
\end{lemma}
\begin{proof}[\textbf{Proof of (i):}]
	By Lemma~\ref{lem:decomp_Sn} and the definition of $\mathcal{S}_n$,
	\begin{align*}
		 &
		\left|
		\mathcal{S}_n(s,\pi_{j^s_{K},K})
		\ - \
		\mathcal{S}_n(
		\pi_{j^s_{K}+1,K}
		,
		\pi_{j^s_{K},K}
		)
		\right|
		\ = \
		\left|
		\mathcal{S}_n(
		\pi_{j^s_{K}+1,K}
		,
		s
		)
		\right|
		\\ &
		  \ \le \
      \sum_{t=1}^{n-T+1}
		  \ind\{[X_{t+\ell}]_\ell \in u_n\cdot (
		  (\pi_{j^s_{K}+1,K}
		\cdot A) \setminus (s \cdot A)
		)\}
		  \ + \
		  \PP[[X_{t+\ell}]_\ell \in u_n\cdot (
			  (\pi_{j^s_{K}+1,K}
			\cdot A) \setminus (s \cdot A)
			)]
		\\ &
		  \qquad + \
		  \norm{
			  \nabla_{\mathbf{y}}G_{\infty,n}(\mathbf{y},
			  \pi_{j^s_{K}+1,K},s,A
			  )\Big|_{\mathbf{y}=\mathbf{0}_T}
		}_{\infty}
		\norm{[X_{t+\ell}]_{\ell}}_1
		  \,.
	\end{align*}
	Since $s \le \pi_{j^s_K,K}$ by definition, and therefore
	$
		\left(
		(\pi_{j^s_{K}+1,K}
		\cdot A) \setminus (s \cdot A)
		\right)
		\subset
		\left(
		(\pi_{j^s_{K}+1,K}
		\cdot A) \setminus
		(\pi_{j^s_K,K}
		\cdot A)
		\right)
	$,
	this is bounded above by
	\begin{align*}
		 &
    \sum_{t=1}^{n-T+1}
		\ind\{[X_{t+\ell}]_\ell \in u_n\cdot (
		(\pi_{j^s_{K}+1,K}
		\cdot A) \setminus (\pi_{j^s_{K},K} \cdot A)
		)\}
		\ + \
		\PP[[X_{t+\ell}]_\ell \in u_n\cdot (
			(\pi_{j^s_{K}+1,K}
			\cdot A) \setminus (s \cdot A)
			)]
		\\ &
		  \qquad + \
		  \norm{
			  \nabla_{\mathbf{y}}G_{\infty,n}(\mathbf{y},
			  \pi_{j^s_{K}+1,K},s,A
			  )\Big|_{\mathbf{y}=\mathbf{0}_T}
		}_{\infty}
		\norm{[X_{t+\ell}]_{\ell}}_1
		\\ &
		  = \mathcal{S}_n(
		  \pi_{j^s_{K}+1,K},
		  \pi_{j^s_{K},K}
		)
		\\ &
		  \qquad + \
      \sum_{t=1}^{n-T+1}
		  \PP[[X_{t+\ell}]_\ell \in u_n\cdot (
			  (\pi_{j^s_{K}+1,K}
			\cdot A) \setminus
			  (\pi_{j^s_{K},K}
			  \cdot A)
			)]
		\ + \
		  \PP[[X_{t+\ell}]_\ell \in u_n\cdot (
			  (\pi_{j^s_{K}+1,K}
			\cdot A) \setminus (s \cdot A)
			)]
		\\ &
		  \qquad + \
		  \norm{
			  \nabla_{\mathbf{y}}G_{\infty,n}(\mathbf{y},
			  \pi_{j^s_{K}+1,K},s,A
			  )\Big|_{\mathbf{y}=\mathbf{0}_T}
		}_{\infty}
		\norm{[X_{t+\ell}]_{\ell}}_1
		\\ &
		  \qquad + \
		  \left(
		  \nabla_{\mathbf{y}}G_{\infty,n}(\mathbf{y},
		  \pi_{j^s_{K}+1,K},
		  \pi_{j^s_{K},K}
		,A
		  )\Big|_{\mathbf{y}=\mathbf{0}_T}
		\right)^{\top}
			  [X_{t+\ell}]_{\ell}
	\end{align*}
	which we further bound, using Lemma~\ref{lem:single}, as
	\begin{align*}
		 &
		\left|
		\mathcal{S}_n(
		\pi_{j^s_{K}+1,K},
		\pi_{j^s_{K},K}
		)
		\right|
		\\ &
		  \qquad + \
      \sum_{t=1}^{n-T+1}
		  \PP[[X_{t+\ell}]_\ell \in u_n\cdot (
			  (\pi_{j^s_{K}+1,K}
			\cdot A) \setminus
			  (\pi_{j^s_{K},K}
			  \cdot A)
			)]
		\ + \
		  \PP[[X_{t+\ell}]_\ell \in u_n\cdot (
			  (\pi_{j^s_{K}+1,K}
			\cdot A) \setminus (s \cdot A)
			)]
		\\ &
		  \qquad + \
		  \left(
		  \norm{
				  \nabla_{\mathbf{y}} G_{\infty,n}(\mathbf{y},
				  \pi_{j^s_{K}+1,K},s,A
				  )\Big|_{\mathbf{y}=\mathbf{0}_T}
			}_{\infty}
		\ + \
		  \norm{
				  \nabla_{\mathbf{y}} G_{\infty,n}(\mathbf{y},
				  \pi_{j^s_{K}+1,K},
				  \pi_{j^s_{K},K},A
				  )\Big|_{\mathbf{y}=\mathbf{0}_T}
			}_{\infty}
		\right)
		\norm{[X_{t+\ell}]_{\ell}}_1
		\\ &
		  \ \lesssim \
		  \left|
		  \mathcal{S}_n(
		  \pi_{j^s_{K}+1,K},
		  \pi_{j^s_{K},K}
		)
		  \right|
		  + \
		  (\pi_{j^s_{K},K} - \pi_{j^s_{K}+1,K})
		  \cdot
      \sum_{t=1}^{n-T+1}
		  \left(
		  1
		  \ + \
		  \norm{[X_{t+\ell}]_{\ell}}_1
		  \right)
		\\ &
		  \ = \
		  \left|
		  \mathcal{S}_n(
		  \pi_{j^s_{K}+1,K},
		  \pi_{j^s_{K},K}
		)
		  \right|
		  \ + \
		  B_n
		  \,.
	\end{align*}
	At the last step we used
	the fact that $\pi_{j,k}-\pi_{j+1,k} = 2^{-k} (\overline{s}-\underline{s})$.
	Here, the constant in $\lesssim$ depends only on $T$, $\underline{s}$, $\gamma$, the coefficients $(a_i)$, and the distribution of $\varepsilon$. This dependence remains for the rest of the proof.
	We conclude that
	\begin{align*}
		\left|
		\mathcal{S}_n(
		s,
		\pi_{j^s_{K},K}
		)
		\right|
		 &
		\ \le \
		\left|
		\mathcal{S}_n(
		s,
		\pi_{j^s_{K},K}
		)
		\ - \
		\mathcal{S}_n(
		\pi_{j^s_{K}+1,K},
		\pi_{j^s_{K},K}
		)
		\right|
		\ + \
		\left|
		\mathcal{S}_n(
		\pi_{j^s_{K}+1,K},
		\pi_{j^s_{K},K}
		)
		\right|
		\\ &
		  \ \lesssim \
		  \left|
		  \mathcal{S}_n(
		  \pi_{j^s_{K}+1,K},
		  \pi_{j^s_{K},K}
		)
		  \right|
		  \ + \
		  B_n
		  \,.
	\end{align*}
	This proves the claim.
	\\
	\textbf{Proof of (ii):}
	By definition of $B_n$ and stationarity of $(X_t)$,
	\begin{align*}
		\EE
		\left[
			|B_n|
			\right]
		\ \le \
		2^{-K}
		\cdot
		n
		\cdot
		\EE
		\left[
			(1+\norm{[X_{\ell}]_{\ell}}_1)
			\right]
		\ \lesssim \
		2^{-K}
		\cdot
		n
		\,.
	\end{align*}
	Since
  $\mathcal{S}_n
		\left(
		\pi_{i,k},
		\pi_{i,k}
    \right)=0$ for any $i,k$, and since $\pi_{j_{k}^s,k},\pi_{j_{k-1}^s,k-1}$ are at most neighbors in the partition at step $k$,
	it follows from Lemma~\ref{lem:partition} that
	\begin{align*}
		\left(
		\sup_{s\in (\underline{s},\overline{s})}
		\left|
		\mathcal{S}_n
		\left(
		\pi_{j_{k}^s,k},
		\pi_{j_{k-1}^s,k-1}
		\right)
		\right|
		\right)^{r_0}
		 &
		\ \le \
		\left(
		\max_{0\le i \le 2^{k}-1}
		\left|
		\mathcal{S}_n
		\left(
		\pi_{i+1,k},
		\pi_{i,k}
		\right)
		\right|
		\right)^{r_0}
				\ \le \
		\sum_{0\le i \le 2^{k}-1}
		\left|
		\mathcal{S}_n
		\left(
		\pi_{i+1,k},
		\pi_{i,k}
		\right)
		\right|^{r_0}
		\,.
	\end{align*}
	We get, by Theorem~\ref{lem:reduction_principle_optimal}, that
	\begin{align*}
		 &
		\EE^{*}
		\left[
			\left(
			\sup_{s\in (\underline{s},\overline{s})}
			\left|
			\mathcal{S}_n
			\left(
			\pi_{j_{k}^s,k},
			\pi_{j_{k-1}^s,k-1}
			\right)
			\right|
			\right)^{r_0}
			\right]
		\ \le \
		\sum_{0\le i \le 2^{k}-1}
		\EE
		\left[
			\left|
			\mathcal{S}_n
			\left(
			\pi_{i+1,k},
			\pi_{i,k}
			\right)
			\right|^{r_0}
			\right]
		\\ &
		  \ \lesssim \
		  n^{r_0(\kappa_0+\delta/2)}
		\sum_{0\le i \le 2^{k}-1}
		\left(
		  \pi_{i,k}
		  -
		  \pi_{i+1,k}
		  \right)
		\ = \
		  n^{r_0(\kappa_0+\delta/2)}
		\cdot
		  \left(
		  \overline{s}
		  -
		  \underline{s}
		  \right)
		\ \lesssim \
		  n^{r_0(\kappa_0+\delta/2)}
		\,,
	\end{align*}
	and
	\begin{align*}
		 &
		\EE^{*}
		\left[
			\left(
			\sup_{s\in (\underline{s},\overline{s})}
			\left|
			\mathcal{S}_n
			\left(
			\pi_{j_{K}^s+1,K},
			\pi_{j_{K}^s,K}
			\right)
			\right|
			\right)^{r_0}
			\right]
		\ \le \
		\sum_{0\le i \le 2^{K}-1}
		\EE
		\left[
			\left|
			\mathcal{S}_n
			\left(
			\pi_{i+1,K},
			\pi_{i,K}
			\right)
			\right|^{r_0}
			\right]
		\\ &
		  \ \lesssim \
		  n^{r_0(\kappa_0+\delta/2)}
		\sum_{0\le i \le 2^{K}-1}
		\left(
		  \pi_{i,K}
		  -
		  \pi_{i+1,K}
		  \right)
		\ \lesssim \
		  n^{r_0(\kappa_0+\delta/2)}
		\,.
	\end{align*}
    The proof is complete.
\end{proof}
\noindent
 \begin{proof}[{\textbf{Proof of Theorem~\ref{thm:main}:}}]
Let $\delta>0$ and assume $u_n\to\infty$ is such that
$
	u_n=o(n^{(d+1/(\nu \land 2)-\kappa_0-\delta)/(\nu+1)})
$,
where $\kappa_0$ is defined in \eqref{eq:kappa_0}.
By
Lemma~\ref{lem:decomp_Sn} and Lemma~\ref{lem:final},
\begin{align*}
	 &
	\sup_{s\in (\underline{s},\overline{s})}
	|\mathcal{S}_n(s,\overline{s})|
	\ \le \
	\sum_{k=1}^K
	\sup_{s\in (\underline{s},\overline{s})}
	\left|
	\mathcal{S}_n
	\left(
	\pi_{j^s_k,k},
	\pi_{j^s_{k-1},k-1}
	\right)
	\right|
	\ + \
	\sup_{s\in (\underline{s},\overline{s})}
	\left|
	\mathcal{S}_n
	\left(
	\pi_{j^s_K+1,K},
	\pi_{j^s_{K},K}
	\right)
	\right|
	\ + \
	\left|
	B_n
	\right|
	\\ &
	  \ \le \
	  2
	  \left(
	  (K+1)
	  \left(
	  \max_{1\le k\le K}
	\sup_{s\in (\underline{s},\overline{s})}
	  \left|
	  \mathcal{S}_n
	  \left(
	  \pi_{j^s_k,k},
	  \pi_{j^s_{k-1},k-1}
	  \right)
	\right|
	  \
	  \lor
	  \
	  \sup_{s\in (\underline{s},\overline{s})}
	  \left|
	  \mathcal{S}_n
	  \left(
	  \pi_{j^s_K+1,K},
	  \pi_{j^s_{K},K}
	  \right)
	\right|
	  \right)
	\ \lor \
	  |B_n|
	  \right)
	  \,,
\end{align*}
so that, fixing $\widetilde{\delta}>0$ and using Markov's inequality and Lemma~\ref{lem:final} again,
\begin{align*}
	 &
	\PP^{*}
	\left[
		\sup_{s\in (\underline{s},\overline{s})}
		|\mathcal{S}_n(s, \overline{s})|
		>
		\widetilde{\delta}\cdot n^{d+1/\alpha} f_{X_0}(u_n)
		\right]
	\\ &
	  \ \le \
	  \sum_{k=1}^K
	  \PP^{*}
	\left[
		  \sup_{s\in (\underline{s},\overline{s})}
		\left|
		  \mathcal{S}_n
		  \left(
		  \pi_{j^s_k,k},
		  \pi_{j^s_{k-1},k-1}
		  \right)
		\right|
		  >
		  \frac{
			  \widetilde{\delta}\cdot n^{d+1/\alpha} f_{X_0}(u_n)
			  }{2(K+1)}
		\right]
	\\ &
	  \qquad + \
	  \PP^{*}
	\left[
		  \sup_{s\in (\underline{s},\overline{s})}
		\left|
		  \mathcal{S}_n
		  \left(
		  \pi_{j^s_{K}+1,K},
		  \pi_{j^s_{K},K}
		  \right)
		\right|
		  >
		  \frac{
			  \widetilde{\delta}\cdot n^{d+1/\alpha} f_{X_0}(u_n)
			  }{2(K+1)}
		\right]
	\ + \
	  \PP
	  \left[
		  |B_n|
		  >
		  \frac{
			  \widetilde{\delta}\cdot n^{d+1/\alpha} f_{X_0}(u_n)
			  }{2}
		\right]
	\\ &
	  \ \le \
	  \left(
	  \widetilde{\delta}\cdot n^{d+1/\alpha} f_{X_0}(u_n)
	  \right)^{-r_0}
	\\ &
	  \qquad
	  \times
	  \left(
	  2^{r_0}
	  (K+1)^{r_0}\left( 
	  \sum_{k=1}^K
	  \EE^{*}
	\left[
		  \sup_{s\in (\underline{s},\overline{s})}
		\left|
		  \mathcal{S}_n
		  \left(
		  \pi_{j^s_k,k},
		  \pi_{j^s_{k-1},k-1}
		  \right)
		\right|^{r_0}
		  \right]
	\ + \
	  \EE^{*}
	\left[
		  \sup_{s\in (\underline{s},\overline{s})}
		\left|
		  \mathcal{S}_n
		  \left(
		  \pi_{j^s_K+1,K},
		  \pi_{j^s_{K},K}
		  \right)
		\right|^{r_0}
		  \right]
	\right) \right)
	\\ &
	  \qquad + \
	  2
	  \left(
	  \widetilde{\delta}\cdot n^{d+1/\alpha} f_{X_0}(u_n)
	  \right)^{-1}
	\EE
	  \left[
		  \left|
		  B_n
		  \right|
		  \right]
	\\ &
	  \ \lesssim \
	  \left(
	  \widetilde{\delta}\cdot n^{d+1/\alpha} f_{X_0}(u_n)
	  \right)^{-r_0}
	\cdot
	  \left(
	  K^{r_0+1}
	\cdot
	  n^{r_0(\kappa_0+\delta/2)}
	\right)
	\ + \
	  \left(
	  \widetilde{\delta}\cdot n^{d+1/\alpha} f_{X_0}(u_n)
	  \right)^{-1}
	2^{-K}\cdot n
	  \,,
\end{align*}
where the constant in $\lesssim$ depends only on $\delta$, $A$, $T$, $\underline{s}$, $\gamma$, $r_0$, the coefficients $(a_i)$, and the distribution of $\varepsilon$, and is independent of $K$ and $n$.
Since
$
\limsup_{n\to\infty}
	2^{-K_n}
	\frac{
		n^{2-d-1/\alpha}
	}{f_{X_0}(u_n)}<\infty
$,
\begin{align}
	\label{eq:cot1}
	\left(
	n^{d+1/\alpha}
	f_{X_0}(u_n)
	\right)^{-1}
	\cdot
	2^{-K}
	n
	\ = \
	\left(
	2^{-K}
	\cdot
	\frac{
			n^{2-d-1/\alpha}
		}{f_{X_0}(u_n)}
	\right)
	\frac{1}{n}
	\ \to \
	0\,,
\end{align}
as $n\to\infty$.
Then, for $\delta>0$ small enough and
$
	u_n=o(n^{(d+1/(\nu \land 2)-\kappa_0-\delta)/(\nu+1)})
$,
\begin{align*}
	\left(
	n^{d+1/\alpha}
	f_{X_0}(u_n)
	\right)^{-r_0}
	\cdot
	K^{r_0+1}
	\cdot
	n^{r_0(\kappa_0 + \delta/2)}
	\ = \
	\frac{n^{r_0(\kappa_0+3\delta/4-d-1/\alpha)}}{(f_{X_0}(u_n))^{r_0}}
	K^{r_0+1}
  n^{-r_0\delta/4}
	\ \to \
	0\,,
\end{align*}
by Potter bounds (showing that $n^{r_0(\kappa_0+3\delta/4-d-1/\alpha)}/(f_{X_0}(u_n))^{r_0}\to 0$), see Theorem~1.5.6 on p.25 in~\cite{binghamRegularVariation1987} and since $K$ grows only logarithmically.
Therefore
\begin{align*}
	\frac{n^{-d-1/\alpha}}{f_{X_0}(u_n)}
	\sup_{s\in (\underline{s},\overline{s})}
	|\mathcal{S}_n(s, \overline{s})|
	\ \to_{\PP^{*}} \
	0
	\,,
\end{align*}
as $n\to\infty$.
By Corollary~\ref{thm:main_simple}, we get
\begin{align*}
	\frac{n^{-d-1/\alpha}}{f_{X_0}(u_n)}
	|\mathcal{S}_n(\overline{s},\infty)|
	\ \to_\PP \
	0
	\,.
\end{align*}
It then follows from Lemma~\ref{lem:decomp_Sn}
\begin{align*}
	\frac{n^{-d-1/\alpha}}{f_{X_0}(u_n)}
	\sup_{s\in (\underline{s},\overline{s})}
	|\mathcal{S}_n(s, \infty)|
	\ \to_{\PP^{*}} \
	0
	\,.
\end{align*}
This finishes the proof.
\end{proof}
\section{Proofs of the results in Section~\ref{sec:clt}}
\label{app:proofsmain}
\subsection{Auxiliary results}
Assumption~\ref{asu:nabla_G_unif} has in the denominator of the first term $f_{X_0}(s\cdot u_n)$. We now provide a different formulation with $f_{X_0}(u_n)$ in the denominator of the first term and an additional factor $s^{-1-\nu}$ in the second term, which comes from $f_{X_0}(s\cdot u_n)/f_{X_0}(u_n)\to s^{-1-\nu}$ locally uniformly as $n\to\infty$.
\begin{lemma}[Different formulation of Assumption~\ref{asu:nabla_G_unif}]
	\label{lem:G_aux}
	Let Assumption~\ref{asu:rv_f} hold, and let $A\in \mathcal{C}_T$ satisfy Assumption~\ref{asu:nabla_G_unif}
	with $H_A$ and $N_A$. Then, as $n\to\infty$,
	\begin{align*}
		\sup_{s\in N_A}
		\left|
		\dfrac{
			\nabla_{\mathbf{y}} G_{\infty,n}(\mathbf{y},
			s,\infty,A
			)\Big|_{\mathbf{y}=\mathbf{0}_T}
		}{f_{X_0}(u_n)}
		\ - \
		s^{-1-\nu}
		H_A(s)
		\right|
		\ \to \ 0
		\,.
	\end{align*}
\end{lemma}
\begin{proof}
	Note that
	\begin{align*}
		 &
		\sup_{s\in N_A}
		\left|
		\frac{
			\nabla_{\mathbf{y}} G_{\infty,n}(\mathbf{y},
			s,\infty,A
			)\Big|_{\mathbf{y}=\mathbf{0}_T}
		}{f_{X_0}(u_n)}
		\ - \
		s^{-1-\nu}
		H_A(s)
		\right|
		\\
		 &
		\ \le \
		\sup_{s\in N_A}
		\left|
		\frac{f_{X_0}(s\cdot u_n)}{f_{X_0}(u_n)}
		\right|
		\sup_{s\in N_A}
		\left|
		\frac{
			\nabla_{\mathbf{y}} G_{\infty,n}(\mathbf{y},
			s,\infty,A
			)\Big|_{\mathbf{y}=\mathbf{0}_T}
		}{f_{X_0}(s\cdot u_n)}
		\ - \
		H_A(s)
		\right|
		\\ &
		  \qquad + \
		  \sup_{s\in N_A}
		\left|
		  H_A(s)
		  \right|
		  \sup_{s\in N_A}
		\left|
		  \frac{f_{X_0}(s\cdot u_n)}{f_{X_0}(u_n)}
		\ - \
		  s^{-1-\nu}
		\right|
		  \,.
	\end{align*}
	By assumption, $N_A\subset(0,\infty)$ is compact,
	so that by
	Assumption~\ref{asu:rv_f} (note that regular variation is always locally uniform) and the continuity of $H_A$,
	\begin{align*}
		\sup_{s\in N_A}
		\left|
		\frac{f_{X_0}(s\cdot u_n)}{f_{X_0}(u_n)}
		\right|
		\,,
		\sup_{s\in N_A}
		\left|
		H_A(s)
		\right|
		\ < \
		\infty\,,
	\end{align*}
	and also by Assumption~\ref{asu:nabla_G_unif}
	\begin{align*}
		\sup_{s\in N_A}
		\left|
		\frac{
			\nabla_{\mathbf{y}} G_{\infty,n}(\mathbf{y},
			s,\infty,A
			)\Big|_{\mathbf{y}=\mathbf{0}_T}
		}{f_{X_0}(s\cdot u_n)}
		\ - \
		H_A(s)
		\right|
		\,,
		\sup_{s\in N_A}
		\left|
		\frac{f_{X_0}(s\cdot u_n)}{f_{X_0}(u_n)}
		\ - \
		s^{-1-\nu}
		\right|
		\ \to \ 0
	\end{align*}
	as $n\to\infty$, so that
	\begin{align*}
		\sup_{s\in N_A}
		\left|
		\frac{
			\nabla_{\mathbf{y}} G_{\infty,n}(\mathbf{y},
			s,\infty,A
			)\Big|_{\mathbf{y}=\mathbf{0}_T}
		}{f_{X_0}(u_n)}
		\ - \
		s^{-1-\nu}
		H_A(s)
		\right|
		\ \to \ 0
	\end{align*}
	as $n\to\infty$.
\end{proof}
In the last result we derived a different formulation of Assumption~\ref{asu:nabla_G_unif}. We now use this to derive an adapted uniform reduction principle from Theorem~\ref{thm:main}, where the coefficient vector of the linear approximation and the partial sum of the $[X_{t+\ell}]_\ell$ have already converged.
\begin{theorem}[Uniform reduction principle with 
weak limit in the linear approximation]
	Let Assumptions~\ref{asu:f}, \ref{asu:coef}, and~\ref{asu:rv_f} hold, assume that $A\in \mathcal{C}_T$ satisfies
	Assumption~\ref{asu:nabla_G_unif}
	with $H_A$ and $N_A$ in the sense of Assumption~\ref{asu:nabla_G_unif}, and let $(\underline{s},\overline{s})\subset N_A$.
  Also assume that $u_n\to\infty$ is such that
$
		u_n=o(n^{(d+1/(\nu \land 2)-\kappa_0-\delta)/(\nu+1)})
	$
	for some $\delta>0$.
	Then, as $n\to\infty$,
	\begin{align*}
		\sup_{s\in(\underline{s},\overline{s})}
		\left|
		\frac{
			n^{-d-1/\alpha}
		}{f_{X_0}(u_n)}
    \sum_{t=1}^{n-T+1}
		\left(
		\ind\{[X_{t+\ell}]_{\ell}\in u_n \cdot s \cdot A\}
		\ - \
		\PP\left[[X_{t+\ell}]_{\ell}\in u_n \cdot s \cdot A\right]
		\right)
		\ - \
		s^{-1-\nu}
		(H_A(s)^\top \mathbf{1}_{T})
		Z_{\alpha}
		\right|
		\ \to_{\PP^{*}} \ 0
		\,,
	\end{align*}
	on a sufficiently rich probability space.
\end{theorem}
\begin{proof}
	By Theorem~\ref{thm:main},
	it remains to prove that there is a probability space on which
	\begin{align*}
		\sup_{s\in (\underline{s},\overline{s})}
		\left|
		n^{-d-1/\alpha}\frac{
			\left(
			\nabla_{\mathbf{y}} G_{\infty,n}(\mathbf{y},
			s,\infty,A
			)\Big|_{\mathbf{y}=\mathbf{0}_T}
			\right)
			^\top}{f_{X_0}(u_n)}
    \sum_{t=1}^{n-T+1}
		[X_{t+\ell}]_{\ell}
		\ - \
		s^{-1-\nu} (H_A(s)^\top \mathbf{1}_T) Z_{\alpha}
		\right|
	\end{align*}
	vanishes in outer probability as $n\to \infty$.
	By Lemma~\ref{lem:G_aux}
	\begin{align*}
		 &
		\sup_{s\in N_A}
		\left|
		\frac{
			\nabla_{\mathbf{y}} G_{\infty,n}(\mathbf{y},
			s,\infty,A
			)\Big|_{\mathbf{y}=\mathbf{0}_T}
		}{f_{X_0}(u_n)}
		\ - \
		s^{-1-\nu}
		H_A(s)
		\right|
		\ \to \ 0
	\end{align*}
	as $n\to\infty$, so that
	\begin{align*}
		 &
		\sup_{s\in (\underline{s},\overline{s})}
		\left|
		n^{-d-1/\alpha}\frac{
			\left(
			\nabla_{\mathbf{y}} G_{\infty,n}(\mathbf{y},
			s,\infty,
			A
			)\Big|_{\mathbf{y}=\mathbf{0}_T}
			\right)
			^\top}{f_{X_0}(u_n)}
    \sum_{t=1}^{n-T+1}
		[X_{t+\ell}]_{\ell}
		\ - \
		s^{-1-\nu} (H_A(s)^\top \mathbf{1}_T) Z_{\alpha}
		\right|
		\\ &
		  \ \le \
		  o(1)
		  \cdot
		  \left\|
      n^{-d-1/\alpha}\sum_{t=1}^{n-T+1} [X_{t+\ell}]_\ell
		  \right\|_1
		  \ + \
		  \sup_{s\in N_A}
		\left\|
		  s^{-1-\nu}
		H_A(s)
		\right\|_{\infty} \times
		  \left\|
      n^{-d-1/\alpha}\sum_{t=1}^{n-T+1} [X_{t+\ell}]_\ell
		  \ -\ Z_{\alpha} \mathbf{1}_{T} 
		\right\|_1
		  \,.
	\end{align*}
	The first summand of the bound vanishes in probability by Theorem~\ref{thm:linear_clt_multi} and Slutsky's theorem.
	For the second summand, note that,
	by Theorem~\ref{thm:linear_clt_multi} and Skorokhod's representation theorem,
	\begin{align*}
    n^{-d-1/\alpha}\sum_{t=1}^{n-T+1} [X_{t+\ell}]_\ell\ -\ Z_{\alpha} \mathbf{1}_{T} 
	\end{align*}
	vanishes in probability as $n\to\infty$
	on a sufficiently rich probability space. Since $N_A$ is compact and bounded away from $0$, and recalling that $H_A$, as a uniform limit of continuous functions, is continuous, 
	\begin{align*}
		\sup_{s\in N_A}
		\left\|
		s^{-1-\nu}
		H_A(s)
		\right\|_{\infty}
		\ < \ \infty
		\,,
	\end{align*}
	so that the second summand vanishes in probability on a sufficiently rich probability space. This proves the claimed convergence.
\end{proof}
As pointed out in Section~\ref{sec:sub:heur}, we now prove a central limit theorem for empirical tail measures with random thresholds and random centering.
\begin{theorem}[Empirical tail measure with random thresholds and random centering]
  \label{thm:conv_multi_inter}
	Let Assumptions~\ref{asu:f},~\ref{asu:coef}, and~\ref{asu:rv_f} hold, let $A_0,\ldots,A_h$, $h\in\NN_0$, be bounded away from 0, and let the threshold sequence $u_n\to\infty$ be such that
	$
		u_n=o(n^{(d+1/(\nu \land 2)-\kappa_0-\delta)/(\nu+1)})
	$ for some $\delta>0$.
    If the sets $A_0,\ldots,A_h$ are in $\mathcal{C}_T$ and satisfy Assumption~\ref{asu:nabla_G_unif} with functions $H_{A_0},\ldots, H_{A_h}$, and if $(Y^{(i)}_n)$, for $i\in \{0,\ldots,h\}$, are random sequences that satisfy $Y^{(i)}_n/u_n\to_{\PP}1$ as $n\to\infty$,
		      \begin{align*}
              &
			      n^{1-d-1/\alpha}
			      u_n
			      \begin{bmatrix}
				      \widehat{\PP}_n(Y^{(i)}_n/u_n,A_i)
				      \ - \
				      \PP_n(s,A_i)
				      \Big|_{s=Y_n^{(i)}/u_n}
			      \end{bmatrix}
			      _{i=0,\ldots,h}
                  \\&
			      \ \to_d \
            \frac{\nu}{2} Z_\alpha
			      [H_{A_i}(1)^{\top}\mathbf{1}_{T}]_{i=0,\ldots,h}
			      \qquad\text{as}\ n\to\infty\,. \qquad
		      \end{align*}
\end{theorem}
\begin{proof}
	We use the Cram\'er-Wold device.
	Let $\lambda_0,\ldots,\lambda_h\in\RR$, and $N_{A_0},\ldots, N_{A_h}$ in the sense of Assumption~\ref{asu:nabla_G_unif}.
	Fix $(\underline{s},\overline{s})\subset \bigcap_{i=0}^h N_{A_i}$, such that $\underline{s}<1<\overline{s}$.
	Since, for all $i\in \{0,\ldots,h\}$, $Y^{(i)}_n/u_n \to_{\PP} 1$ as $n\to\infty$, we
	have $Y^{(i)}_n /u_n \in(\underline{s},\overline{s})$ for all $i\in \{0,\ldots,h\}$ with probability converging to 1 as $n\to\infty$. Then, for any $\delta>0$, as $n\to\infty$, noting that we have $n/(n-T+1)\in (1,2)$,
	\begin{align*}
		 &
		\PP
		\left[
			\left|
			\sum_{i=0}^h
			\lambda_i
			\left(
			\Huge
			n^{1-d-1/\alpha}
			\frac{\PP[|X_0|>u_n]}{f_{X_0}(u_n)}
			\left(
			\widehat{\PP}_n(Y^{(i)}_n/u_n,A_i)
			\ - \
			\PP_n(s,A_i)
			\Big|_{s=Y_n^{(i)}/u_n}
			\right)
			\right.
			\right.
			\right.
		\\ &
			  \left.
			  \left.
			  \left.
			  \qquad
			  \qquad
			  \qquad - \
        \frac{n}{n-T+1}
			  n^{-d-1/\alpha}
      \sum_{t=1}^{n-T+1}
			  \frac{
				  \left(\nabla_{\mathbf{y}} G_{\infty,n}(\mathbf{y},s,\infty, A_i)\Big|_{
						  (\mathbf{y},s)=(\mathbf{0}_{T},Y_n^{(i)}/u_n)
						  }\right)^{\top}[X_{t+\ell}]_\ell
				  }{f_{X_0}(u_n)}
			\right)
			  \right|
			  \ > \ \delta
			  \right]
		\\ & \normalsize
		  \ \le \
		  \sum_{i=0}^h
		  \PP^{*}
		\left[
			  |\lambda_i|
			  \sup_{s\in(\underline{s},\overline{s})}
			\left|
			  \frac{n^{-d-1/\alpha}}{f_{X_0}(u_n)}
      \sum_{t=1}^{n-T+1}
			  \left.
			  \Bigg(
			  \ind\{[X_{t+\ell}]_\ell \in u_n \cdot s\cdot A_i \}
			  \ - \
			  \PP[[X_{\ell}]_\ell \in u_n \cdot s\cdot A_i]
			\right.
			  \right.
			  \right.
		\\ &
			  \qquad
			  \qquad
			  \qquad
			  \qquad
			  \qquad
			  \qquad
			  \qquad
			  \qquad
			  \left.
			  \left.
			  \left.
			  \qquad - \
			  \left(
			  \nabla_{\mathbf{y}} G_{\infty,n}(\mathbf{y}, s, \infty,  A_i)
			  \Big|_{\mathbf{y}=\mathbf{0}_T}
			\right)
			  ^{\top}[X_{t+\ell}]_\ell
			  \right)
			  \right|
			  \ > \
        \frac{\delta}{2(h+1)}
			\right]
		\\ &
		  \qquad + \ o(1)
		\\ &
		  \ \to \ 0
	\end{align*}
	by Theorem~\ref{thm:main}.
	Thus, by Slutsky's theorem,
	\begin{align*}
		\sum_{i=0}^h
		\lambda_i
		\left(
		n^{1-d-1/\alpha}
		\frac{\PP[|X_0|>u_n]}{f_{X_0}(u_n)}
		\left(
		\widehat{\PP}_n(Y^{(i)}_n/u_n,A_i)
		\ - \
		\PP_n(s,A_i)
		\Big|_{s=Y_n^{(i)}/u_n}
		\right)
		\right)
	\end{align*}
	and
	\begin{align*}
		\sum_{i=0}^h
		\lambda_i
		\frac{
			\left(\nabla_{\mathbf{y}} G_{\infty,n}(\mathbf{y},s,\infty, A_i)\Big|_{
					(\mathbf{y},s)=(\mathbf{0}_{T},Y_n^{(i)}/u_n)}\right)^{\top}
		}{f_{X_0}(u_n)}
		\left(
		n^{-d-1/\alpha}
    \sum_{t=1}^{n-T+1}
		[X_{t+\ell}]_\ell
		\right)
	\end{align*}
    have the same weak limiting behavior. 
	By Theorem~\ref{thm:linear_clt_multi},
	\begin{align*}
		n^{-d-1/\alpha}
    \sum_{t=1}^{n-T+1}
			[X_{t+\ell}]_\ell
		\ \to \
		Z_{\alpha}\mathbf{1}_{T}
	\end{align*}
	weakly as $n\to\infty$.
	On the other hand, by Lemma~\ref{lem:G_aux} and the continuous mapping theorem,
	\begin{align*}
		\sum_{i=0}^h
		\lambda_i
		\frac{
			\nabla_{\mathbf{y}} G_{\infty,n}(\mathbf{y},s,\infty, A_i)\Big|_{
					(\mathbf{y},s)=(\mathbf{0}_{T},Y_n^{(i)}/u_n)}
		}{f_{X_0}(u_n)}
		\ \to_{\PP} \
		\sum_{i=0}^h
		\lambda_i
		H_{A_i}(1)
		\,.
	\end{align*}
	Again, by Slutsky's theorem,
	\begin{align*}
		\sum_{i=0}^h
		\lambda_i
		\frac{
			\left(\nabla_{\mathbf{y}} G_{\infty,n}(\mathbf{y},s,\infty, A_i)\Big|_{
					(\mathbf{y},s)=(\mathbf{0}_{T},Y_n^{(i)}/u_n)}\right)^{\top}
		}{f_{X_0}(u_n)}
		\left(
		n^{-d-1/\alpha}
    \sum_{t=1}^{n-T+1}
		[X_{t+\ell}]_\ell
		\right)
		\ \to_d \
		\sum_{i=0}^h
		\lambda_i
		\left(
		H_{A_i}(1)^{\top} \mathbf{1}_{T}
		\right) Z_{\alpha}
		\,.
	\end{align*}
    Finally, note that, by Karamata's theorem~\citep[see][Proposition 1.5.10 on p.27]{binghamRegularVariation1987} and symmetry of $\varepsilon$ (and hence of $X_0$)
	\begin{align*}
		\frac{\PP[|X_0|>u_n]}{f_{X_0}(u_n)}
		\ \sim \ u_n 
    \frac{2}{\nu}
		\,.
	\end{align*}
    This finishes the proof. 
\end{proof}
\subsection{Proofs of the main results}
\begin{proof}[\textbf{Proof of Theorem~\ref{thm:conv_multi_simple}:}]
	We use the Cram\'er-Wold device.
	Let $\lambda_0,\ldots,\lambda_h\in\RR$, 
  and note that for $n\ge 2T$ we have $n/(n-T+1)\in (1,2)$. 
	Then, for any $\delta>0$, as $n\to\infty$
	\begin{align*}
		 &
		\PP
		\Bigg[
			\Bigg|
			\sum_{i=0}^h
			\lambda_i
			\Bigg(
			\Huge
			n^{1-d-1/\alpha}
			\frac{\PP[|X_0|>u_n]}{f_{X_0}(u_n)}
			\left(
			\widehat{\PP}_n(1,A_i)
			\ - \
			\PP_n(1,A_i)
			\right)
		\\ &
			  \qquad
			  \qquad
			  \qquad - \
        \frac{n}{n-T+1}
			  n^{-d-1/\alpha}
      \sum_{t=1}^{n-T+1}
			  \frac{
				  \left(\nabla_{\mathbf{y}} G_{\infty,n}(\mathbf{y},1, \infty,A_i)\Big|_{
						  \mathbf{y}=\mathbf{0}_{T}
					}\right)^{\top}[X_{t+\ell}]_\ell
				  }{f_{X_0}(u_n)}
			\Bigg)
			  \Bigg|
			  \ > \ \delta
			  \Bigg]
		\\ & \normalsize
		  \ \le \
		  \sum_{i=0}^h
		  \PP^{}
		\left[
			  |\lambda_i|
			  \left|
			  \frac{n^{-d-1/\alpha}}{f_{X_0}(u_n)}
      \sum_{t=1}^{n-T+1}
			  \left.
			  \Bigg(
			  \ind\{[X_{t+\ell}]_\ell \in u_n \cdot A_i \}
			  \ - \
			  \PP[[X_{\ell}]_\ell \in u_n \cdot  A_i]
			\right.
			  \right.
			  \right.
		\\ &
			  \qquad
			  \qquad
			  \qquad
			  \qquad
			  \qquad
			  \qquad
			  \qquad
			  \qquad
			  \left.
			  \left.
			  \left.
			  \qquad - \
			  \left(
			  \nabla_{\mathbf{y}} G_{\infty,n}(\mathbf{y}, 1,\infty,  A_i)
			  \Big|_{\mathbf{y}=\mathbf{0}_T}
			\right)
			  ^{\top}[X_{t+\ell}]_\ell
			  \right)
			  \right|
			  \ > \
        \frac{\delta}{2(h+1)}
			\right]
		\\ &
		  \ \to \ 0
	\end{align*}
	by Corollary~\ref{thm:main_simple}.
  Thus, by Slutsky's theorem and $n/(n-T+1)\to 1$,
	\begin{align*}
		\sum_{i=0}^h
		\lambda_i
		\left(
		n^{1-d-1/\alpha}
		\frac{\PP[|X_0|>u_n]}{f_{X_0}(u_n)}
		\left(
		\widehat{\PP}_n(1,A_i)
		\ - \
		\PP_n(1,A_i)
		\right)
		\right)
	\end{align*}
	has the same weak limiting behavior as
	\begin{align*}
		\sum_{i=0}^h
		\lambda_i
		\frac{
			\left(\nabla_{\mathbf{y}} G_{\infty,n}(\mathbf{y},1,\infty, A_i)\Big|_{
					\mathbf{y}=\mathbf{0}_T
				}\right)^{\top}
		}{f_{X_0}(u_n)}
		\left(
		n^{-d-1/\alpha}
    \sum_{t=1}^{n-T+1}
		[X_{t+\ell}]_\ell
		\right)
		\,.
	\end{align*}
	Again, by Karamata's theorem and symmetry of $X_0$,
	\begin{align*}
		\frac{\PP[|X_0|>u_n]}{f_{X_0}(u_n)}
		\ \sim \ 
    u_n
    \frac{2}{\nu}
		\,,
	\end{align*}
	so that we make this replacement in the final result.
	By Theorem~\ref{thm:linear_clt_multi},
	\begin{align*}
		n^{-d-1/\alpha}
    \sum_{t=1}^{n-T+1}
			[X_{t+\ell}]_\ell
		\ \to \
		Z_{\alpha}\mathbf{1}_{T}
	\end{align*}
	weakly as $n\to\infty$.
	On the other hand,
	by Assumption~\ref{asu:nabla_G}
	\begin{align*}
		\sum_{i=0}^h
		\lambda_i
		\frac{
			\nabla_{\mathbf{y}} G_{\infty,n}(\mathbf{y},1, \infty,A_i)\Big|_{
					\mathbf{y}=\mathbf{0}_T}
		}{f_{X_0}(u_n)}
		\ \to_{\PP} \
		\sum_{i=0}^h
		\lambda_i
		H_{A_i}(1)
		\,,
	\end{align*}
	so that, by Slutsky's theorem,
	\begin{align*}
		\sum_{i=0}^h
		\lambda_i
		\frac{
			\left(\nabla_{\mathbf{y}} G_{\infty,n}(\mathbf{y},1, \infty,A_i)\Big|_{
					\mathbf{y}=\mathbf{0}_T
				}\right)^{\top}
		}{f_{X_0}(u_n)}
		\left(
		n^{-d-1/\alpha}
    \sum_{t=1}^{n-T+1}
		[X_{t+\ell}]_\ell
		\right)
		\ \to \
		\sum_{i=0}^h
		\lambda_i
		\left(
		H_{A_i}(1)^{\top} \mathbf{1}_{T}
		\right) Z_{\alpha}
		\,,
	\end{align*}
	which finishes the proof.
\end{proof}
An overview of the proof of Theorem~\ref{thm:det_centering}, coming next, is given in Section~\ref{sec:sub:heur}.
\begin{proof}[\textbf{Proof of Theorem~\ref{thm:det_centering}:}] Replacing random centering as in Theorem~\ref{thm:conv_multi_inter} by deterministic centering adds a correction to the limit. To account for this correction,
	we set $r_n
		= 
		n^{1-d-1/\alpha}
		u_n$ and we make the decomposition
	\begin{align*}
			 &
			r_n
			\left(
			\widehat{\PP}_n(X_{n-k:n}/u_n,A_i)
			\ - \
			\PP_n(1,A_i)
			\right)
			\\ &
			  \ = \
			  r_n
			  \left(
			  \widehat{\PP}_n(X_{n-k:n}/u_n,A_i)
			  \ - \
			  \PP_n(s,A_i)\Big|_{s=X_{n-k:n}/u_n}
			\right)
			\\ &
			  \qquad + \
			  r_n
			  \left(
			  \PP_n(s,A_i)\Big|_{s=X_{n-k:n}/u_n}
			  \ - \
			  \PP_n(1,A_i)
			\right)
			\\ &
			  \ = \
			  r_n
			  \left(
			  \widehat{\PP}_n(X_{n-k:n}/u_n,A_i)
			  \ - \
			  \PP_n(s,A_i)\Big|_{s=X_{n-k:n}/u_n}
			\right)
            \\ &
			  \qquad + \frac{\partial}{\partial s}
			\PP_n(s,A_i)
			\Big|_{s=1} \cdot
			  r_n
			  \left(
			  \frac{X_{n-k:n}}{u_n}
			  \ - \ 1
			  \right)
			  + r_n \int_1^{X_{n-k:n}/u_n} \left( \frac{\partial}{\partial s}
			\PP_n(s,A_i) - \frac{\partial}{\partial s}
			\PP_n(s,A_i)
			\Big|_{s=1} \right) \mathrm{d}s
			\,.
	\end{align*}
    Note that $\frac{\partial}{\partial s}\PP_n(s,A)$ exists and is continuous due to Lemma~\ref{lem:multi_swap}(ii); see also Remark~\ref{rem:bremen}. Now 
    \begin{align*}
    &\left| r_n \int_1^{X_{n-k:n}/u_n} \left( \frac{\partial}{\partial s}
			\PP_n(s,A_i) - \frac{\partial}{\partial s}
			\PP_n(s,A_i)
			\Big|_{s=1} \right) \mathrm{d}s \right| \\
            &\leq r_n
			  \left|
			  \frac{X_{n-k:n}}{u_n}
			  \ - \ 1
        \right| \times \sup_{s\in [1\land X_{n-k:n}/u_n, 1 \lor X_{n-k:n}/u_n]} \left| \frac{\partial}{\partial s}
			\PP_n(s,A_i) - \frac{\partial}{\partial s}
			\PP_n(s,A_i)
			\Big|_{s=1} \right| \\
            &= r_n
			  \left|
			  \frac{X_{n-k:n}}{u_n}
			  \ - \ 1
			  \right| \times o_{\PP}(1) 
    \end{align*}
	since
	$X_{n-k:n}/u_n\to_{\PP} 1$ as $n\to\infty$ and the (deterministic) integrand converges uniformly to a continuous limit on $N_{A_i}$, see Remark~\ref{rem:stuttgart}. Consequently 
    \begin{align}
		\label{eq:prev_decomp}
		\begin{split}
			 & r_n
			\left(
			\widehat{\PP}_n(X_{n-k:n}/u_n,A_i)
			\ - \
			\PP_n(1,A_i)
			\right)
			\\ &
			  \ = \
			  r_n
			  \left(
			  \widehat{\PP}_n(X_{n-k:n}/u_n,A_i)
			  \ - \
			  \PP_n(s,A_i)\Big|_{s=X_{n-k:n}/u_n}
			\right) \\ &
			  \qquad + \left( \frac{\partial}{\partial s}
			\PP_n(s,A_i)
			\Big|_{s=1} + o_{\PP}(1) \right) \cdot
			  r_n
			  \left(
			  \frac{X_{n-k:n}}{u_n}
			  \ - \ 1
			  \right)
			  \,.
		\end{split}
	\end{align}
    We now study the joint convergence of the two terms in the right-hand side of~\eqref{eq:prev_decomp}. For this, we remark that by the properties of quantile functions, for $s^*\in \RR$, the inequalities
	\begin{align}
		\label{eq:can_be_written}
		r_n
		\left(
		\frac{X_{n-k:n}}{u_n}
		\ - \
		1
		\right)
		\ \le \ s^*
	\end{align}
	and
	\begin{align*}
		r_n
		\left(
		\frac{1}{n}
		\sum_{t=1}^n
		\frac{\ind\{X_t > u_n(1+s^*/r_n)\}}{\PP[X_0>u_n(1+s^*/r_n)]}
		\ - \ 1
		\right)
		\ \le \
		r_n
		\left(
		\frac{\PP[X_0>u_n]}{\PP[X_0>(1+s^*/r_n)u_n]}
		\ - \
		1
		\right)
	\end{align*}
	are equivalent. Set 
	\begin{align*}
		\xi(u_n)
		\ = \
		\frac{u_n\cdot f_{X_0}(u_n)}{\PP[X_0>u_n]}
		\,.
	\end{align*}
	As shown in the proof of Lemma~D.1 in \cite{scheffelCentralLimitTheory2025}
	if,
	$
		u_n= o
		\left(
		n^{(d+1/(\nu\wedge 2)-\kappa_0-\delta)/(\nu+1)}
		\right)
	$
	for some $\delta>0$, it holds for $s^*\in\RR$
	\begin{align*}
		r_n
		\left(
		\frac{\PP[X_0>u_n]}{\PP[X_0>(1+s^*/r_n)u_n]}
		\ - \
		1
		\right)
		\ = \
		s^*\cdot \xi(u_n)
		\left(
		1+o(1)
		\right)
	\end{align*}
	and
	\begin{align*}
		\frac{r_n/\xi(u_n)}{n}
		\sum_{t=1}^n
		\left(
		\frac{\ind\{X_t > u_n\}}{\PP[X_0>u_n]}
		\ - \
		\frac{\ind\{X_t > u_n(1+s^*/r_n)\}}{\PP[X_0>u_n(1+s^*/r_n)]}
		\right)
		\ = \
		o_{\PP}(1)
		\,.
	\end{align*}
    Recall that $\xi(u_n) \to \nu$ as $n\to\infty$ by Karamata's theorem. Then~\eqref{eq:can_be_written} is equivalent to 
    \[
    \frac{r_n}{n}
		\sum_{t=1}^n
		\left(
		\frac{\ind\{X_t > u_n\}}{\PP[X_0>u_n]}
		\ - \
		1
		\right) \leq s^*\cdot \nu
		\left(
		1+o(1)
		\right) + o_{\PP}(1).
    \]
	Also note that
	\begin{align*}
    &
		\frac{r_n}{n}
		\sum_{t=1}^n
		\left( \frac{\ind\{X_t > u_n\}}{\PP[X_0>u_n]}
		\ - \ 1 \right)
    \\&
    \ = \ 
		\frac{r_n}{n}
    \sum_{t=n-T+2}^n
		\left( \frac{\ind\{X_t > u_n\}}{\PP[X_0>u_n]} - 1 \right)
    \ + \ 
    2
    \frac{n-T+1}{n}
    \left(
    \frac{r_n}{n-T+1}
    \sum_{t=1}^{n-T+1} \left( 
		\frac{\ind\{X_t > u_n\}}{\PP[|X_0|>u_n]}
		\ - \ 
    \frac{
    \PP[X_0>u_n]
    }{
    \PP[|X_0|>u_n]
    }
    \right) \right)
    \\&
    \ = \ 
    o_{\PP}(1)
    \ + \ 
    2(1 + o(1)) r_n
    \left(
    \widehat{\PP}_n(1,B)
    \ - \ 
    \PP_n(1,B)
    \right)
   	\end{align*}
	with
	$
		B
		\ = \
		(1,\infty) \times \RR^{T-1}
		\in \mathcal{C}_T
		\,.
	$
	Therefore, \eqref{eq:can_be_written} is equivalent to
	\begin{align}
		\label{eq:3rd_equiv}
		r_n
		\left(
		\widehat{\PP}_n(1,B)
		\ - \
		\PP_n(1,B)
		\right)
		\ \le \
		s^*\cdot
    \frac{\nu}{2}
		\cdot
		(1+o(1))
		\ + \
		o_{\PP}(1)
		\,.
	\end{align}
    By Equation~\eqref{eq:3rd_equiv}, the joint limit of the first and the second term in~\eqref{eq:prev_decomp} can be studied in a unified way using Theorem~\ref{thm:conv_multi_inter}, at the level of distribution functions.
	Let $s_0,\ldots,s_h, s^*\in \RR$. Write
	\begin{align*}
		E_n(s_0,\ldots,s_h)
		\ := \
		\bigcap_{i=0}^h
		\left\{
		r_n
		\left(
		\widehat{\PP}_n(X_{n-k:n}/u_n,A_i)
		\ - \
		\PP_n(s,A_i)\Big|_{s=X_{n-k:n}/u_n}
		\right)
		\ \le \ s_i
		\right\}
	\end{align*}
	and, by Equation~\eqref{eq:3rd_equiv}
	\begin{align*}
		 &
		\PP
		\left[
			E_n(s_0,\ldots,s_h)
			\,,
			\quad
			r_n
			\left(
			\frac{X_{n-k:n}}{u_n}
			\ - \ 1
			\right)
			\ \le \ s^*
			\right]
		\\ &
		  \ = \
		  \PP
		  \left[
			  E_n(s_0,\ldots,s_h)
			  \,,
			  \quad
			  \frac{r_n}{\nu/2}
			\left(
			  \widehat{\PP}_n(1, B)
			  \ - \
			  \PP_n(1,B)
			\right)
			  \ \le \ s^*\cdot (1+o(1)) + o_{\PP}(1)
			  \right]
		\,.
	\end{align*}
	We shall now apply Theorem~\ref{thm:conv_multi_inter}, with $h$ replaced by $h+1$, $A_{h+1}=B$, $Y^{(i)}_n=X_{n-k:n}$ for $i\in \{0,\ldots,h\}$ and $Y^{(h+1)}_n=u_n$. For this, note that the set $B$ satisfies Assumption~\ref{asu:nabla_G_unif}, because 
    \begin{align*}
		\frac{
			\nabla_{\mathbf{y}}
			G_{\infty,n}
			(\mathbf{y},s,\infty,B)
			\Big|_{\mathbf{y}=\mathbf{0}_T}
		}{f_{X_0}(s\cdot u_n)}
		\ = \
		\frac{
			\nabla_{\mathbf{y}}
			\PP[X_0 > s\cdot u_n- y_0]
			\Big|_{\mathbf{y}=\mathbf{0}_T}
		}{f_{X_0}(s\cdot u_n)}
		\ = \
		\mathbf{e}_{1}
		\ = \ H_B(s) 
        \ = \ H_B(1).
	\end{align*}
    In particular $H_B(1)^{\top} \mathbf{1}_T = 1$. Also note that $X_{n-k:n}/u_n \to_{\PP} 1$ by Lemma~D.1 in~\cite{scheffelCentralLimitTheory2025}, so that $X_{n-k:n}$ is a valid random threshold in the application of Theorem~\ref{thm:conv_multi_inter}. Then this theorem provides that
	\begin{align*}
		 &
		r_n
		\left[
			\left[
				\left(
				\widehat{\PP}_n(X_{n-k:n}/u_n,A_i)
				\ - \
				\PP_n(s,A_i)\Big|_{s=X_{n-k:n}/u_n}
				\right)
				\right]_{i=0,\ldots,h}
			\,,
			\quad
			\frac{1}{\nu/2}
      \left(
			\widehat{\PP}_n(1,B)
			\ - \
			\PP_n(1,B)
      \right)
			\right]
		\,,
	\end{align*}
	converges in distribution to
	\begin{align*}
    \frac{\nu}{2} Z_{\alpha}
		\left[
			[H_{A_i}(1)^{\top}\mathbf{1}_T]_{i=0,\ldots,h}\,,
			\quad
			\frac{2}{\nu}
			\right]
		\,.
	\end{align*}
	Finally, by \eqref{eq:prev_decomp}, Remark~\ref{rem:stuttgart}, and Slutsky's theorem,
	\begin{align*}
		r_n
		\left[
			\widehat{\PP}_n(X_{n-k:n}/u_n,A_i)
			\ - \
			\PP_n(1,A_i)
			\right]_{i=0,\ldots,h}
	\end{align*}
	converges in distribution to
	\begin{align*}
    \frac{\nu}{2} Z_{\alpha}
		\left[
			\left(
			H_{A_i}(1)
			\right)^{\top}
			\mathbf{1}_T
			\ + \
			p_{A_i}(1)
			\right]_{i=0,\ldots,h}
		\,.
	\end{align*}
    The proof is complete.
\end{proof}
\begin{proof}[\textbf{Proof of Theorem~\ref{thm:final}:}]
We closely follow the argument of the proof of
Corollary~3.3 in~\cite{davisExtremogramCorrelogramExtreme2009}.
	Note that
	\begin{align*}
		 &
		\frac{\widehat{\PP}_n(s,A_i)}{\widehat{\PP}_n(s,A_0)}
		\ - \
		\frac{\PP_n(1,A_i)}{\PP_n(1,A_0)}
		\\ &
		  \ = \
		  \frac{1}{\PP_n(1,A_0)\cdot\widehat{\PP}_n(s,A_0)}
		\left(
		  \PP_n(1,A_0)
		\left(
		  \widehat{\PP}_n(s,A_i)
		  \ - \
		  \PP_n(1,A_i)
		\right)
		  \ - \
		  \PP_n(1,A_i)
		\left(
		  \widehat{\PP}_n(s,A_0)
		  \ - \
		  \PP_n(1,A_0)
		\right)
		  \right)
		  \,,
	\end{align*}
	so that
	\begin{align*}
		 &
		\left[
			\frac{\widehat{\PP}_n(s,A_i)}{\widehat{\PP}_n(s,A_0)}
			\ - \
			\frac{\PP_n(1,A_i)}{\PP_n(1,A_0)}
			\right]_{i=1,\ldots,h}
		\\ &
		  \ = \
		  \frac{1}{\PP_n(1,A_0)\cdot \widehat{\PP}_n(s,A_0)}
		\begin{bmatrix}
			\PP_n(1,A_0)\cdot \mathbf{I}_{h} & [-\PP_n(1,A_i)]_{i=1,\ldots,h}
		\end{bmatrix}
		  \begin{bmatrix}
			\left[
				\widehat{\PP}_n(s,A_i)
				\ - \
				\PP_n(1,A_i)
				\right]_{i=1,\ldots,h}
			\\
			\widehat{\PP}_n(s,A_0)
			\ - \
			\PP_n(1,A_0)
		\end{bmatrix}.
	\end{align*}
	It holds, for $Y_n=u_n$ or $Y_n = X_{n-k:n}$ respectively, that
	\begin{align*}
		\widehat{\PP}_n(Y_n/u_n,A_0)\ -\ \PP_n(1,A_0)\ \to_\PP \ 0
		\qquad\text{as}\ n\to\infty,
	\end{align*}
  by Theorem~\ref{thm:conv_multi_simple} and Theorem~\ref{thm:det_centering} respectively.
	Since $\PP_n(1,A_0)\to \mu_{T}(A_0)>0$ as $n\to\infty$,
	\begin{align*}
		\widehat{\PP}_n(Y_n/u_n,A_0)\cdot\PP_n(1,A_0)\ \to_{\PP}\  (\mu_T(A_0))^2
	\end{align*}
	for $Y_n=u_n$ or $Y_n=X_{n-k:n}$.
	Therefore,
	\begin{align*}
		 &
		\frac{1}{\PP_n(1,A_0)\cdot \widehat{\PP}_n(Y_n/u_n,A_0)}
		\begin{bmatrix}
			\PP_n(1,A_0)\cdot \mathbf{I}_{h} & [-\PP_n(1,A_i)]_{i=1,\ldots,h}
		\end{bmatrix}
		\\ &
		  \ \to_{\PP} \
		  \frac{1}{
			  \left(
			  \mu_T(A_0)
			  \right)^2
			  }
		\begin{bmatrix}
			\mu_T(A_0)\cdot \mathbf{I}_{h} & [-\mu_T(A_i)]_{i=1,\ldots,h}
		\end{bmatrix}
		  \,.
	\end{align*}
  By Theorem~\ref{thm:conv_multi_simple} 
	\begin{align*}
		n^{1-d-1/\alpha}
		u_n
		\begin{bmatrix}
			\left[
				\widehat{\PP}_n(1,A_i)
				\ - \
				\PP_n(1,A_i)
				\right]_{i=1,\ldots,h}
			\\
			\widehat{\PP}_n(1,A_0)
			\ - \
			\PP_n(1,A_0)
		\end{bmatrix}
		\ \to_d \
    \frac{
		\nu
    }{2} Z_{\alpha}
		\begin{bmatrix}
			\left[
				H_{A_i}(1)^{\top} \mathbf{1}_{T}
				\right]_{i=1,\ldots,h}
			\\
			H_{A_0}(1)^{\top} \mathbf{1}_{T}
		\end{bmatrix}
		\,,
	\end{align*}
	and by Theorem~\ref{thm:det_centering}
	\begin{align*}
		n^{1-d-1/\alpha}
		q_{X_0}
		\left(
		1 - \frac{k}{n}
		\right)
		\begin{bmatrix}
			\left[
				\widehat{\PP}_n(X_{n-k:n}/u_n,A_i)
				\ - \
				\PP_n(1,A_i)
				\right]_{i=1,\ldots,h}
			\\
			\widehat{\PP}_n(X_{n-k:n}/u_n,A_0)
			\ - \
			\PP_n(1,A_0)
		\end{bmatrix}
	\end{align*}
	converges in distribution to
	\begin{align*}
    \frac{\nu}{2} Z_{\alpha}
		\begin{bmatrix}
			\left[
        H_{A_i}(1)^{\top} \mathbf{1}_{T} \ +\
        p_{A_i}(1)
				\right]_{i=1,\ldots,h}
			\\
      H_{A_0}(1)^{\top} \mathbf{1}_{T} \ +\
        p_{A_0}(1)
		\end{bmatrix}
		\,.
	\end{align*}
	The claimed convergence in the two cases follows from Slutsky's theorem.
\end{proof}
\section{Proofs of results in Section~\ref{sec:rect}}
\label{app:proofsrect}
\begin{proof}[\textbf{Proof of Proposition~\ref{prop:imply}:}]
  \textbf{Verifying Assumption~\ref{asu:cond_limit}:}
  We apply partial integration over each coordinate. 
  First, by Lemma~\ref{lem:multi_density}(iv) and dominated convergence,
  \begin{align*}
    &
	\int_{\RR^T}
		\ind\{\mathbf{x}\in u_n \cdot s \cdot A\}
		\cdot
		\mathbf{x}^{\top}
		\cdot
		\nabla f_{[X_\ell]_\ell}(\mathbf{x})
		\,\mathrm{d}\mathbf{x}
		\\ &
		  \ = \
		  \sum_{i=1}^T
		  \int_{\RR^T}
		\ind\left\{\min_{j\in J_A}(x_{j-1} - u_n \cdot s \cdot \mathfrak{a}_j)\ge 0\right\}
		  x_{i-1}
		\frac{\partial}{\partial x_{i-1}}
		f_{[X_{\ell}]_\ell}(\mathbf{x})
		  \,\mathrm{d}\mathbf{x}
        \\&
        \ = \
		  \sum_{i=1}^T
		  \lim_{M\to \infty} \int_{\RR^T}
		\ind\left\{\min_{j\in J_A}(x_{j-1} - u_n \cdot s \cdot \mathfrak{a}_j)\ge 0\,, |x_{i-1}|\leq M \right\} 
		  x_{i-1}
		\frac{\partial}{\partial x_{i-1}}
		f_{[X_{\ell}]_\ell}(\mathbf{x})
		  \,\mathrm{d}\mathbf{x}\,.
    \end{align*}
    Fix $i\notin J_A$. By partial integration over the variable $x_{i-1}$,  
    \begin{align*}
    &\lim_{M\to \infty}\int_{\RR^T}
		\ind\left\{\min_{j\in J_A}(x_{j-1} - u_n \cdot s \cdot \mathfrak{a}_j)\ge 0\,, |x_{i-1}|\leq M \right\} 
		  x_{i-1}
		\frac{\partial}{\partial x_{i-1}}
		f_{[X_{\ell}]_\ell}(\mathbf{x})
		  \,\mathrm{d}\mathbf{x} \\
    &\ =\  \lim_{M\to \infty}\int_{\RR^{T-1}}
		\ind\left\{\min_{j\in J_A}(x_{j-1} - u_n \cdot s \cdot \mathfrak{a}_j)\ge 0 \right\} 
		  M
		f_{[X_{\ell}]_\ell}(\mathbf{x}_{\ell\neq i-1},x_{i-1}=M)
		  \,\mathrm{d}\mathbf{x}_{\ell\neq i-1} \\
    &\qquad +\  \lim_{M\to \infty}\int_{\RR^{T-1}}
		\ind\left\{\min_{j\in J_A}(x_{j-1} - u_n \cdot s \cdot \mathfrak{a}_j)\ge 0 \right\} 
		  M
		f_{[X_{\ell}]_\ell}(\mathbf{x}_{\ell\neq i-1},x_{i-1}=-M)
		  \,\mathrm{d}\mathbf{x}_{\ell\neq i-1}\\
    &\qquad -\  \PP[[X_{\ell}]_{\ell}\in u_n \cdot s\cdot A].
    \end{align*}
    Now 
    \begin{align*}
    &\int_{\RR^{T-1}}
		\ind\left\{\min_{j\in J_A}(x_{j-1} - u_n \cdot s \cdot \mathfrak{a}_j)\ge 0 \right\} 
		  M
		f_{[X_{\ell}]_\ell}(\mathbf{x}_{\ell\neq i-1},x_{i-1}=\pm M)
		  \,\mathrm{d}\mathbf{x}_{\ell\neq i-1} \\
    &\leq \int_{\RR^{T-1}}
		  M
		f_{[X_{\ell}]_\ell}(\mathbf{x}_{\ell\neq i-1},x_{i-1}=\pm M)
		  \,\mathrm{d}\mathbf{x}_{\ell\neq i-1} = M f_{X_{i-1}}(\pm M)\  =\  M f_{X_0}(\pm M)
    \end{align*}
    converges to 0 as $M\to\infty$ by Assumption~\ref{asu:rv_f}. Therefore, for all $i\notin J_A$,
    \[
    \lim_{M\to \infty}\int_{\RR^T}
		\ind\left\{\min_{j\in J_A}(x_{j-1} - u_n \cdot s \cdot \mathfrak{a}_j)\ge 0\,, |x_{i-1}|\leq M \right\} 
		  x_{i-1}
		\frac{\partial}{\partial x_{i-1}}
		f_{[X_{\ell}]_\ell}(\mathbf{x})
		  \,\mathrm{d}\mathbf{x} = - \PP[[X_{\ell}]_{\ell}\in u_n \cdot s\cdot A].
    \]
    For $i\in J_A$, partial integration analogously yields
    \begin{align*}
    &\lim_{M\to \infty}\int_{\RR^T}
		\ind\left\{\min_{j\in J_A}(x_{j-1} - u_n \cdot s \cdot \mathfrak{a}_j)\ge 0\,, |x_{i-1}|\leq M \right\} 
		  x_{i-1}
		\frac{\partial}{\partial x_{i-1}}
		f_{[X_{\ell}]_\ell}(\mathbf{x})
		  \,\mathrm{d}\mathbf{x} \\
    &= \lim_{M\to \infty}\int_{\RR^{T-1}}
		\ind\left\{\min_{j\in J_A\setminus\{i\}}(x_{j-1} - u_n \cdot s \cdot \mathfrak{a}_j)\ge 0 \right\} 
		  M
		f_{[X_{\ell}]_\ell}(\mathbf{x}_{\ell\neq i-1},x_{i-1}=M)
		  \,\mathrm{d}\mathbf{x}_{\ell\neq i-1} \\
    &\ - \int_{\RR^{T-1}}
		\ind\left\{\min_{j\in J_A\setminus\{i\}}(x_{j-1} - u_n \cdot s \cdot \mathfrak{a}_j)\ge 0 \right\} 
		u_n \cdot s \cdot \mathfrak{a}_{i} \, 
		f_{[X_{\ell}]_\ell}(\mathbf{x}_{\ell\neq i-1},x_{i-1}=u_n \cdot s \cdot \mathfrak{a}_{i})
		  \,\mathrm{d}\mathbf{x}_{\ell\neq i-1}\\
    &\ - \PP[[X_{\ell}]_{\ell}\in u_n \cdot s\cdot A] \\
    &= -s\cdot u_n\cdot \mathfrak{a}_i
\cdot
		f_{X_0}(s\cdot u_n\cdot \mathfrak{a}_i)
		  \cdot
		  \PP[[X_\ell]_\ell\in u_n \cdot s\cdot A \mid X_{i-1} = u_n \cdot s\cdot \mathfrak{a}_i] \ - \ 
    \PP[[X_{\ell}]_{\ell}\in u_n \cdot s\cdot A]
    \,.
    \end{align*}
    Consequently 
    \begin{align*}
    &\int_{\RR^T}
		\ind\{\mathbf{x}\in u_n \cdot s \cdot A\}
		\cdot
		\mathbf{x}^{\top}
		\cdot
		\nabla f_{[X_\ell]_\ell}(\mathbf{x})
		\,\mathrm{d}\mathbf{x}
      \\&
      =- \ 
		  \sum_{i\in J_A}
s\cdot u_n\cdot \mathfrak{a}_i
\cdot
		f_{X_0}(s\cdot u_n\cdot \mathfrak{a}_i)
		  \cdot
		  \PP[[X_\ell]_\ell\in u_n \cdot s\cdot A \mid X_{i-1} = u_n \cdot s\cdot \mathfrak{a}_i]
    \\&
    \qquad - \ 
    T\cdot \PP[[X_{\ell}]_{\ell}\in u_n \cdot s\cdot A]
    \,.
  \end{align*}
    By regular variation of $f_{X_0}$ (Assumption~\ref{asu:rv_f}), symmetry of $\varepsilon$ and $X_0$, and Karamata's theorem, 
	\begin{align*}
		\frac{
s\cdot u_n\cdot \mathfrak{a}_i\cdot
			f_{X_0}(s\cdot u_n\cdot \mathfrak{a}_i)
		}{\PP[|X_0|>u_n]}
		\ = \
    \frac{1}{2}
		\frac{
s\cdot u_n\cdot \mathfrak{a}_i\cdot
			f_{X_0}(s\cdot u_n\cdot \mathfrak{a}_i)
		}
		{\PP[X_0>u_n \cdot s\cdot \mathfrak{a}_i]}
		\cdot
		\frac
		{\PP[X_0>u_n \cdot s\cdot \mathfrak{a}_i]}
		{\PP[X_0>u_n]}
		\ \to \
    \frac{\nu}{2}
		\cdot (s\cdot \mathfrak{a}_i)^{-\nu}
	\end{align*}
	and 
	\begin{align*}
		\frac{
			\PP[[X_\ell]_\ell\in u_n \cdot s\cdot A]
		}{\PP[|X_0|>u_n]}
		\ \to \
		\mu_T(s\cdot A)
		\ = \
		s^{-\nu}
		\mu_T(A)
    \,,
	\end{align*}
  both locally uniformly in $s$,
  so that
	\begin{align*}
		 &
		\frac{1}{\PP[|X_0|>u_n]}
		\int_{\RR^T}
		\ind\{\mathbf{x}\in u_n \cdot s \cdot A\}
		\cdot
		\mathbf{x}^{\top}
		\cdot
		\nabla f_{[X_\ell]_\ell}(\mathbf{x})
		\,\mathrm{d}\mathbf{x}
		\\ &
		  \ \to \
		  s^{-\nu}
		\left(
    -
    \frac{\nu}{2}
		  \sum_{i\in J_A}
		  \mathfrak{a}_i^{-\nu}
		p_{i,A}
		  \ -\ 
		  T
		  \mu_T(A)
		\right)
		  \,,
	\end{align*}
  locally uniformly in $s$.
  It follows that
	\begin{align*}
		p_A(s)
		\ = \
		p_A(1)
    \ = \ 
		-
		\sum_{i\in J_A}
		\mathfrak{a}_i^{-\nu}
		p_{i,A}
		\,.
	\end{align*}
  \textbf{Verifying Assumption~\ref{asu:nabla_G_unif}:} Rewrite $G_{\infty,n}$ as 
	\begin{align*}
		G_{\infty,n}(\mathbf{y},s,\infty,A)
		\ = \
		\PP
		\left[
			\min_{j\in J_A}
			\left(
			X_{j-1} + y_{j-1} - u_n \cdot s\cdot \mathfrak{a}_j
			\right)
			\ \ge \ 0
			\right]
		\,,
	\end{align*}
	so that
	\begin{align*}
		 &
		\nabla_{\mathbf{y}}
		G_{\infty,n}(\mathbf{y},s,\infty,A)
		\Big|_{\mathbf{y}=\mathbf{0}_T}
		\ = \
		\sum_{j\in J_A}
		f_{X_0}(s\cdot u_n\cdot \mathfrak{a}_j)
		\cdot
		\PP
		\left[
			[X_{\ell}]_{\ell}\in s\cdot u_n \cdot A
			\mid
			X_{j-1} = s\cdot u_n \cdot \mathfrak{a}_j
			\right] \cdot \mathbf{e}_j
		\,.
	\end{align*}
	Therefore
	\begin{align*}
		\frac{
			\nabla_{\mathbf{y}}
			G_{\infty,n}(\mathbf{y},s,\infty,A)
			\Big|_{\mathbf{y}=\mathbf{0}_T}
		}{f_{X_0}(s\cdot u_n)}
		\ = \
		\sum_{j\in J_A}
		\frac{
			f_{X_0}(s\cdot u_n\cdot \mathfrak{a}_j)
		}{f_{X_0}(s\cdot u_n)}
		\cdot
		\PP
		\left[
			[X_{\ell}]_{\ell}\in s\cdot u_n \cdot A
			\mid
			X_{j-1} = s\cdot u_n \cdot \mathfrak{a}_j
			\right] \cdot \mathbf{e}_j.
	\end{align*}
  By local uniformity of regular variation and Assumption~\ref{asu:simplified}, this converges uniformly in $s\in N_A$ to
	\begin{align*}
		H_A(s)
		\ = \
		H_A(1)
    \ = \ 
		\sum_{j\in J_A}
		\mathfrak{a}_j^{-(\nu+1)}
		\cdot
		p_{j,A} \cdot \mathbf{e}_j
		\,.
	\end{align*}
    This is the announced result.
\end{proof}
\begin{proof}[\textbf{Proof of Theorem~\ref{thm:example}}]
  Let $(c_i)_{i\in \NN_0}\subset [0,\infty)$ satisfy $0<\sum_{i=0}^{\infty}c_i^{\nu}<\infty$. Then, under the assumptions of unimodality of $\varepsilon$, it holds that $f_{\sum_{i=0}^\infty c_i\varepsilon_{-i}}$ is symmetric and regularly varying at $+\infty$ with index $-1-\nu$; see Remark \ref{rem:unimodality}. 
	Throughout the proof we use that for sufficiently large $u$ and $(c_i)_{i\in\NN_0}\subset [0,\infty)$ such that $\sum_{i=0}^{\infty}c_i^{\nu}<\infty$ it follows from
	Karamata's theorem and Equation~(15.3.5) in \cite{kulikHeavyTailedTimeSeries2020} that 
	\begin{align}
		  \label{eq:KKS}
      \begin{split}
        &
		f_{\sum_{i=0}^{\infty}c_i \varepsilon_{-i}}(u)
		\ = \
		\frac{
			uf_{\sum_{i=0}^{\infty}c_i \varepsilon_{-i}}(u)
		}{\nu \PP[\sum_{i=0}^{\infty}c_i \varepsilon_{-i}>u]}
		\cdot
		\frac{\nu\PP[\varepsilon>u]}{uf_{\varepsilon}(u)}
		\cdot
		\frac{
			\PP[\sum_{i=0}^{\infty}c_i \varepsilon_{-i}>u]
		}{\PP[\varepsilon > u]}
		f_{\varepsilon}(u)
		\ \sim\
		\frac{
			\PP[\sum_{i=0}^{\infty}c_i \varepsilon_{-i}>u]
		}{\PP[\varepsilon > u]}
		f_{\varepsilon}(u)
		\\ &
		\ \sim \
		  \sum_{i=0}^{\infty}c_i^{\nu}
		f_{\varepsilon}(u)
		  \, \quad \mbox{ as } \quad u\to\infty
      \,.
      \end{split}
	\end{align}
	We first show by mathematical induction that, for all $k\in \NN_0$,
	\begin{align*}
		\frac{1}{f_{X_0}(u)}
		\int_{u\mathfrak{a}_1-y_1}^{\infty}
		f_{X_{0,k},X_{1,k+1}}(u-y_0, x)\,\mathrm{d}x
		\ \to \ 0
	\end{align*}
	and
	\begin{align*}
		\frac{1}{f_{X_0}(u)}
		\int_{u-y_0}^{\infty}
		f_{X_{0,k},X_{1,k+1}}(x, u\mathfrak{a}_1-y_1)\,\mathrm{d}x
		\ \to\
    \mathfrak{a}_1^{-1-\nu}
		\frac{\sum_{i=1}^{k+1}a_i^{\nu}}{\sum_{i=0}^{\infty}a_i^{\nu}}
		\,,
	\end{align*}
	both uniformly in 
  $\mathbf{y}\in N_{\mathbf{y}}$
	as $u\to\infty$, where $N_{\mathbf{y}}$ 
 is an arbitrary compact subset of $\mathbb{R}^2$, which we fix from now on.
 Then, in a second step, we show that we can take the limit $k\to\infty$.
 \\
  \textbf{Analysis for $k\in\NN_0$:}
	The arguments for both terms are mostly the same, so we introduce some notation to unify the proof for both cases.
	Let $\sigma:\{0,1\}\to \{0,1\}$ be a permutation and write
	\begin{align*}
		X_{\sigma(0),k+\sigma(0)}
		\ = \
		\sum_{i=0}^{k+\sigma(0)}a_{i}\varepsilon_{\sigma(0)-i}
		\ = \
		\sum_{i=-\sigma(0)}^{k}a_{\sigma(0)+i}\varepsilon_{-i}
		\,.
	\end{align*}
	Our goal is then to show that, for all $k\in \NN_0$,
	\begin{align}
		\label{eq:claim}
		\begin{split}
			 &
			\frac{1}{f_{\varepsilon}(u)}
			\int_{u\mathfrak{a}_{\sigma(0)}-y_{\sigma(0)}}^{\infty}
			f_{
					X_{\sigma(0),k+\sigma(0)},
					X_{\sigma(1),k+\sigma(1)}
				}(x,u\mathfrak{a}_{\sigma(1)}-y_{\sigma(1)})
			\,\mathrm{d}x
			\\ &
			  \ \to \
			  \begin{cases}
				\mathfrak{a}_1^{-1-\nu}
				\cdot
				\sum_{i=1}^{k+1}a_i^{\nu}
				\,, & \qquad \sigma(0) = 0\,, \\
				0
				\,, & \qquad \sigma(0) = 1\,,
			\end{cases}
		\end{split}
	\end{align}
	uniformly in 
  $\mathbf{y}\in N_{\mathbf{y}}$
	as $u\to\infty$.
	\\
	\textbf{Base case $k=0$:}
	Since $a_{-i}=0$ for $i\in\NN$, and $a_0=1$,
	\begin{align*}
		\begin{bmatrix}
			X_{\sigma(0),k+\sigma(0)} \\
			X_{\sigma(1),k+\sigma(1)}
		\end{bmatrix}
		\ = \
		\begin{bmatrix}
			a_{\sigma(0)-1} &
			a_{\sigma(0)}
			\\
			a_{\sigma(1)-1} &
			a_{\sigma(1)}
		\end{bmatrix}
		\begin{bmatrix}
			\varepsilon_1 \\
			\varepsilon_0
		\end{bmatrix}
		\,,
	\end{align*}
	where
	\begin{align*}
		\left|
		\det
		\begin{bmatrix}
			a_{\sigma(0)-1} &
			a_{\sigma(0)}
			\\
			a_{\sigma(1)-1} &
			a_{\sigma(1)}
		\end{bmatrix}
		\right|
		\ = \
		\left|
		a_{\sigma(0)-1}
		a_{\sigma(1)}
		\ - \
		a_{\sigma(1)-1}
		a_{\sigma(0)}
		\right|
		\ = \
		1
		\,.
	\end{align*}
	Note also that, applying the inverse transformation,
	\begin{align*} 
        \pm
		\begin{bmatrix}
			a_{\sigma(1)}    &
			-a_{\sigma(0)}
			\\
			-a_{\sigma(1)-1} &
			a_{\sigma(0)-1}
		\end{bmatrix}
		\begin{bmatrix}
			x \\
			u\mathfrak{a}_{\sigma(1)}-y_{\sigma(1)}
		\end{bmatrix}
		\ = \ 
        \pm
		\begin{bmatrix}
			a_{\sigma(1)}
			\cdot
			x
			\ -\
			a_{\sigma(0)}
			\cdot
			(u\mathfrak{a}_{\sigma(1)}-y_{\sigma(1)})
			\\
			-
			a_{\sigma(1)-1}
			\cdot
			x
			\ +\
			a_{\sigma(0)-1}
			\cdot
			(u\mathfrak{a}_{\sigma(1)}-y_{\sigma(1)})
		\end{bmatrix}.
	\end{align*}
	Since the determinant is $\pm1$, it holds by the symmetry of $\varepsilon$,
	\begin{align*}
		 &
		f_{
				X_{\sigma(0),k+\sigma(0)},
				X_{\sigma(1),k+\sigma(1)}
			}(x,u\mathfrak{a}_{\sigma(1)}-y_{\sigma(1)})
		\\
		 &
		\ =\
		f_{
				\varepsilon
			}(a_{\sigma(1)}
		\cdot
		x
		\ -\
		a_{\sigma(0)}
		\cdot
		(u\mathfrak{a}_{\sigma(1)}-y_{\sigma(1)})) f_{
				\varepsilon
			}(-a_{\sigma(1)-1}
		\cdot
		x
		\ +\
		a_{\sigma(0)-1}
		\cdot
		(u\mathfrak{a}_{\sigma(1)}-y_{\sigma(1)}))\,.
	\end{align*}
	Also note that $a_{\sigma(1)}\in \{1,a_1\}$, so that
	making the change of variables
	\begin{align*}
		x
		\ \mapsto\
		\frac{
			x\ +\ a_{\sigma(0)}\cdot (u\mathfrak{a}_{\sigma(1)}-y_{\sigma(1)})
		}{a_{ \sigma(1)}}
		\,,
	\end{align*}
	the right-hand side becomes
	\[
		f_{
				\varepsilon
			}(x) f_{
				\varepsilon
			}\left(-
		\dfrac{
			a_{\sigma(1)-1}
		}{a_{\sigma(1)}}
		\cdot
		x
		\ +\
		\left(
		a_{\sigma(0)-1}
		-
		\dfrac{a_{\sigma(1)-1}a_{\sigma(0)}}{a_{\sigma(1)}}
		\right)
		\cdot
		(u\mathfrak{a}_{\sigma(1)}-y_{\sigma(1)})\right).
	\]
		The lower bound in \eqref{eq:claim} for the change of variables in the integral is
		\begin{align*}
		a_{\sigma(1)}
		\left(
		u\mathfrak{a}_{\sigma(0)}-y_{\sigma(0)}
		\right)
		\ -\
		a_{\sigma(0)}
		\left( u\mathfrak{a}_{\sigma(1)}-y_{\sigma(1)} \right)
		\ = \
		u\cdot(a_{\sigma(1)}\mathfrak{a}_{\sigma(0)}-a_{\sigma(0)}\mathfrak{a}_{\sigma(1)})
		\ + \ O_{N_{\mathbf{y}}}(1)
		\,,
	\end{align*}
	where $O_{N_{\mathbf{y}}}(1)$ denotes a quantity that is bounded in $\mathbf{y}\in N_{\mathbf{y}}$.
	Then,
	\begin{align*}
		 &
		\frac{1}{f_{\varepsilon}(u)}
		\int_{u\mathfrak{a}_{\sigma(0)}-y_{\sigma(0)}}^{\infty}
		f_{
				X_{\sigma(0),k+\sigma(0)},
				X_{\sigma(1),k+\sigma(1)}
			}(x,u\mathfrak{a}_{\sigma(1)}-y_{\sigma(1)})
		\,\mathrm{d}x
		\\
		 &
		\ = \
		\frac{1}{a_{\sigma(1)}}
		\frac{1}{f_{\varepsilon}(u)}
		\int_{
			u\cdot(a_{\sigma(1)}\mathfrak{a}_{\sigma(0)}-a_{\sigma(0)}\mathfrak{a}_{\sigma(1)})
			\ + \ O_{N_{\mathbf{y}}}(1)
		}^{\infty}
		f_{
				\varepsilon
			}
		(x)
		\\ &
		  \qquad
		  \times f_{
				  \varepsilon
				  }
		\left(
		  -
		  \dfrac{
			  a_{\sigma(1)-1}
		}{a_{\sigma(1)}}
		\cdot
		  x
		  \ +\
		  \left(
		  a_{\sigma(0)-1}
		  -
		  \dfrac{a_{\sigma(1)-1}a_{\sigma(0)}}{a_{\sigma(1)}}
		  \right)
		\cdot
		  u\mathfrak{a}_{\sigma(1)}
		\ + \ O_{N_{\mathbf{y}}}(1)
		  \right)
		  \,\mathrm{d}x
		  \,.
	\end{align*}
	Since $\mathfrak{a}_1\ge 1$, we have $a_0\mathfrak{a}_1 \ge a_0 > a_1$, the lower bound of the integration range satisfies
	\begin{align}
    \label{eq:conv_post}
		u(a_{\sigma(1)}\mathfrak{a}_{\sigma(0)}-a_{\sigma(0)}\mathfrak{a}_{\sigma(1)})
		\ + \ O_{N_{\mathbf{y}}}(1)
		 & \ = \
		\begin{cases}
			u(a_1-\mathfrak{a}_1)
			\ + \ O_{N_{\mathbf{y}}}(1)
			\,, & \qquad \sigma(0) = 0\,, \\
			u(\mathfrak{a}_1-a_1)
			\ + \ O_{N_{\mathbf{y}}}(1)
			\,, & \qquad \sigma(0) = 1
			\,.
		\end{cases}
	\end{align}
	Thus, the argument of the second density term satisfies
	on the integration range in \eqref{eq:conv_post} that
	\begin{align*}
		-
		\frac{
			a_{\sigma(1)-1}
		}{a_{\sigma(1)}}
		\cdot
		x
		\ +\
		\left(
		a_{\sigma(0)-1}
		-
		\frac{a_{\sigma(1)-1}a_{\sigma(0)}}{a_{\sigma(1)}}
		\right)
		\cdot
		u\mathfrak{a}_{\sigma(1)}
		\ + \ O_{N_{\mathbf{y}}}(1)
		 & \ = \
		\begin{cases}
			-\dfrac{x+u\mathfrak{a}_{1}}{a_1}
			\ + \ O_{N_{\mathbf{y}}}(1)
			\,, & \qquad \sigma(0) = 0\,, \\
			u
			\ + \ O_{N_{\mathbf{y}}}(1)
			\,, & \qquad \sigma(0) = 1\,,
		\end{cases}
		\\ &
		  \begin{cases}
			\ \le \
			-u
			\ + \ O_{N_{\mathbf{y}}}(1)
			\,, & \qquad \sigma(0) = 0\,, \\
			\ \ge \
			u
			\ + \ O_{N_{\mathbf{y}}}(1)
			\,, & \qquad \sigma(0) = 1\,.
		\end{cases}
	\end{align*}
    For $\sigma(0)=0$, 
    note that, by monotonicity of $f_{\varepsilon}$ (see Remark~\ref{rem:unimodality}) and local uniformity of regular variation,
    \begin{align*}
     \dfrac{1}{f_{\varepsilon}(u)} f_{
					\varepsilon
				}
			\left(
			-\dfrac{x+u\mathfrak{a}_{1}}{a_1}
			\ + \ O_{N_{\mathbf{y}}}(1)
      \right) \ \leq\  
      \dfrac{
      f_{\varepsilon}(-u+O_{N_{\mathbf{y}}}(1))
      }
      {f_{\varepsilon}(u)}  
      \ 
      <
      \ 
      \infty  
    \end{align*}
      uniformly as $u\to\infty$. Thus, it follows from dominated convergence and the uniform convergence of the lower bound $u(a_1-\mathfrak{a}_1)
			\ + \ O_{N_{\mathbf{y}}}(1)$ of the integration range in Equation~\eqref{eq:conv_post} to $-\infty$ that
	\begin{align*}
		 &
		\frac{1}{a_{\sigma(1)}}
		\frac{1}{f_{\varepsilon}(u)}
		\int_{
			u\cdot(a_{\sigma(1)}\mathfrak{a}_{\sigma(0)}-a_{\sigma(0)}\mathfrak{a}_{\sigma(1)})
			\ + \ O_{N_{\mathbf{y}}}(1)
		}^{\infty}
		f_{
				\varepsilon
			}
		\left(
		x
		\right)
		\\ &
		  \qquad
		  \times f_{
				  \varepsilon
				  }
		\left(
		  -
		  \dfrac{
			  a_{\sigma(1)-1}
		}{a_{\sigma(1)}}
		\cdot
		  x
		  \ +\
		  \left(
		  a_{\sigma(0)-1}
		  -
		  \dfrac{a_{\sigma(1)-1}a_{\sigma(0)}}{a_{\sigma(1)}}
		  \right)
		\cdot
		  u\mathfrak{a}_{\sigma(1)}
		\ + \ O_{N_{\mathbf{y}}}(1)
		  \right)
		  \,\mathrm{d}x
		\\ &
		  \ = \
		  \dfrac{1}{a_{1}}
			\displaystyle\int_{
				u(a_1-\mathfrak{a}_{1}) \ + \ O_{N_{\mathbf{y}}}(1)
			}^{\infty}
			f_{
					\varepsilon
				}
			\left(
			x
			\right) \cdot \dfrac{1}{f_{\varepsilon}(u)} f_{
					\varepsilon
				}
			\left(
			-\dfrac{x+u\mathfrak{a}_{1}}{a_1}
			\ + \ O_{N_{\mathbf{y}}}(1)
			\right)
			\,\mathrm{d}x
		\\ &
		  \ \to \
		  \dfrac{1}{a_1}
			(a_1/\mathfrak{a}_1)^{1+\nu}
			=
			a_1^{\nu}
			\mathfrak{a}_1^{-1-\nu}
	\end{align*}
    	uniformly in  $\mathbf{y}\in N_{\mathbf{y}}$ as $u\to\infty$ due to local uniformity of regular variation.
        
    For $\sigma(0)=1$, we use uniform convergence of the lower bound $u(\mathfrak{a}_1-a_1)\ + \ O_{N_{\mathbf{y}}}(1)$ of the integration range in Equation~\eqref{eq:conv_post} to $\infty$ and local uniformity of regular variation to see that
	\begin{align*}
		 &
		\frac{1}{a_{\sigma(1)}}
		\frac{1}{f_{\varepsilon}(u)}
		\int_{
			u\cdot(a_{\sigma(1)}\mathfrak{a}_{\sigma(0)}-a_{\sigma(0)}\mathfrak{a}_{\sigma(1)})
			\ + \ O_{N_{\mathbf{y}}}(1)
		}^{\infty}
		f_{
				\varepsilon
			}
		\left(
		x
		\right)
		\\ &
		  \qquad
		  \times f_{
				  \varepsilon
				  }
		\left(
		  -
		  \dfrac{
			  a_{\sigma(1)-1}
		}{a_{\sigma(1)}}
		\cdot
		  x
		  \ +\
		  \left(
		  a_{\sigma(0)-1}
		  -
		  \dfrac{a_{\sigma(1)-1}a_{\sigma(0)}}{a_{\sigma(1)}}
		  \right)
		\cdot
		  u\mathfrak{a}_{\sigma(1)}
		\ + \ O_{N_{\mathbf{y}}}(1)
		  \right)
		  \,\mathrm{d}x
		\\ &
		  \ = \
		\dfrac{f_{\varepsilon}(u+O_{N_{\mathbf{y}}}(1))}{f_{\varepsilon}(u)}
			\displaystyle\int_{
				u(\mathfrak{a}_1-a_1) \ + \ O_{N_{\mathbf{y}}}(1)
			}^{\infty}
			f_{
					\varepsilon
				}
			\left(
			x
			\right)\,\mathrm{d}x
		  \ \to \
			0\,,
	\end{align*}
	uniformly in 
  $\mathbf{y}\in N_{\mathbf{y}}$
	as $u\to\infty$.
	This proves \eqref{eq:claim} for $k=0$.
	\\
	\textbf{Induction statement:}
	Assume that, for some $k\in \NN_0$, it holds \eqref{eq:claim}.
	\\
	\textbf{Induction step:}
	Throughout, we use the disintegration formula
	\begin{align*}
		 &
		\int_{u\mathfrak{a}_{\sigma(0)}-y_{\sigma(0)}}^{\infty}
		f_{X_{\sigma(0),k+1+\sigma(0)},X_{\sigma(1),k+1+\sigma(1)}}(x, u\mathfrak{a}_{\sigma(1)}-y_{\sigma(1)})\,\mathrm{d}x
		\\ &
		  \ = \
		  \int_{\RR}
		\left(
		  \int_{u\mathfrak{a}_{\sigma(0)}
			-  y_{\sigma(0)}
		}
		  ^{\infty}
		f_{X_{\sigma(0),k+\sigma(0)},X_{\sigma(1),k+\sigma(1)}}(x-a_{k+1+\sigma(0)}z, u\mathfrak{a}_{\sigma(1)} - a_{k+1+\sigma(1)}z
		  -  y_{\sigma(1)}
		)\,\mathrm{d}x
		  \right)
		  f_{\varepsilon}(z)
		  \,\mathrm{d}z
		\\ &
		  \ = \
		  \int_{\RR}
		\left(
		  \int_{
			  u\mathfrak{a}_{\sigma(0)}
			- a_{k+1+\sigma(0)}z
			  -  y_{\sigma(0)}
		}
		  ^{\infty}
		f_{X_{\sigma(0),k+\sigma(0)},X_{\sigma(1),k+\sigma(1)}}(x, u\mathfrak{a}_{\sigma(1)} - a_{k+1+\sigma(1)}z \ -  y_{\sigma(1)}
		)\,\mathrm{d}x
		  \right)
		  f_{\varepsilon}(z)
		  \,\mathrm{d}z
		  \,,
	\end{align*}
	and fix $M>0$.
	We split the analysis of the integral into several cases for $z$, namely, we write
	\begin{align}
		\nonumber
		 &
		\frac{1}{f_{\varepsilon}(u)}\int_{u\mathfrak{a}_{\sigma(0)}-y_{\sigma(0)}}^{\infty}
		f_{X_{\sigma(0),k+1+\sigma(0)},X_{\sigma(1),k+1+\sigma(1)}}(x, u\mathfrak{a}_{\sigma(1)}-y_{\sigma(1)})\,\mathrm{d}x
		\label{eq:disintegration}
		\\ &
		= \
      \sum_{j=1}^5 \int_{D_j}
		\left(
		  \frac{1}{f_{\varepsilon}(u)} \int_{
			  u\mathfrak{a}_{\sigma(0)}
			- a_{k+1+\sigma(0)}z
			  -  y_{\sigma(0)}
		}
		  ^{\infty}
		f_{X_{\sigma(0),k+\sigma(0)},X_{\sigma(1),k+\sigma(1)}}(x, u\mathfrak{a}_{\sigma(1)} - a_{k+1+\sigma(1)}z \ -  y_{\sigma(1)}
		)\,\mathrm{d}x
		  \right)
		  f_{\varepsilon}(z)
		  \,\mathrm{d}z
		  \,,
	\end{align}
  with
  \begin{align*}
    D_1
    &
    \ = \ \{|z|
      \le M
    \}
    \,,
    \\
    D_2
    &
    \ = \ \{
			|u\mathfrak{a}_{\sigma(1)}-a_{k+1+\sigma(1)}z-y_{\sigma(1)}|\le M
      \qquad\text{and}\qquad
      |z|>M
    \}
    \,,
    \\
    D_3
    &
    \ = \ 
    \left\{
    |z|\in (M,\delta u]
    \qquad\text{and}\qquad
		|u\mathfrak{a}_{\sigma(1)}-a_{k+1+\sigma(1)}z
			-y_{\sigma(1)}|>M
    \right\}
    \,,
    \\
    D_4
    &
    \ = \ 
    \left\{
    |z|>\delta u
    \qquad\text{and}\qquad
		|u\mathfrak{a}_{\sigma(1)}-a_{k+1+\sigma(1)}z
      -y_{\sigma(1)}|\in (M,\delta u]
    \right\}
    \,,
    \\
    D_5
    &
    \ = \ 
    \left\{
    |z|
    \land
		|u\mathfrak{a}_{\sigma(1)}-a_{k+1+\sigma(1)}z
      -y_{\sigma(1)}| > \delta u
    \right\}\,,
  \end{align*}
  for fixed $\delta \in (0,1/4)$,
  where the sets form a disjoint partition of $\RR$.
  \\
  \textbf{Analysis of $D_1$:} 
	If $|z|$ is bounded by $M$, then
	it follows from the induction statement that the integral with domain $D_1$ in~\eqref{eq:disintegration} converges to
	\begin{align*}
		\PP[|\varepsilon| \le M]
		\cdot
		\begin{cases}
			\mathfrak{a}_1^{-1-\nu}
			\sum_{i=1}^{k+1}a_i^{\nu}
			\,, & \qquad \sigma(0) = 0\,, \\
			0
			\,, & \qquad \sigma(0) = 1\,,\end{cases}
	\end{align*}
	uniformly in 
  $\mathbf{y}\in N_{\mathbf{y}}$
	as $u\to\infty$.
	\\
  \textbf{Analysis of $D_2$:} 
	Note that $a_{k+1+\sigma(1)}\in \{a_{k+1},a_{k+2}\}$ and $a_{k+1}>a_{k+2}>0$.
	We apply the change of variables
	\begin{align*}
		z\ \mapsto\
		\frac{u\mathfrak{a}_{\sigma(1)}-y_{\sigma(1)}-z}{a_{k+1+\sigma(1)}}
	\end{align*}
  so that $D_2$ is transformed to 
  \begin{align*}
    |z|\ \le\  M
    \qquad\text{and}\qquad
    \left|
		\frac{u\mathfrak{a}_{\sigma(1)}-y_{\sigma(1)}-z}{a_{k+1+\sigma(1)}}
    \right|\ > \ M
    \,.
  \end{align*}
  After the change of variables, the lower bound of the integration range in $x$ becomes 
	\begin{align*}
		u
		\left(
		\mathfrak{a}_{\sigma(0)}
		-
		\mathfrak{a}_{\sigma(1)}
		\frac{a_{k+1+\sigma(0)}}{a_{k+1+\sigma(1)}}
		\right)
		\ + \
		z
		\frac{a_{k+1+\sigma(0)}}{a_{k+1+\sigma(1)}}
		\ + \
		O_{N_{\mathbf{y}}}(1)
		\,.
	\end{align*}
	The integral then becomes
	\begin{align*}
		 &
		\frac{1}{a_{k+1+\sigma(1)}}
		\frac{1}{f_{\varepsilon}(u)}
    \int_{-M}^M
		\left(
		\int_{
			u
			\left(
			\mathfrak{a}_{\sigma(0)}
			-
			\mathfrak{a}_{\sigma(1)}
			\frac{a_{k+1+\sigma(0)}}{a_{k+1+\sigma(1)}}
			\right)
			\ + \
			z
			\frac{a_{k+1+\sigma(0)}}{a_{k+1+\sigma(1)}}
			\ + \
			O_{N_{\mathbf{y}}}(1)
		}^{\infty}
		f_{X_{\sigma(0),k+\sigma(0)},X_{\sigma(1),k+\sigma(1)}}(x,
		z
		)\,\mathrm{d}x
		\right)
		\\ & \qquad
		  \qquad
		  \qquad
		  \qquad
		  \qquad
		  f_{\varepsilon}
		\left(
		  \frac{u\mathfrak{a}_{\sigma(1)}+O_{N_{\mathbf{y}}}(1)-z}{a_{k+1+\sigma(1)}}
		  \right)
      \ind
      \left\{
    \left|
		\frac{u\mathfrak{a}_{\sigma(1)}-y_{\sigma(1)}-z}{a_{k+1+\sigma(1)}}
    \right|\ > \ M
      \right\}
		\,\mathrm{d}z
		  \,.
	\end{align*}
	Since $a_{k+1}>a_{k+2}$ and $\mathfrak{a}_1\ge 1$,
	\begin{align*}
		 &
		\mathfrak{a}_{\sigma(0)}
		-
		\mathfrak{a}_{\sigma(1)}
		\frac{a_{k+1+\sigma(0)}}{a_{k+1+\sigma(1)}}
		\ = \
		\begin{cases}
			1
			-
			\mathfrak{a}_{1}
			\frac{a_{k+1}}{a_{k+2}} < 0
			\,, & \qquad \sigma(0) = 0\,, \\
			\mathfrak{a}_{1}
			-
			\frac{a_{k+2}}{a_{k+1}} > 0
			\,, & \qquad \sigma(0) = 1\,,\end{cases}
	\end{align*}
	so that
	\begin{align*}
		u
		\left(
		\mathfrak{a}_{\sigma(0)}
		-
		\mathfrak{a}_{\sigma(1)}
		\frac{a_{k+1+\sigma(0)}}{a_{k+1+\sigma(1)}}
		\right)
		\ + \
		z
		\frac{a_{k+1+\sigma(0)}}{a_{k+1+\sigma(1)}}
		\ + \
		O_{N_{\mathbf{y}}}(1)
		\ \to \
		\begin{cases}
			-\infty
			\,, & \qquad \sigma(0) = 0\,, \\
			+\infty
			\,, & \qquad \sigma(0) = 1\,,\end{cases}
	\end{align*}
	uniformly 
  in
  $\mathbf{y}\in N_{\mathbf{y}}$ and 
  $|z
			|\le M$
	as $u\to\infty$.
  In the case of $\sigma(0)=0$, we have
  \begin{align}
    \label{eq:inder}
    \frac{1}{f_\varepsilon(u)}
		  f_{\varepsilon}
		\left(
		  \frac{u\mathfrak{a}_{\sigma(1)}+O_{N_{\mathbf{y}}}(1)-z}{a_{k+1+\sigma(1)}}
		  \right)
      \ind
      \left\{
    \left|
		\frac{u\mathfrak{a}_{\sigma(1)}-y_{\sigma(1)}-z}{a_{k+1+\sigma(1)}}
    \right|\ > \ M
      \right\}
      \ < \  
      \infty
  \end{align}
  uniformly as $u\to\infty$. Since
  \begin{align*}
    &
	\int_{
			u
			\left(
			\mathfrak{a}_{\sigma(0)}
			-
			\mathfrak{a}_{\sigma(1)}
			\frac{a_{k+1+\sigma(0)}}{a_{k+1+\sigma(1)}}
			\right)
			\ + \
			z
			\frac{a_{k+1+\sigma(0)}}{a_{k+1+\sigma(1)}}
			\ + \
			O_{N_{\mathbf{y}}}(1)
		}^{\infty}
		f_{X_{\sigma(0),k+\sigma(0)},X_{\sigma(1),k+\sigma(1)}}(x,
		z
		)\,\mathrm{d}x
    \\&
    \ \to \ 
	\int_\RR
		f_{X_{\sigma(0),k+\sigma(0)},X_{\sigma(1),k+\sigma(1)}}(x,
    z)
    \,\mathrm{d}x
    \ = \ 
    f_{X_{\sigma(1),k+\sigma(1)}}
    (z)\,,
  \end{align*}
  it follows from dominated convergence and the pointwise convergence of~\eqref{eq:inder} to $(\mathfrak{a}_1/a_{k+2})^{-1-\nu}$
  that the integral with domain $D_2$ in~\eqref{eq:disintegration} converges to
  \begin{align*}
\PP[
			|X_{\sigma(1),k+\sigma(1)}+o_{N_{\mathbf{y}}}(1)|
			\le M
		]
		\cdot
			\mathfrak{a}_1^{-1-\nu}
			a_{k+2}^{\nu}
      \qquad\text{as $u\to\infty$.}\qquad
  \end{align*}
  If $\sigma(0)=1$, then
  \begin{align*}
	\int_{
			u
			\left(
			\mathfrak{a}_{\sigma(0)}
			-
			\mathfrak{a}_{\sigma(1)}
			\frac{a_{k+1+\sigma(0)}}{a_{k+1+\sigma(1)}}
			\right)
			\ + \
			z
			\frac{a_{k+1+\sigma(0)}}{a_{k+1+\sigma(1)}}
			\ + \
			O_{N_{\mathbf{y}}}(1)
		}^{\infty}
		f_{X_{\sigma(0),k+\sigma(0)},X_{\sigma(1),k+\sigma(1)}}(x,
		z
		)\,\mathrm{d}x
\ \to \ 0
\,,
  \end{align*}
  so that the integral with domain $D_2$ in~\eqref{eq:disintegration} converges to $0$ as $u\to\infty$. 
	\\
  \textbf{Analysis of $D_3$:}
	On this domain one has, for $u$ large enough, regardless of $\mathbf{y}\in N_{\mathbf{y}}$,
	\begin{align*}
		u\mathfrak{a}_{\sigma(1)}-a_{k+1+\sigma(1)}z
		-y_{\sigma(1)}
		\ > \ u\mathfrak{a}_{\sigma(1)}/2
		\,.
	\end{align*}
	Therefore, the integral with domain $D_3$ in~\eqref{eq:disintegration} is bounded above by
	\begin{align*}
		 &
		\frac{1}{f_{\varepsilon}(u)}
		\sup_{x>u\mathfrak{a}_{\sigma(1)}/2}
		f_{X_{\sigma(1),k+\sigma(1)}}(x)
		\cdot
		\PP[|\varepsilon|> M]
		\ \lesssim\
		\PP[|\varepsilon|> M]
		\,,
	\end{align*}
    using~\eqref{eq:KKS} and the regular variation property of $f_{\varepsilon}$. Here, the multiplicative constant in $\lesssim$ depends only on $\mathfrak{a}_0, \mathfrak{a}_1$ and is independent of $k$ and $u$.  
	\\
	\textbf{
  Analysis of $D_4$:}
  One has $|a_{k+1+\sigma(1)}z - (u\mathfrak{a}_{\sigma(1)} - y_{\sigma(1)})| \leq \delta u$, which, for sufficiently large $u$, results in
			\begin{align*}
			 &
			z
			\ \ge\
			\frac{\mathfrak{a}_{\sigma(1)}-2\delta}{a_{k+1+\sigma(1)}}
			u
			\ \ge\
			\frac{1}{2 a_{k+1+\sigma(1)}}
			u
		\end{align*}
		for all $\mathbf{y}\in N_{\mathbf{y}}$, so that the integral with domain $D_4$ in~\eqref{eq:disintegration} is bounded above by
		\begin{align*}
			\frac{1}{a_{k+1+\sigma(1)}} \frac{1}{f_{\varepsilon}(u)}
			\sup_{|z|\ge u/(2 a_{k+1+\sigma(1)})}
			f_{\varepsilon}(|z|)
			\cdot
			\PP[|X_{\sigma(1),k+\sigma(1)}|>M]
			\ \lesssim\
			\PP[|X_{\sigma(1),k+\sigma(1)}|>M]\,,
		\end{align*}
    where the multiplicative constant in $\lesssim$ depends only on $(a_i)$ and is independent of $k$ and $u$. 
	\\
  \textbf{Analysis of $D_5$:}
		The integral with domain $D_5$ in~\eqref{eq:disintegration} is bounded above by
		\begin{align*}
			 &
			\frac{1}{f_{\varepsilon}(u)}
			\sup_{|x|>\delta u} f_{X_{\sigma(1),k+\sigma(1)}}(x)
			\PP[|\varepsilon|>\delta u]
			\ \lesssim\
			\PP[|\varepsilon|>\delta  u]
			\ \to \ 0
			\,,
		\end{align*}
		using~\eqref{eq:KKS} and the regular variation property of $f_{\varepsilon}$, uniformly in 
    $\mathbf{y}\in N_{\mathbf{y}}$
	as $u\to\infty$.
  \\
  \textbf{Conclusion of analysis for $k\in\NN_0$:}
	Taking $M\to\infty$ yields \eqref{eq:claim} for $k+1$, and by mathematical induction, the claim holds for all $k\in\NN_0$.
	\\
	\textbf{Analysis for $k\to\infty$:}
  Throughout, we use the independence of 
  \begin{align*}
  (X_{\sigma(0),k+\sigma(0)}, X_{\sigma(1),k+\sigma(1)})
  \qquad\text{and}\qquad
 (X_{\sigma(0)} - X_{\sigma(0),k+\sigma(0)},X_{\sigma(1)} - X_{\sigma(1),k+\sigma(1)})
  \end{align*}
 to obtain the disintegration formula
	\begin{align*}
		 &
		\frac{1}{f_{\varepsilon}(u)}
		\int_{u\mathfrak{a}_{\sigma(0)}}^{\infty}
		f_{X_{\sigma(0)},X_{\sigma(1)}}(x,u\mathfrak{a}_{\sigma(1)})
		\,\mathrm{d}x
		\\ &
		  \ = \
		  \frac{1}{f_{\varepsilon}(u)}
		\int_{\RR^2}
		\left(
		  \int_{u\mathfrak{a}_{\sigma(0)}}^{\infty}
		f_{X_{\sigma(0),k+\sigma(0)},X_{\sigma(1),k+\sigma(1)}}(x-r_0,u\mathfrak{a}_{\sigma(1)}-r_1)
		  \,\mathrm{d}x
		  \right)
		\\ & \qquad
		  \qquad
		  \qquad
		  \qquad
		  f_{X_{\sigma(0)}-X_{\sigma(0),k+\sigma(0)}, X_{\sigma(1)}-X_{\sigma(1),k+\sigma(1)}}(r_0,r_1)
		  \,\mathrm{d}(r_0,r_1)
		\,,
	\end{align*}
  valid for any $k\in \NN_0$. Again, we write
		\begin{align}
			\label{eq:full_disintegration}
			\begin{split}
				 &
				\frac{1}{f_{\varepsilon}(u)}
				\int_{u\mathfrak{a}_{\sigma(0)}}^{\infty}
				f_{X_{\sigma(0)},X_{\sigma(1)}}(x,u\mathfrak{a}_{\sigma(1)})
				\,\mathrm{d}x
				\\ &
				  \ = \
          \sum_{j=1}^4 \int_{D'_j}
				\left( \frac{1}{f_{\varepsilon}(u)}
				\int_{u\mathfrak{a}_{\sigma(0)}-r_0}^{\infty}
				f_{X_{\sigma(0),k+\sigma(0)},X_{\sigma(1),k+\sigma(1)}}(x,u\mathfrak{a}_{\sigma(1)}-r_1)
				  \,\mathrm{d}x
				  \right)
				\\ & \qquad
				  \qquad
				  \qquad
				  \qquad
				  f_{X_{\sigma(0)}-X_{\sigma(0),k+\sigma(0)}, X_{\sigma(1)}-X_{\sigma(1),k+\sigma(1)}}(r_0,r_1)
				  \,\mathrm{d}(r_0,r_1)
				\,,
			\end{split}
		\end{align}
    with
    \begin{align*}
      D'_1
      &
      \ = \ \{|r_0|\lor |r_1|\le 1\}
      \,,
      \\
      D'_2 &\ = \ \{|r_0|\lor |r_1|\in (1,\delta u]\}
      \,,
      \\
      D'_3 &\ = \ \{|r_1|> \delta u\}
      \,,
      \\
      D'_4 &\ = \ \{|r_0|> \delta u
      \quad\text{and}\quad
      |r_1| \le \delta u
      \}
      \,,
    \end{align*}
    for fixed $\delta \in (0,1/4)$, where the sets form a disjoint partition of $\RR^2$.
	\\
  \textbf{Analysis of $D'_1$:}
	By \eqref{eq:claim}, it holds that the integral with domain $D_1'$ in~\eqref{eq:full_disintegration} converges to
	\begin{align*}
		\PP[
			|X_{0}-X_{0,k}|\lor |X_{1}-X_{1,k+1}|\le 1
		]
		\cdot
		\begin{cases}
      \mathfrak{a}_1^{-1-\nu}
			\sum_{i=1}^{k+1}a_i^{\nu}
			\,, & \qquad \sigma(0) = 0\,, \\
			0
			\,, & \qquad \sigma(0) = 1\,,
		\end{cases}
	\end{align*}
  as $u\to\infty$.
	\\
  \textbf{Analysis of $D'_2$:}
	It holds $u\mathfrak{a}_{\sigma(1)}-r_1\in
		[(1-\delta)u\mathfrak{a}_{\sigma(1)}, (1+\delta)u\mathfrak{a}_{\sigma(1)}]$, so that the integral with domain $D_2'$ in~\eqref{eq:full_disintegration} is bounded above by
	\begin{align*}
		 &
		\frac{1
		}{f_{\varepsilon}(u)}
		\sup_{z\in [(1-\delta)u\mathfrak{a}_{\sigma(1)},(1+\delta)u\mathfrak{a}_{\sigma(1)}]}
		f_{X_{\sigma(1),k+\sigma(1)}}(z)
		\cdot
		\left( \PP[|X_{0}-X_{0,k}|>1] + \PP[|X_{1}-X_{1,k+1}|>1] \right)
		\\ &
		  \ \lesssim\
		  \PP[|X_{0}-X_{0,k}|>1] + \PP[|X_{1}-X_{1,k+1}|>1]
		\,,
	\end{align*}
  using~\eqref{eq:KKS} and the regular variation property of $f_{\varepsilon}$. Here the multiplicative constant in $\lesssim$ depends only on $\delta$ and $\mathfrak{a}_0, \mathfrak{a}_1$ and is independent of $k$ and $u$. 
	\\
  \textbf{Analysis of $D'_3$:}
  It holds
  \begin{align*}
    &
    \int_{|r_1|>\delta u}
    \Bigg(
    \int_{\RR}
				\left( \frac{1}{f_{\varepsilon}(u)}
				\int_{u\mathfrak{a}_{\sigma(0)}-r_0}^{\infty}
				f_{X_{\sigma(0),k+\sigma(0)},X_{\sigma(1),k+\sigma(1)}}(x,u\mathfrak{a}_{\sigma(1)}-r_1)
				  \,\mathrm{d}x
				  \right)
				\\ & \qquad
				  \qquad
				  \qquad
				  \qquad
				  f_{X_{\sigma(0)}-X_{\sigma(0),k+\sigma(0)}, X_{\sigma(1)}-X_{\sigma(1),k+\sigma(1)}}(r_0,r_1)
          \,\mathrm{d}r_0
          \Bigg)
          \,\mathrm{d}r_1
          \\&
          \ \le \ 
    \int_{|r_1|>\delta u}
        \frac{1}{f_{\varepsilon}(u)}
				\left( 
    \int_{\RR}
				f_{X_{\sigma(0),k+\sigma(0)},X_{\sigma(1),k+\sigma(1)}}(x,u\mathfrak{a}_{\sigma(1)}-r_1)
				  \,\mathrm{d}x
				  \right)
				\\ & \qquad
				  \qquad
				  \qquad
				  \qquad
        \left(
				\int_{\RR}
				  f_{X_{\sigma(0)}-X_{\sigma(0),k+\sigma(0)}, X_{\sigma(1)}-X_{\sigma(1),k+\sigma(1)}}(r_0,r_1)
          \,\mathrm{d}r_0
        \right)
				  \,\mathrm{d}r_1
          \\&
          \ = \ 
    \int_{|r_1|>\delta u}
				f_{X_{\sigma(1),k+\sigma(1)}}(u\mathfrak{a}_{\sigma(1)}-r_1)
        \frac{
          f_{X_{\sigma(1)}-X_{\sigma(1),k+\sigma(1)}}(r_1)
        }{f_{\varepsilon}(u)}
				  \,\mathrm{d}r_1
          \,.
  \end{align*}
  By~\eqref{eq:KKS},
  this is bounded above by
	\begin{align*}
		\frac{1}{f_{\varepsilon}(u)}
		\sup_{|r_1|> \delta u}
		f_{X_{\sigma(1)}-X_{\sigma(1),k+\sigma(1)}}(|r_1|)
		\ \lesssim \
		\sum_{i=k+1}^{\infty}
		a_i^{\nu}
		\,,
	\end{align*}
  where the multiplicative constant in $\lesssim$ depends only on $\delta$ and is independent of $k$ and $u$.
	\\
  \textbf{Analysis of $D'_4$:}
	It holds that the integral with domain $D_4'$ in~\eqref{eq:full_disintegration} is bounded above by
	\begin{align*}
		\frac{1}{f_{\varepsilon}(u)}
    \sup_{z\in [(1-\delta)u\mathfrak{a}_{\sigma(1)},(1+\delta)u\mathfrak{a}_{\sigma(1)}]}
			f_{X_{\sigma(1),k+\sigma(1)}}(z)
		\PP[|X_{\sigma(0)}-X_{\sigma(0),k+\sigma(0)}|>\delta u]
    \ \to \ 0
	\end{align*}
  as $u\to\infty$.
  \\
  \textbf{Conclusion of the analysis for $k\to\infty$:}
  For each $k\in \NN_0$,
		\begin{align*}
			 & \PP[
				|X_{0}-X_{0,k}|\lor |X_{1}-X_{1,k+1}|\le 1
			]
			\cdot
			\begin{cases}
        \mathfrak{a}_1^{-1-\nu}
				\sum_{i=1}^{k+1}a_i^{\nu}
				\,, & \qquad \sigma(0) = 0\,, \\
				0
				\,, & \qquad \sigma(0) = 1\,,
			\end{cases}                                                \\
			 & \ \le \ \liminf_{u\to\infty} \frac{1}{f_{\varepsilon}(u)}
			\int_{u\mathfrak{a}_{\sigma(0)}}^{\infty}
			f_{X_{\sigma(0)},X_{\sigma(1)}}(x,u\mathfrak{a}_{\sigma(1)})
			\,\mathrm{d}x 
      \ \le \ 
       \limsup_{u\to\infty} \frac{1}{f_{\varepsilon}(u)}
			\int_{u\mathfrak{a}_{\sigma(0)}}^{\infty}
			f_{X_{\sigma(0)},X_{\sigma(1)}}(x,u\mathfrak{a}_{\sigma(1)})
			\,\mathrm{d}x                                                                \\
			 & \ \le\ 
			\PP[
				|X_{0}-X_{0,k}|\lor |X_{1}-X_{1,k+1}|\le 1
			]
			\cdot
			\begin{cases}
        \mathfrak{a}_1^{-1-\nu}
				\sum_{i=1}^{k+1}a_i^{\nu}
				\,, & \qquad \sigma(0) = 0\,, \\
				0
				\,, & \qquad \sigma(0) = 1\,,
			\end{cases}                                                \\
			 & \qquad +\  
       C
             \left( 
       \PP[|X_{0}-X_{0,k}|>1] +\PP[|X_{1}-X_{1,k+1}|>1] + \sum_{i=k+1}^{\infty}
			a_i^{\nu} 
            \right)
      \,,
		\end{align*}
    where $C>0$ is a constant that is independent of $k$. 
    Letting $k\to\infty$ yields
		\[
			\PP[
				|X_{0}-X_{0,k}|\lor |X_{1}-X_{1,k+1}|\le 1
			] \to 1, \quad
       \PP[|X_{0}-X_{0,k}|>1] +\PP[|X_{1}-X_{1,k+1}|>1] + \sum_{i=k+1}^{\infty}
			a_i^{\nu} 
      \ \to \ 0
			\,,
		\]
    which proves the result.
\end{proof}

\begin{proof}[\textbf{Proof of Theorem~\ref{thm:comp:iid}:}]
	Note that $k=k(n)$ satisfies $k\to \infty$, $k/n\to 0$ and $\mathfrak{A}(n/k)=O(1/\sqrt{k})$.
	Let $u'_n=q_{X}(1-\frac{k}{n})$. Then, from the regular variation of $X$, $u_n\sim u'_n$, and
	\begin{align*}
    \frac{u'_n}{u_n} - 1 & = \left( \frac{q_X(1-\lfloor n\cdot \PP[X>u_n]\rfloor/n)}{q_X(1-\PP[X>u_n])} - \left( \frac{n\cdot \PP[X>u_n]}{\lfloor n\cdot \PP[X>u_n]\rfloor} \right)^{1/\nu} \right) \\
		                     & + \left( \left( \frac{n\cdot \PP[X>u_n]}{\lfloor n\cdot \PP[X>u_n]\rfloor} \right)^{1/\nu} - 1 \right)                                                                   \\
		                     & = o(\mathfrak{A}(1/\mathbb{P}[X>u_n])) + O(1/\lfloor n\cdot \PP[X>u_n]\rfloor) = o(1/\sqrt{k})
	\end{align*}
	by the uniform inequality in Theorem~2.3.9 on p.48 in~\cite{dehaanExtremeValueTheory2006}. Thus, likewise,
	\[
		\frac{\PP[X>\mathfrak{a}\cdot u'_n] - \PP[X>\mathfrak{a}\cdot u_n]}{\PP[X > u_n]} = o(1/\sqrt{k}).
	\]
	We therefore focus, without loss of generality, on the case when $u_n=q_{X}(1-\frac{k}{n})$.
	Then 
	\begin{align*}
		 &
		\left\{
				\sqrt{k}
		\left(
		\frac{1}{n}
		\sum_{i=1}^n
		\left(
			\frac{\ind\{X_i>\mathfrak{a}\cdot X_{n-k:n}\}}{\PP[X>u_n]}
			\ - \
			\frac{\PP[X>\mathfrak{a}\cdot u_n]}{\PP[X > u_n]}
			\right)
		\right)
		\ \le \ t
		\right\}
		\\ &
		  \ = \
		  \left\{
		  \widehat{\overline{F}}_n
		  \left(
		  \mathfrak{a}\cdot X_{n-k:n}
		  \right)
		\ - \
		  \PP
		  \left[
			  X >
			  \mathfrak{a}
			\cdot
			  q_{X}
			\left(
			  1
			  -
			  \frac{k}{n}
			  \right)
			\right]
		\ \le \
		  \frac{t}{\sqrt{k}}
		\PP
		  \left[
			  X > q_{X}\left(
			  1 - \frac{k}{n}
			  \right)
			\right]
		\right\}
		\\ &
		  \ = \
		  \left\{
		  \PP
		  \left[
			  X \le \mathfrak{a}\cdot q_{X}\left(
			  1-\frac{k}{n}
			  \right)
			\right]
		\ - \
		  \frac{t}{\sqrt{k}}
		\PP
		  \left[
			  X > q_{X}\left(
			  1 - \frac{k}{n}
			  \right)
			\right]
		\ \le \
		  \widehat{F}_n
		  \left(
		  \mathfrak{a}\cdot X_{n-k:n}
		  \right)
		\right\}.
	\end{align*}
	Using that for any $\tau\in (0,1)$ and $x\in \mathbb{R}$, $\tau\leq \widehat{F}_n(x)$ if and only if $\widehat{q}_n(\tau)\leq x$, where $\widehat{q}_n(\tau)=X_{\lceil n\tau \rceil:n}$ is the empirical quantile function, we get that, for $n$ large enough, the above event is equal to
	\begin{align*}
		 &
		\left\{
		\widehat{q}_n
		\left(
		\PP
		\left[
			X \le \mathfrak{a}\cdot q_{X}\left(
			1-\frac{k}{n}
			\right)
			\right]
		\ - \
		\frac{t}{\sqrt{k}}
		\PP
		\left[
			X > q_{X}\left(
			1 - \frac{k}{n}
			\right)
			\right]
		\right)
		\ \le \
		\mathfrak{a}\cdot \widehat{q}_n
		\left(
		1 -
		\frac{k}{n}
		\right)
		\right\}
		\\ &
		  \ = \
		  \left\{
		  \frac{
			  \widehat{q}_n
			  \left(
			  \PP
			  \left[
				  X \le \mathfrak{a}\cdot q_{X}\left(
				  1-\frac{k}{n}
				  \right)
				\right]
			  \ - \
			  \frac{t}{\sqrt{k}}
			  \PP
			  \left[
				  X > q_{X}\left(
				  1 - \frac{k}{n}
				  \right)
				\right]
			  \right)
		}{q_{X}\left(
			  1-\frac{k}{n}
			  \right)}
		\ - \
		  \mathfrak{a}
		\cdot
		  \frac{
			  \widehat{q}_n
			  \left(
			  1-\frac{k}{n}
			  \right)
		}{
			  q_X
			  \left(
			  1-\frac{k}{n}
			  \right)
		}
		\ \le \ 0
		  \right\}
		  \,.
	\end{align*}
	Write
	\begin{align*}
		\alpha_n(t)
		\  & := \
		\PP
		\left[
			X \le \mathfrak{a}\cdot q_{X}\left(
			1-\dfrac{k}{n}
			\right)
			\right]
		\ - \
		\dfrac{t}{\sqrt{k}}
		\PP
		\left[
			X > q_{X}\left(
			1 - \dfrac{k}{n}
			\right)
			\right]
		\,,
		\\
		\beta_n(t)
		\  & := \
		\frac{n}{k}
		\left(
		1 - \alpha_n(t)
		\right)
		\\
		\  & = \
		\frac{n}{k}
		\left(
		\PP
		\left[
			X > \mathfrak{a}\cdot q_{X}\left(
			1-\dfrac{k}{n}
			\right)
			\right]
		\ + \
		\dfrac{t}{\sqrt{k}}
		\PP
		\left[
			X > q_{X}\left(
			1 - \dfrac{k}{n}
			\right)
			\right]
		\right)
		\,.
	\end{align*}
	Since $k/n = \PP[X>q_{X}(1-\frac{k}{n})]$, it follows from regular variation of $X$ that 
	\begin{align*}
		\beta_n(t)
		\ = \
		\frac{
			\PP[X>\mathfrak{a}\cdot u_n]
		}{
			\PP[X> u_n]
		}
		\ + \
		\frac{t}{\sqrt{k}}
		\ \to \ \mathfrak{a}^{-\nu}\in (0,\infty)
		\qquad\text{as}\ n\to\infty\,.
	\end{align*}
	By Theorem~2.4.8 on p.52 in~\cite{dehaanExtremeValueTheory2006},
	there is a sequence of Brownian motions $(W_n)$ such that, uniformly in $s\in (0,1]$, and for arbitrarily small $\varepsilon>0$,
	\begin{align*}
		\frac{
			\widehat{q}_n
			\left(
			1 - \frac{ks}{n}
			\right)
		}{
			q_X
			\left(
			1 - \frac{k}{n}
			\right)
		}
		\ = \
		s^{-1/\nu}
		\ + \
		\frac{1}{\sqrt{k}}
		\left(
		\frac{1}{\nu}
		s^{-1/\nu-1}
		W_n(s)
		+ \sqrt{k}
		\cdot
		\mathfrak{A}
		\left(
		\frac{n}{k}
		\right)
		s^{-1/\nu}
		\frac{
			s^{-\rho}
			-1
		}{\rho}
		+ s^{-1/\nu-1/2-\varepsilon} o_{\PP}(1)
		\right)
		\,.
	\end{align*}
	An inspection of the proof of this result reveals that it holds in fact uniformly on any interval of the form $(0,s_0]$, with $s_0>0$ fixed.
	Then
	\begin{align}
		\label{eq:ratioquant1}
		\mathfrak{a}
		\cdot
		\frac{
			\widehat{q}_n
			\left(
			1 - \frac{k}{n}
			\right)
		}{
			q_X
			\left(
			1 - \frac{k}{n}
			\right)
		}
		\ = \
		\mathfrak{a}
		\left(
		1
		\ + \
		\frac{1}{\sqrt{k}}
		\frac{1}{\nu}
		W_n(1)
		\ + \
		o_\PP(1/\sqrt{k})
		\right)
		\ = \
		\mathfrak{a}
		\ + \
		\frac{1}{\nu}
		\frac{\mathfrak{a}}{\sqrt{k}}
		W_n(1)
		\ + \ o_\PP(1/\sqrt{k})
	\end{align}
	and, by $\alpha_n(t)=1-(k/n)\beta_n(t)$ with $\beta_n(t)\to \mathfrak{a}^{-\nu}\in (0,\infty)$, and continuity properties of the Brownian motion,
	\begin{align*}
		 &
		\frac{
			\widehat{q}_n
			\left(
			\alpha_n(t)
			\right)
		}{
			q_X
			\left(
			1 - \frac{k}{n}
			\right)
		}
		\\ &
		  \ = \
		  \left(
		  \beta_n(t)
		\right)^{-1/\nu}
		\ + \
		  \frac{1}{\sqrt{k}}
		\left(
		  \frac{1}{\nu}
		\left(
		  \beta_n(t)
		\right)^{-1/\nu-1}
		W_n(\beta_n(t))
		  + \sqrt{k}
		\cdot
		  \mathfrak{A}
		\left(
		  \frac{n}{k}
		  \right)
		\left(
		  \beta_n(t)
		\right)^{-1/\nu}
		\frac{
			  \left(
			  \beta_n(t)
			\right)^{-\rho}
			-1
			  }{\rho}
		+ o_{\PP}(1)
		  \right)
		\\ &
		  \ = \
		  \left(
		  \beta_n(t)
		\right)^{-1/\nu}
		\ + \
		  \frac{1}{\sqrt{k}}
		\left(
		  \frac{1}{\nu}
		\mathfrak{a}^{\nu + 1}
		W_n(\mathfrak{a}^{-\nu})
		  + \sqrt{k}
		\cdot
		  \mathfrak{A}
		\left(
		  \frac{n}{k}
		  \right)
		\mathfrak{a}
		\frac{
			  \mathfrak{a}^{\rho\nu}
			-1
			  }{\rho}
		+ o_{\PP}(1)
		  \right)
		  \,.
	\end{align*}
	Besides,
	\begin{align*}
		\beta_n(t)
		 &
		\ = \
		\frac{
			\PP[X>\mathfrak{a}\cdot u_n]
		}{
			\PP[X> u_n]
		}
		\ + \
		\frac{t}{\sqrt{k}}
		\\ &
		  \ = \
		  \mathfrak{A}
		\left(
		  \frac{n}{k}
		  \right)
		\left(
		  \frac{
			  \dfrac{\PP[X>\mathfrak{a}\cdot u_n]}{\PP[X>u_n]}
			  \ - \ \mathfrak{a}^{-\nu}
		}{\mathfrak{A}(n/k)}
		  \ - \
		  \mathfrak{a}^{-\nu}
		\frac{\mathfrak{a}^{\rho\nu}- 1}{\rho/\nu}
		  \right)
		\ + \
		  \mathfrak{a}^{-\nu}
		\ + \
		  \mathfrak{A}
		\left(
		  \frac{n}{k}
		  \right)
		\mathfrak{a}^{-\nu}
		\frac{\mathfrak{a}^{\rho\nu}- 1}{\rho/\nu}
		\ + \ \frac{t}{\sqrt{k}}
		\\ &
		  \ = \
		  \mathfrak{a}^{-\nu}
		\ + \
		  \mathfrak{A}
		\left(
		  \frac{n}{k}
		  \right)
		\mathfrak{a}^{-\nu}
		\frac{\mathfrak{a}^{\rho\nu}- 1}{\rho/\nu}
		\ + \ \frac{t}{\sqrt{k}}
		\ + \ o(1/\sqrt{k})
		\\ &
		  \ = \
		  \mathfrak{a}^{-\nu}
		\left(
		  1
		  \ + \
		  \frac{1}{\sqrt{k}}
		\left(
		  \sqrt{k}
		\cdot
		  \mathfrak{A}
		\left(
		  \frac{n}{k}
		  \right)
		\frac{\mathfrak{a}^{\rho\nu}- 1}{\rho/\nu}
		\ + \
		  \mathfrak{a}^{\nu}
		t
		  \ + \ o(1)
		  \right)
		  \right)
		\\ &
		  \ = \
		  \left(
		  \mathfrak{a}
		\left(
		  1
		  \ + \
		  \frac{1}{\sqrt{k}}
		\left(
		  \sqrt{k}
		\cdot
		  \mathfrak{A}
		\left(
		  \frac{n}{k}
		  \right)
		\frac{\mathfrak{a}^{\rho\nu}- 1}{\rho/\nu}
		\ + \
		  \mathfrak{a}^{\nu}
		t
		  \ + \ o(1)
		  \right)
		  \right)^{-1/\nu}
		\right)^{-\nu}
		\\ &
		  \ = \
		  \left(
		  \mathfrak{a}
		\left(
		  1
		  \ - \
		  \frac{1}{\nu}
		\frac{1}{\sqrt{k}}
		\left(
		  \sqrt{k}
		\cdot
		  \mathfrak{A}
		\left(
		  \frac{n}{k}
		  \right)
		\frac{\mathfrak{a}^{\rho\nu}- 1}{\rho/\nu}
		\ + \
		  \mathfrak{a}^{\nu}
		t
		  \ + \ o(1)
		  \right)
		  \right)
		  \right)^{-\nu}
		\,,
	\end{align*}
	so that
	\begin{align*}
		\left(
		\beta_n(t)
		\right)^{-1/\nu}
		\ = \
		\mathfrak{a}
		\left(
		1
		\ - \
		\frac{1}{\sqrt{k}}
		\left(
		\sqrt{k}
		\cdot
		\mathfrak{A}
		\left(
		\frac{n}{k}
		\right)
		\frac{\mathfrak{a}^{\rho\nu}- 1}{\rho}
		\ + \
		\frac{1}{\nu}
        \mathfrak{a}^{\nu}
		t
		\ + \ o(1)
		\right)
		\right)
		\,.
	\end{align*}
	Therefore
	\begin{align}
		\label{eq:ratioquant2}
		\frac{
			\widehat{q}_n
			\left(
			\alpha_n(t)
			\right)
		}{
			q_X
			\left(
			1 - \frac{k}{n}
			\right)
		}
		\ = \
		\mathfrak{a}
		\ + \
		\frac{1}{\nu}
		\frac{\mathfrak{a}^{\nu+1}}{\sqrt{k}}
		W_n(\mathfrak{a}^{-\nu})
		\ - \
		\frac{1}{\nu}
		\mathfrak{a}^{\nu+1}
		\frac{t}{\sqrt{k}}
		\ + \
		o_{\PP}(1/\sqrt{k})
		\,.
	\end{align}
	Combining~\eqref{eq:ratioquant1} with~\eqref{eq:ratioquant2}, we get
	\begin{align*}
		 &
		\frac{
			\widehat{q}_n
			\left(
			\alpha_n(t)
			\right)
		}{
			q_X
			\left(
			1 - \frac{k}{n}
			\right)
		}
		\ - \
		\mathfrak{a}
		\frac{
			\widehat{q}_n
			\left(
			1 - \frac{k}{n}
			\right)
		}{
			q_X
			\left(
			1 - \frac{k}{n}
			\right)
		}
		\ = \
		\mathfrak{a}
		\frac{1}{\nu}
		\frac{1}{\sqrt{k}}
		\left(
		\mathfrak{a}^{\nu}
		W_n(\mathfrak{a}^{-\nu})
		\ - \ W_n(1)
		\ - \ \mathfrak{a}^{\nu}t
		\ + \ o_{\PP}(1)
		\right)
		\,.
	\end{align*}
	Therefore
	\begin{align*}
		 &
		\PP
		\left[
			\sqrt{k}
			\left(
			\frac{1}{n}
			\sum_{i=1}^n
			\left(
			\frac{\ind\{X_i>\mathfrak{a}\cdot X_{n-k:n}\}}{\PP[X>u_n]}
			\ - \
			\frac{\PP[X>\mathfrak{a}\cdot u_n]}{\PP[X > u_n]}
			\right)
			\right)
			\ \le \ t
			\right]
		\\ &
		  \ = \
		  \PP
		  \left[
			  \frac{
				  \widehat{q}_n
				  \left(
				  \alpha_n(t)
				\right)
				  }{
				  q_X
				  \left(
				  1 - \frac{k}{n}
				  \right)
			}
			\ - \
			  \mathfrak{a}
			\frac{
				  \widehat{q}_n
				  \left(
				  1 - \frac{k}{n}
				  \right)
			}{
				  q_X
				  \left(
				  1 - \frac{k}{n}
				  \right)
			}
			\ \le \
			  0
			  \right]
		\\ &
		  \ = \
		  \PP
		  \left[
			  \mathfrak{a}^{\nu}
			W_n(\mathfrak{a}^{-\nu})
			  \ - \ W_n(1)
			  \ + \ o_{\PP}(1)
			  \ \le \
			  \mathfrak{a}^{\nu}t
			  \right]
		\\ &
		  \ \to\
		  \PP
		  \left[
			  W(\mathfrak{a}^{-\nu})
			  \ - \
			  \mathfrak{a}^{-\nu}
			W(1)
			  \ \le \
			  t
			  \right]
            \qquad\text{as}\ n\to\infty\,,
	\end{align*}
	where $(W(t))_{t\ge 0}$ is a Brownian motion.
	We have that
	\begin{align*}
		 &
		\Var
		\left[
			W(\mathfrak{a}^{-\nu})
			\ - \
			\mathfrak{a}^{-\nu}
			W(1)
			\right]
      \\&
		\ = \
    \Var[W(\mathfrak{a}^{-\nu})]
    \ + \
    \mathfrak{a}^{-2\nu}
    \Var[W(1)]
    \ - \ 
    2
    \mathfrak{a}^{-\nu}
    \Cov[W(\mathfrak{a}^{-\nu}),W(1)]
    \\&
    \ = \ 
		\mathfrak{a}^{-\nu}
		\ + \
		\mathfrak{a}^{-2\nu}
		\ - \ 2
		\mathfrak{a}^{-\nu}
		\left(
		\mathfrak{a}^{-\nu}
		\land 1
		\right)
		\ = \
		\mathfrak{a}^{-\nu}
		\left(
		1 \ + \ \mathfrak{a}^{-\nu}
		- 2 (\mathfrak{a}^{-\nu}\land 1)
		\right)
		\\ &
		  \ = \
		  \mathfrak{a}^{-\nu}
		\left|
		  \frac{1}{\mathfrak{a}^{\nu}}
		\ - \ 1
		  \right|
		  \,,
	\end{align*}
	which proves the claimed limiting distribution.
\end{proof}

\end{document}